\documentclass[12pt]{article}
\usepackage[utf8]{inputenc}
\usepackage{setspace}
\usepackage{johd}
\usepackage[margin=1.25in]{geometry}
\usepackage{graphicx}
\graphicspath{ {./figures/} }
\usepackage{subcaption}
\usepackage{amsmath}
\usepackage{lineno}
\usepackage{hyperref}
\usepackage{algorithm}
\usepackage{algpseudocode}
\usepackage{amssymb}
\usepackage{amsthm}
\usepackage{comment}
\usepackage{mathrsfs}
\usepackage[mathcal]{eucal}
\renewcommand{\phi}{\varphi}
\usepackage{enumerate}
\usepackage{bbold}
\usepackage{booktabs}

\usepackage{romannum}
\usepackage{mathtools}
\newtheorem{theorem}{Theorem}[section]
\newtheorem{lemma}[theorem]{Lemma}
\newtheorem{assumption}{Assumption}
\newtheorem{prop}{Proposition}

\usepackage{tikz}
\usetikzlibrary{arrows.meta, positioning, calc}



\usepackage{stackengine,scalerel}

\usepackage{xcolor}

\title{Long Memory in Intrinsically Dynamic Factor Models}

\author{Qin Wen, Clifford M. Hurvich \\
        \small Stern School of Business, New York University \\\\
        \small $^{*}$Corresponding author: Qin Wen; \tt{qw936@stern.nyu.edu} \\
}
\date{}

\begin{document}
\pagenumbering{arabic}
\maketitle
\begin{abstract} 
\noindent 
We study the generalized dynamic factor model in a long-memory setting. Unlike most recent work, which assumes a finite-dimensional factor space and short memory, our framework allows the factor space to be infinite-dimensional and the common components to exhibit long memory. We employ the two-sided estimation method of \cite{dynamicfactorForni2sided} to recover the common component. The long memory structure of the common component poses a challenge, as it introduces unboundedness/discontinuity in the spectral density. We address this issue by leveraging two key facts: First, the estimated operator is a projection onto the leading eigenspace and thus the eigengap provides an intrinsic scaling that partially mitigates the blow-up. Second, we perform most of our estimation in $L^p$-norm, rather than pointwise. Experimental results are presented to provide evidence supporting the theory, as well as potential improvements to it.

\end{abstract}



\section{Introduction}

Large collections of observed variables often exhibit an underlying structure. Rather than treating the observed data as purely noisy or unorganized, one assumes that they can be decomposed into a systematic component and an idiosyncratic or noise component. When the systematic component can be represented through a linear structure, the resulting model is often both interpretable and useful for explaining patterns in the observed data. Factor models impose a further form of structure: the systematic component across a large collection of variables is driven by a small number of unobserved common sources. When the data are observed over time, as we assume here, these common factors and their effects may also vary over time.
Factor models have been widely used in economics and finance to explain the behavior of macroeconomic variables, financial assets, and other high-dimensional time series.


Dynamic factor models have been a recent focus of research nowadays. Unlike static factor models, dynamic factor models have a richer lead and lag structure, allowing the underlying factors to be time-varying, which makes it more flexible and better suited to capture the dynamics of real-world systems. It is worth noting that some dynamic factor models, by stacking the factors and their lags (e.g., \cite{Stock01122002}), can be written in a static form. This case is often referred to as a model with a finite-dimensional factor space in the existing literature (e.g., \cite{FORNI2015_dfm_infinite_method, FORNIinfinitedim}), and thus can be handled using the usual principal component analysis after the transformation.

However, models that do not have such a static representation require other methods to estimate the structure within. The generalized dynamic factor model (GDFM), proposed by \cite{dynamicfactorForni2sided}, gives flexibility to the factor filters and thus allows for a potentially infinite-dimensional factor space. Because of this, in such an intrinsically dynamic factor model it is not always possible to stack the data to obtain a static representation.

The existing literature provides two ways to explore the structure of a dynamic factor model without a static representation. \cite{dynamicfactorForni2sided} was the first to address this problem, whose framework largely depends on dynamic principal components, see \cite{brillinger}. This is also the approach we will be investigating in this paper. Another approach, \cite{FORNI2015_dfm_infinite_method} and \cite{FORNIinfinitedim}, imposes a stronger structure on the underlying filters of the factors: the operators are assumed to be rational, and after applying a one-sided filter, the filtered process has a static factor structure. This model is parametric over the common component part and nonparametric over the idiosyncratic error, whereas the approach we will investigate in this paper, \cite{dynamicfactorForni2sided}, admits a fully nonparametric structure.

In real-world applications, dynamic factor models are often used to study the latent structure underlying indicators such as GDP growth or stock returns. In many cases, these indicators exhibit long memory, meaning that shocks can have persistent effects and the dependence between observations decays slowly, as a power law, over time. For a more detailed discussion of long memory in factor models, we refer the reader to \cite{Cheunglongmem}.
\cite{Cheunglongmem} studies factor estimation for a long-memory factor model, but its scope is limited to static settings. \cite{ERGEMENstackableDFM} extends long memory factor estimation to dynamic cases, but only for those with static representations. To the best of our knowledge, there is no existing work on factor estimation for long-memory GDFMs, without (or even with) a parametric factor structure. Our paper is the first to address this problem in an intrinsically infinite-dimensional factor space.

For a long-memory factor model, the infinite-dimensional nature of the factor space becomes even more consequential. The long-memory component inherently contributes to an infinite-dimensional space of factors and allows for more variation in how the long-memory part manifests. For instance, all existing long-memory factor literature relies on the static representation nature of the model, and that only permits the long memory $d$ to vary across factors, but not across rows. Additional discussion and examples are provided in Section \ref{sec:model}. In any case, the distinction between the finite and infinite-dimensional dynamic factor space is essential and can lead to the failure of the usual PCA method, as used in \cite{Cheunglongmem}, \cite{BARIGOZZI2021455}, and \cite{ERGEMENstackableDFM}. Tables \ref{tab:time_domain_pca_spca_same_d}, and \ref{tab:dynamic_pca_truncation_same_d} support this viewpoint.

This paper provides the first theoretical guarantees for the dynamic principal components approach of \cite{dynamicfactorForni2sided} in a long-memory setting. The extension from short to long memory is not trivial. Whereas the existing literature focuses only on the short-memory factor, the long-memory structure of the common component fundamentally alters the frequency-domain behavior by introducing low-frequency singularity and weakening regularity, so the relevant operators may only be controlled in an $L^1$ space. This substantially changes the analysis underlying factor estimation. 




These features create at least two major obstacles when proposing an approach based on dynamic PCA in \cite{dynamicfactorForni2sided} and establishing theoretical guarantees. The first obstacle concerns spectral density estimation. Discrete periodogram smoothing, the standard nonparametric approach for estimating the spectral density, can introduce bias at each fixed frequency because the spectral density diverges at zero frequency and thus is not smooth there. In this paper, we continue to work with this estimator for three reasons. First, the structure of the factor model, comprising a common component plus noise, may degrade the regularity of the short-memory component of the spectral density, precluding efficient estimation of the memory parameters, especially when the long-memory parameter $d$ is heterogeneous. This makes it difficult to apply prewhitening techniques. Second, under Assumption \ref{assump: spec_den_specific_form} below, we can infer that the eigengap between the first $q$ eigenvalues and the rest is of order $\mathcal{O}(n|\theta|^{-2\Tilde{d}}), \Tilde{d} = \min_j d_j$. Again, we would like to emphasize that Assumption~\ref{assump: spec_den_specific_form}, while important for our analysis, is still much milder than what is assumed in the existing literature. Since the estimated operator is a projection onto the leading eigenspace, by applying perturbation theory to the projection operator at a fixed frequency, we show that the eigengap provides an intrinsic scaling (a form of “prewhitening”) that partially mitigates the blow-up, see (\ref{eq:perturbation_prewhiten_scale}). Third, we are primarily concerned with the integrability of the spectral density estimator rather than its pointwise behavior. We prove that discrete periodogram smoothing estimator provides $L^1$-consistent estimates of the spectral density, even though it may not perform well pointwise for each frequency.

Another obstacle we encountered is the data availability (or “edge”) effect. The population construction in \cite{dynamicfactorForni2sided} is infeasible as it requires data points at all times $t \in \mathbb{Z}$, while in practice, we will have data at only finitely many time points. \cite{dynamicfactorForni2sided} deals with this infeasibility by truncation of magnitude $M(T)$, which is often used as a standard remedy. While truncation might not cause too much loss when factors are sufficiently regular (such as in the short-memory case), it can have a more significant impact when the series exhibit long-range dependence across lags and leads. From a mathematical perspective, the regularity of the eigenvector function is difficult to infer directly due to the blow-up behavior of the spectral density near zero, and therefore the Fourier remainder requires more delicate control.

To the best of our knowledge, within the existing literature on GDFMs, this is the first work to provide a rigorous consistency rate that includes the
truncation error for the method of \cite{dynamicfactorForni2sided}, even in the benchmark short-memory case $d=0$. While \cite{FORNI2004231} developed
theoretical guarantees for \cite{dynamicfactorForni2sided} in a short-memory
setting, their analysis does not explicitly account for the truncation (edge)
term, which can lead to seemingly sharper but incomplete rates. For the homogeneous long-memory parameter scenario, the regularity is the
strongest, which allows us to obtain a sharp result. For the scenario where the long-memory parameters differ only across factors, as stated in Assumption \ref{assump: spec_den_specific_form} case (a), we are able to establish H\"older continuity and a cusp-like behavior near zero. This implies that the norm of the distributed lead and lag coefficients $\left\|\mathbf{K}_{n i, h}\right\|_2$ of the projection operator $\underline{\mathbf{K}}_{ni}(\theta)$ decay faster than $|h|^{-1}$. To the best of our knowledge, our analysis is the first result in the existing GDFM literature to study the regularity of the eigenprojection under memory-parameter heterogeneity across factors. 

In the other scenario, where the long-memory parameters differ only across rows, as stated in Assumption \ref{assump: spec_den_specific_form} case (b), our proof establishes a preliminary bound that applies under certain conditions, specifically when $4\Delta - 2d +1 < 0, \Delta = \max_j d_j - \min_j d_j, d = \max_j d_j$. We acknowledge that the edge term in the heterogeneous long-memory case is the most challenging part to control. Our theoretical bound, the resulting sufficient conditions, and consequently the convergence rate, although the first in the literature to accommodate memory parameter heterogeneity across rows, a feature not covered by models with static representations, are conservative relative to what we observe in simulations: the results in Section \ref{sec:simulation} suggest that $\left\|\mathbf{K}_{n i, h}\right\|_2$ likely exhibits the same order of decay as in Assumption \ref{assump: spec_den_specific_form} case (a).

The rest of the paper is organized as follows. Section \ref{sec:model}
introduces the model setup and assumptions, and discusses the difference between models with or without static representation. Section \ref{sec:DPCA_methodology} reviews the
dynamic PCA approach of \cite{dynamicfactorForni2sided}. Section
\ref{sec:main_results} presents the main theoretical results on estimator
consistency. Section \ref{sec:simulation} reports the simulation results. We conclude in Section \ref{sec:conclusion}. All proofs are collected in the Appendix \ref{sec:assump_proof}, \ref{sec:proof_main_prop}, \ref{sec:proof_lemma_for_main_prop}, \ref{sec:proof_lemma_for_main_lemma_1}, \ref{sec:proof_lemma_for_lemma_2}.




\section{The Model}\label{sec:model}
In this paper, we study a doubly infinite sequence $\{x_{it}, i \in \mathbb{N}, t \in \mathbb{Z}\}$, where 
\begin{align}\label{eq:model0}
    x_{it} &= \chi_{it} + \xi_{it} ,\\
    \chi_{it} &= b_{i1}(L)u_{1t} + b_{i2}(L)u_{2t} + ... + b_{iq}(L)u_{qt} 
\end{align}

Here, $L$ denotes the lag operator, and $b_{i1}(L),\ldots, b_{iq}(L)$ are filters of potentially infinite order, not necessarily one-sided. The variables $u_{lt}, l = 1,\ldots,q$, are referred to as common shocks, while $\chi_{it}$ and $\xi_{it}$ denote the common component and the idiosyncratic component of $x_{it}$, respectively. In practice, only a finite amount of data may be observed, with $1 \leq i \leq n$ and $1 \leq t \leq T$. Our objective is to estimate $\chi_{it}$ using the observed sequence $\{x_{it}, 1 \le i \le n,, 1 \le t \le T\}$. 

\subsection{Stackable v.s.\ Non‑Stackable}
As discussed in the Introduction, there are two types of dynamic factor models. One can be reformulated as a static model by stacking the factors into a finite-dimensional factor vector; we call such models stackable. Another type does not admit such a representation, and we call it non-stackable. In \cite{FORNI2015_dfm_infinite_method} and \cite{FORNIinfinitedim}, the stackable/non-stackable distinction is framed in terms of whether the (static) factor space is finite-dimensional. However, the restriction of having a static representation, i.e., in a finite-dimensional factor space, which most papers focus on, can be quite limiting. As noted in \cite{FORNI2015_dfm_infinite_method}, even a simple specification such as
\begin{align}\label{eq:basic_ar_factor}
    x_{it} = \frac{1}{1-\alpha_i L} u_t + \xi_{it}
\end{align}
falls outside the class with a static representation. In this setup, for any fixed $t$, the collection of common components $\{\chi_{it}\}_{i\in\mathbb{N}}$ spans an  infinite-dimensional (static) factor space unless the coefficients $\{\alpha_i\}$ take only finitely many distinct values.

More specifically, the factor model dynamics in \eqref{eq:model0} can be written in the form
\begin{align}\label{eq:DFM_general_form}
  x_{it} &= \Phi_i(L)\,f_t  +  \xi_{it},   &&\text{(cross–sectional structure)}\\
  \Psi(L)\,f_t &= u_t,   &&\text{(factor dynamics)}
\end{align}
where $f_t$ denotes the factor process, $\Phi_i(L)$ may depend on $i$ but not on $t$, and $\Psi(L)$ depends only on the factors but not on $i$.

Such a model is stackable if one of the following holds:
\begin{enumerate}
  \item $\displaystyle \Phi_i(L)$ depends on $i$, and each entry is a finite‐degree polynomial in~$L$ of degree at most $K_{\text{degree}}$, where $K_{\text{degree}}$ does not depend on $i$.
  \item There are only finitely many distinct filters $\Phi_i(L)$. That is, there exist filters $\Phi^{(1)}(L),\dots,\Phi^{(M)}(L)$
  such that $\Phi_i(L)\in\{\Phi^{(1)}(L),\dots,\Phi^{(M)}(L)\}$ for all $i$. In this case, $\displaystyle \Phi_i(L)$ can be the same for all $i$, $M=1$, i.e.\ $\Phi_i(L)\equiv\Phi(L)$.
\end{enumerate}

A sufficient condition for a model to be non-stackable is that $\displaystyle \Phi_i(L)$ depends on $i$, the family $\{\Phi_i(L)\}_{i\in\mathbb{N}}$ contains infinitely many distinct filters, and each entry is an infinite-degree polynomial in~$L$.  

This provides a concrete perspective on what finite- and infinite-dimensional factor spaces actually mean. The basic AR factor model \eqref{eq:basic_ar_factor} does not have a static representation, because it violates the finite-degree and finitely-many-filters restrictions on $\Phi_i(L)$. The infinite lag length makes it impossible to apply time domain PCA, and thus models exhibiting long memory will further undermine the low-dimensional static factor structure.  

\cite{Cheunglongmem} and \cite{ERGEMENstackableDFM} both address factor models with long memory, but their frameworks rely on the assumption of finite-dimensionality in a static factor space mentioned above. More precisely, \cite{Cheunglongmem} studies static factor models, while \cite{ERGEMENstackableDFM} considers dynamic factor models with a static representation; both assume that memory parameters are shared across all rows. In contrast, our model not only allows different memory parameters across factors, but also permits each row $i$ to have its own long-memory parameter, attributing persistence to the loadings rather than solely to the factors.

Last but not least, we mention a point put forward by \cite{FORNI2015_dfm_infinite_method}, showing that stacking then ordinary PCA is a fragile way of estimating factor structures. In their Example~(1.6), 
\(x_{1t}=u_t+a u_{t-1}+\xi_{1t}\), and \(x_{it}=u_t+\xi_{it}\) for \(i>1\).
This is an example where the model does have a static representation by stacking \((u_t,u_{t-1})\); however, it still cannot be estimated correctly by the usual static (time-domain) PCA. In the stackable factor model, stacking the factors increases the dimension of factor space and thus requires pervasiveness for a larger $q$, for example, in this example, $q = 2$. The issue is that the lagged term \(u_{t-1}\) is non-pervasive, loading only on the first series, and therefore does not generate a second static diverging eigenvalue (in the covariance matrix); as a result, it is absorbed into the idiosyncratic component.
Consequently, the impulse response of \(x_{1t}\) to the common shock \(u_t\), namely \(1+aL\), cannot be recovered by the static PCA approach.
In contrast, dynamic (frequency-domain) PCA treats lags as part of the transfer function, rather than as additional stacked factors. Thus, it does not increase the dimension and only requires pervasiveness for $q=1$ in the spectral density.
Thus, the frequency-domain dynamic PCA method correctly recovers the common component in this example.

The following examples illustrate which class each dynamic factor model falls into:
\begin{enumerate}
  \item $\Phi_i(L) = \bigl(1-\alpha_{1i}L,\;1-\alpha_{2i}L\bigr),\quad
      |\alpha_{1i}|,|\alpha_{2i}|<1,\quad
      \Psi(L)=1-0.8L,$
    then one may define
    \[
      \widetilde f_t =
       \bigl(f_{1t},f_{1,t-1},f_{2t},f_{2,t-1}\bigr)',
       \quad
      \widetilde\gamma_i = (1,\alpha_{1i},1,\alpha_{2i})',
    \]
    so that $x_{it}=\widetilde\gamma_i'\,\widetilde f_t$ is a static representation of the model.
  \item $\Phi_i(L)=\frac1{1-\alpha_iL},\;|\alpha_i|<1,$ with $\Psi(L)=1-0.8L$, and the $\alpha_i$'s take infinitely many values. Each entry has infinite lag length and the model is non-stackable.
  \item $\Phi_i(L)=(1-L)^{-d},\ d<\tfrac12,$ then one can absorb the fractional filter into $\Psi(L)$ and obtain a static representation.
  \item If $\Phi_i(L)=(1-L)^{-d_i},\ d_i<\tfrac12,$ with $d_i$ taking infinitely many values, the model is non-stackable. 
  \item $\Phi_i(L)=\frac{(1-L)^{-d}}{1-\alpha_iL},\;|\alpha_i|<1$. Even though in this case one can again absorb the fractional filter into the factor dynamics, since $\frac{1}{1-\alpha_iL}$ has infinite degree, this model still does not have a static representation.
  \item $\Phi_i(L) = \left[(1-L)^{-d_{il}} \frac{\Theta_{il}(L)}{\Phi_{il}(L)}\right]_{l = 1,\ldots,q}$; whether $d_{il}$ varies in both indices or only across rows ($d_i$) or factors ($d_l$), these models are all non-stackable.
\end{enumerate}


The distinction between the finite and infinite-dimensional
dynamic factor space is essential and can lead to the failure of the usual PCA method. Consider a simple one-factor model 
\begin{align*}
    x_{it} &= \chi_{it} + \xi_{it}, \\
    \chi_{it} &= \frac{(1-L)^{-d}}{1-\alpha_i L}u_t
    \quad |\alpha_i|<1, \quad d = 0.4
\end{align*}

We compare the performance of the ordinary PCA and the dynamic PCA from \cite{dynamicfactorForni2sided} in estimating the common component $\chi_{it}$, and use the global criterion
\begin{align}\label{eq:global_criterion}
R(\hat{\chi},\chi)
=
\frac{\sum_{i,t}\bigl(\hat{\chi}_{it}-\chi_{it}\bigr)^2}{\sum_{i,t}\chi_{it}^2}
\end{align}
Table~\ref{tab:time_domain_pca_spca_same_d}, \ref{tab:dynamic_pca_truncation_same_d} report the average and the standard deviation (in brackets) of this statistic across the experiments.

\begin{table}[htbp]
\centering
\caption{Time domain PCA, \cite{BARIGOZZI2021455}, \cite{Cheunglongmem}, \cite{ERGEMENstackableDFM}}
\label{tab:time_domain_pca_spca_same_d}
\begin{tabular}{lcccccc}
\toprule
$n\backslash T$ & 20 & 50 & 100 & 200 & 500 & 800 \\
\midrule
20  & 0.515(0.184) & 0.669(0.184) & 0.683(0.158) & 0.775(0.120) & 0.734(0.135) & 0.736(0.150) \\
50  & 0.589(0.143) & 0.688(0.122) & 0.744(0.087) & 0.788(0.109) & 0.676(0.118) & 0.760(0.118) \\
100 & 0.590(0.165) & 0.630(0.118) & 0.759(0.075) & 0.744(0.071) & 0.800(0.059) & 0.782(0.077) \\
200 & 0.562(0.170) & 0.779(0.065) & 0.765(0.077) & 0.801(0.075) & 0.773(0.029) & 0.832(0.034) \\
\bottomrule
\end{tabular}
\end{table}

\begin{table}[htbp]
\centering
\caption{Dynamic(frequency domain) PCA from \cite{dynamicfactorForni2sided}}
\label{tab:dynamic_pca_truncation_same_d}
\begin{tabular}{lcccccc}
\toprule
$n\backslash T$ & 20 & 50 & 100 & 200 & 500 & 800 \\
\midrule
20  & 0.124(0.067) & 0.129(0.079) & 0.062(0.018) & 0.040(0.017) & 0.043(0.019) & 0.026(0.010) \\
50  & 0.211(0.095) & 0.089(0.022) & 0.064(0.026) & 0.040(0.013) & 0.035(0.014) & 0.028(0.014) \\
100 & 0.223(0.108) & 0.113(0.034) & 0.067(0.008) & 0.050(0.018) & 0.041(0.012) & 0.032(0.007) \\
200 & 0.243(0.106) & 0.144(0.067) & 0.068(0.026) & 0.053(0.013) & 0.045(0.008) & 0.031(0.010) \\
\bottomrule
\end{tabular}
\end{table}

It is clear that time domain PCA does not provide any meaningful estimation of the common component, while frequency domain dynamic PCA performs well even when $d$ is relatively large (e.g., $d=0.4$).


\subsection{Assumptions}

Throughout the paper, $\Pi = (-\pi, \pi]$.

\begin{assumption}\label{assump:model}
(Model assumption)
    \begin{enumerate}[(i)]
    \item $(u_{1t}, u_{2t}, ..., u_{qt})'$ is orthonormal white noise with with finite fourth-order moments.
    \item $\xi_n = \{(\xi_{1t}, \xi_{2t}, ..., \xi_{nt})', t \in \mathbb{Z}\}$  is a zero-mean stationary vector process for each $n$, and $\xi_{it}$ is independent of $ u_{l, t-k}$ for any $i, l, t$, and $k$. Furthermore, the idiosyncratic component admits the representation
\begin{align}\label{eq:xi_representation}
    \xi_t = B_\xi(L)\epsilon_t
= \sum_{m \in \mathbb{Z}}B_{\xi,m}\epsilon_{t-m},
\end{align}
where \(\epsilon_t\) is an \(n\)-dimensional second-order white noise process.
    \item $b_{il}(L)$s are filters of potentially infinite order, not necessarily one-sided. The coefficients are square summable for each $i,l$.  
\end{enumerate}
\end{assumption}

\begin{assumption}\label{assump:semiparametric_long_memory}
    The common component $\chi_{it}$ may exhibit long memory characteristics, meaning that the transfer function $b_{il}(L)$ in (\ref{eq:model0}) admits the decomposition
\[
b_{il}(L) = (1-L)^{-d_{il}} g_{il}(L),
\]
where $g_{il}(L)$ represents the remaining short-memory component. The memory parameters satisfy $d_{il}\in [0,0.5)$.
\end{assumption}

\begin{assumption}\label{assump:pervasive_short_mem} Define an $n\times q$ matrix $G(\theta) = \left[g_{il}(e^{-\iota \theta})\right]_{i=1,...,n;l = 1,...,q}$, where $g_{il}(L)$ is the remaining short memory component defined in Assumption \ref{assump:semiparametric_long_memory}.
Let $S(\theta) = \frac{1}{n}G^*(\theta)G(\theta)$.
    There exist constants $0 < s^{-} \le s^{+} < \infty$ such that the eigenvalues of $S(\theta)$ satisfies $s^{-}\leq \lambda_{q}^S(\theta)\leq \cdots\leq\lambda_{1}^S(\theta) \leq s^{+}$ for all $\theta \in \Pi$. 

    Equivalently, 
    \begin{align}\label{eq:pervasive_S_GG}
        s^- \mathbf{I}_{q\times q} \preceq S(\theta) \preceq s^+ \mathbf{I}_{q\times q}
    \end{align}

    Moreover, we assume the regularity of $G(\theta)$, $G(\theta)$ is second differentiable and thus Lipschitz continuous: for all $\theta_1, \theta_2 \in \Pi$,
    \begin{align}\label{eq:lipschitz_short_memory}
    &\left\|G^{(r)}(\theta)\right\|_2 \leq C_{G,r}\sqrt{n}, \quad r = 1, 2\\
        &\left\|G(\theta_2) - G(\theta_1)\right\| \leq \sqrt{n}L_G\left|\theta_2 - \theta_1  \right|.
    \end{align}
\end{assumption}

We briefly illustrate that Assumption~\ref{assump:semiparametric_long_memory}, the Lipschitz condition~(\ref{eq:lipschitz_short_memory}), and the upper bound in~(\ref{eq:pervasive_S_GG}) are satisfied by a standard ARFIMA-type specification under the usual stability/invertibility conditions. 
In particular, consider the model below 
\begin{align}
    x_{it} &= \chi_{it} + \xi_{it},\\
    \chi_{it} &= b_{i1}(L)u_{1t} + b_{i2}(L)u_{2t} + \cdots + b_{iq}(L)u_{qt},\\
    b_{il}(L) &= (1-L)^{-d_{il}} \frac{\Theta_{il}(L)}{\Phi_{il}(L)}, \qquad l=1,\dots,q,
\end{align}
where
\[
\Theta_{il}(z) = \sum_{k=0}^{p_{il}} \theta_{il,k} z^k,
\qquad
\Phi_{il}(z) = 1 - \sum_{k=1}^{r_{il}}\phi_{il,k}z^k.
\]
Assume the usual ARMA regularity condition that $\Phi_{il}(z)$ has no zeros on $|z|\leq 1$, so that it is uniformly bounded away from $0$ on the unit circle; or equivalently, there exists $\delta>0$ such that $\inf_{\theta\in\Pi}\big|\Phi_{il}(e^{-\iota\theta})\big| \ge \delta$ for all $i,l$.

Then $g_{il}(L) = {\Theta_{il}(L)}/{\Phi_{il}(L)}$. Furthermore, if 
\[
\max_{i,l}\max_{k=0,...,p_{il}} |\theta_{il,k}| \leq C_{\Theta}, \max_{il}p_{il} \leq p
\] where $p$ and $C_{\Theta}$ are independent of $n, T$. This is also standard in an ARFIMA setting. Thus, $\max_{i,l,\theta} |g_{il}(e^{-\iota \theta})| \leq pC_{\Theta}/\delta \leq M_g$, where $M_g$ is independent of $n, T$. Consequently, $\max_{\theta}\left\|S(\theta)\right\|_2 \leq \frac{1}{n}\max_{\theta}\left\|G(\theta)\right\|_2^2 \leq \frac{1}{n}\max_{\theta}\left\|G(\theta)\right\|_F^2 \leq q \max_{i,l,\theta}|g_{il}(e^{-\iota \theta})|^2 \leq qM_g^2$.

Moreover, under the additional standard uniform coefficient bounds
\[
\sup_{i,l}\sum_k k\,|\theta_{il,k}| < L_{\Theta} < \infty,
\qquad
\sup_{i,l}\sum_k k\,|\phi_{il,k}| < L_{\Phi} < \infty,
\]
it follows that the derivatives of the transfer functions are uniformly controlled. In particular, there exists a constant $L_g<\infty$, independent of $n$ and $T$, such that $\max_{i,l}\sup_{\theta\in\Pi}
\left|
\frac{d}{d\theta} g_{il}(e^{-\iota\theta})
\right|
\leq L_g.$
Since $g_{il}$ has no poles on the unit circle and is smooth in $\theta$, the mean value theorem yields
\begin{align*}
    \left\|  G(\theta_2) - G(\theta_1) \right\|_2 \leq \left\|  G(\theta_2) - G(\theta_1) \right\|_F \leq \sqrt{nq} \max_{il}{|g_{il}(\theta_2) - g_{il}(\theta_1)|} \leq \sqrt{nq}L_g|\theta_2 - \theta_1|.
\end{align*}
Therefore, inequality~(\ref{eq:lipschitz_short_memory}) holds.

Here we emphasize that the lower inequality in (\ref{eq:pervasive_S_GG}) is imposed as an assumption. It rules out some cases. In particular, $S(\theta) \succeq s^- \mathbf{I}_{q\times q}$ implies that for any $\alpha \in \mathbb{R}^q$, $\alpha^* S(\theta) \alpha \geq s^- \|\alpha\|_2^2$, or equivalently, $\|G(\theta)\alpha\|_2^2 \geq ns^- \|\alpha\|_2^2$. 

There are two simple situations that might break this lower bound. First, if one column of $G(\theta)$, say the $l_{th}$ column $g_l(\theta)$, is too small and does not grow with $n$, then $e_l^* S(\theta) e_l = \frac{1}{n}\left\|g_l(\theta)\right\|_2^2 \leq s^-$. This breaks the lower bound of (\ref{eq:pervasive_S_GG}). That said, the lower inequality of (\ref{eq:pervasive_S_GG}) requires at least $\|g_l(\theta)\|_2 \geq \sqrt{ns^-}$ for all $l = 1,...,n$ and $\theta \in \Pi$. Second, if there exists some $\alpha_0 \in \mathbb{R}^q, \|\alpha_0\|_2 = 1$, such that $G(\theta)\alpha_0 \approx 0$, then the columns of $G(\theta)$ are nearly collinear. This case again violates  $\|G(\theta)\alpha_0\|_2^2 \geq ns^-$.

Such an lower bound guarantees an eigengap of order $n$ for the common spectral density $\Sigma_\chi(\theta)$ between the first $q$ eigenvalues and the remaining $n-q$ eigenvalues, for each fixed $\theta$. Under Assumption~\ref{assump:bdd_error_eval}, this separation is preserved for the full spectral density $\Sigma_x(\theta)$ as well. See Lemma \ref{lemma:div_component_eval_specific_form} below for a more precise formulation.

\begin{assumption}\label{assump: spec_den_specific_form}

In this paper, we mainly consider two cases:
\begin{enumerate}[(a)]
    \item $d_{il} = d_l \in [0, 0.5)$. Define a $q\times q$ diagonal matrix $D(\theta) = diag\left( (1-e^{-\iota \theta})^{-d_l} \right)_{l = 1,...,q}. $ By Assumption \ref{assump:model}, the spectral density of $u_t=(u_{1t}, u_{2t}, ..., u_{qt})'$ is $\sigma_u^2 \mathbf{I}_{q\times q}$. Without loss of generality, set $\sigma_u=1$.  Then $\Sigma_{\chi}(\theta) =  G(\theta) D(\theta) D^{\ast}(\theta) G^{\ast}(\theta) = G(\theta) \Bar{D}^{2}(\theta) G^{\ast}(\theta)$ where $\Bar{D}(\theta) = diag\left( |1-e^{-\iota \theta}|^{-d_l} \right)_{l = 1,...,q}$.
    \item $d_{il} = d_i \in [0, 0.5)$. Define an $n\times n$ diagonal matrix $D_n(\theta) = diag\left( (1-e^{-\iota \theta})^{-d_i} \right)_{i = 1,...,n}. $   Then $\Sigma_{\chi}(\theta) =  D_n(\theta)G(\theta) G^{\ast}(\theta)D^{\ast}_n(\theta)$.
\end{enumerate}

Throughout the paper, we let $d = \max_{il} d_{il}$, and $\tilde d = \min_{il} d_{il}$.

\end{assumption}


\begin{assumption}[Boundness and regularity of idiosyncratic component]\label{assump:bdd_error_eval}
There exist constants $\Lambda_{\xi}^{-}, \Lambda_{\xi}^{+}$ such that the eigenvalues of the idiosyncratic spectral density matrix $\Sigma_{\xi}(\theta)$ satisfy $\Lambda_{\xi}^{-}\leq \lambda_{nn}^\xi(\theta)\leq \cdots\leq\lambda_{n1}^\xi(\theta) \leq \Lambda_{\xi}^{+}$ for all $\theta \in \Pi$. 
Furthermore, we assume regularity on $\Sigma_{\xi}(\theta)$, $\Sigma_{\xi}(\theta)$ is second differentiable and thus Lipschitz continuous on $\Pi$.

Furthermore, we assume that the fourth-order cumulant of $\epsilon_t$ in \eqref{eq:xi_representation} is uniformly bounded in multilinear operator norm:
\begin{align}\label{eq:eps_Q_finite}
    \left\|Q^\epsilon(0,0,0)\right\|_{\mathrm{op}} < \infty,
\end{align}
where
\[
\left\|Q^\epsilon(0,0,0)\right\|_{\mathrm{op}}
:=
\sup_{\|v_1\|_2=\cdots=\|v_4\|_2=1}
\left|
\operatorname{cum}
(v_1'\epsilon_t, v_2'\epsilon_t, v_3'\epsilon_t, v_4'\epsilon_t)
\right|.
\]
\end{assumption}

\textbf{Remark.} 
Condition \eqref{eq:eps_Q_finite} is imposed to obtain a sharper rate in Lemma \ref{lm:var_conv}. 
A sufficient condition for \eqref{eq:eps_Q_finite} is that the coordinates of \(\epsilon_t\) are mutually independent and have uniformly bounded fourth cumulants. 
Moreover, if \(\epsilon_t\) is Gaussian, then \eqref{eq:eps_Q_finite} holds trivially as \(Q^\epsilon(0,0,0)=0\).

Here we impose an extra structure on $G(\theta)$, that is besides Assumption \ref{assump:pervasive_short_mem} we further assume

\begin{assumption}\label{assump:row_boundness_G} Assumption \ref{assump:pervasive_short_mem} holds. Let $G_i(\theta)$ to be the $i_{th}$ row of $G(\theta)$, then we assume $\left\| G_i(\theta)\right\|^2 \leq C$, C is independent of $n$ and $T$.
\end{assumption}

Notice that such an assumption is compatible with Assumption \ref{assump:pervasive_short_mem}.

We further impose a gap dominance assumption to obtain a better consistency rate 
\begin{assumption}\label{assump:gap_dominance} In Assumption \ref{assump: spec_den_specific_form} case (a), let
\[
\rho^{(q)}_m = \min \{\rho^{(q)}_{m-1}, \alpha_{m,m+1}, \min_{l\leq m-1}\{2\rho^{(q)}_l - \alpha_{l,m}\} \}
\]
for all $m \leq q$, $\rho^{(q)}_0 = 1$, $\alpha_{j_1j_2} \coloneqq 2|d_{(j_1)}-d_{(j_2)}|>0$, $\alpha_{0,j} = 1, j = 1, \dots, q$, $d_{(q+1)}\coloneqq 0$. We assume
\[
2\rho^{(q)}_l > \alpha_{l,m}
\] for all $l = 1, \dots, m-1$, and $m= 2,\dots, q$, if $q\geq 2$.
\end{assumption}

\textbf{Remark:} This assumption imposes a dominance condition on the gaps among the memory parameters. In particular, it requires each relevant gap to be sufficiently large relative to the accumulated effect of the remaining gaps. For example, the condition holds automatically when \(q=1\) or \(q=2\). When \(q=3\), it requires that the first gap dominates the second gap, namely $d_{(1)} - d_{(2)} > d_{(2)} - d_{(3)}$.

Under Assumption \ref{assump:pervasive_short_mem} and \ref{assump: spec_den_specific_form}, the following lemma holds. See Appendix \ref{sec:assump_proof} for the proof.

\begin{lemma}\label{lemma:div_component_eval_specific_form}
Under Assumptions~\ref{assump:pervasive_short_mem} and~\ref{assump: spec_den_specific_form}, the diverging eigenvalues of $\Sigma^{\chi}(\theta)$ diverge at rate $n$, $\theta-$a.e.; that is, there exist functions $\alpha^{\chi}_q(\theta),  \beta^{\chi}_{0}(\theta) \in L^1(\Pi)$ such that 
    \begin{align*}
        \beta^{\chi}_0(\theta) \geq \frac{\lambda^{\chi}_{n1}(\theta)}{n}, \dots, \frac{\lambda^{\chi}_{nq}(\theta)}{n} \geq \alpha^{\chi}_q(\theta) >  \beta^{\chi}_q(\theta) =0 
    \end{align*}

In the case where
\begin{enumerate}
    \item $d_{il} = d_l, l = 1,\dots, q$, assume $d^{(1)}>\cdots>d^{(q)}$, 
\begin{equation}\label{eq:alpha_beta_defn}
\alpha_q^{\chi}(\theta)=c_q^{-}\,|\theta|^{-2d^{(q)}},
\qquad
\beta_0^{\chi}(\theta)=c_{1}^{+}\,|\theta|^{-2d^{(1)}} ,
\end{equation}
\item $d_{il} = d_i, i = 1, \dots, n$, assume $d^{(1)}>\cdots>d^{(n)}$, 
\begin{align}\label{eq:alpha_beta_defn_across_row}
&\alpha_q^{\chi}(\theta)=c_q^{-}\,\min\!\big\{ |\theta|^{-2d_{(1)}},\, |\theta|^{-2d_{(n)}}\big\},\\
&\beta_0^{\chi}(\theta)=c_1^{+}\,\max\!\big\{ |\theta|^{-2d_{(1)}},\, |\theta|^{-2d_{(n)}}\big\}.
\end{align}
\end{enumerate}
for some constants $0 <c_q^{-} \leq c_1^{+}<\infty$.

Moreover, if every row and series shares the same the long memory parameter $d$, there exists $ c_{q}^{-}\leq c_{1}^{+} $ such that the eigenvalue $\lambda^{b}_{j,\chi}(\theta)$ of $|\theta|^{2d}\Sigma^{\chi}(\theta)$, satisfies
    \[
    c_1^{+}\geq \frac{\lambda^{b}_{1,\chi}(\theta)}{n}, \dots \frac{\lambda^{b}_{q,\chi}(\theta)}{n}\geq c_q^{-}
    \]
    For $\theta$-a.e., $\lambda^{\chi}_{nj}(\theta) = |\theta|^{-2d}\lambda^{b}_{j,\chi}(\theta)$, $j = 1,\dots, q$. $\alpha_q^{\chi}(\theta) = c_q^{-}|\theta|^{-2d}, \beta_0^{\chi}(\theta) = c_{1}^{+}|\theta|^{-2d}$.
\end{lemma}

\textbf{Remark.} Since our paper mainly focuses on the behavior of the spectral density near zero frequency, for simplicity we take the lower bound to be $c_q^- |\theta|^{-2d_{(n)}}$ and the upper bound to be $c_1^+ |\theta|^{-2d_{(1)}}$.

\section{Dynamic PCA Approach} \label{sec:DPCA_methodology}
\cite{dynamicfactorForni2sided} proposed a method based on dynamic principal components, originally introduced in \cite{brillinger} as dynamic principal component analysis. In this section, we briefly review the dynamic PCA approach of \cite{dynamicfactorForni2sided} for estimating the common component.

Let $\{X_t\}_{t\in\mathbb{Z}}$ denote the $n$-dimensional observed time series, $X_t = (X_{1t},\ldots,X_{nt})'$.
Let $\Sigma_X(\theta)$ denote the $n\times n$ spectral density matrix of $\{X_t\}$.

Denote by
\[
\Sigma_X(\theta) = P(\theta)\Lambda(\theta)P(\theta)'
\]
the eigen-decomposition of $\Sigma_X(\theta)$, where $P(\theta)=\big(\mathbf{p}_1(\theta),\ldots,\mathbf{p}_n(\theta)\big)$
is the orthogonal matrix of eigenvectors and $\Lambda(\theta)=\mathrm{diag}(\lambda_1(\theta),\ldots,\lambda_n(\theta))$
contains the eigenvalues ordered decreasingly. $p_{j,i}(\theta)$ denote the $i$th component of eigenvector $\mathbf{p}_j(\theta)$.

Following \cite{dynamicfactorForni2sided}, the oracle filter for recovering the $i$th common component is defined as

\begin{equation}
\chi_{it,n} = \underline{\mathbf{K}}_{ni}(L)X_t,
\end{equation}

where $\underline{\mathbf{K}}_{ni}(L)$ is the linear filter with transfer function

\begin{equation}
\underline{\mathbf{K}}_{ni}(L)
=
\frac{1}{2\pi}\sum_{h=-\infty}^{\infty}
\mathbf{K}_{ni,h}L^h, \quad \mathbf{K}_{ni,h} = \int_{\Pi}\underline{\mathbf{K}}_{ni}(\theta)e^{-\iota h\theta}\,d\theta
\end{equation}
with
\begin{equation}
\underline{\mathbf{K}}_{ni}(\theta)
=
\tilde{p}_{1,i}(\theta)\mathbf{p}_1(\theta)
+
\tilde{p}_{2,i}(\theta)\mathbf{p}_2(\theta)
+\cdots+
\tilde{p}_{q,i}(\theta)\mathbf{p}_q(\theta).
\end{equation}

Here $\tilde p_{j,i}(\theta)$ denotes the complex conjugate of $p_{j,i}(\theta)$.

This estimator is called an oracle estimator because it depends on the unknown population spectral density $\Sigma_X(\theta)$, as well as an infinite amount of data.

\bigskip

To construct a feasible estimator, we continue to adopt the spectral density estimator based on the observed data $\{X_t\}_{t=1}^T$ as used in \cite{dynamicfactorForni2sided}. Specifically, the spectral density matrix is estimated by the smoothed periodogram
\begin{align}\label{eq:smoothed_periodogram}
     \widehat{\Sigma}_{n}(\theta) &= \frac{2\pi}{B_T T}\sum_{j=-T_0}^{T_0} W\left(\frac{\theta - \lambda_j}{B_T}\right)\mathbf{I}_{XX}(\lambda_j) \mathbb{1}(j \neq 0) \\\nonumber
     &= \frac{2\pi}{B_T T}\sum_{j=1}^{T_0} \left(W\left(\frac{\theta - \lambda_j}{B_T}\right) + W\left(\frac{\theta + \lambda_j}{B_T}\right) \right)\mathbf{I}_{XX}(\lambda_j)  \\
     \mathbf{I}_{XX}(\omega) &:=\frac{1}{2 \pi T}\left[\sum_{t=0}^{T-1} X_t \exp (-i \omega t)\right]\left[\sum_{t=0}^{T-1} X_t^{ \prime} \exp (i \omega t)\right]
 \end{align}
where $T_0 = \lfloor (T-1)/2 \rfloor$. Moreover, if we further denote 
\[
\widetilde{W}_{B_T}(x,y) \coloneqq W\left(\frac{x-y}{B_T}\right) + W\left(\frac{x+y}{B_T}\right)
\]
Then 
\[
 \widehat{\Sigma}_{n}(\theta) = \frac{2\pi}{B_T T}\sum_{j=1}^{T_0} \widetilde{W}_{B_T}\left(\theta, \lambda_j\right) \mathbf{I}_{XX}(\lambda_j)
\]

$W(\cdot)$ is called the spectral window that satisfies the assumption below.
\begin{assumption}\label{assump:W_kernel}
    $W(\cdot)$ is a symmetric, non-negative kernel function satisfying \[
    \int_{-\infty}^{\infty} W(\lambda) d\lambda = 1, \text{ and } \int_{-\infty}^{\infty} \lambda^2 W(\lambda) d\lambda < \infty
    \]
    $W(\cdot)$ has a compact support on $[-\rho, \rho]$, and is Lipschitz continuous.
\end{assumption}

Let
\[
\widehat{\Sigma}_n(\theta)
=
\widehat P_n(\theta)
\widehat\Lambda_n(\theta)
\widehat P_n(\theta)'
\]
denote the eigen-decomposition of $\widehat{\Sigma}_n(\theta)$, where $\widehat{P}_n(\theta)=\big(\widehat{\mathbf{p}}_{n1}(\theta),\ldots,\widehat{\mathbf{p}}_{nn}(\theta)\big)$.
Define
\begin{align}
    \underline{\widehat{\mathbf{K}}}_{ni}(\theta)
=
\widetilde{\widehat{p}}_{n1,i}(\theta)\widehat{\mathbf{p}}_{n1}(\theta)
+
\widetilde{\widehat{p}}_{n2,i}(\theta)\widehat{\mathbf{p}}_{n2}(\theta)
+\cdots+
\widetilde{\widehat{p}}_{nq,i}(\theta)\widehat{\mathbf{p}}_{nq}(\theta).
\end{align}

The feasible filter is obtained by truncating the infinite filter
\begin{equation}
\underline{\widehat{\mathbf{K}}}_{ni}(L)
=
\sum_{h=\max\{t-T,-M(T)\}}^{\min\{t-1,M(T)\}}
\underline{\widehat{\mathbf{K}}}_{ni,h}L^h,
\end{equation}

where

\begin{align}
    \underline{\widehat{\mathbf{K}}}_{ni,h}
=
\frac{1}{2\pi}\int_{-\pi}^{\pi}
\underline{\widehat{\mathbf{K}}}_{ni}(\theta)e^{-ih\theta}\,d\theta,
\end{align}

Finally, the feasible estimator of the common component is
\[
\widehat\chi_{it}
=
\widehat K_{ni}(L)X_t.
\]

$M(T)\to\infty$ as $T\to\infty$. For the moment, we postpone the discussion of how the truncation parameter $M(T)$ should vary with $d$, $n$, $T$, and $\Delta$ in order to achieve a consistency rate. This point will be revisited in Section~\ref{subsec:feasible_estimator} after the consistency theorem is established.

\section{Consistency of the Long Memory DFM}\label{sec:main_results}
\subsection{Consistency of the Oracle Estimator}\label{subsec:oracle}
\begin{prop}\label{proposition_oracle}
    Suppose that Assumptions \ref{assump:model}, \ref{assump:semiparametric_long_memory}, \ref{assump:pervasive_short_mem}, and \ref{assump:bdd_error_eval}
    hold, then
        for all $\eta>0$, there exists $B_{\eta}, N_{\eta}$, such that for any $n \geq N_{\eta}$,
    \[
P\left[n^{\frac{1}{2}}\left|{\underline{\mathbf{K}}}_{n i}(L)X_t -\chi_{i t}\right|>B_{\eta}\right] \leq \eta
    \]
\end{prop}

\subsection{Feasible Estimation of the Common Component}\label{subsec:feasible_estimator}

\begin{prop}\label{thm:main_theorem}
Assumptions \ref{assump:model}, \ref{assump:semiparametric_long_memory}, \ref{assump:pervasive_short_mem}, \ref{assump: spec_den_specific_form}, \ref{assump:bdd_error_eval}, \ref{assump:W_kernel}, and \ref{assump:row_boundness_G} hold. Then for all $\eta>0$, there exist $B_\eta$ and $N_\eta$ such that
    \[
P\left[r^{(n)}\left|\widehat{\chi}_{it} -\chi_{i t}\right|>B_{\eta}\right] \leq \eta, \quad \widehat{\chi}_{it} = \widehat{\underline{\mathbf{K}}}_{n i}(L)X_t
    \] for all $t = t^*(T)$ satisfying 
    \begin{align}\label{eq:center_part_T}
    0<a \leq \liminf _{T \rightarrow \infty} \frac{t^*(T)}{T} \leq \limsup _{T \rightarrow \infty} \frac{t^*(T)}{T} \leq b<1 
    \end{align}
    and for all $n\geq N_{\eta}, T \geq T(n)$, where $T(n):\mathbb{N}\to\mathbb{N}$ is an increasing function such that $T(n)\to\infty$ as $n\to\infty$.
    Then $\widehat{\chi}_{it}$ achieves the consistency rate $r^{(n)}$, where
    \begin{equation}\label{eq:consistency_rate}
\begin{aligned}
    r^{(n)} &= \min \left(r_0^{(n)},r_1^{(n)},r_2^{(n)}\right) \\
    r_0^{(n)} &= n^{1/2};\\
    r_1^{(n)} &= M(T)^{-1}\left(\delta^{(T,n)}\right)^{-1}n^{-\frac{1}{2}};\\
    \delta^{(T,n)} &= B_T^{1-2\Delta} + \frac{\log T}{TB_T} 
\end{aligned}
\end{equation}
and that\begin{enumerate}
        \item when all rows and shocks share the same $d$
        \begin{align}\label{eq:common_d_edge_rate}
            r_2^{(n)} \asymp \frac{1}{\log M(T)}\begin{cases}\dfrac{M(T)}{\sqrt{n}}, & d=0,\\[6pt] \min\!\left(\dfrac{M(T)}{\sqrt{n}},\; \sqrt{n}\, M(T)^{2d}\right), & d>0.\end{cases}, 
        \end{align}
\begin{enumerate}[(i)]
    \item When the idiosyncratic error $\xi_t$ is orthonormal white noise,  \begin{align*}
            r_2^{(n)} \asymp \frac{M(T)}{\log M(T) \sqrt{n}}, 
        \end{align*}
\end{enumerate}
\item $d_{il}$ differs, and it satisfies case(a) in Assumption \ref{assump: spec_den_specific_form}
\begin{align}\label{eq:hetero_d_across_factor_edge_rate}
    r_2^{(n)} \asymp n^{-1/2}\,M(T)^{\frac{1}{2} +\rho^{(q)}_q -d}.
\end{align}
where
     \begin{align}\label{eq:defn_rho_q_q}
     \rho^{(q)}_q = \min \{\rho^{(q)}_{q-1}, \alpha_{q,q+1}, \min_{0 \leq l\leq m-1}\{2\rho^{(q)}_l - \alpha_{l,q}\} \}
    \end{align}
     $\alpha_{j_1j_2} = 2|d_{(j_1)}-d_{(j_2)}|$, $d_{(q+1)} \coloneqq 0$, $\rho^{(j)}_0 \coloneqq 1$, $\alpha_{0,j} \coloneqq 1, j = 1,\dots,q$.
        \item $d_{il}$ differs, and it satisfies case(b) in Assumption \ref{assump: spec_den_specific_form}. 
        When $4\Delta + 2d - 1 <0$, $\Delta = d_{max} - d_{min}$,
        \begin{align}\label{eq:hetero_d_edge_rate}
    r_2^{(n)} \asymp n^{-\frac{1}{2}}M(T)^{\frac{1}{2} - 2\Delta - d} .
\end{align}
    \end{enumerate}
The proposition provides a consistency rate only if $\left(r^{(n)}\right)^{-1} = o(1)$.
\end{prop}

\textbf{Remark:} To the best of our knowledge, within the existing literature on GDFM, this is the first work to provide a rigorous consistency rate for the method of \cite{dynamicfactorForni2sided} that explicitly incorporates the truncation error, even in the benchmark short-memory case $d=0$. While \cite{FORNI2004231} developed theoretical guarantees for \cite{dynamicfactorForni2sided} in a short-memory setting, their analysis does not explicitly account for the truncation (edge) term, which can lead to seemingly sharper but incomplete rates. Specifically, in the case $d=0$, \eqref{eq:common_d_edge_rate} implies that $r^{(n)}$ is at least $\frac{M(T)}{\log M(T)\sqrt{n}}$, which slows down the rate derived in \cite{FORNI2004231}. Second, since \cite{FORNI2004231} does not explicitly account for truncation, no $M(T)$ term appears in the main term $r_1^{(n)}$ or in their theorem.


Throughout, asymptotic statements are understood along paths $\{(n,T(n)): n \in \mathbb{N}\}$. To obtain a consistency rate, we examine what these paths look like and under what growth conditions such a rate can be achieved. \paragraph{How to pick $(b,m)$ (plain power-law logic) and what $\kappa(d)$ becomes.}

To understand Proposition~\ref{thm:main_theorem} under polynomial tuning, set
\[
M(T)=T^{m},\qquad B_T=T^{-b},\qquad T=n^{\kappa},
\]
with $m,b\in(0,1)$. Since $r^{(n)}=\mathcal{O}(\min(r_0^{(n)},r_1^{(n)},r_2^{(n)}))$ and $r_0^{(n)}=n^{1/2}$ never collapses, the overall rate stays non-degenerate (up to logarithms) once we enforce
\[
\left(r_1^{(n)}\right)^{-1}=o(1)\qquad\text{and}\qquad \left(r_2^{(n)}\right)^{-1}=o(1).
\]

\smallskip
\noindent Step 1: rewrite $r_1$ in terms of powers of $T$.
Recall
\[
r_1^{(n)}=M(T)^{-1}(\delta^{(T,n)})^{-1}n^{-1/2},
\qquad
\delta^{(T,n)}=B_T^{1-2\Delta}+\frac{\log T}{TB_T}.
\]
Ignoring the polylog factors (they only affect the final choice by an arbitrarily small $\varepsilon$), the three terms of $\delta^{(T,n)}$ have polynomial orders
\[
B_T^{1-2\Delta}=T^{-b(1-2\Delta)},\qquad
(B_TT)^{-1}=T^{\,b-1}
\]
Hence, up to logs,
\[
\delta^{(T,n)} \asymp T^{\max\{-b(1-2\Delta),\,b - 1\}},
\qquad
(\delta^{(T,n)})^{-1}\asymp T^{\gamma(b,d)},
\]
where 
\[
\gamma(b,\Delta)\coloneqq
\min\Big\{\,b(1-2\Delta),\;1-b\Big\}.
\]
Therefore,
\[
r_1^{(n)} \asymp n^{-1/2}\,T^{-m}\,T^{\gamma(b,\Delta)}
= n^{-1/2+\kappa(\gamma(b,\Delta)-m)}.
\]
Thus $\left(r_1^{(n)}\right)^{-1}=o(1)$ is equivalent to
\begin{equation}\label{eq:kappa_r1_plain}
-\frac12+\kappa(\gamma(b,\Delta)-m)> 0
\quad\Longleftrightarrow\quad
\kappa > \frac{1}{2(\gamma(b,\Delta)-m)}.
\end{equation}

\smallskip
\noindent Step 2: choose $b$.

The best $b$ is the one that maximizes $\gamma(b,\Delta)$:
\[
b^*(\Delta)\in\arg\max_{b\in(0,1)}\ \gamma(b,\Delta)
=\arg\max_{b\in(0,1)}\ \min\Big\{b(1-2\Delta),\ 1-b\Big\}.
\]
Thus,
\[
\boxed{\ b^*(\Delta)=\frac{1}{2(1-\Delta)}, d \neq 0\ }.
\]

In practice (because of the $\log T$ factors) one takes $b=b^*(\Delta)\pm\varepsilon$ for an arbitrarily small $\varepsilon>0$.

\smallskip
\noindent Step 3: combine with $r_2$ and choose $m$ by balancing.
\paragraph{Optimal polynomial tuning in the common $d$ cases} 
In the common-$d$ case (including the orthonormal white-noise specialization), the bottleneck part of $r_2$ is
\[
r_2^{(n)}\asymp \frac{M(T)}{\sqrt n}
= n^{-1/2}T^{m}
= n^{-1/2+\kappa m},
\]
when $M(T)/\sqrt{n} \lesssim \sqrt{n}M(T)^{2d}$, i.e. $\kappa < \frac{1}{m(1-2d)}$. When $M(T)/\sqrt{n} \lesssim \sqrt{n}M(T)^{2d}$, 
\[
r_2^{(n)}\asymp \sqrt{n}M(T)^{2d}
\] always blows up so it never collapses.

Thus, $\left(r_2^{(n)}\right)^{-1}=o(1)$ is equivalent to
\begin{equation}\label{eq:kappa_r2_plain}
 -\frac12+\kappa m > 0
\quad\Longleftrightarrow\quad
\kappa > \frac{1}{2m}.
\end{equation} 
Since $\Delta = 0$ for the common $d$ case, now fix $b=b^*(0)$ so that $\gamma(b,0)=\gamma^*(0)\coloneqq \gamma(b^*(0),0) = \frac{1}{2}$. The two lower bounds \eqref{eq:kappa_r1_plain}--\eqref{eq:kappa_r2_plain} become
\[
\kappa > \frac{1}{2(\gamma^*(0)-m)} = \frac{1}{1-2m}
\qquad\text{and}\qquad
\kappa > \frac{1}{2m}.
\]
Hence the optimal $m$ is obtained by balancing the two constraints:
\[
\frac{1}{1-2m}=\frac{1}{2m}
\quad\Longleftrightarrow\quad
m^*(0,d)=\frac{1}{4}.
\]
Plugging back gives
\[
\kappa^*(d)=\frac{1}{2m^*(0,d)}=\frac{1}{\gamma^*(0)}=2,
\]
i.e.
\[
\boxed{
b^*(0) = \frac{1}{2},
\qquad
m^*(0, d)=\frac{1}{4},
\qquad
\kappa^*(d)=
2.
}
\]

\paragraph{Optimal polynomial tuning in the heterogeneous-$d_{il}$ cases: across factors} 
For the truncation remainder, we now use
\[
r_2^{(n)}
\asymp
n^{-1/2}M(T)^{\frac12+\rho_q^{(q)}-d}.
\]

Thus, \((r_2^{(n)})^{-1}=o(1)\) requires
\[
\kappa>
\frac{1}{2m\left(\frac12+\rho_q^{(q)}-d\right)}.
\]

Now fixing \(b=b^*(\Delta)\), the optimal \(m\) is obtained by balancing the two constraints:
\[
\kappa>
\frac{1}{2(\gamma^*(\Delta)-m)}
\qquad\text{and}\qquad
\kappa>
\frac{1}{2m\left(\frac12+\rho_q^{(q)}-d\right)}.
\]
that is,
\[
\frac{1}{2(\gamma^*(\Delta)-m)}
=
\frac{1}{2m\left(\frac12+\rho_q^{(q)}-d\right)}.
\]
Equivalently,
\[
m^*(d,\Delta,\rho_q^{(q)})
=
\frac{\gamma^*(\Delta)}{\frac32+\rho_q^{(q)}-d}.
\]
Plugging this back, as well as $\gamma^*(\Delta)=\frac{1-2\Delta}{2(1-\Delta)}$ gives
\[
\boxed{
\kappa^*(d,\Delta,\rho_q^{(q)})
=
\frac{(1-\Delta)\left(\frac32+\rho_q^{(q)}-d\right)}
{(1-2\Delta)\left(\frac12+\rho_q^{(q)}-d\right)}}
\]

For illustration, we fix $\rho_q^{(q)}=0.25.$
Since the recursive definition of $\rho_q^{(q)}$ implies $\rho_q^{(q)}\leq 2\Delta,$
we impose the necessary restriction $\Delta \geq \frac{\rho_q^{(q)}}{2}=0.125.$ We also impose $d\geq \Delta.$

The corresponding values of $\kappa^*(d,\Delta,0.25)$ are reported in the following table:
\[
\begin{array}{c|ccccccc}
\Delta\backslash d
&0.15&0.20&0.25&0.30&0.35&0.40&0.45\\
\hline
0.15&3.238&3.422&3.643&3.913&4.250&4.684&5.262\\
0.20&-&3.758&4.000&4.296&4.667&5.143&5.778\\
0.25&-&-&4.500&4.833&5.250&5.786&6.500\\
0.30&-&-&-&5.639&6.125&6.750&7.583\\
0.35&-&-&-&-&7.583&8.357&9.389\\
0.40&-&-&-&-&-&11.571&13.000\\
0.45&-&-&-&-&-&-&23.833
\end{array}
\]

For the case where $\rho_q^{(q)}=0.1$, $\Delta \geq 0.05, d \geq \Delta$,
\[
\begin{array}{c|ccccccccc}
\Delta\backslash d
&0.05&0.10&0.15&0.20&0.25&0.30&0.35&0.40&0.45\\
\hline
0.05&2.975&3.167&3.401&3.694&4.071&4.574&5.278&6.333&8.093\\
0.10&-&3.375&3.625&3.938&4.339&4.875&5.625&6.750&8.625\\
0.15&-&-&3.913&4.250&4.684&5.262&6.071&7.286&9.310\\
0.20&-&-&-&4.667&5.143&5.778&6.667&8.000&10.222\\
0.25&-&-&-&-&5.786&6.500&7.500&9.000&11.500\\
0.30&-&-&-&-&-&7.583&8.750&10.500&13.417\\
0.35&-&-&-&-&-&-&10.833&13.000&16.611\\
0.40&-&-&-&-&-&-&-&18.000&23.000\\
0.45&-&-&-&-&-&-&-&-&42.167
\end{array}
\]

When $\rho_q^{(q)}=0.01$, $\Delta \geq 0.005, d \geq \Delta$,

\[
\begin{array}{c|ccccccccccc}
\Delta\backslash d
&0.005&0.01&0.05&0.10&0.15&0.20&0.25&0.30&0.35&0.40&0.45\\
\hline
0.005&2.995&3.015&3.190&3.456&3.797&4.247&4.871&5.791&7.287&10.142&17.756\\
0.01&-&3.031&3.206&3.474&3.816&4.269&4.896&5.821&7.324&10.194&17.847\\
0.05&-&-&3.350&3.630&3.988&4.461&5.115&6.082&7.653&10.652&18.648\\
0.10&-&-&-&3.869&4.250&4.754&5.452&6.482&8.156&11.352&19.875\\
0.15&-&-&-&-&4.587&5.131&5.885&6.997&8.804&12.253&21.452\\
0.20&-&-&-&-&-&5.634&6.462&7.683&9.667&13.455&23.556\\
0.25&-&-&-&-&-&-&7.269&8.643&10.875&15.136&26.500\\
0.30&-&-&-&-&-&-&-&10.083&12.687&17.659&30.917\\
0.35&-&-&-&-&-&-&-&-&15.708&21.864&38.278\\
0.40&-&-&-&-&-&-&-&-&-&30.273&53.000\\
0.45&-&-&-&-&-&-&-&-&-&-&97.167
\end{array}
\]

\paragraph{Optimal polynomial tuning in the heterogeneous-$d_{il}$ cases: across rows} 
We continue to use the polynomial parametrization
\[
M(T)=T^{m},\qquad B_T=T^{-b},\qquad T=n^{\kappa},
\]
and the same logic as in the common-$d$ case.

In this case, \[
r_2^{(n)}\asymp n^{-\frac{1}{2}}M(T)^{\frac{1}{2} - 2\Delta - d} = n^{(\frac{1}{2} - 2\Delta - d)\kappa m - \frac{1}{2}}
\]
Thus, when $\frac{1}{2} - 2\Delta - d > 0$ $\left(r_2^{(n)}\right)^{-1}=o(1)$ is equivalent to
\begin{equation}\label{eq:kappa_r2_plain_hetero}
\kappa > \frac{1}{(1-4\Delta - 2d)m}.
\end{equation}
Then again, plugging back in $\gamma(b,\Delta)=\gamma^*(\Delta)\coloneqq \gamma(b^*(\Delta),\Delta) = \frac{1-2\Delta}{2(1-\Delta)}$, the optimal $m$ is obtained by balancing the two constraints \eqref{eq:kappa_r1_plain}, and \eqref{eq:kappa_r2_plain_hetero} 
\[
\frac{1}{2(\gamma^*(\Delta)-m)}= \frac{1}{(1-4\Delta - 2d)m}
\Longleftrightarrow
m^*(\Delta,d)=\frac{2\gamma^*(\Delta)}{3-4\Delta - 2d} = \frac{1-2\Delta}{(1-\Delta)(3-4\Delta - 2d)}.
\]
Plugging back gives
\[
\boxed{
\kappa^*(d)=\frac{(1-\Delta)(3-4\Delta-2d)}{(1-2\Delta)(1-4\Delta -2d)}.
}
\]



\begin{table}[!htbp]
\centering
\caption{$b_h^*(\Delta)=\dfrac{1}{2(1-\Delta)}$.}
\label{tab:bh_grid}
\resizebox{\textwidth}{!}{%
\begin{tabular}{c|cccccccccc}
\toprule
$\Delta$ & 0.00 & 0.05 & 0.10 & 0.15 & 0.20 & 0.25 & 0.30 & 0.35 & 0.40 & 0.45 \\
\midrule
$b_h^*(\Delta)$
& 0.5000 & 0.5263 & 0.5556 & 0.5882 & 0.6250 & 0.6667 & 0.7143 & 0.7692 & 0.8333 & 0.9091 \\
\bottomrule
\end{tabular}}
\end{table}

\begin{table}[!htbp]
\centering
\caption{$m_h^*(d,\Delta)$ restricted to valid pairs:
$\Delta<d$ and $1-4\Delta-2d>0$.}
\label{tab:mh_grid}
\resizebox{\textwidth}{!}{%
\begin{tabular}{c|cccccccccc}
\toprule
$\Delta\backslash d$
& 0.00 & 0.05 & 0.10 & 0.15 & 0.20 & 0.25 & 0.30 & 0.35 & 0.40 & 0.45 \\
\midrule
0.00 & -- & 0.3448 & 0.3571 & 0.3704 & 0.3846 & 0.4000 & 0.4167 & 0.4348 & 0.4545 & 0.4762 \\
0.05 & -- & -- & 0.3644 & 0.3789 & 0.3947 & 0.4119 & 0.4306 & 0.4511 & -- & -- \\
0.10 & -- & -- & -- & 0.3865 & 0.4040 & 0.4233 & -- & -- & -- & -- \\
\bottomrule
\end{tabular}}
\end{table}

\begin{table}[!htbp]
\centering
\caption{$\kappa_h^*(d,\Delta)$ restricted to valid pairs:
$\Delta<d$ and $1-4\Delta-2d>0$.}
\label{tab:kappah_grid}
\resizebox{\textwidth}{!}{%
\begin{tabular}{c|cccccccccc}
\toprule
$\Delta\backslash d$
& 0.00 & 0.05 & 0.10 & 0.15 & 0.20 & 0.25 & 0.30 & 0.35 & 0.40 & 0.45 \\
\midrule
0.00 & -- & 3.2222 & 3.5000 & 3.8571 & 4.3333 & 5.0000 & 6.0000 & 7.6667 & 11.0000 & 21.0000 \\
0.05 & -- & -- & 4.5741 & 5.2778 & 6.3333 & 8.0926 & 11.6111 & 22.1667 & -- & -- \\
0.10 & -- & -- & -- & 8.6250 & 12.3750 & 23.6250 & -- & -- & -- & -- \\
\bottomrule
\end{tabular}}
\end{table}

\section{Simulations}\label{sec:simulation}

In our analysis for proving Proposition \ref{thm:main_theorem}, we decompose the probability into the main part \eqref{eq:main_prob} and the remainder part \eqref{eq:remainder_prob}. These two parts were separated because of the data infeasibility at infinite lags and leads, as well as the introduction of \(M(T)\), whose necessity we have not yet fully understood intuitively. That is, at this stage, we do not yet have a clear intuition for why \(M(T)\) is needed; however, within our current framework, it still serves as a useful proxy for carrying out the proof. We welcome discussion of possible alternative approaches to proving these theorems. To provide better intuition for what the structure may look like, in this section we present simulation evidence on some of its inner components. 

\subsection{Simulation corroborating Lemma \ref{lemma:div_component_eval_specific_form} }
Since the operator $\underline{\mathbf{K}}_{ni}(L)$ is a projection matrix, its behavior is governed by perturbation results such as the Davis--Kahan theorem (or, more generally, existing eigenspace perturbation results). In particular, the stability of such an operator depends crucially on whether there exists a nonvanishing eigengap separating the first $q$ eigenvalues from the rest of the spectrum. In some cases, it is even helpful to understand the behavior of each individual eigengap.

Therefore, as a first step, we examine whether the key theoretical result on the eigengap, namely Lemma \ref{lemma:div_component_eval_specific_form}, holds as claimed in our paper.

\paragraph{Case 1: heterogeneous memory across factors.}
We consider a rank-\(2\) factor model where
\begin{align}\label{eq:rank_2_heter_factor_simulation_model}
x_{it}
=
\frac{(1-L)^{-0.40}}{1-\alpha_{i1}L}\,u_{1t}
+
\frac{(1-L)^{-0.25}}{1-\alpha_{i2}L}\,u_{2t}
+
\xi_{it},
\qquad i=1,\dots,n.
\end{align}

Figure \ref{fig:eigengap_rank_2_factor_d} shows that indeed, the eigengap are of order $\mathcal{O}(n|\theta|^{-2d_2}), d_2 = 0.25$. 

\begin{figure}[htbp]
    \centering
    \includegraphics[width=1.2\textwidth]{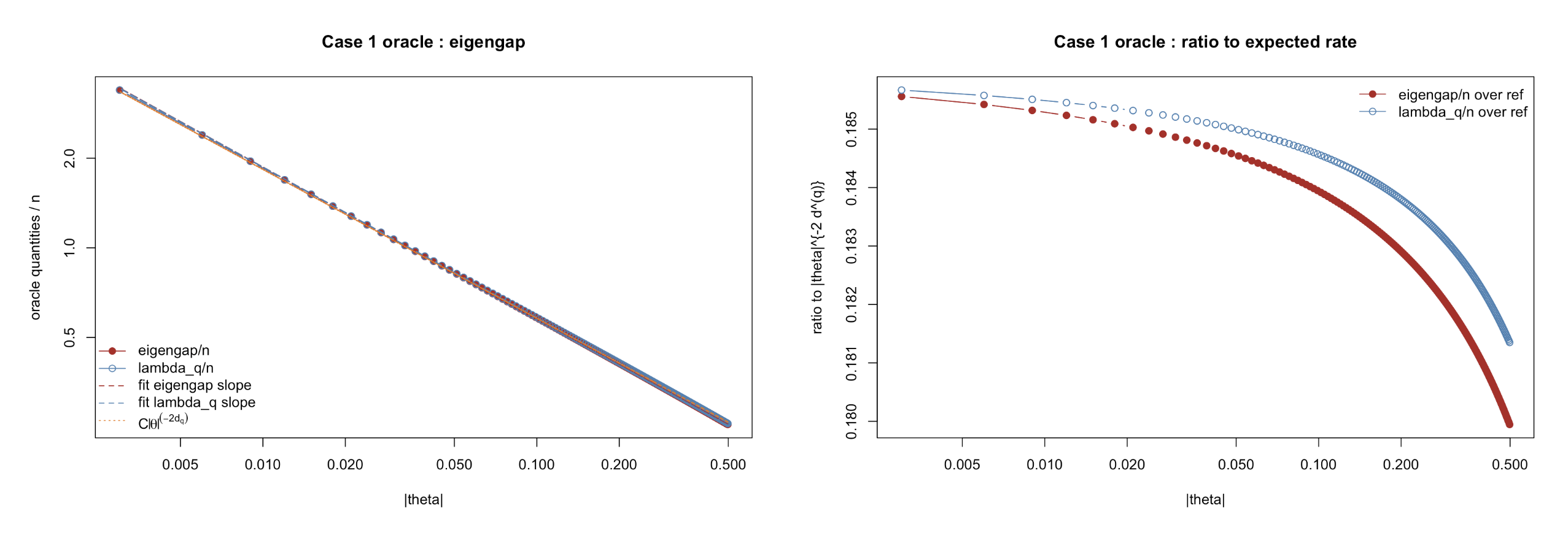}
    \caption{Oracle eigengap check with \(q=2\).}
    \label{fig:eigengap_rank_2_factor_d}
\end{figure}

\subsection{Simulation of the rate of the main part}
\paragraph{Case 1: heterogeneous memory across factors.} 

First of all, we would like to provide simulation evidence supporting the rates of the inner structure, i.e. the \(L^1\) convergence rate of \(\widehat{\underline{\mathbf{K}}}_{n i}(\theta)\) in Lemma \ref{lemma:main_theorem_of_spectral_estimate}. 

\begin{figure}[htbp]
    \centering
    \includegraphics[width=1.2\textwidth]{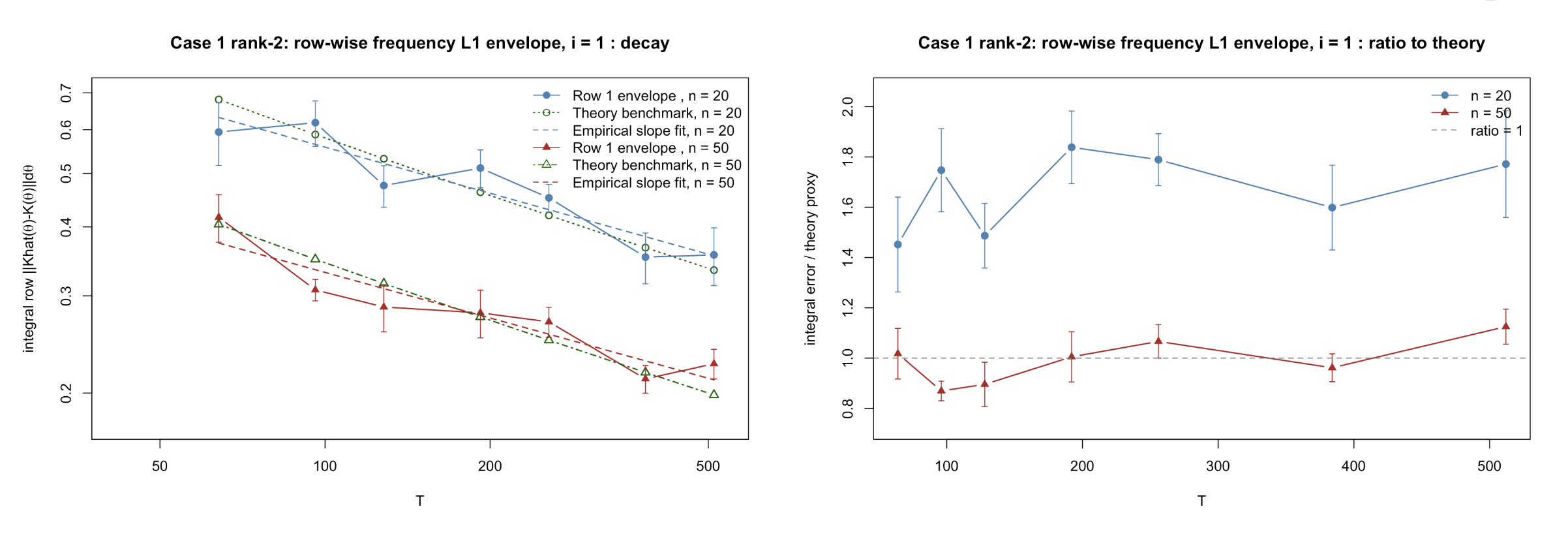}
    \caption{$L_1$ convergence rate for $\widehat{\underline{\mathbf{K}}}_{n i}(\theta)$ }
    \label{fig:L1_conv_rate_for_K_est_factor_differ}
\end{figure}

Figure \ref{fig:L1_conv_rate_for_K_est_factor_differ} shows the approximated value of  $\mathbb{E} \left[\int_{\Pi} \left\| \widehat{\underline{\mathbf{K}}}_{n i}(\theta) - \underline{\mathbf{K}}_{n i}(\theta)  \right\|_{op}   d\theta  \right]$ compares to the theoretical \(L^1\) convergence rate of \(\widehat{\underline{\mathbf{K}}}_{n i}(\theta)\) in Lemma \ref{lemma:main_theorem_of_spectral_estimate}, that is,
\begin{align*}
    &\delta^{(T,n)} = B_T^{1-2\Delta} + \frac{1}{TB_T} + \frac{ \log^{\frac{3}{2}}T}{\sqrt{T}}.
\end{align*}
According to the simulation, our bound to \(L^1\) convergence rate of \(\widehat{\underline{\mathbf{K}}}_{n i}(\theta)\) is tight.

Secondly, we examine, under the model \eqref{eq:rank_2_heter_factor_simulation_model}, whether our \(r^{(n)}\) captures the variance of the main term
\[
Z_{i,t} = \left|\sum_{h=-M(T)}^{M(T)}\left(\widehat{\mathbf{K}}_{n i, h}-\mathbf{K}_{n i, h}\right) L^h X_t\right|,
\]
where, in the simulation below, we manually fix \(i = 1, t = 50\) and let \(T = 160, 192, 256, 384, 512, 768, 1024\), \(n = 20, 50\), and \(B_T = T^{-b}\), with \(b = 0.4\). For the choice of \(M(T)\), we apply the truncation rule
\[
M(T,n)
=
\left\lfloor
\left\{
\frac{C_M}{\delta^{(T,n)}\sqrt n}
\right\}^{\beta}
\right\rfloor, \qquad \beta \in (0,1),
\]
and since \(n\) is fixed in the numerical experiment, up to an \(n\)-dependent constant, such an \(M(T,n)\) is equivalent to $M(T)\asymp \{\delta^{(T,n)}\}^{-\beta}.$
In the present design, \(\delta^{(T,n)}\approx B_T^{(1-2\Delta)} = T^{-b(1-2\Delta)} = T^{-0.28}\).

We report the moment approximation result for \(\sqrt{\mathbb E Z_{i,t}^2}\), and examine whether it aligns with our rate $(r^{(n)})^{-1} = M(T)\,\delta^{(T,n)}\sqrt{n}.$
We also include \(\delta^{(T,n)}\sqrt{n}\) as a benchmark, and explore whether the bound should depend on \(M(T)\).

Specifically, for each pair \((T,n)\), let \(Z_{i,t}^{(r)}\) denote the value of \(Z_{i,t}\) from the \(r\)-th replication, \(r=1,\ldots,R\). We estimate \(\sqrt{\mathbb E Z_{i,t}^2}\) by
\[
\left\{
\frac{1}{R}
\sum_{r=1}^R
\left(Z_{i,t}^{(r)}\right)^2
\right\}^{1/2}.
\]

Figure \ref{fig:sample_var_vs_T_beta_06} shows the result when \(\beta = 0.6\).
\begin{figure}[htbp]
    \centering
    \includegraphics[width=1.2\textwidth]{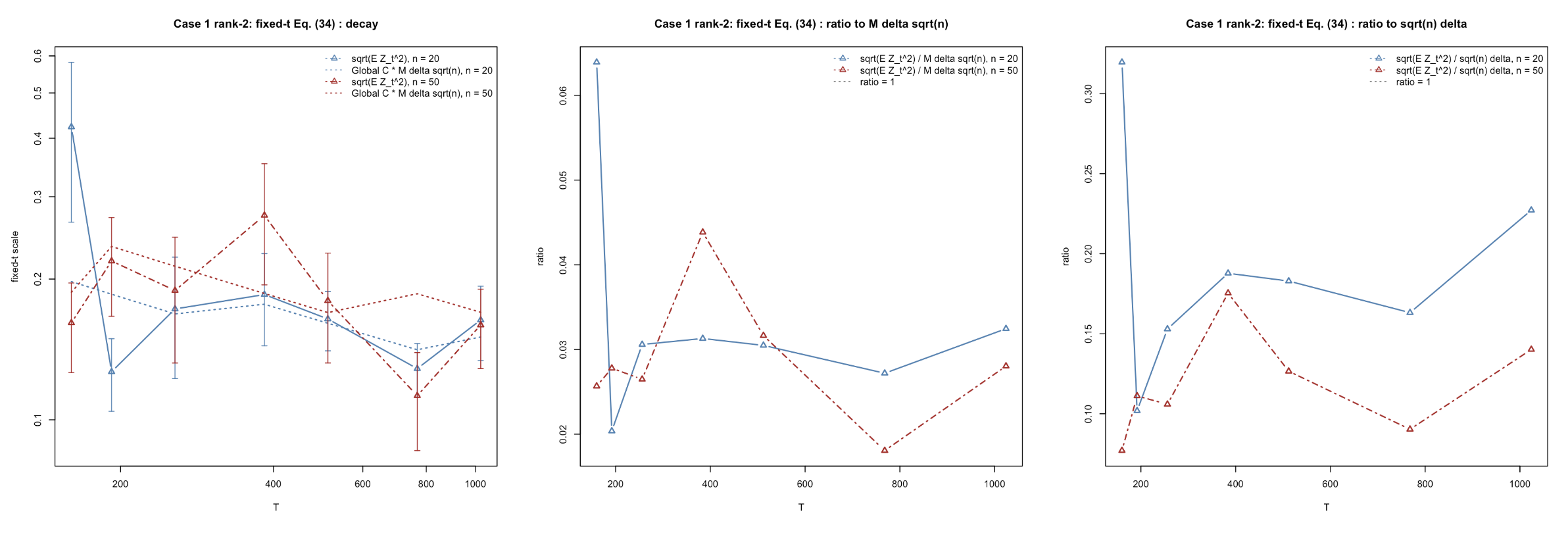}
    \caption{Sample variance of the main term}
    \label{fig:sample_var_vs_T_beta_06}
\end{figure}

The second plot shows the ratio relative to the rate \(M(T)\delta^{(T,n)}\sqrt{n}\). 
The third plot shows the ratio relative to the rate \(\delta^{(T,n)}\sqrt{n}\).



Based on Figure \ref{fig:sample_var_vs_T_beta_06}, we do not get a clear picture of whether this bound can be improved. In any case, it would be interesting to see whether there is an alternative route to obtain the rate suggested by this structure.

\subsection{Simulation to check the decay rate of $\left\|\mathbf{K}_{ni,h}\right\|_2$}
\subsubsection{Long-memory parameters vary factor-wise}

For Assumption \ref{assump: spec_den_specific_form} case (a), where the memory parameters differ across factors, we present simulations matching the theory in Lemma \ref{lemma:d_differ_K_piecewise_cont_fourier_coef_decay_rate}. 

We consider a rank-2 factor-memory model
\[
x_t = B(L)u_t + \xi_t,
\]
where 
\[
B_{i\ell}(\theta)=c_{i\ell}\frac{(1-e^{-i\theta})^{-d_\ell}}{1-\alpha_i e^{-i\theta}},
\qquad \ell=1,2,
\]
with factor-specific memory parameters \(d_1=0.35\) and \(d_2=0.1\), and row-specific short-memory parameters \(\alpha_i\) equally spaced between \(0.2\) and \(0.8\). We set \(n=80\). The idiosyncratic component \(\xi_t\) is modeled as a dense short-memory AR(1)-type process with cross-sectional covariance
\[
\Sigma_\xi(\theta)=\frac{\Sigma_0}{|1-\phi_\varepsilon e^{-i\theta}|^2},
\qquad
(\Sigma_0)_{ij}=\sigma_\varepsilon r_{cs}^{|i-j|},
\]
with \(\sigma_\varepsilon=1\), \(\phi_\varepsilon=0.4\), and \(r_{cs}=0.6\). See Figure \ref{fig:K_h_factor_pert_rank_2}.

\begin{figure}[htbp]
    \centering
    \includegraphics[width=0.8\textwidth]{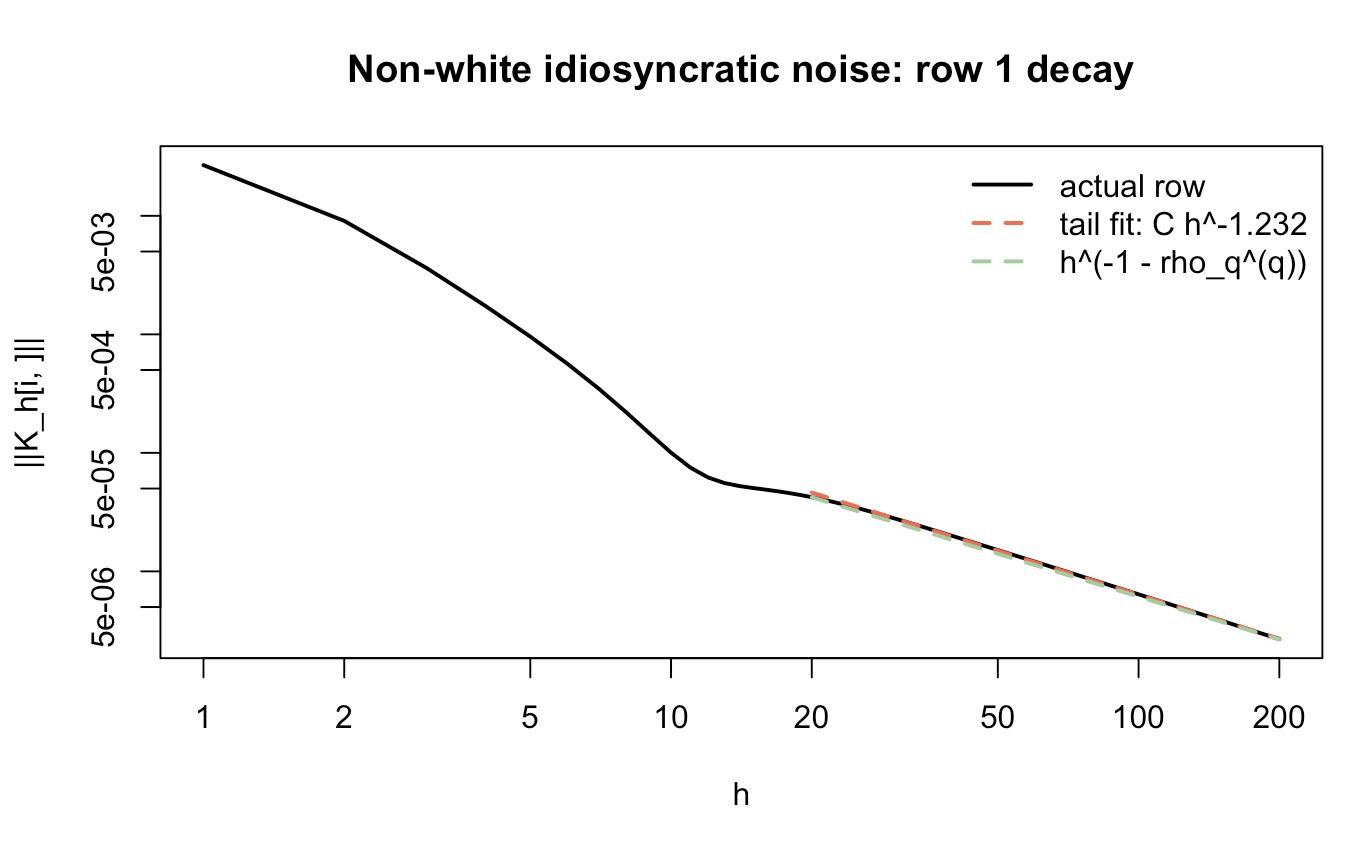}
    \caption{\(\left\|\mathbf{K}_{ni,h}\right\|\) decay rate: \(\left\|\mathbf{K}_{ni,h}\right\|\) vs. \(|h|\) on a log-log scale}
    \label{fig:K_h_factor_pert_rank_2}
\end{figure}

We also consider a rank-3 factor-memory model
\[
x_t = B(L)u_t + \xi_t,
\]
where 
\[
B_{i\ell}(\theta)=c_{i\ell}\frac{(1-e^{-i\theta})^{-d_\ell}}{1-\alpha_i e^{-i\theta}},
\qquad \ell=1,2,3,
\]
with factor-specific memory parameters \(d_1=0.45\), \(d_2=0.25\), and \(d_3=0.2\). The other parameters are taken to be the same as in the rank-2 example. See Figure \ref{fig:K_h_factor_pert_rank_3}

\begin{figure}[htbp]
    \centering
    \includegraphics[width=0.8\textwidth]{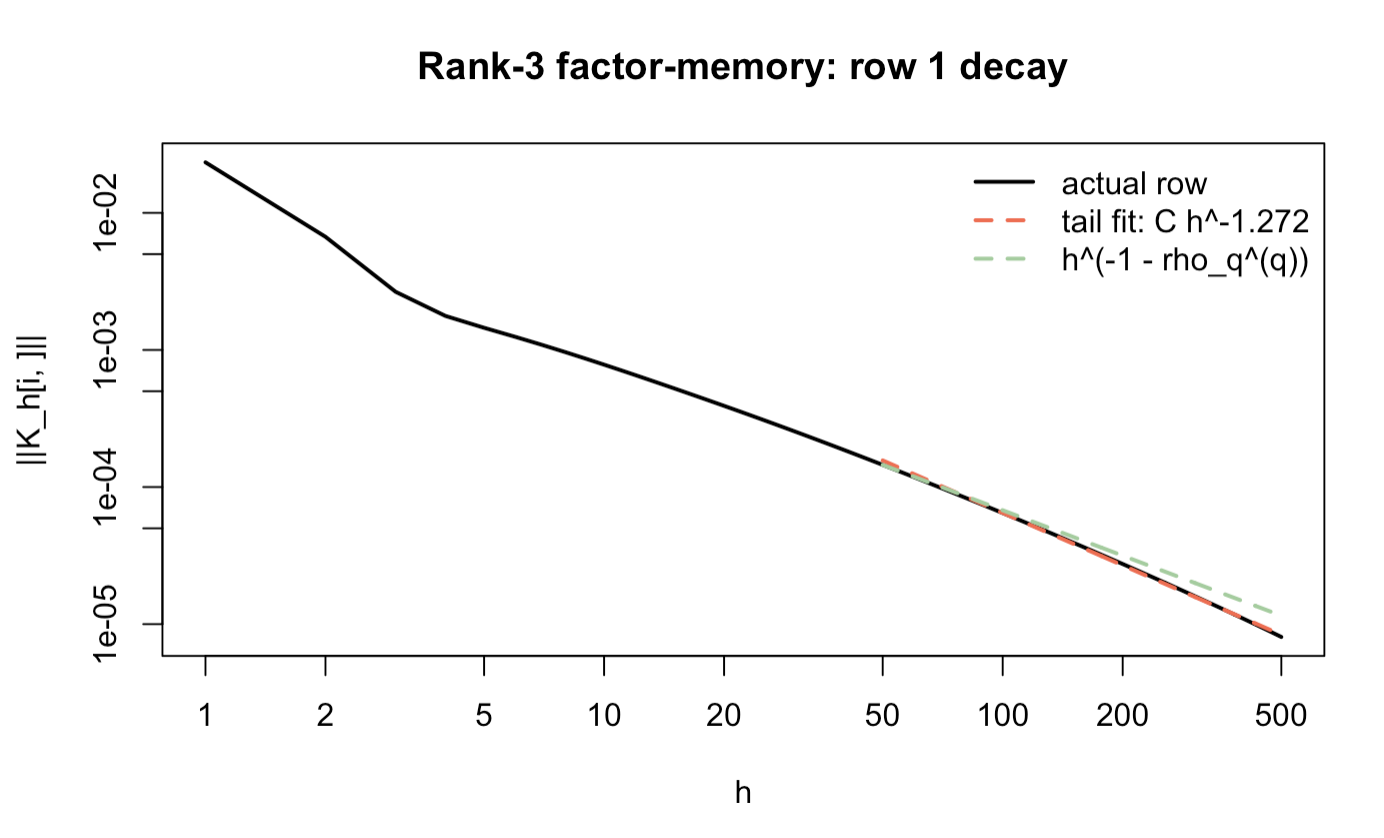}
    \caption{\(\left\|\mathbf{K}_{ni,h}\right\|\) decay rate: \(\left\|\mathbf{K}_{ni,h}\right\|\) vs. \(|h|\) on a log-log scale}
    \label{fig:K_h_factor_pert_rank_3}
\end{figure}

\subsubsection{Long-memory parameters vary row-wise}

For Assumption \ref{assump: spec_den_specific_form} case (b), where the memory parameters differ across rows, we present simulations suggesting that the decay rate proved in \eqref{eq:decay_rate_K_nih} may be further improved. Specifically, $\left\|\mathbf{K}_{ni,h}\right\|_2$ likely exhibits the same order of decay as in Assumption \ref{assump: spec_den_specific_form} case (a), namely, the rate established in Lemma \ref{lemma:d_differ_K_piecewise_cont_fourier_coef_decay_rate}.


Consider a simple rank-1 common-component case, where 
\[
x_t = B(L)u_t + \xi_t,
\]
where
\[
B(\theta)=\bigl[B_{i1}(\theta)\bigr]_{i=1,\dots,n},
\qquad
B_{i1}(\theta)=\frac{(1-e^{-\iota\theta})^{-d_i}}{1-\alpha_{i}e^{-\iota\theta}},
\]
with
\[
d_i=
\begin{cases}
0.35, & i\neq 10,\\[4pt]
0.25, & i=10.
\end{cases}
\]
and $\alpha_i$ are taken to be equally spaced between 0.2 and 0.8.

Figure \ref{fig:K_h_one_pert_rank_1} shows that as $h \rightarrow \infty$, $\left\|\mathbf{K}_{ni,h}\right\|$ decreases at the rate of $h^{-1.095}$, which is much faster than $h^{2\Delta -1} = h^{-0.8}$.

\begin{figure}[htbp]
    \centering
    \includegraphics[width=0.8\textwidth]{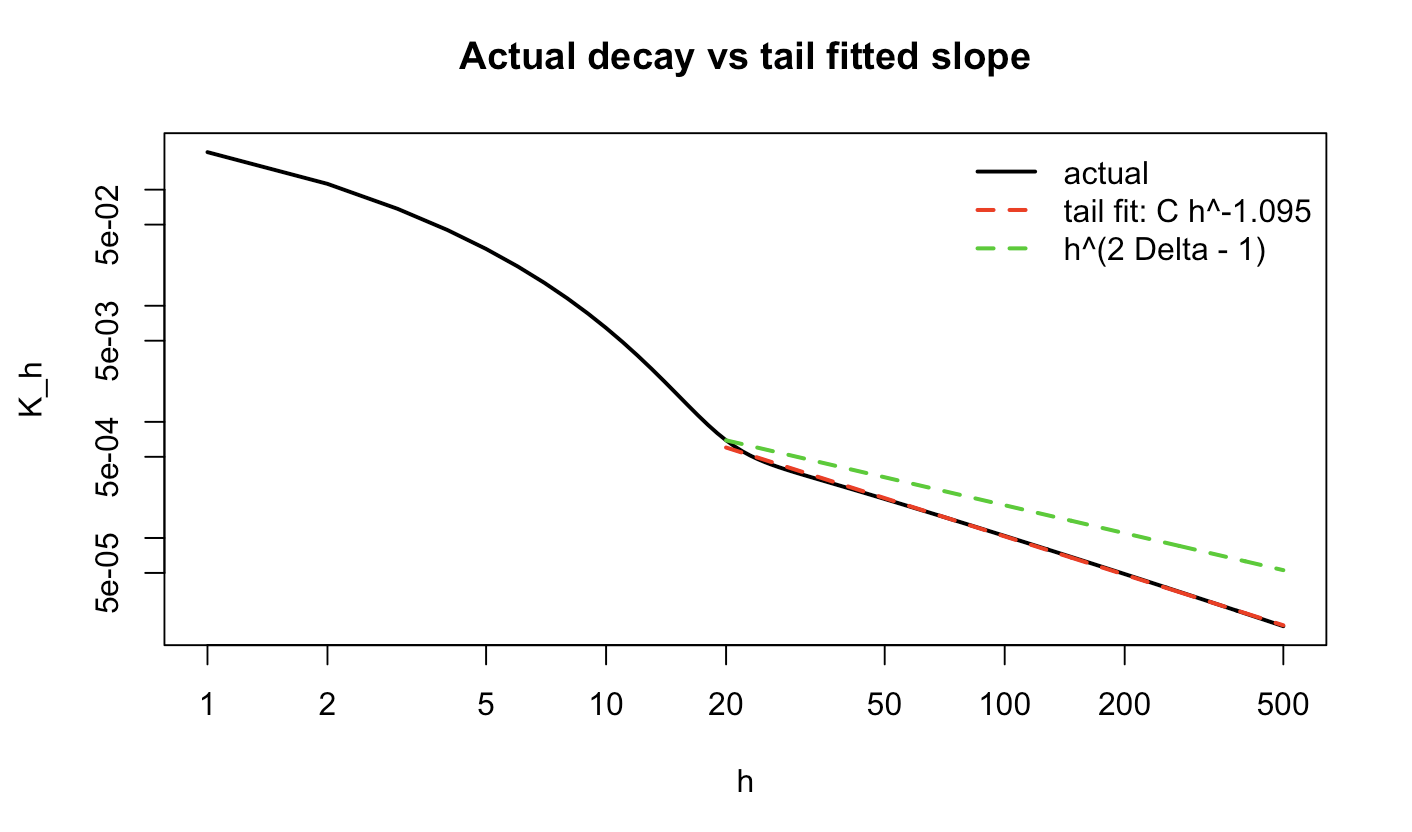}
    \caption{$\left\|\mathbf{K}_{ni,h}\right\|$ decay rate, $\left\|\mathbf{K}_{ni,h}\right\|$ vs. $|h|$ on a log-log scale}
    \label{fig:K_h_one_pert_rank_1}
\end{figure}

We then examine the effect of increasing \(\Delta\). Using the same rank-1 model, we vary the perturbation size over the grid
\[
\Delta \in \{0,\,0.001,\,0.005,\,0.01,\,0.02,\,0.05,\,0.1,\,0.15,\,0.2,\,0.25,\,0.3\}.
\]
For each value of \(\Delta\), we compute \(\left\|\mathbf{K}_{ni,h}\right\|\) and estimate its tail decay rate by regressing 
\(\log \left\|\mathbf{K}_{ni,h}\right\|\) on \(\log h\), using only the large-\(h\) region \(h\ge 80\). The results are shown in Figure \ref{fig:tail_slope_one_row_pert_h_80}. 

\begin{figure}[htbp]
    \centering
    \includegraphics[width=0.8\textwidth]{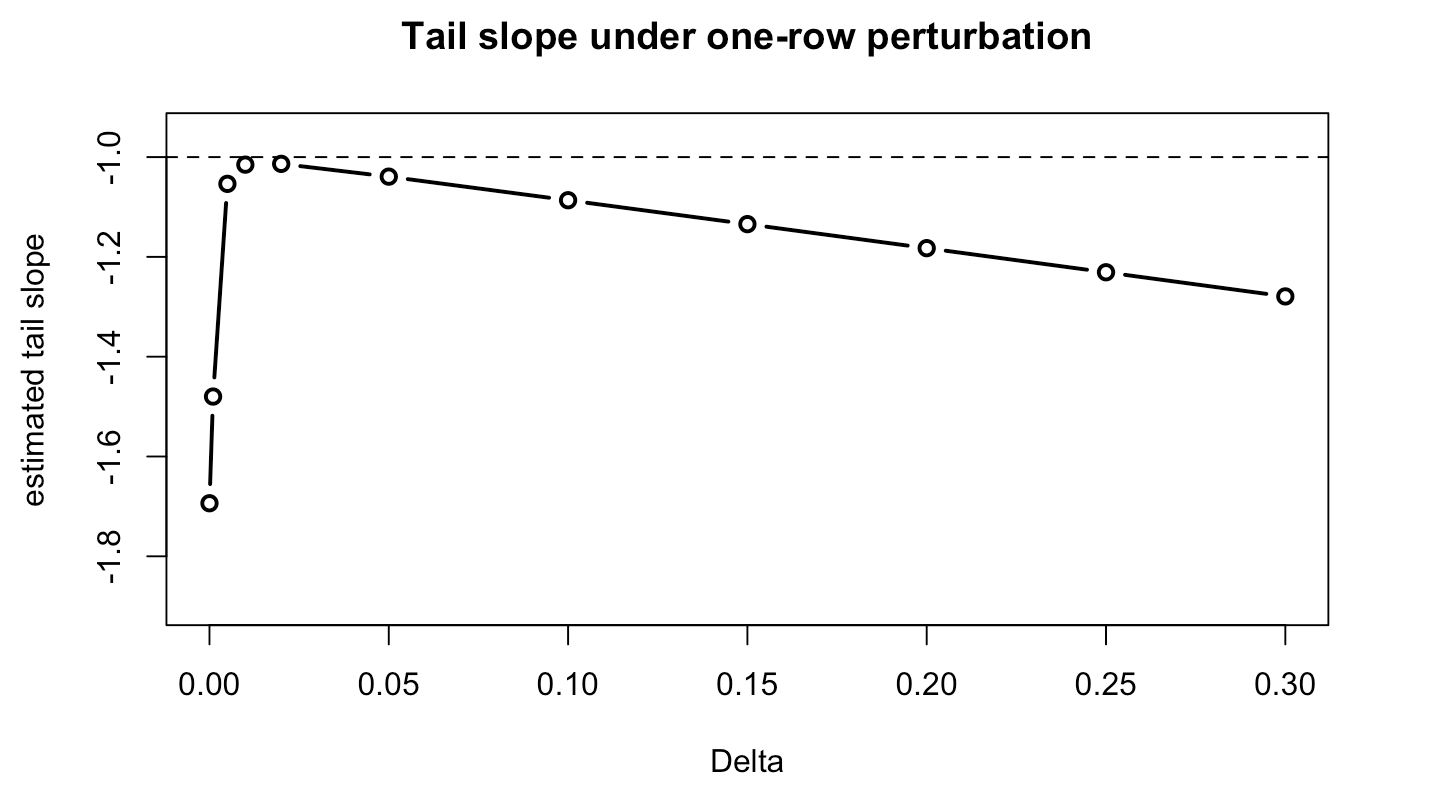}
    \caption{tail slope vs. perturbation size}
    \label{fig:tail_slope_one_row_pert_h_80}
\end{figure}

\subsubsection{Long-memory parameters vary entry-wise}

Although this is not the case included in Assumption \ref{assump: spec_den_specific_form}, we still present some results to show that the bound in Lemma \ref{lemma:d_differ_K_fourier_coef_decay_rate_old} holds in certain scenarios, while in some situations, when the model admits a more organized structure, it has a faster decay rate.

Consider the rank-2 matrix below, with only one entry perturbed:
\[
x_t = B(L)u_t + \xi_t,
\]
where
\[
B(\theta)=\bigl[B_{i\ell}(\theta)\bigr]_{i=1,\ldots,n;\,\ell=1,2},
\qquad
B_{i\ell}(\theta)=\frac{(1-e^{-\iota\theta})^{-d_{i\ell}}}{1-\alpha_i e^{-\iota\theta}},
\]
with
\[
d_{i\ell}=
\begin{cases}
0.25, & (i,\ell)=(10,1),\\[4pt]
0.35, & \text{otherwise}.
\end{cases}
\]
and \(\alpha_i\) are taken to be equally spaced between \(0.2\) and \(0.8\). The error component \(\xi_t\) is an orthonormal white noise.

Figure \ref{fig:K_h_one_pert_rank_2} shows that as \(h \rightarrow \infty\), \(\left\|\mathbf{K}_{ni,h}\right\|\) decreases at the rate of \(h^{-0.754}\). In this case, where the long-memory parameters vary entry-wise and do not admit either case in Assumption \ref{assump: spec_den_specific_form}, the result aligns with the bound in Lemma \ref{lemma:d_differ_K_fourier_coef_decay_rate_old}.

\begin{figure}[htbp]
    \centering
    \includegraphics[width=0.8\textwidth]{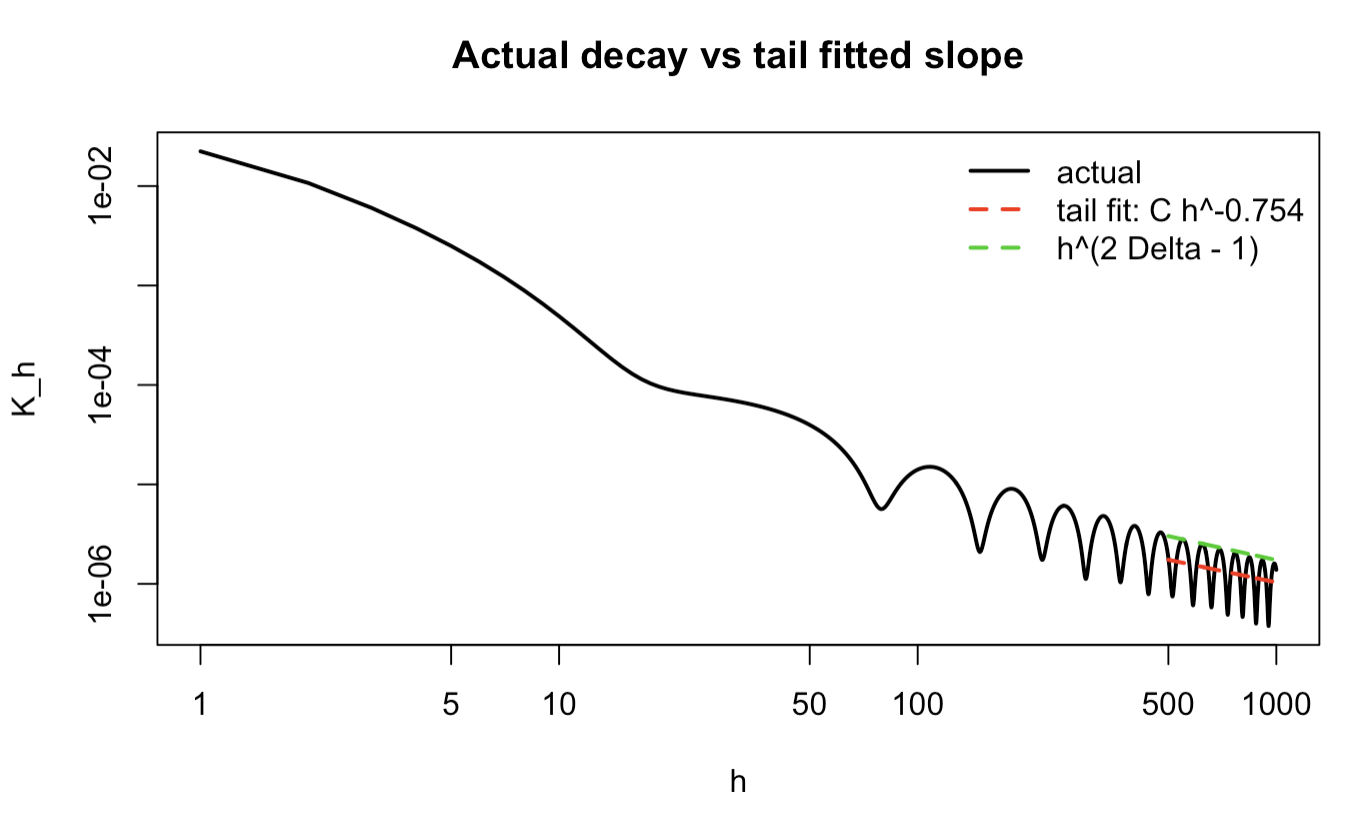}
    \caption{$\left\|\mathbf{K}_{ni,h}\right\|\) decay rate, $\left\|\mathbf{K}_{ni,h}\right\|\) vs. $|h|$ on a log-log scale}
    \label{fig:K_h_one_pert_rank_2}
\end{figure}

Now we show that, even in the entry-wise case, once it admits some structure, it is very likely that \(\left\|\mathbf{K}_{ni,h}\right\|\) decays at a faster rate. Consider the rank-2 matrix below:
\[
x_t = B(L)u_t + \xi_t,
\]
where
\[
B(\theta) = [B_{i\ell}(\theta)]_{i=1,\ldots,n;\,\ell=1,2},
\]
with
\[
B_{i1}(\theta)
=
\frac{(1-e^{-i\theta})^{-d_{i1}}}{1-\alpha_{i1}e^{-i\theta}},
\qquad
B_{i2}(\theta)
=
c\,
\frac{(1-e^{-i\theta})^{-d_{i2}}}{1-\alpha_{i2}e^{-i\theta}},
\]
and \(d_{i1}\) are taken to be equally spaced between \(0.1\) and \(0.3\); \(d_{i2}\) are taken to be equally spaced between \(0.15\) and \(0.35\); while \(\alpha_{i1}\) are taken to be equally spaced between \(0.2\) and \(0.7\); and \(\alpha_{i2}\) are taken to be equally spaced between \(0.5\) and \(0.9\). The idiosyncratic error component \(\xi_t\) is taken to be a diagonal short-memory AR(1) process.

Figure \ref{fig:K_h_equal_pert_rank_2}, \ref{fig:K_h_equal_pert_rank_2_zoomin} show that as \(h \rightarrow \infty\), \(\left\|\mathbf{K}_{ni,h}\right\|\) decreases at the rate of \(h^{-1.042}\), and this is much faster than \(h^{2\Delta -1} = h^{-0.8}\).

\begin{figure}[htbp]
    \centering
    \includegraphics[width=0.8\textwidth]{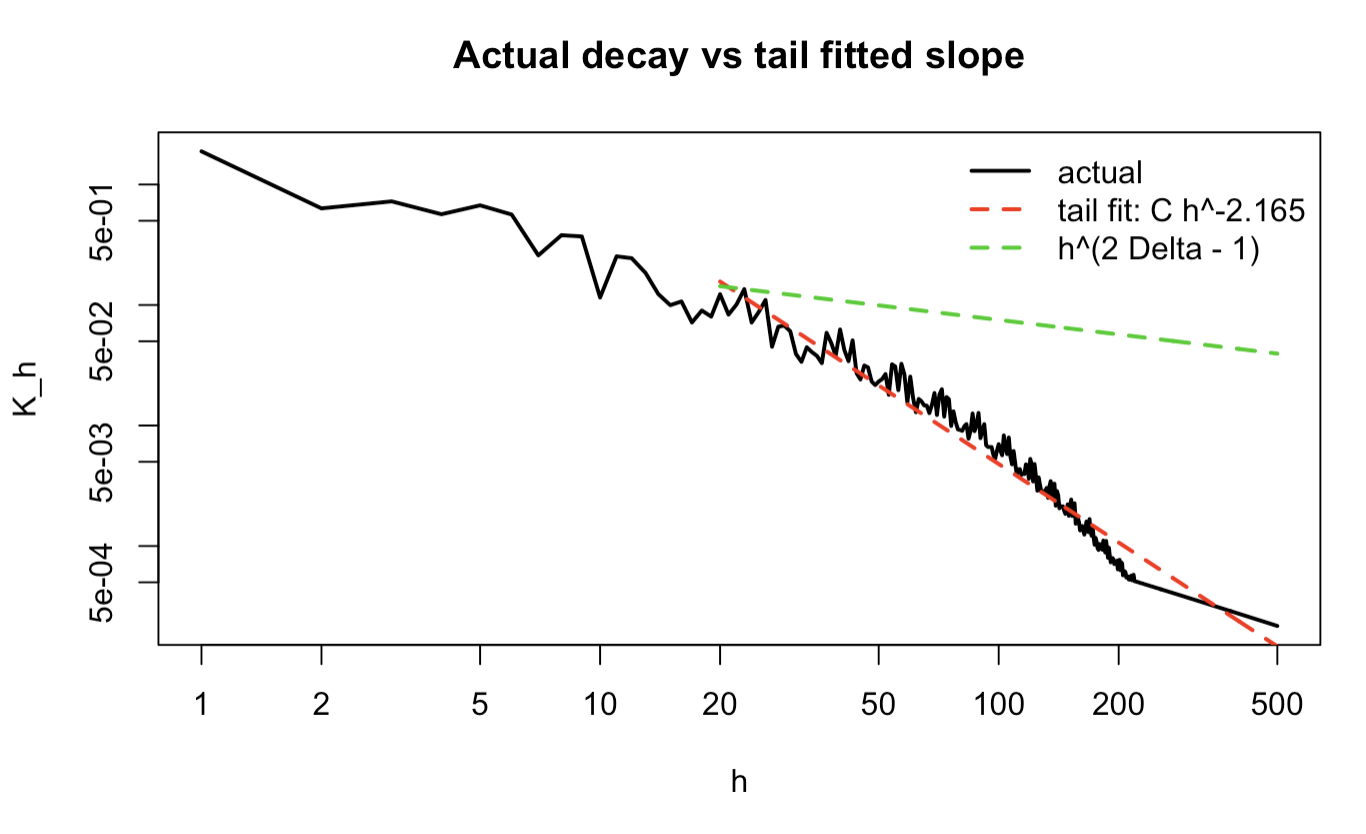}
    \caption{$\left\|\mathbf{K}_{ni,h}\right\|$ decay rate, $\left\|\mathbf{K}_{ni,h}\right\|$ vs. $|h|$ on a log-log scale}
    \label{fig:K_h_equal_pert_rank_2}
\end{figure}

\begin{figure}[htbp]
    \centering
    \includegraphics[width=0.8\textwidth]{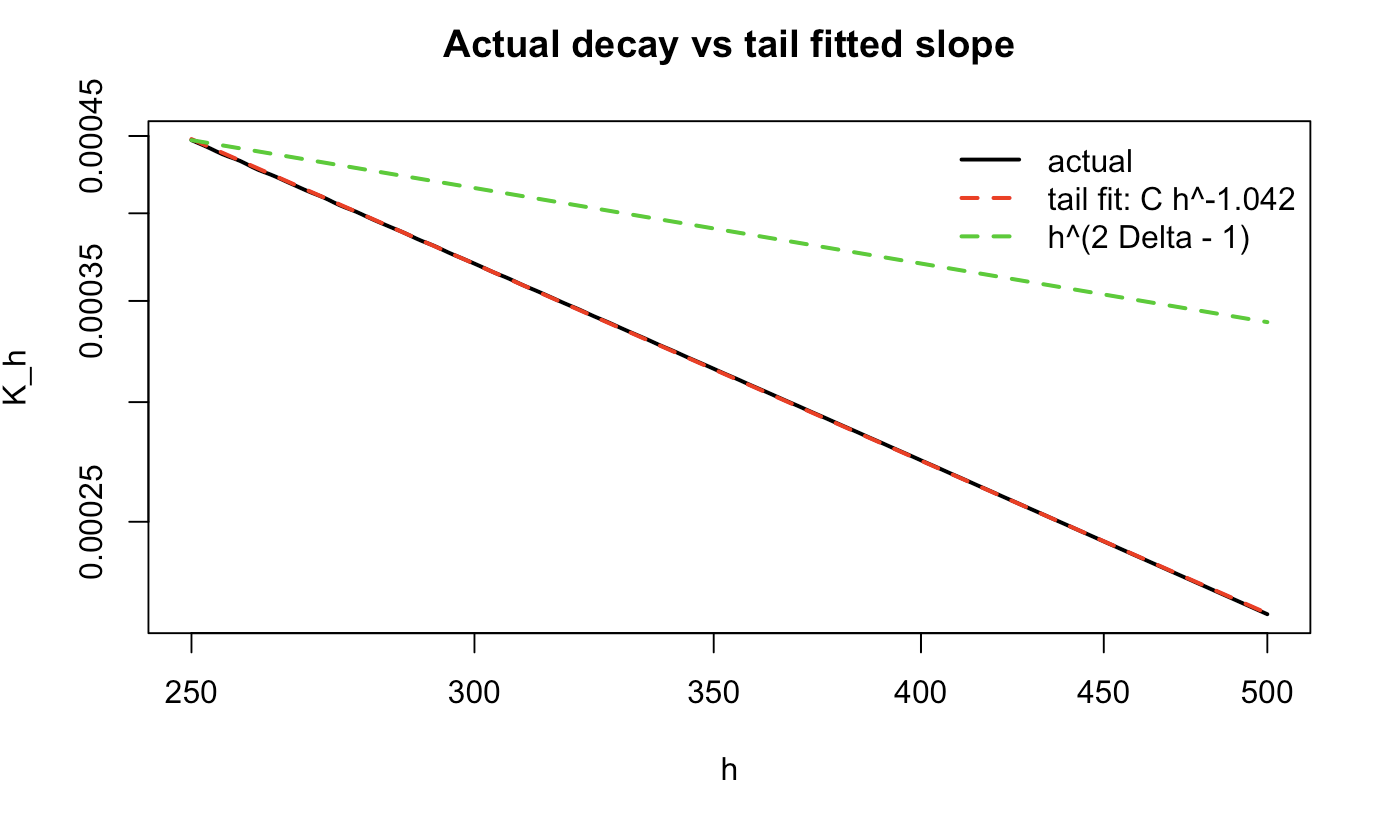}
    \caption{Figure \ref{fig:K_h_equal_pert_rank_2} zoom in. $\left\|\mathbf{K}_{ni,h}\right\|$ decay rate, $\left\|\mathbf{K}_{ni,h}\right\|$ vs. $|h|$ on a log-log scale}
    \label{fig:K_h_equal_pert_rank_2_zoomin}
\end{figure}

\section{Conclusion}\label{sec:conclusion}

In this paper, we examine the methodology of \cite{dynamicfactorForni2sided} under a long-memory assumption. The model we study is the generalized dynamic factor model (GDFM) proposed by \cite{dynamicfactorForni2sided}, which provides a flexible framework for modeling high-dimensional time series. The long-memory structure of the common component creates substantial technical challenges, since it introduces unboundedness and discontinuity in the spectral density. We address these challenges and establish convergence rates for the estimated common component.

Our analysis is based on the traditional discrete periodogram smoothing estimator for the spectral density. This choice is motivated by two considerations. First, in the factor model structure, where the observed process consists of a common component plus noise, the regularity of the short-memory component of the spectral density may be degraded. This makes efficient estimation of the memory parameters difficult, especially when different rows have different memory parameters, and complicates the use of prewhitening techniques. Second, our analysis mainly requires integrability properties of the spectral density estimator rather than sharp pointwise control. We show that the discrete periodogram smoothing estimator provides $L^1$-consistent estimates of the spectral density, even though its pointwise behavior may be less satisfactory near the singularity.

A key part of our proof is the control of the variation component of the estimator. To this end, we extend the results of \cite{kim_cumulant_long_memory} to the multivariate setting. These results may also be useful in other problems involving long-memory multivariate time series and may be of independent interest. In addition, we establish regularity properties of the eigenvectors and eigenprojections of the spectral density matrix function in the long-memory factor model setting. This regularity is essential for controlling the truncation remainder term in the two-sided estimation procedure.

Several questions remain open. For example, when the long-memory parameters differ across rows, the regularity of the corresponding eigenstructure has not yet been fully characterized. A sharper understanding of this case may lead to improved convergence rates and a more complete theory for generalized dynamic factor models with heterogeneous long-memory behavior. We leave this direction for future research.

\clearpage
\appendix
We provide a roadmap for the proof of Proposition~\ref{thm:main_theorem} in Figure~\ref{fig:proof-roadmap}.

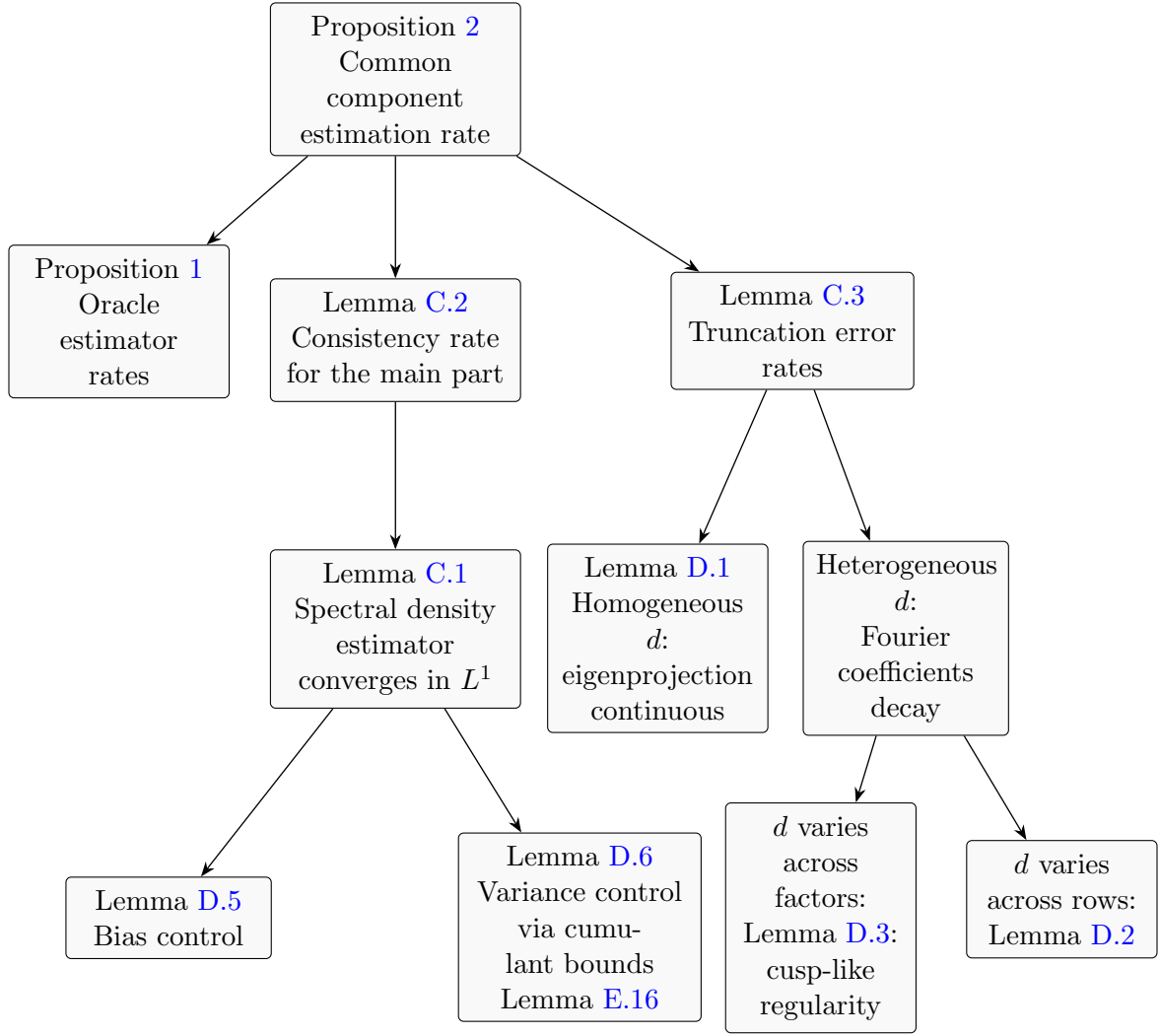
\begin{figure}[htbp]
\hspace*{-0.5cm}
\begin{tikzpicture}[
    x=0.95cm,
    y=1.25cm,
    >=Stealth,
    font=\small,
    box/.style={
        draw,
        rounded corners=2pt,
        align=center,
        inner sep=5pt,
        line width=0.4pt,
        fill=gray!5,
        minimum height=0.95cm
    },
    arr/.style={
        -{Stealth[length=2mm]},
        line width=0.5pt
    }
]

\node[box, text width=3.0cm] (prop2) at (0,0)
{Proposition \ref{thm:main_theorem}\\Common component\\estimation rate};

\node[box, text width=2.6cm] (prop1) at (-3.9,-2.6)
{Proposition \ref{proposition_oracle}\\Oracle estimator\\rates};

\node[box, text width=3.0cm] (b2) at (0,-2.8)
{Lemma \ref{lemma:consistency_main_part_estimator}\\Consistency rate for the main part};

\node[box, text width=2.9cm] (b3) at (5.6,-2.7)
{Lemma \ref{lemma:consistency_residual_part_estimator}\\Truncation error\\rates};

\node[box, text width=3.0cm] (b1) at (0,-5.9)
{Lemma \ref{lemma:main_theorem_of_spectral_estimate}\\Spectral density\\estimator\\converges in $L^1$};

\node[box, text width=2.4cm] (c3) at (-3.2,-9.0)
{Lemma \ref{lemma:bias_conv}\\Bias control};

\node[box, text width=2.9cm] (c4) at (2.6,-9.1)
{Lemma \ref{lm:var_conv}\\Variance control\\via cumulant bounds Lemma \ref{lm:cum_property_long_mem}};

\node[box, text width=2.6cm] (same) at (3.7,-6.0)
{Lemma \ref{lemma:holder_cont_K}\\Homogeneous $d$:\\eigenprojection\\continuous};

\node[box, text width=2.4cm] (diff) at (7.2,-6.0)
{Heterogeneous $d$:\\Fourier\\coefficients\\decay};

\node[box, text width=2.2cm] (extra1) at (6.0,-9.0)
{$d$ varies across factors:\\Lemma \ref{lemma:K_holder_factor_wise}:\\cusp-like\\regularity };

\node[box, text width=2.2cm] (extra2) at (9.4,-8.8)
{$d$ varies across rows:\\Lemma \ref{lemma:d_differ_K_fourier_coef_decay_rate_old}};

\draw[arr] (prop2) -- (prop1);
\draw[arr] (prop2) -- (b2);
\draw[arr] (prop2) -- (b3);

\draw[arr] (b2) -- (b1);

\draw[arr] (b1) -- (c3);
\draw[arr] (b1) -- (c4);

\draw[arr] (b3) -- (same);
\draw[arr] (b3) -- (diff);
\draw[arr] (diff) -- (extra1);
\draw[arr] (diff) -- (extra2);

\end{tikzpicture}
\caption{Roadmap for the proof of Proposition 2.}
\label{fig:proof-roadmap}
\end{figure}

\section{Proof of Lemma in Assumption Section}\label{sec:assump_proof}
\begin{proof}
\begin{enumerate}
    \item $d_{il}=d_l$.

$\Sigma_{\chi}(\theta) =  G(\theta) D(\theta) D^*(\theta) G^*(\theta) = G(\theta) \Bar{D}^{2}(\theta) G^*(\theta)$ where $\Bar{D}(\theta) = diag\left( |1-e^{-\iota \theta}|^{-d_l} \right)_{l = 1,...,q}$. Notice that $G(\theta) \Bar{D}^{2}(\theta) G^*(\theta)$ shares the same eigenvalues as $\Bar{D}(\theta)G^*(\theta)G(\theta)\Bar{D}(\theta)$.

Under Assumption \ref{assump:pervasive_short_mem} we have, 
\begin{align*}
        s^- \Bar{D}^2(\theta) \preceq \Bar{D}(\theta) S(\theta) \Bar{D}(\theta) \preceq s^+ \Bar{D}^2(\theta)
    \end{align*}
Multiplying by a diagonal matrix preserves the PSD order,
\begin{align*}
    n s^- \Bar{D}^2(\theta) \preceq \Sigma_{\chi}(\theta) \preceq n s^+ \Bar{D}^2(\theta)
\end{align*}
Eigenvalues are monotone under $\preceq$, thus $n s^- |1-e^{-\iota \theta}|^{-2d^{(r)}} \leq \lambda^{\chi}_{nr}(\theta) \leq n s^+ |1-e^{-\iota \theta}|^{-2d^{(r)}} $, for $r = 1,...,q.$

In this case,  \begin{equation*}
\alpha_q^{\chi}(\theta)=c_q^{-}\,|\theta|^{-2d^{(q)}},
\qquad
\beta_0^{\chi}(\theta)=c_{1}^{+}\,|\theta|^{-2d^{(1)}} .
\end{equation*}

\item $d_{il}=d_i$.\\
 
$\Sigma_{\chi}(\theta) =  D_n(\theta)G(\theta) G^*(\theta)D^*_n(\theta)$, where $D_n(\theta) = diag\left( (1-e^{-\iota \theta})^{-d_i} \right)_{i = 1,...,n}. $ 

Notice that $D_n(\theta)G(\theta) G^*(\theta)D^*_n(\theta)$ shares the same eigenvalues as $G^*(\theta) \Bar{D}_n^{2}(\theta) G(\theta)$ where $\Bar{D}_n(\theta) = diag\left( |1-e^{-\iota \theta}|^{-d_i} \right)_{i = 1,...,n} = diag\left( |2 \sin(\frac{\theta}{2})|^{-d_i} \right)_{i = 1,...,n}$. Since
\begin{align*}
    |2 \sin(\frac{\theta}{2})|^{-2d^{(n)}} \preceq \Bar{D}_n^{2}(\theta) \preceq |2 \sin(\frac{\theta}{2})|^{-2d^{(1)}}, |\theta| < \pi/3 
\end{align*}
then
\begin{align*}
    n\left|2 \sin\left({\theta}/{2}\right)\right|^{-2d^{(n)}}S(\theta) \preceq G^*(\theta) \Bar{D}_n^{2}(\theta) G(\theta) \preceq n\left|2 \sin\left({\theta}/{2}\right)\right|^{-2d^{(1)}} S(\theta)
\end{align*}
Thus, under Assumption \ref{assump:pervasive_short_mem} the eigenvalues of $\Sigma_{\chi}(\theta)$ satisfies
\begin{align*}
    & n  |1-e^{-\iota \theta}|^{-2d^{(n)}}\lambda_S^{(r)}(\theta) \leq \lambda^{\chi}_{nr}(\theta) \leq n |1-e^{-\iota \theta}|^{-2d^{(1)}} \lambda_S^{(r)}(\theta) ,\\
    &n s^{-} \left|2 \sin\left({\theta}/{2}\right)\right|^{-2d^{(n)}} \leq \lambda^{\chi}_{nr}(\theta)  \leq n s_r^{+} \left|2 \sin\left({\theta}/{2}\right)\right|^{-2d^{(1)}}, |\theta| < \pi/3
\end{align*}
For $|\theta|>\pi/3$, the inequality is 
\begin{align*}
    &n s^{-} \left|2 \sin\left({\theta}/{2}\right)\right|^{-2d^{(1)}} \leq \lambda^{\chi}_{nr}(\theta)  \leq n s_r^{+} \left|2 \sin\left({\theta}/{2}\right)\right|^{-2d^{(n)}}, |\theta| > \pi/3
\end{align*}

Since our paper mainly focuses on the low frequency around zero. For simplicity, we generally take the lower bound to be $c_q^- |\theta|^{-2d_{(n)}}$, and the upper bound to be $c_1^+ |\theta|^{-2d_{(1)}}$.
\end{enumerate}

\end{proof}

\section{Proofs of Proposition \ref{proposition_oracle} and \ref{thm:main_theorem}}\label{sec:proof_main_prop}
\subsection{Proof of Proposition \ref{proposition_oracle}}

\begin{proof}

We adopt the same proof framework as in Lemma~4.1 and Corollary~4.1 in \cite{FORNI2004231}. Denote by \(\mathbf B^{(n)}(\theta)\) the \(n\times q\) matrix with entries 
\(b_{ij}(e^{-i\theta})\), by \(\mathbf I^{(n)}_{i}\) the \(i\)th \(n\)-dimensional unit row vector, by $\mathbf{P}^{(n)}(\theta)$ the 
$q \times n$ matrix with rows $\mathbf{p}_j^{(n)}(\theta)$, by $\boldsymbol{\Lambda}^{(n)}(\theta)$ the 
$q \times q$ diagonal matrix with diagonal elements 
$\lambda_j^{(n)}(\theta)$. Let $\mathbf{M}^{(n)}(\theta) := \big(\boldsymbol{\Lambda}^{(n)}(\theta)\big)^{-1/2}$, since $1/\sqrt{\lambda_j^{(n)}(\theta)}$ is a.e.\ correctly defined. The decomposition in the proof of Lemma 4.1 in \cite{FORNI2004231} shows
\begin{align}\label{eq:K_decompose_signal_noise}
    \mathbf{K}_i^{(n)}
  &= \mathbf{I}_i^{(n)} \boldsymbol{\Sigma}_\xi^{(n)} 
     \tilde{\mathbf{P}}^{(n)} \mathbf{M}^{(n)} \mathbf{M}^{(n)} \mathbf{P}^{(n)}
   + \mathbf{I}_i^{(n)} \mathbf{B}^{(n)} \tilde{\mathbf{B}}^{(n)} 
     \tilde{\mathbf{P}}^{(n)} \mathbf{M}^{(n)} \mathbf{M}^{(n)} \mathbf{P}^{(n)} \\\nonumber
  &=: \mathbf{H}_i^{(n)} + \mathbf{G}_i^{(n)}. \nonumber
\end{align} Hence, $\chi_{it}^{(n)} = \mathbf{K}_i^{(n)}(L) X_t^{(n)}$ decomposes into $\chi_{it}^{(n)} = \mathbf{H}_i^{(n)}(L) X_t^{(n)} + \mathbf{G}_i^{(n)}(L) X_t^{(n)}$.

We only need to notice that, in the long-memory setting, the spectral density of \(\chi_{it}\),
\(\mathbf I^{(n)}_{i}\,\mathbf B^{(n)}(\theta)\,\tilde{\mathbf B}^{(n)}(\theta)\,\tilde{\mathbf I}^{(n)}_{i}\),
is bounded by \(C^{2}|\theta|^{-2d}\) by Assumption~\ref{assump: spec_den_specific_form}. 
Then, repeating the same steps as in the proof of Lemma~4.1 in \cite{FORNI2004231}, the spectral density of 
\(\underline{\mathbf K}_{ni}(L)X_t-\chi_{it}\) is upper-bounded by 
\(\mathcal{O}\!\left(n^{-1}|\theta|^{-2d}\right)\) (up to a constant). 
Therefore, the variance of ${\underline{\mathbf{K}}}_{n i}(L)X_t - \chi_{it}$ is of order $\mathcal{O}\left(n^{-1}\right)$. 

As a result, for any \(\eta>0\) there exist \(B_{\eta},N_{\eta}>0\) such that for \(n>N_{\eta}\),
\[
\mathbb{P}\!\left[r^{(n)}\big|\underline{\mathbf K}_{ni}(L)\mathbf x_{nt}-\chi_{it}\big|>B_{\eta}\right]
\le \frac{(r^{(n)})^{2}}{B_{\eta}^{2}}\,
\mathrm{var}\!\left(\underline{\mathbf K}_{ni}(L)\mathbf x_{nt}-\chi_{it}\right)
\le \frac{(r^{(n)})^{2}}{B_{\eta}^{2}}\,n^{-1}\le \eta,
\]
It suffices to take \(r^{(n)}\asymp\!n^{1/2}\).

\end{proof}

\subsection{Proof of Proposition \ref{thm:main_theorem}}


\begin{proof}[Proof of Proposition \ref{thm:main_theorem}]
    Notice that for any $B>0,$ 
    \begin{align*}
P\left[r^{(n)}\left|\widehat{\underline{\mathbf{K}}}_{n i}(L) X_{t}-\chi_{i t}\right|>B\right] &\leq P\left[r^{(n)}\left|\left(\widehat{\underline{\mathbf{K}}}_{n i}(L)-\underline{\mathbf{K}}_{n i}(L)\right) X_t\right|>B / 2\right] \\
& +P\left[r^{(n)}\left|\underline{\mathbf{K}}_{n i}(L) X_t-\chi_{i t}\right|>B / 2\right]
\end{align*}

Furthermore, 
\begin{align*}
&P\left[r^{(n)}\left|\left(\widehat{\underline{\mathbf{K}}}_{n i}(L)-\underline{\mathbf{K}}_{n i}(L)\right) X_t\right|>B / 2\right]\\
&\leq P\left[r^{(n)}\left|\sum_{h=-M(T)}^{M(T)}\left(\widehat{\mathbf{K}}_{n i, h}-\mathbf{K}_{n i, h}\right) L^h X_t\right|>\frac{B}{4}\right] \\
&+P\left[r^{(n)}\left|\left( -\sum_{h=-\infty}^{-M(T)-1} \mathbf{K}_{n i, h} L^h-\sum_{h=M(T)+1}^{\infty} \mathbf{K}_{n i, h} L^h\right)X_t\right|>\frac{B}{4}\right]\\
&= P\left[r^{(n)}\left|\sum_{h=-M(T)}^{M(T)}\left(\widehat{\mathbf{K}}_{n i, h}-\mathbf{K}_{n i, h}\right) L^h X_t\right|>\frac{B}{4}\right] + P\left[r^{(n)}\left|\sum_{|h|>M(T)} \mathbf{K}_{n i, h} X_{t-h}\right|>\frac{B}{4}\right]
\end{align*}

By Proposition \ref{proposition_oracle}, for all $\eta>0$, there exists $B_{0,\eta}, N_{0,\eta}$ such that 
\[
P\left[r_0^{(n)}\left|\underline{\mathbf{K}}_{n i}(L) \mathbf{x}_{n t}-\chi_{i t}\right|> \frac{B_{0,\eta} }{2}\right] \leq \frac{\eta}{2}, \text{ where }r_0^{(n)} \asymp n^{1/2}
\]

Lemma \ref{lemma:consistency_main_part_estimator} gives that for $0<\eta<4$, there exists a $B_{1,\eta}$ and $N_{1,\eta}$ such that for all $n \geq N_{1,\eta}$,
    \begin{align}
        P\left[r_1^{(n)}\left|\sum_{h=-M(T)}^{M(T)}\left(\widehat{\mathbf{K}}_{n i, h}-\mathbf{K}_{n i, h}\right) L^h X_t\right|>\frac{B_{1,\eta}}{4}\right] \leq \frac{\eta}{4}
    \end{align}
    where 
    \begin{align*}
        r_1^{(n)} \asymp M(T)^{-1}\left(\delta^{(T,n)}\right)^{-1}n^{-\frac{1}{2}} , \delta^{(T,n)} = B_T^{1-2\Delta} + \frac{\log T}{TB_T}     \end{align*}

Meanwhile, Lemma \ref{lemma:consistency_residual_part_estimator} shows that if the rows and factors shares the same $d$, then there exists a $B_{2,\eta}$ and $N_{2,\eta}$ such that for all $n \geq N_{2,\eta}$,
    \begin{align}
        P\left[r_2^{(n)}\left|\sum_{|h|>M(T)} \mathbf{K}_{n i, h} X_{t-h}\right|>\frac{B_{2,\eta}}{4}\right]<\frac{\eta}{4}
    \end{align}
where $r_2^{(n)}$ is defined in Lemma \ref{lemma:consistency_residual_part_estimator}.

Finally, letting $B_{\eta} = \max(B_{0,\eta}, B_{1,\eta}, B_{2,\eta})$ and $N_{\eta} = \max(N_{0,\eta},$ $ N_{1,\eta}, N_{2,\eta})$, we have that for all $n \geq N_{\eta}$,
\begin{align*}
    P\left[r^{(n)}\left|\widehat{\underline{\mathbf{K}}}_{n i}(L) X_{t}-\chi_{i t}\right|>B_{\eta}\right] 
    &\leq \frac{\eta}{2} + \frac{\eta}{4} + \frac{\eta}{4} = \eta
\end{align*}
and that
\begin{align*}
    r^{(n)} &= \min \left(r_0^{(n)},r_1^{(n)},r_2^{(n)}\right)\\
    r_0^{(n)} &= n^{1/2};\\
    r_1^{(n)} &= M(T)^{-1}\left(\delta^{(T,n)}\right)^{-1}n^{-\frac{1}{2}};\\
    \delta^{(T,n)} &= B_T^{1-2\Delta} + \frac{\log T}{TB_T}
\end{align*}
and $r_2^{(n)}$ is such that 

\begin{enumerate}
        \item when all rows and series share the same $d$
        \begin{align*}
            r_2^{(n)} \asymp \frac{1}{\log M(T)}\begin{cases}\dfrac{M(T)}{\sqrt{n}}, & d=0,\\[6pt] \min\!\left(\dfrac{M(T)}{\sqrt{n}},\; \sqrt{n}\, M(T)^{2d}\right), & d>0.\end{cases}, 
        \end{align*}
\begin{enumerate}[(i)]
    \item When the idiosyncratic error $\xi_t$ is an orthonormal white noise,  \begin{align*}
            r_2^{(n)} \asymp \frac{M(T)}{\log M(T) \sqrt{n}}, 
        \end{align*}
\end{enumerate}
\item 
$d_{il}$ differs only across factors, i.e. it satisfies case (a) in Assumption \ref{assump: spec_den_specific_form}.
\begin{align*}
    r_2^{(n)} \asymp n^{-1/2}\,M(T)^{\frac{1}{2} +\rho^{(q)}_q -d},
\end{align*} where \begin{align*}
     \rho^{(q)}_m = \min \{\rho^{(q)}_{m-1}, \alpha_{m,m+1}, \min_{0 \leq l\leq m-1}\{2\rho^{(q)}_l - \alpha_{l,m}\} \}, m  = 1, \dots, q,
    \end{align*}
     $\alpha_{j_1j_2} = 2|d_{(j_1)}-d_{(j_2)}|$, $d_{(q+1)} \coloneqq 0$, $\rho^{(j)}_0 \coloneqq 1$, $\alpha_{0,j} \coloneqq 1, j = 1,\dots,q$.
        \item $d_{il}$ differs only across rows, i.e. it satisfies case (b) in Assumption \ref{assump: spec_den_specific_form}. 

When $4\Delta + 2d - 1 <0$, $\Delta = d_{max} - d_{min}$,
\begin{align*}
    r_2^{(n)} \asymp n^{-\frac{1}{2}}M(T)^{\frac{1}{2} - 2\Delta - d} .
\end{align*}
    \end{enumerate}
It provides a consistency rate only if $\left(r^{(n)}\right)^{-1} = o(1)$.

\end{proof}

\section{Lemmas Used in Appendix \ref{sec:proof_main_prop}}\label{sec:proof_lemma_for_main_prop}
\subsection{$L^1$ Consistency of the Spectral Density Estimator}\label{subsec:L_1_consistency_periodogram_smoothing}
In order to prove Lemma~\ref{lemma:consistency_main_part_estimator}, 
we first establish the consistency of the smoothed periodogram estimator~(\ref{eq:smoothed_periodogram}). A takeaway of the theorem below is that it can be extended to \(L^p\) spaces, due to the continuity of translation in \(L^p\) and the fact that the variance vanishes as \(T\to\infty\), but for simplicity and to obtain a better rate, we present consistency only in \(L^1\) here.


\begin{lemma}\label{lemma:main_theorem_of_spectral_estimate}
    Assume that assumptions \ref{assump:model}, \ref{assump:semiparametric_long_memory}, \ref{assump:pervasive_short_mem}, \ref{assump: spec_den_specific_form}, \ref{assump:bdd_error_eval}, \ref{assump:W_kernel}, and \ref{assump:row_boundness_G} 
hold, then 
    \begin{align}\label{eq:conv_rate_expectation_sigma_hat_in_L1}
        \mathbb{E}\left[ \int_{\Pi} |\theta|^{2\tilde{d}} \left\|\widehat{\Sigma}_{n}(\theta) - \Sigma(\theta)\right\|_{op} d\theta \right] =\mathcal{O}\left(nB_T^{1-2\Delta} + \frac{n\log T}{TB_T} \right)
    \end{align}
    where $\Delta = \max_j d_j - \min_j d_j$.
\end{lemma}

\begin{proof}[Proof of Theorem \ref{lemma:main_theorem_of_spectral_estimate}]
Since the integrand below on the left side is measurable and nonnegative, by Tonelli’s theorem we may interchange expectation and integral.

\begin{align*}
    \mathbb{E}\left[ \int_{\Pi} |\theta|^{2\tilde{d}} \left\|\widehat{\Sigma}_{n}(\theta) - \mathbb{E} \widehat{\Sigma}_{n}(\theta) \right\|_{op} d\theta \right] &= \int_{\Pi} |\theta|^{2\tilde{d}} \mathbb{E}\left[  \left\|\widehat{\Sigma}_{n}(\theta) - \mathbb{E} \widehat{\Sigma}_{n}(\theta) \right\|_{op} \right] d\theta\\
    &\leq \int_{\Pi} |\theta|^{2\tilde{d}} \sqrt{\mathbb{E} \left\|\widehat{\Sigma}_{n}(\theta) - \mathbb{E} \widehat{\Sigma}_{n}(\theta) \right\|_{op}^2  } d\theta\\
    &\leq \int_{\Pi} |\theta|^{2\tilde{d}} \sqrt{\mathbb{E} \left\|\widehat{\Sigma}_{n}(\theta) - \mathbb{E} \widehat{\Sigma}_{n}(\theta) \right\|_{F}^2  } d\theta\\
    &= \int_{\Pi} |\theta|^{2\tilde{d}} \sqrt{\mathbb{E}\left[ \sum_{k,l} \left| \hat{\sigma}_{kl}(\theta) - \mathbb{E}\left(\hat{\sigma}_{kl}(\theta)\right)\right|^2  \right]} d\theta\\
    &\leq \int_{\Pi} |\theta|^{2\tilde{d}} \sqrt{n^2\max_{kl}\mathbb{E}\left[ \left| \hat{\sigma}_{kl}(\theta) - \mathbb{E}\left(\hat{\sigma}_{kl}(\theta)\right)\right|^2  \right]} d\theta\\
    &\leq n\int_{\Pi} |\theta|^{2\tilde{d}} \max_{kl}\sqrt{var\left(\hat{\sigma}_{kl}(\theta)\right)}  d\theta
\end{align*}
Combining with the rate from Lemma \ref{lm:var_conv} that 
\begin{align*}
    \max_{kl} var\left(\hat{\sigma}_{kl}(\theta)\right) = \mathcal{O}\left(  \frac{|\theta|^{-4d}\log^2 T}{TB_T}\mathbb{1}(|\theta|>2\rho B_T) + \left(\frac{\log^2 T}{TB_T^{1+4d}} + \frac{\log^2 T}{T^{2-4d}B_T^{2}}\right)\mathbb{1}(|\theta|\leq 2\rho B_T) \right)
\end{align*}
$\theta\in \Pi$ a.e., 
we obtain
\begin{align*}
    \mathbb{E}\left[ \int_{\Pi} |\theta|^{2\tilde{d}} \left\|\widehat{\Sigma}_{n}(\theta) - \mathbb{E} \widehat{\Sigma}_{n}(\theta) \right\|_{op} d\theta\right]  = \mathcal{O}\left( n\log T \left( \frac{1}{TB_T} + \frac{B_T^{\frac{1}{2} - 2\Delta}}{\sqrt{T}} + \frac{B_T^{2\tilde{d}}}{T^{1-2d}} \right) \right)
\end{align*}

Combining with the rate for bias term from Lemma \ref{lemma:bias_conv} yields, 
    \begin{align*}
        &\mathbb{E}\left[  \int_{\Pi} |\theta|^{2\tilde{d}} \left\|\widehat{\Sigma}_{n}(\theta) - \Sigma(\theta)\right\|_{op} d\theta  \right]\\ 
        &\leq \mathbb{E}\left[ \int_{\Pi} |\theta|^{2\tilde{d}} \left\|\widehat{\Sigma}_{n}(\theta) - \mathbb{E} \widehat{\Sigma}_{n}(\theta) \right\|_{op} d\theta\right] +  \int_{\Pi} |\theta|^{2\tilde{d}}\left\|\mathbb{E}\left(\widehat{\Sigma}_{n}(\theta)\right) - \Sigma(\theta)\right\|_{op}  d\theta  \\
        &\leq \mathcal{O}\left( n\log T \left( \frac{1}{TB_T} + \frac{B_T^{\frac{1}{2} - 2\Delta}}{\sqrt{T}} + \frac{B_T^{2\tilde{d}}}{T^{1-2d}} \right)  + nB_T^{1-2\Delta} + \frac{n}{TB_T}\right)\\
        &= \mathcal{O}\left( nB_T^{1-2\Delta} + \frac{n\log T}{TB_T}\right)
    \end{align*}
    
\end{proof}

\subsection{Consistency of the main part of the estimator}
\begin{lemma}\label{lemma:consistency_main_part_estimator}
For all $0<\eta<4$, there exists a $B_{\eta}$ and $N_{\eta}$ such that for all $n \geq N_{\eta}$,
    \begin{align}\label{eq:main_prob}
        P\left[r^{(n)}\left|\sum_{h=-M(T)}^{M(T)}\left(\widehat{\mathbf{K}}_{n i, h}-\mathbf{K}_{n i, h}\right) L^h X_t\right|>\frac{B_{\eta}}{4}\right] \leq \frac{\eta}{4}
    \end{align}
    where 
    \begin{align*}
        r^{(n)} = \mathcal{O}\left( M(T)^{-1}\left(\delta^{(T,n)}\right)^{-1}n^{-\frac{1}{2}} \right)
    \end{align*}
\end{lemma}
\begin{proof}
    Recall that 
    \begin{align*}
    \underline{\mathbf{K}}_{n i}(\theta) &=  \tilde{p}_{n1,i}(\theta) \mathbf{p}_{n1}(\theta) + \tilde{p}_{n2,i}(\theta) \mathbf{p}_{n2}(\theta) + ... + \tilde{p}_{nq,i}(\theta) \mathbf{p}_{nq}(\theta)\\
        \widehat{\underline{\mathbf{K}}}_{n i}(\theta) &=  \tilde{\hat{p}}_{n1,i}(\theta) \mathbf{\hat{p}}_{n1}(\theta) + \tilde{\hat{p}}_{n2,i}(\theta) \mathbf{\hat{p}}_{n2}(\theta) + ... + \tilde{\hat{p}}_{nq,i}(\theta) \mathbf{\hat{p}}_{nq}(\theta)
    \end{align*}
    Let $V(\theta) = \left( \mathbf{p}_{n1}^{H}(\theta), ...,  \mathbf{p}_{nq}^{H}(\theta) \right)$, $\hat{V}(\theta) = \left( \mathbf{\hat{p}}_{n1}^{H}(\theta), ...,  \mathbf{\hat{p}}_{nq}^{H}(\theta) \right)$,
    then by the first statement in \cite{DavisKahan}, Theorem 2, we have 
    \begin{align*}
        \left\|  \sin \Theta \left(\hat{V}(\theta), V(\theta)\right)  \right\|_F \leq  \frac{C \left\|\hat{\Sigma}(\theta) - \Sigma(\theta) \right\|_{op}}{\lambda_q(\theta) - \lambda_{q+1}(\theta)}
    \end{align*}
    \cite{DavisKahan} treats real symmetric matrices, but the first statement of Theorem~2 extends straightforwardly to Hermitian matrices over $\mathbb{C}$. Proof details are omitted. The definition of the canonical angles between two complex subspaces is given in \cite{sunstewart1990matrix} Definition~I.5.3. 

    Weyl inequalities and $\Sigma(\theta) = \Sigma^{\chi}(\theta) + \Sigma^{\xi}(\theta)$ imply that
    \begin{align*}
        \lambda_{q}^{\chi}(\theta) + \lambda_{n}^{\xi}(\theta) \leq \lambda_{q}^{x}(\theta) \leq \lambda_{q}^{\chi}(\theta) + \lambda_{1}^{\xi}(\theta).
    \end{align*}
    Further by Assumption \ref{assump:bdd_error_eval} and Lemma \ref{lemma:div_component_eval_specific_form}, there exists some $n^x$, for $n \geq n^{x}$,
    \begin{align*}
        \lambda_{q}^{x}(\theta) - \lambda_{q+1}^{x}(\theta) &\geq n(\alpha_q^{\chi}(\theta) - \beta_q^{\chi}(\theta)) - (\lambda_{1}^{\xi}(\theta) - \lambda_{n}^{\xi}(\theta)) \\
        &\geq n(\alpha_q^{\chi}(\theta) - \beta_q^{\chi}(\theta)) - \Lambda^{\xi}\\
        &\geq Cn(\alpha_q^{\chi}(\theta) - \beta_q^{\chi}(\theta)) \geq Cn|\theta|^{2\tilde{d}}
    \end{align*} where $\tilde{d} = \min_{il} d_{il}$.

    Meanwhile, by \cite{sunstewart1990matrix} Theorem~I.5.5, we have
    \begin{align*}
        \left\|V(\theta)V^H(\theta) - \hat{V}(\theta)\hat{V}^H(\theta)\right\|_{op} = \left\|\sin \Theta \left(\hat{V}(\theta), V(\theta)\right)\right\|_{op} = q\left\|\sin \Theta \left(\hat{V}(\theta), V(\theta)\right)\right\|_{F}
    \end{align*}

    Then use the relation that $\underline{\mathbf{K}}_{n i}(\theta) - \widehat{\underline{\mathbf{K}}}_{n i}(\theta) = e_i \left( V(\theta)V^H(\theta) - \hat{V}(\theta)\hat{V}^H(\theta)  \right)$, where $e_i$ is the standard basis vector, we have
    \begin{align}\label{eq:perturbation_prewhiten_scale}
        \left\| \widehat{\underline{\mathbf{K}}}_{n i}(\theta) - \underline{\mathbf{K}}_{n i}(\theta)  \right\| &\leq \frac{C\left\|\widehat{\Sigma}_{n}(\theta) - \Sigma(\theta)\right\|_{op}}{n\alpha_q^{\chi}(\theta) }\\
        &\leq \frac{Cc_j |\theta|^{2\tilde{d}}}{n} \left\|\widehat{\Sigma}_{n}(\theta) - \Sigma(\theta)\right\|_{op}
    \end{align}

    From Lemma \ref{lemma:main_theorem_of_spectral_estimate}, we know that 
    \begin{align*}
    &\mathbb{E} \left[\int_{\Pi} |\theta|^{2\tilde{d}} \left\|\widehat{\Sigma}_{n}(\theta) - \Sigma(\theta)\right\|_{op}   d\theta  \right] = \mathcal{O}\left(n\delta^{(T,n)}\right),\\
    &\delta^{(T,n)} = B_T^{1-2\Delta} + \frac{1}{TB_T} + \frac{ \log^{\frac{3}{2}}T}{\sqrt{T}}
    \end{align*}
    where $\Delta = \max_j d_j - \min_j d_j$.
    Then by Markov's inequality,
    for all $\eta>0$, there exists $ N_{3,\eta}$ such that for $n\geq N_{3,\eta}$,
    \begin{align*}
        \mathbb{P}\left( \max_h \left\| \widehat{\mathbf{K}}_{n i, h}-\mathbf{K}_{n i, h} \right\| > \frac{8}{\eta}\delta^{(T,n)}  \right)&\leq \mathbb{P}\left(\int_{\Pi} \left\|\widehat{\underline{\mathbf{K}}}_{n i}(\theta) -\underline{\mathbf{K}}_{n i}(\theta) \right\| d\theta  > \frac{8}{\eta}\delta^{(T,n)}  \right)\\
        &\leq \mathbb{P}\left( \int_{\Pi} |\theta|^{2\tilde{d}} \left\|\widehat{\Sigma}_{n}(\theta) - \Sigma(\theta)\right\|_{op}   d\theta > \frac{8}{\eta}n\delta^{(T,n)}  \right)\\
        &\leq \mathbb{E} \left[\int_{\Pi} |\theta|^{2\tilde{d}} \left\|\widehat{\Sigma}_{n}(\theta) - \Sigma(\theta)\right\|_{op}   d\theta  \right]\frac{\eta}{8n\delta^{(T,n)}} \leq \frac{\eta}{8}
    \end{align*}
    For $\eta \leq 4,$ moreover we have 
    \begin{align*}
        \mathbb{P}\left( \max_h \left\| \widehat{\mathbf{K}}_{n i, h}-\mathbf{K}_{n i, h} \right\| \leq \frac{8}{\eta}\delta^{(T,n)}  \right) 
 = 1- \mathbb{P}\left( \max_h \left\| \widehat{\mathbf{K}}_{n i, h}-\mathbf{K}_{n i, h} \right\| > \frac{8}{\eta}\delta^{(T,n)}  \right) \geq \frac{1}{2}
    \end{align*}
    \begin{align*}
        &\text{LHS of }\eqref{eq:main_prob} \\
        &\leq \mathbb{P}\left[r^{(n)}\left|\sum_{h=-M(T)}^{M(T)}\left(\widehat{\mathbf{K}}_{n i, h}-\mathbf{K}_{n i, h}\right) L^h X_t\right|>\frac{B}{4}\Bigg| \max_h \left\| \widehat{\mathbf{K}}_{n i, h}-\mathbf{K}_{n i, h} \right\| \leq \frac{8}{\eta}\delta^{(T,n)}\right]\\
        &+ \mathbb{P}\left[ \max_h \left\| \widehat{\mathbf{K}}_{n i, h}-\mathbf{K}_{n i, h} \right\| > \frac{8}{\eta}\delta^{(T,n)}  \right]\\
    &\leq \frac{16}{B^2} \mathbb{E}\left[\left(r^{(n)}\right)^2
  \left|\sum_{h=-M(T)}^{M(T)}\left(\widehat{\mathbf{K}}_{n i, h}-\mathbf{K}_{n i, h}\right)X_{t-h}\right|^2 \Bigg| \max_h \left\| \widehat{\mathbf{K}}_{n i, h}-\mathbf{K}_{n i, h} \right\| \leq \frac{8}{\eta}\delta^{(T,n)}\right] + \frac{\eta}{8}\\
&\leq \frac{16\left(r^{(n)}\right)^2}{B^2} \mathbb{E} \left[
  \left(\sum_{h=-M(T)}^{M(T)} \left\| \widehat{\mathbf{K}}_{n i, h}-\mathbf{K}_{n i, h} \right\|^2\right) 
  \left( \sum_{h=-M(T)}^{M(T)} \left\|{X}_{t-h} \right\|^2 \right)\Bigg|\cdots\right] + \frac{\eta}{8}\\
&\leq \frac{16\left(r^{(n)}\right)^2}{B^2}(2M(T)+1) \mathbb{E} \left[
  \max_h \left\| \widehat{\mathbf{K}}_{n i, h}-\mathbf{K}_{n i, h} \right\|^2
  \left( \sum_{h=-M(T)}^{M(T)} \left\|{X}_{t-h} \right\|^2 \right) \Bigg| \cdots \right]+ \frac{\eta}{8}\\
&\leq \frac{16\left(\frac{8}{\eta}r^{(n)}\right)^2}{B^2}(2M(T)+1)^2 \left(\delta^{(T,n)}\right)^2 \mathbb{E} \left[\left\|{X}_{t} \right\|^2  \right]/\mathbb{P}\left[ \left\| \widehat{\mathbf{K}}_{n i, h}-\mathbf{K}_{n i, h} \right\| \leq \frac{8}{\eta}\delta^{(T,n)} \right]+ \frac{\eta}{8}\\
&\leq \frac{32\left(r^{(n)}\right)^2}{(B/\frac{8}{\eta})^2}(2M(T)+1)^2 \left(\delta^{(T,n)}\right)^2 \mathbb{E} \left[\left\|{X}_{t} \right\|^2  \right]+ \frac{\eta}{8}
    \end{align*}
    Notice that
    \begin{align*}
        \mathbb{E} \left[\left\|{X}_{t} \right\|^2\right] &= \mathbb{E} \left[ trace\left( X_t X_t'  \right) \right] = trace\left(\Gamma_0\right) = trace \int_{\Pi} \Sigma(\theta) d\theta\\
        &= \sum_{j=1}^q \int_{\Pi} \lambda_j(\theta) d\theta + \sum_{j=q+1}^n \int_{\Pi} \lambda_j(\theta) d\theta \leq Cnq + \Lambda (n-q) = \mathcal{O}(n)
    \end{align*}
    Then, for all $0<\eta<4$, there exists a $B_{1,\eta}$ and $N_{1,\eta}$ such that for all $n \geq \max(N_{1,\eta}, N_{2,\eta},N_{3,\eta})$, 
    \[
    \frac{128}{B_{1,\eta}} n\left(\delta^{(T,n)}r^{(n)}M(T)\right)^2 \leq \left(\frac{\eta}{8}\right)^3
    \]

which holds whenever
\[
r^{(n)} \le C_\eta\, M(T)^{-1}\,(\delta^{(T,n)})^{-1}\,n^{-1/2}.
\]
In particular, we may take
\[
r^{(n)} \asymp M(T)^{-1}(\delta^{(T,n)})^{-1}n^{-1/2}.
\]
    In this case, $R_1 \leq \eta/8 + \eta/8 \leq \eta/4$.
\end{proof}

\subsection{Consistency of the remainder part of the oracle estimator}
\begin{lemma}For all $0<\eta<4$, there exists a $B_{\eta}$ and $N_{\eta}$ such that for all $n \geq N_{\eta}$,
    \begin{align}\label{eq:remainder_prob}
        P\left[r^{(n)}\left|\sum_{|h|>M(T)} \mathbf{K}_{n i, h} X_{t-h}\right|>\frac{B}{4}\right]<\frac{\eta}{4}
    \end{align}
    where
    \begin{enumerate}
        \item when all rows and series share the same $d$
        \begin{align*}
            r_2^{(n)} \asymp \frac{1}{\log M(T)}\begin{cases}\dfrac{M(T)}{\sqrt{n}}, & d=0,\\[6pt] \min\!\left(\dfrac{M(T)}{\sqrt{n}},\; \sqrt{n}\, M(T)^{2d}\right), & d>0.\end{cases}, 
        \end{align*}
\begin{enumerate}[(i)]
    \item When the idiosyncratic error $\xi_t$ is an orthonormal white noise,  \begin{align*}
            r_2^{(n)} \asymp \frac{M(T)}{\log M(T) \sqrt{n}}, 
        \end{align*}
\end{enumerate}
\item $d_{il}$ differs, and it satisfies case(a) in Assumption \ref{assump: spec_den_specific_form}
\begin{align*}
    r_2^{(n)} \asymp n^{-1/2}\,M(T)^{\frac{1}{2} +\rho^{(q)}_q -d}.
\end{align*}
        \item $d_{il}$ differs, and it satisfies case(b) in Assumption \ref{assump: spec_den_specific_form}. 
        When $4\Delta + 2d - 1 <0$, $\Delta = d_{max} - d_{min}$
        \begin{align*}
    r_2^{(n)} \asymp n^{-\frac{1}{2}}M(T)^{\frac{1}{2} - 2\Delta - d} .
\end{align*}
    \end{enumerate}
\end{lemma}\label{lemma:consistency_residual_part_estimator}
\begin{proof}
\begin{enumerate}
    \item  we first show the case when all rows and series share the same $d$
    Denote $\mathbf{K}^{tail}_{ni,h} = \mathbf{K}_{n i, h} \mathbb{I}(|h|>M(T)), \underline{\mathbf{K}}^{tail}_{n i}(\theta) = \sum_{h}\mathbf{K}^{tail}_{ni,h} e^{\iota h \theta}$\\ $= \sum_{|h|>M(T)}\mathbf{K}_{ni,h} e^{\iota h \theta}$, then by Chebyshev's inequality,
\begin{align*}
    & P\left[r^{(n)}\left|\sum_{|h|>M(T)} \mathbf{K}_{n i, h} X_{t-h}\right|>\frac{B}{4}\right]\leq \frac{16\left(r^{(n)}\right)^2}{B^2}\mathbb{E}\left[ \left|\sum_{|h|>M(T)} \mathbf{K}_{n i, h} X_{t-h}\right|^2 \right]\\
    &= \frac{16\left(r^{(n)}\right)^2}{B^2} var\left( \sum_{h} \mathbf{K}_{n i, h}^{tail} X_{t-h}  \right) = \frac{16\left(r^{(n)}\right)^2}{B^2} \int_{\Pi} \underline{\mathbf{K}}^{tail}_{n i}(\theta) \Sigma(\theta) \widetilde{\underline{\mathbf{K}}}^{tail}_{n i}(\theta) d\theta\\
    &\leq \frac{16\left(r^{(n)}\right)^2}{B^2} \int_{\Pi} \lambda_1(\theta)\left\|\underline{\mathbf{K}}^{tail}_{n i}(\theta)\right\|_2^2 d\theta
\end{align*}

Meanwhile, since by Lemma~\ref{lemma:holder_cont_K}, $\underline{\mathbf{K}}_{ni}(\theta)$ is a H\"older continuous function and according to Theorem 10.8 in \cite{Zygmund_2003} Chapter II, page 64, we obtain that
\begin{align*}
    \sup_{\theta} \left\|\underline{\mathbf{K}}^{tail}_{n i}(\theta)\right\|_2 &= \sup_{\theta} \left\|\sum_{|h|>M(T)} \mathbf{K}_{n i, h}e^{\iota h \theta}\right\|_2\\ 
    &= \mathcal{O}\left(\max\left(\frac{\log M(T)}{M(T)},\frac{\log M(T)}{n\left(M(T)\right)^{2d}}\mathbb{1}(d>0)\right)\right)
\end{align*}
Then 
for all $\eta > 0$, there exists a $B_{\eta}$ and $N_{\eta}$ such that for all $n \geq N_{\eta}$,
\begin{align*}
    &P\left[r_{1,1}^{(n)}\left|\sum_{|h|>M(T)} \mathbf{K}_{n i, h} X_{t-h}\right|>\frac{B_{\eta}}{4}\right] \leq \frac{16\left(r_{1,1}^{(n)}\right)^2}{B_{\eta}^2} \int_{\Pi} \lambda_1(\theta)\left\|\underline{\mathbf{K}}^{tail}_{n i}(\theta)\right\|_2^2 d\theta\\
    &\leq \frac{16C\left(r_{1,1}^{(n)}\right)^2n}{B_{\eta}^2} \left[\left(\frac{\log M(T)}{M(T)}\right)^2 + \left(\frac{\log M(T)}{n\left(M(T)\right)^{2d}}\right)^2 \mathbb{1}(d>0)\right] \leq \frac{\eta}{4}
\end{align*}
where 
\begin{align*}
            r_2^{(n)} \asymp \frac{1}{\log M(T)}\begin{cases}\dfrac{M(T)}{\sqrt{n}}, & d=0,\\[6pt] \min\!\left(\dfrac{M(T)}{\sqrt{n}},\; \sqrt{n}\, M(T)^{2d}\right), & d>0.\end{cases}, 
        \end{align*}

\begin{enumerate}[(i)]
    \item Furthermore, when the idiosyncratic error $\xi_t$ is an orthonormal white noise, that is when $\Sigma_{\xi}(\theta) = \sigma_{\xi}^2 I$, by the second part of Lemma ~\ref{lemma:holder_cont_K}, $\underline{\mathbf{K}}_{ni}(\theta)$ is a Lipschitz function, and that results in 
    \begin{align*}
    \sup_{\theta} \left\|\underline{\mathbf{K}}^{tail}_{n i}(\theta)\right\|_2 &= \sup_{\theta} \left\|\sum_{|h|>M(T)} \mathbf{K}_{n i, h}e^{\iota h \theta}\right\|_2 
    = \mathcal{O}\left(\frac{\log M(T)}{M(T)}\right)
\end{align*}
as well as
\begin{align*}
    r_{1,1}^{(n)} \asymp \frac{M(T)}{\log M(T) \sqrt{n}} 
\end{align*}
\end{enumerate}



\item Case (a) in Assumption \ref{assump: spec_den_specific_form}: $d$ differs across factors. 

By Lemma \ref{lemma:d_differ_K_piecewise_cont_fourier_coef_decay_rate}, the decay rate for $\left\|\mathbf{K}_{ni,h}\right\|_2$ is $\mathcal{O}\left(|h|^{-1-\rho^{(q)}_q}\right)$. 

 Then 
 \begin{align*}
     \left\|\underline{\mathbf{K}}^{tail}_{n i}(\theta)\right\|_2 &\leq  \sum_{|h|\leq M(T)}\left\|\mathbf{K}_{n i, h}\right\|_2 \leq \sum_{|h|\leq M(T)}|h|^{-1-\rho^{(q)}_q} \leq C M(T)^{-\rho^{(q)}_q}
 \end{align*}
and by Parseval's inequality
\begin{align*}
    \frac{1}{2\pi}\int_{\Pi}\left\|\underline{\mathbf{K}}^{tail}_{n i}(\theta)\right\|_2^2 d\theta  = \sum_{|h|\geq M(T)} \left\|\mathbf{K}_{n i, h}\right\|^2 \leq C \left(M(T)\right)^{-1-2\rho^{(q)}_q}.
\end{align*}
 
Let $\delta_0 = M(T)^{-1}$,
 \begin{align*}
    & P\left[r^{(n)}\left|\sum_{|h|>M(T)} \mathbf{K}_{n i, h} X_{t-h}\right|>\frac{B}{4}\right] \leq \frac{16\left(r^{(n)}\right)^2}{B^2} \int_{\Pi} \lambda_1(\theta)\left\|\underline{\mathbf{K}}^{tail}_{n i}(\theta)\right\|_2^2 d\theta\\
    &= \frac{16n\left(r^{(n)}\right)^2}{B^2} \int_{|\theta|\leq \delta_0} |\theta|^{-2d}\left\|\underline{\mathbf{K}}^{tail}_{n i}(\theta)\right\|_2^2 d\theta + \frac{16n\left(r^{(n)}\right)^2}{B^2} \int_{|\theta|> \delta_0} |\theta|^{-2d}\left\|\underline{\mathbf{K}}^{tail}_{n i}(\theta)\right\|_2^2 d\theta\\
    &\leq C\frac{n\left(r^{(n)}\right)^2}{B^2}\left(M(T)\right)^{2d-1-2\rho^{(q)}_q }
\end{align*}

Then 
for all $\eta > 0$, there exists a $B_{\eta}$ and $N_{\eta}$ such that for all $n \geq N_{\eta}$,
\begin{align*}
    P\left[r_{2}^{(n)}\left|\sum_{|h|>M(T)} \mathbf{K}_{n i, h} X_{t-h}\right|>\frac{B_{\eta}}{4}\right]\leq \frac{C\left(r_{2}^{(n)}\right)^2n}{B_{\eta}^2} \left(M(T)\right)^{2d-1-2\rho^{(q)}_q} \leq \frac{\eta}{4}
\end{align*}
where 
\begin{align*}
            r_2^{(n)} \asymp n^{-1/2}\,M(T)^{\frac{1}{2} +\rho^{(q)}_q -d}.
        \end{align*}

\item Case (b) in Assumption \ref{assump: spec_den_specific_form}: $d$ differs across rows. 
 By Lemma \ref{lemma:d_differ_K_fourier_coef_decay_rate_old}, the decay rate for $\left\|\mathbf{K}_{ni,h}\right\|_2$ is $\mathcal{O}\left(|h|^{2\Delta-1}\right)$. 

 Then 
 \begin{align*}
     \left\|\underline{\mathbf{K}}^{tail}_{n i}(\theta)\right\| &\leq \left\|\underline{\mathbf{K}}_{n i}(\theta)\right\| + \sum_{|h|\leq M(T)}\left\|\mathbf{K}_{n i, h}\right\|\\
     &\leq 1 + \sum_{|h|\leq M(T)}|h|^{2\Delta-1} \leq C M(T)^{2\Delta}
 \end{align*}
and by Parseval's inequality
\begin{align*}
    \frac{1}{2\pi}\int_{\Pi}\left\|\underline{\mathbf{K}}^{tail}_{n i}(\theta)\right\|_2^2 d\theta  = \sum_{|h|\geq M(T)} \left\|\mathbf{K}_{n i, h}\right\|^2 \leq C \left(M(T)\right)^{4\Delta-1}, \Delta \leq \frac{1}{4}.
\end{align*}
 
Let $\delta_0 = M^{-1}(T)$,
 \begin{align*}
    & P\left[r^{(n)}\left|\sum_{|h|>M(T)} \mathbf{K}_{n i, h} X_{t-h}\right|>\frac{B}{4}\right] \leq \frac{16\left(r^{(n)}\right)^2}{B^2} \int_{\Pi} \lambda_1(\theta)\left\|\underline{\mathbf{K}}^{tail}_{n i}(\theta)\right\|_2^2 d\theta\\
    &= \frac{16n\left(r^{(n)}\right)^2}{B^2} \int_{|\theta|\leq \delta_0} |\theta|^{-2d}\left\|\underline{\mathbf{K}}^{tail}_{n i}(\theta)\right\|_2^2 d\theta + \frac{16n\left(r^{(n)}\right)^2}{B^2} \int_{|\theta|> \delta_0} |\theta|^{-2d}\left\|\underline{\mathbf{K}}^{tail}_{n i}(\theta)\right\|_2^2 d\theta\\
    &\leq \frac{16n\left(r^{(n)}\right)^2}{B^2} \left[ \left(M(T)\right)^{4\Delta} \delta_0^{1-2d} + \delta_0^{-2d}\left(M(T)\right)^{4\Delta-1} \right]\\
    &\leq C\frac{n\left(r^{(n)}\right)^2}{B^2}\left(M(T)\right)^{4\Delta+2d-1}
\end{align*}

Then 
for all $\eta > 0$, there exists a $B_{\eta}$ and $N_{\eta}$ such that for all $n \geq N_{\eta}$,
\begin{align*}
    P\left[r_{2}^{(n)}\left|\sum_{|h|>M(T)} \mathbf{K}_{n i, h} X_{t-h}\right|>\frac{B_{\eta}}{4}\right]\leq \frac{C\left(r_{2}^{(n)}\right)^2n}{B_{\eta}^2} \left(M(T)\right)^{4\Delta+2d-1} \leq \frac{\eta}{4}
\end{align*}
where 
\begin{align*}
            r_2^{(n)} \asymp n^{-1/2}\,M(T)^{\frac{1}{2}-2\Delta-d},\quad \Delta = d_{max} - d_{min}.
        \end{align*}


\end{enumerate}

\end{proof}

\section{Lemmas Used in Appendix \ref{sec:proof_lemma_for_main_prop}}\label{sec:proof_lemma_for_main_lemma_1}
\subsubsection{Lemmas for $\underline{\mathbf{K}}_{n i}(\theta)$ and $\mathbf{K}_{n i, h}$}

\begin{lemma}\label{lemma:holder_cont_K} 
    Assumption \ref{assump:model}, \ref{assump:semiparametric_long_memory}, \ref{assump:pervasive_short_mem}, \ref{assump: spec_den_specific_form}, \ref{assump:bdd_error_eval}  hold.
    If long memory parameter $d$ remains the same across rows and series, then $\underline{\mathbf{K}}_{n i}(\theta)$ is a H\"older continuous function, i.e. for any $\theta_1, \theta_2 \in \Pi$,
    \begin{align}\label{eq:holder_K}
        \left\|\underline{\mathbf{K}}_{n i}(\theta_1) - \underline{\mathbf{K}}_{n i}(\theta_2)\right\| = \mathcal{O}\left(  \max\left( \frac{1}{n} |\theta_1 - \theta_2|^{2d}\mathbb{1}(d>0), |\theta_1 - \theta_2| \right)  \right) 
    \end{align}
\end{lemma}

\begin{proof}
    Denote $\Sigma_{\chi}^b(\theta) = |\theta|^{2d}\Sigma_{\chi}(\theta)$, $\Sigma_{\xi}^b(\theta) = |\theta|^{2d}\Sigma_{\xi}(\theta)$, and $\Sigma^b(\theta) = \Sigma_{\chi}^b(\theta) + \Sigma_{\xi}^b(\theta)$. It is easy to notice that $\mathbf{p}_{j}(\theta)$ is also the $j$th eigenvector of $\Sigma^b(\theta)$, and the corresponding eigenvalue is $\lambda_j^{x,b}(\theta) \coloneqq |\theta|^{2d}\lambda_j^x(\theta)$. 

    Analogous to the proof in Lemma \ref{lemma:consistency_main_part_estimator}, let $V(\theta) = \left( \mathbf{p}_{n1}^{H}(\theta), ...,  \mathbf{p}_{nq}^{H}(\theta) \right)$,
    then by the first statement in \cite{DavisKahan}, Theorem 2, we have 
    \begin{align*}
        \left\|  \sin \Theta \left(V(\theta_1), V(\theta_2)\right)  \right\|_F \leq  \frac{C \left\|\Sigma^b(\theta_2) - \Sigma^b(\theta_1) \right\|_{op}}{\lambda^{x,b}_q(\theta) - \lambda^{x,b}_{q+1}(\theta)}
    \end{align*}

    Again, \cite{DavisKahan} treats real symmetric matrices, but the first statement of Theorem~2 extends straightforwardly to Hermitian matrices over $\mathbb{C}$. Proof details are omitted. The definition of the canonical angles between two complex subspaces is given in \cite{sunstewart1990matrix} Definition~I.5.3. 


    By Assumption \ref{assump:bdd_error_eval} and Lemma \ref{lemma:div_component_eval_specific_form}, we have the eigengap 
    \begin{align*}
        \lambda_{q}^{x,b}(\theta) - \lambda_{q+1}^{x,b}(\theta) &\geq n|\theta|^{2d}\alpha_q^{\chi}(\theta) - |\theta|^{2d}(\lambda_{1}^{\xi}(\theta) - \lambda_{n}^{\xi}(\theta)) \\
        &\geq n|\theta|^{2d}\alpha_q^{\chi}(\theta)  - C\Lambda^{\xi}\\
        &\geq Cn
    \end{align*}
    Then 
    \begin{align*}
        \left\|\underline{\mathbf{K}}_{n i}(\theta_1) - \underline{\mathbf{K}}_{n i}(\theta_2)\right\| &\leq \left\| V(\theta_1) V^H(\theta_1) - V(\theta_2) V^H(\theta_2) \right\|_{op}\\
        &=  \left\| \sin \Theta \left( V(\theta_2) - V(\theta_1) \right)   \right\|_F\\
        &\leq \frac{C  \left\|\Sigma^b(\theta_2) - \Sigma^b(\theta_1) \right\|_{op}}{n}\\
        &\leq \frac{C}{n} \left\| \Sigma_{\chi}^b(\theta_1) - \Sigma_{\chi}^b(\theta_2) \right\|_{op} + \frac{C}{n} \left\| \Sigma_{\xi}^b(\theta_1) - \Sigma_{\xi}^b(\theta_2) \right\|_{op}.
    \end{align*}
    For the second term,
    \begin{align*}
        \left\| \Sigma_{\xi}^b(\theta_1) - \Sigma_{\xi}^b(\theta_2) \right\|_{op} &\leq |\theta_1|^{2d}\left\| \Sigma_{\xi}(\theta_1) - \Sigma_{\xi}(\theta_2) \right\|_{op} + \left\|\Sigma_{\xi}(\theta_2)\right\| \big||\theta_1|^{2d} - |\theta_2|^{2d}\big|\\ &= \mathcal{O}\left(|\theta_1 - \theta_2|^{2d}\right)
    \end{align*}
    For the first term, 
    \begin{align*}
    \left\| \Sigma_{\chi}^b(\theta_1) - \Sigma_{\chi}^b(\theta_2) \right\|_{op} &\leq n \left\| S(\theta_2) - S(\theta_1)  \right\| 
    \leq n|\theta_2 - \theta_1|
    \end{align*}
    where $S(\theta)$ is defined in Assumption \ref{assump:pervasive_short_mem}.
    Hence,
    \begin{align*}
        \left\|\underline{\mathbf{K}}_{n i}(\theta_1) - \underline{\mathbf{K}}_{n i}(\theta_2)\right\| = \mathcal{O}\left(  \max\left( \frac{1}{n} |\theta_1 - \theta_2|^{2d}, |\theta_1 - \theta_2| \right)  \right) 
    \end{align*}
\end{proof}

\begin{lemma}\label{lemma:d_differ_K_fourier_coef_decay_rate_old}
    Under Assumption \ref{assump:model}, \ref{assump:semiparametric_long_memory}, \ref{assump:pervasive_short_mem}, \ref{assump: spec_den_specific_form}, and \ref{assump:bdd_error_eval}, 
    \begin{align*}
        \left\|\mathbf{K}_{n i, h}\right\| = \mathcal{O}\left(|h|^{2\Delta-1}\right)
    \end{align*}
\end{lemma}

\begin{proof}
    Although we know that $\left\|\underline{\mathbf{K}}_{n i}(\theta)\right\|_2$ is bounded, we still need its pointwise behavior to get the decay rates. In the worst-case scenario, by Assumption \ref{assump: spec_den_specific_form} \ref{assump:pervasive_short_mem}, and thus Lemma \ref{lemma:div_component_eval_specific_form},
\begin{align*}
    \left|p_{ji}(\theta)\right|^2 \leq \frac{\sigma_{ii}(\theta)}{\lambda_j(\theta)} \leq \frac{C|\theta|^{-2d}}{n \alpha_q (\theta)}\leq \frac{C}{n}|\theta|^{-2\Delta}
\end{align*}
Then 
\begin{align}\label{eq:K_pointwise_behavior}
    \left\|  \underline{\mathbf{K}}_{n i}(\theta)  \right\|\leq \sum_{j=1}^q \left|p_{ji}(\theta)\right| \leq \frac{C}{\sqrt{n}}|\theta|^{-\Delta}
\end{align}

Meanwhile, 
\begin{align*}
    \mathbf{K}_{ni,h} &= \frac{1}{2\pi} \int_{\Pi} \underline{\mathbf{K}}_{n i}(\theta) e^{-\iota h \theta} d\theta \\
    &= \frac{1}{2\pi} \int_{|\theta|<\frac{\pi}{|h|}} \underline{\mathbf{K}}_{n i}(\theta) e^{-\iota h \theta} d\theta + \frac{1}{2\pi} \int_{\Pi\setminus \{|\theta|<\frac{\pi}{|h|}\}} \underline{\mathbf{K}}_{n i}(\theta) e^{-\iota h \theta} d\theta\\
    &\coloneqq I_1^{(h)} + I_2^{(h)}
\end{align*}

The first term $I_1^{(h)}$ can be controlled by the pointwise behavior of $\underline{\mathbf{K}}_{n i}(\theta)$, 
\begin{align*}
    \left\| I_1^{(h)} \right\| = \left\|\frac{1}{2\pi} \int_{|\theta|<\frac{\pi}{|h|}} \underline{\mathbf{K}}_{n i}(\theta) e^{-\iota h \theta} d\theta\right\| &\leq \frac{1}{2\pi} \int_{|\theta|<\frac{\pi}{|h|}} \left\|\underline{\mathbf{K}}_{n i}(\theta) \right\| d\theta \\
    &\leq \frac{1}{2\pi} \int_{|\theta|<\frac{\pi}{|h|}} \frac{C}{\sqrt{n}} |\theta|^{-\Delta} d\theta  \leq \frac{C}{\sqrt{n}} |h|^{\Delta -1}.
\end{align*}

To control the second term $I_2^{(h)}$, first notice that
\begin{align*}
    e^{-\iota h \theta} = \frac{e^{-\iota h \theta} - e^{-\iota h (\theta \pm \frac{\pi}{h})}}{1-e^{\mp \iota \pi}}, 1-e^{\mp \iota \pi} = 2
\end{align*}
Choosing that for $h>0$, we choose it to be $\theta + \frac{\pi}{h}$ and for $h<0$, we choose it to be $\theta - \frac{\pi}{h}$. Then 
plugging them into $I_2^{(h)}$ yields
\begin{align*}
     I_2^{(h)} = \frac{1}{4\pi} \int_{\Pi\setminus \{|\theta|<\frac{\pi}{|h|}\}} \underline{\mathbf{K}}_{n i}(\theta) \left(e^{-\iota h \theta} - e^{-\iota h (\theta + \frac{\pi}{|h|})}\right) d\theta
\end{align*}

Notice that $\underline{\mathbf K}_{ni}(\theta)=\mathbf I^{(n)}_{i}\,\mathbf P(\theta)$, 
where $\mathbf I^{(n)}_{i}$ is the $i$th $n$-dimensional unit row vector, and 
$\mathbf P(\theta) = \sum_{j=1}^q \mathbf{p}'_{j}(\theta) \mathbf{p}_{j}(\theta)$.
According to \cite{YousefSaad}, §3.1.3–§3.1.4 and Example 1.2, 
$\mathbf P(\theta)$ can be written as 
$\frac{1}{2\pi \iota}\!\int_{\Gamma}(z\mathbf I-\Sigma(\theta))^{-1}\,dz$, 
where $\Gamma$ is a contour that encloses the largest $q$ eigenvalues; 
see also \cite{kato}, Ch.~III, §6.4. 
Recall that Lemma~\ref{lemma:div_component_eval_specific_form} shows that a uniform eigengap between the first $q$ eigenvalues and the rest, so we may choose $\Gamma$ independent of $\theta$. Thus, $\underline{\mathbf K}_{ni}(\theta)$ has the same $2\pi$-periodicity as $\Sigma(\theta)$.

Throughout our paper, we have set that $\Pi = [-\pi, \pi)$. Then
\begin{align*}
4\pi I_2^{(h)}
&= \int_{[-\pi,\,-\frac{\pi}{|h|}] \,\cup\, [\frac{\pi}{|h|},\,\pi]}
    \underline{\mathbf{K}}_{ni}(\theta)
    \Bigl(e^{-\iota h \theta} - e^{-\iota h (\theta + \frac{\pi}{|h|})}\Bigr)\, d\theta \\
&= \int_{[-\pi,\,-\frac{\pi}{|h|}] \,\cup\, [\frac{\pi}{|h|},\,\pi]}
    \underline{\mathbf{K}}_{ni}(\theta)e^{-\iota h \theta}\, d\theta
   - \int_{[-\pi+\frac{\pi}{|h|},\,0] \,\cup\, [\frac{2\pi}{|h|},\,\pi+\frac{\pi}{|h|}]}
    \underline{\mathbf{K}}_{ni}\!\left(\theta - \frac{\pi}{|h|}\right)e^{-\iota h \theta}\, d\theta \\
&= \int_{[-\pi,\,-\frac{\pi}{|h|}] \,\cup\, [\frac{\pi}{|h|},\,\pi]}
    \underline{\mathbf{K}}_{ni}(\theta)e^{-\iota h \theta}\, d\theta
   - \int_{[-\pi,\,0] \,\cup\, [\frac{2\pi}{|h|},\,\pi]}
    \underline{\mathbf{K}}_{ni}\!\left(\theta - \frac{\pi}{|h|}\right)e^{-\iota h \theta}\, d\theta \\
&\qquad\text{(here we used $2\pi$–periodicity to replace $[\pi,\pi+\frac{\pi}{|h|}]$ by $[-\pi,-\pi+\frac{\pi}{|h|}]$)} \\
&= \int_{-\pi}^{-\frac{\pi}{|h|}} \underline{\mathbf{K}}_{ni}(\theta)e^{-\iota h \theta}\, d\theta
   - \int_{-\pi}^{0} \underline{\mathbf{K}}_{ni}\!\left(\theta - \frac{\pi}{|h|}\right)e^{-\iota h \theta}\, d\theta \\
&\quad + \int_{\frac{\pi}{|h|}}^{\frac{2\pi}{|h|}} \underline{\mathbf{K}}_{ni}(\theta)e^{-\iota h \theta}\, d\theta
   + \int_{\frac{2\pi}{|h|}}^{\pi}
      \left(\underline{\mathbf{K}}_{ni}(\theta)
            - \underline{\mathbf{K}}_{ni}\!\left(\theta - \frac{\pi}{|h|}\right)\right)
      e^{-\iota h \theta}\, d\theta \\
&= \int_{-\pi + \frac{\pi}{|h|}}^{-\frac{\pi}{|h|}} \left(\underline{\mathbf{K}}_{ni}(\theta)- \underline{\mathbf{K}}_{ni}\!\left(\theta - \frac{\pi}{|h|}\right)\right) e^{-\iota h \theta} d\theta + \int_{-\pi}^{-\pi + \frac{\pi}{|h|}}  \underline{\mathbf{K}}_{ni}(\theta) e^{-\iota h \theta} d\theta   \\
&- \int_{-\pi}^{-\pi + \frac{\pi}{|h|}} \underline{\mathbf{K}}_{ni}\!\left(\theta - \frac{\pi}{|h|}\right) d\theta - \int_{-\frac{\pi}{|h|}}^{0} \underline{\mathbf{K}}_{ni}\!\left(\theta - \frac{\pi}{|h|}\right) d\theta \\
&+ \int_{\frac{\pi}{|h|}}^{\frac{2\pi}{|h|}} \underline{\mathbf{K}}_{ni}(\theta)e^{-\iota h \theta}\, d\theta
   + \int_{\frac{2\pi}{|h|}}^{\pi}
      \left(\underline{\mathbf{K}}_{ni}(\theta)
            - \underline{\mathbf{K}}_{ni}\!\left(\theta - \frac{\pi}{|h|}\right)\right)
      e^{-\iota h \theta}\, d\theta\\
&= \int_{[-\pi + \frac{\pi}{|h|}, -\frac{\pi}{|h|}]\, \cup \, [\frac{2\pi}{|h|}, \pi]} \left(\underline{\mathbf{K}}_{ni}(\theta)
            - \underline{\mathbf{K}}_{ni}\!\left(\theta - \frac{\pi}{|h|}\right)\right)
      e^{-\iota h \theta}\, d\theta\\
&+ \int_{ [-\frac{2\pi}{|h|}, -\frac{\pi}{|h|}] \, \cup \, [\frac{\pi}{|h|}, \frac{2\pi}{|h|}] } \underline{\mathbf{K}}_{ni}(\theta) e^{-\iota h \theta}\, d\theta +  \int_{[-\pi, -\pi + \frac{\pi}{|h|}]  \, \cup \, [\pi - \frac{\pi}{|h|}, \pi]} \underline{\mathbf{K}}_{ni}(\theta) e^{-\iota h \theta}\, d\theta
\end{align*}

For the second term, applying the boundness of $\left\| \underline{\mathbf{K}}_{n i}(\theta)  \right\|$, we have 
\begin{align*}
    \left\|\int_{\frac{\pi}{|h|}}^{\frac{2\pi}{|h|}} \underline{\mathbf{K}}_{n i}(\theta)e^{-\iota h \theta} d\theta\right\| \leq  \int_{\frac{\pi}{|h|}}^{\frac{2\pi}{|h|}} \left\| \underline{\mathbf{K}}_{n i}(\theta)  \right\| d\theta \leq |h|^{-1}
\end{align*}
On the other hand, for $\theta \in [-\pi, -\pi + \frac{\pi}{|h|}]  \, \cup \, [\pi - \frac{\pi}{|h|}, \pi]$, $|h| > \frac{3}{2}$, $|\theta|^{-\Delta} < \left(\frac{\pi}{3}\right)^{\Delta}$,
\begin{align}\label{eq:I_2_third_term}
     \left\|\int_{-\pi}^{-\pi + \frac{\pi}{|h|}} \underline{\mathbf{K}}_{n i}(\theta)e^{-\iota h \theta} d\theta\right\| \leq \frac{C}{\sqrt{n}} \left(\frac{\pi}{3}\right)^{\Delta} \int_{-\pi}^{-\pi + \frac{\pi}{|h|}} d\theta = \frac{C}{\sqrt{n}} |h|^{-1}
\end{align}

As for the first term, we first notice that for $\theta \in A, A \coloneqq [-\pi + \frac{\pi}{|h|}, -\frac{\pi}{|h|}]\, \cup \, [\frac{2\pi}{|h|}, \pi]$, $|\theta|$ and $ \left|\theta - \frac{\pi}{|h|}\right| \in \left[\frac{\pi}{|h|}, \pi\right]$.

    By Lemma \ref{lm:increment_bound_for_Sigma}, we have for $|\theta|,\left|\theta - \frac{\pi}{|h|}\right|  > \frac{\pi}{|h|}$,
    \begin{align*}
        \left\| \Sigma(\theta) - \Sigma\left( \theta-\frac{\pi}{|h|}  \right) \right\|_{op} \lesssim n|h|^{-1}\left( |\theta|^{-2d}+ |\theta|^{-2d-1} + \left|\theta-\frac{\pi}{|h|}\right|^{-2d} +  \left|\theta-\frac{\pi}{|h|}\right|^{-2d-1} \right)
    \end{align*}

Moreover, by Lemma \ref{lemma:div_component_eval_specific_form}, we have the eigengap 
    \begin{align*}
        \lambda_{q}^{x}(\theta) - \lambda_{q+1}^{x}(\theta) &\geq \lambda_q^{\chi}(\theta) - (\lambda_{1}^{\xi}(\theta) - \lambda_{n}^{\xi}(\theta)) \\
        &\geq n\alpha_q^{\chi}(\theta)  - C\Lambda^{\xi}\\
        &\geq Cn|\theta|^{-2\tilde{d}}
    \end{align*}

Then proceeding with Davis-Kahan in \cite{sunstewart1990matrix}, \cite{DavisKahan}, 
\begin{align*}
    &\left\| \int_A \left( \underline{\mathbf{K}}_{n i}(\theta) - \underline{\mathbf{K}}_{n i}\left(\theta - \frac{\pi}{|h|} \right) \right)   d\theta  \right\| \leq  \int_A \left\|  \underline{\mathbf{K}}_{n i}(\theta) - \underline{\mathbf{K}}_{n i}\left(\theta - \frac{\pi}{|h|}  \right)  \right\|d\theta \\
    &\leq \int_{A} \frac{\left\| \Sigma(\theta) - \Sigma\left( \theta-\frac{\pi}{|h|}   \right) \right\|_{op}}{\lambda_{q}^{x}(\theta) - \lambda_{q+1}^{x}(\theta)} d\theta\\
    &\leq C |h|^{-1} \int_{A} |\theta|^{-2\Delta} + |\theta|^{-2\Delta-1} + |\theta|^{2\tilde{d}} \left|\theta-\frac{\pi}{|h|}\right|^{-2d} + |\theta|^{2\tilde{d}} \left|\theta-\frac{\pi}{|h|}\right|^{-2d-1} d\theta\\
    &\lesssim |h|^{2\Delta - 1}
\end{align*}
The last inequality comes from when $\theta \in A$, $\left|\theta-\frac{\pi}{|h|}\right|\geq |\theta|/2$,
\begin{align*}
    \int_A |\theta|^{2\tilde{d}} \left|\theta-\frac{\pi}{|h|}\right|^{-2d-1} d\theta \lesssim  \int_A |\theta|^{-2\Delta - 1} d\theta \lesssim |h|^{2\Delta} 
\end{align*}

Putting everything together yields
\begin{align*}
    \left\| I_2^{(h)}  \right\| &\leq C_1 |h|^{2\Delta - 1} + C_2|h|^{-1} = \mathcal{O}\left( |h|^{2\Delta - 1} \right)
\end{align*}
Thus, 
\begin{align}\label{eq:decay_rate_K_nih}
    \left\|\mathbf{K}_{ni,h}\right\|_2 \leq \left\| I_1^{(h)}  \right\| + \left\| I_2^{(h)}  \right\| \leq C|h|^{2\Delta-1}.
\end{align}
\end{proof}

\begin{lemma}\label{lemma:K_holder_factor_wise}
     Under Assumption \ref{assump:model}, \ref{assump:semiparametric_long_memory}, \ref{assump:pervasive_short_mem}, \ref{assump: spec_den_specific_form} case (a),  \ref{assump:bdd_error_eval}, \ref{assump:row_boundness_G}, and \ref{assump:gap_dominance}, $\underline{\mathbf{K}}_{n i}(\theta)$ is H\"older continuous with exponent $\rho^{(q)}_q$, where 
     \begin{align}\label{eq:defn_rho_q_m}
     \rho^{(q)}_m = \min \{\rho^{(q)}_{m-1}, \alpha_{m,m+1}, \min_{0 \leq l\leq m-1}\{2\rho^{(q)}_l - \alpha_{l,m}\} \}, m  = 1, \dots, q,
    \end{align}
     $\alpha_{j_1j_2} = 2|d_{(j_1)}-d_{(j_2)}|$, $d_{(q+1)} \coloneqq 0$, $\rho^{(j)}_0 \coloneqq 1$, $\alpha_{0,j} \coloneqq 1, j = 1,\dots,q$. The H\"older constant is independent of $n, T$.
     Furthermore, 
\begin{align}\label{eq:first_second_derivative_K_bnd}
    \left\|\underline{\mathbf{K}}_{n i}^{(r)}(\theta)\right\|_2 \leq C|\theta|^{\rho^{(q)}_q - r},\quad r = 1,2.
\end{align}
\end{lemma}
\begin{proof}
In this proof, without loss of generality, we assume that $d_1 > d_2 > \cdots >d_q$, and in the following analysis, we take $|1- e^{-\iota \theta}|^{-2d_l} \asymp |\theta|^{-2d_l}$, for simplicity. 

\paragraph{First, we prove that $\mathbf{p}_1^*(\theta)\mathbf{p}_1(\theta)$ is H\"older continuous.} Here, $\mathbf{p}_1(\theta)$ denotes the row eigenvector of $\Sigma(\theta)$ associated with the largest eigenvalue. We first show that $\lim_{\theta\to 0}\mathbf{p}_1^*(\theta)\mathbf{p}_1(\theta)$ exists.



Denote $U_1(\theta) = {\Sigma(\theta)}/{\lambda_1(\theta)}$, where $\lambda_1(\theta)$ is the largest eigenvalue of $\Sigma(\theta)$. 

Since $\Sigma_\xi(\theta)$ is continuous, it is bounded on a neighborhood of zero. 
Moreover, for $\theta\neq 0$, $\Sigma(\theta)$ is continuous, and since the eigenvalues of a Hermitian matrix are continuous functions of its entries, $\lambda_1(\theta)$ is continuous. Moreover, $\Sigma_X(\theta)\succeq \Sigma_\xi(\theta)\succeq \Lambda_\xi^- I_n$, so $\lambda_1(\theta)\geq \Lambda_\xi^- >0$. Thus, $\Sigma_\xi(\theta)/\lambda_1(\theta)$ is continuous away from zero. Near zero, Lemma~\ref{lemma:div_component_eval_specific_form} and Weyl's inequality implies given $n$, $\lambda_1(\theta) \geq \lambda_1^{\chi}(\theta) - \Lambda_{\xi}^{+}\asymp n |\theta|^{-2d_1}$, 
\[
\left\|\frac{\Sigma_\xi(\theta)}{\lambda_1(\theta)}\right\|_{\mathrm{op}}
\leq C |\theta|^{2d_1}/n \to 0.
\]
Therefore, fixing $n$, by defining 
$\Sigma_\xi(0)/\lambda_1(0):=0$, the ratio admits a continuous extension at $\theta=0$. Thus, $\Sigma_{\xi}(\theta)/\lambda_1(\theta)$ is continuous. If we further use the second part of Assumption \ref{assump:bdd_error_eval}, namely that $\Sigma_{\xi}(\theta)$ is Lipschitz continuous, and the fact that $\lambda_1^{-1}(\theta)$ is regular away from zero, then $\Sigma_{\xi}(\theta)/\lambda_1(\theta)$ is H\"older continuous with exponent $2d_1$.

Let $M_1(\theta) \coloneqq G(\theta) diag(1, |\theta|^{2(d_1 - d_2)}, \dots, |\theta|^{2(d_1 - d_q)})G^*(\theta)$, $G(\theta) \coloneqq (g_1(\theta), \dots, g_q(\theta))$ then $\Sigma_{\chi}(\theta) = |\theta|^{-2d_1}M_1(\theta)$, $\lambda^{\chi}_j(\theta) = |\theta|^{-2d_1}\lambda_j\left(M_1(\theta)\right)$, 
\begin{align*}
    \frac{\Sigma_{\chi}(\theta)}{\lambda_1(\theta)} = \frac{|\theta|^{-2d_1}}{\lambda_1(\theta)} M_1(\theta) = \frac{|\theta|^{-2d_1}}{\lambda^{\chi}_1(\theta)} \frac{\lambda^{\chi}_1(\theta)}{\lambda_1(\theta)} M_1(\theta) = \frac{\lambda^{\chi}_1(\theta)}{\lambda_1(\theta)} \frac{M_1(\theta)}{\lambda_1\left(M_1(\theta)\right)}
\end{align*}
By Weyl's inequality, for $j = 1,...,q,$
\begin{align*}
    1 - \frac{\Lambda^{+}_{\xi}}{\lambda_j(\theta)} \leq 1 - \frac{\lambda^{\xi}_1(\theta)}{\lambda_j(\theta)} \leq \frac{\lambda^{\chi}_j(\theta)}{\lambda_j(\theta)} \leq 1 - \frac{\lambda^{\xi}_n(\theta)}{\lambda_j(\theta)} \leq 1 - \frac{\Lambda^{-}_{\xi}}{\lambda_j(\theta)}
\end{align*}
By Lemma \ref{lemma:div_component_eval_specific_form}, as $\theta \rightarrow 0$, $\lambda^{-1}_j(\theta) \asymp (\lambda^{\chi}_j(\theta))^{-1} \asymp C|\theta|^{2d_j}/n$. Thus, $\lambda^{\chi}_j(\theta)/\lambda_j(\theta) \asymp 1- C|\theta|^{2d_j}/n \rightarrow 1$, as $\theta \rightarrow 0$. Meanwhile, $M_1(\theta)$ is continuous and well-defined, and \begin{align}\label{eq:M_1_decomp}
    M_1(\theta) = g_1(\theta)g_1^*(\theta) + \sum_{j=2}^q g_j(\theta)g_j^*(\theta)|\theta|^{2(d_1 - d_j)}
\end{align} then $M_1(0) = G(0) diag(1,0,\dots, 0) G^*(0) = g_1(0)g_1^*(0)$, $\lambda_1\left(M_1(0)\right) = \left\|g_1(0)\right\|_2^2$. Therefore,
\begin{align}\label{eq:limit_U_1}
    \lim_{\theta \rightarrow 0}U_1(\theta) =\lim_{\theta \rightarrow 0}\frac{\Sigma_{\chi}(\theta)}{\lambda_1(\theta)} = \frac{M_1(0)}{\lambda_1(M_1(0))} = \frac{g_1(0)g_1^*(0)}{\left\|g_1(0)\right\|_2^2}
\end{align}
Thus, by defining $U_1(0) = \frac{g_1(0)g_1^*(0)}{\left\|g_1(0)\right\|_2^2}$, $U_1(\theta)$ admits a continuous extension at $\theta=0$. Furthermore, $\mathbf{p}^*_1(U_1(0)) = g_1(0)/\left\|g_1(0)\right\|_2$. Since $\mathbf{p}_1(U_1(\theta)) = \mathbf{p}_1(\theta),$ when $ \theta \neq 0$, by 
\cite{DavisKahan}, 
\begin{align*}
    \left\|\mathbf{p}^*_1(\theta)\mathbf{p}_1(\theta) - \mathbf{p}^*_1(U_1(0))\mathbf{p}_1(U_1(0))\right\| &= \left\|\mathbf{p}^*_1(U_1(\theta))\mathbf{p}_1(U_1(\theta)) - \mathbf{p}^*_1(U_1(0))\mathbf{p}_1(U_1(0))\right\|\\
    &\leq \frac{\left\|U_1(\theta) - U_1(0)\right\|}{\lambda_1(U_1(0))- \lambda_2(U_1(0))} = \left\|U_1(\theta) - U_1(0)\right\|
\end{align*}
where the last inequality is from \eqref{eq:limit_U_1} that $\lambda_1(U_1(0)) = 1$, $\lambda_2(U_1(0)) = \lambda_2(M_1(0))/\lambda_1(M_1(0)) = 0$. Since $U_1(\theta)$ is continuous at $\theta = 0$,
the same holds for $\mathbf{p}^*_1(\theta)\mathbf{p}_1(\theta)$, and we define $\mathbf{p}^*_1(0)\mathbf{p}_1(0) \coloneqq \lim_{\theta \rightarrow 0} \mathbf{p}^*_1(\theta)\mathbf{p}_1(\theta) = \mathbf{p}^*_1(U_1(0))\mathbf{p}_1(U_1(0)) = \frac{g_1(0)g_1^*(0)}{\left\|g_1(0)\right\|_2^2}$.

To see whether it is H\"older continuous on $\Pi$, let $\theta_1, \theta_2 \neq 0$, and notice that for any Hermitian matrix $S_1, S_2$, by Weyl's inequality,
\begin{align}\label{eq:S_divide_lambda_S}
    \left\|\frac{S_1}{\lambda_j(S_1)} - \frac{S_2}{\lambda_j(S_2)}\right\|_2&\leq \left\|\frac{S_1 - S_2}{\lambda_j(S_1)}\right\|_2 + \left\|\frac{S_2}{\lambda_j(S_1)} - \frac{S_2}{\lambda_j(S_2)}\right\|_2\\\nonumber
    &\leq \frac{\left\|S_1 - S_2\right\|_2}{\lambda_j(S_1)} + \left\|S_2\right\|_2\frac{\left\|S_1 - S_2\right\|_2}{\lambda_j(S_1)\lambda_j(S_2)}\\\nonumber
    &\leq \frac{1}{\lambda_j(S_1)}\left(1 + \frac{\lambda_1(S_2)}{\lambda_j(S_2)}\right)\left\|S_1 - S_2\right\|_2\nonumber
\end{align}
and thus, 
\begin{align*}
    \left\|\frac{S_1}{\lambda_1(S_1)} - \frac{S_2}{\lambda_1(S_2)}\right\|_2 \leq \frac{2\left\|S_1 - S_2\right\|_2}{\lambda_1(S_2)}.
\end{align*}
Therefore, since $\lambda_1(U_1(\theta_1)) = \lambda_1(\theta_1)/\lambda_1(\theta_1) = 1$, $\lambda_2(U_1(\theta_1)) = \lambda_2(\theta_1)/\lambda_1(\theta_1) \asymp C_{\lambda_2}|\theta|^{2(d_1 - d_2)}$,
\begin{align*}
    &\left\|\mathbf{p}^*_1(\theta_1)\mathbf{p}_1(\theta_1) - \mathbf{p}^*_1(\theta_2)\mathbf{p}_1(\theta_2)\right\| \leq \frac{\left\|U_1(\theta_1) - U_1(\theta_2)\right\|_2}{\lambda_1(U_1(\theta_1)) - \lambda_2(U_1(\theta_1))} \leq C\left\|U_1(\theta_1) - U_1(\theta_2)\right\|_2\\
    &= \left\|\frac{|\theta_1|^{-2d_1}M_1(\theta_1) + \Sigma_{\xi}(\theta_1)}{\lambda_1(\theta_1)} -  \frac{|\theta_2|^{-2d_1}M_1(\theta_2) + \Sigma_{\xi}(\theta_2)}{\lambda_1(\theta_2)}\right\|\\
    &=  \left\|\frac{M_1(\theta_1) + |\theta_1|^{2d_1}\Sigma_{\xi}(\theta_1)}{|\theta_1|^{2d_1}\lambda_1(\theta_1)} -  \frac{M_1(\theta_2) + |\theta_2|^{2d_1}\Sigma_{\xi}(\theta_2)}{|\theta_2|^{2d_1}\lambda_1(\theta_2)}\right\|\\
    &= \left\|\frac{M_1(\theta_1) + |\theta_1|^{2d_1}\Sigma_{\xi}(\theta_1)}{\lambda_1(M_1(\theta_1) + |\theta_1|^{2d_1}\Sigma_{\xi}(\theta_1))} -  \frac{M_1(\theta_2) + |\theta_2|^{2d_1}\Sigma_{\xi}(\theta_2)}{\lambda_1(M_1(\theta_2) + |\theta_2|^{2d_1}\Sigma_{\xi}(\theta_2))}\right\|\\
    &\leq \frac{2\left\|M_1(\theta_1) - M_1(\theta_2) \right\|_2}{\lambda_1(M_1(\theta_1) + |\theta_1|^{2d_1}\Sigma_{\xi}(\theta_1))} + \frac{2\left\||\theta_1|^{2d_1}\Sigma_{\xi}(\theta_1) - |\theta_2|^{2d_1}\Sigma_{\xi}(\theta_2)\right\|_2}{\lambda_1(M_1(\theta_1) + |\theta_1|^{2d_1}\Sigma_{\xi}(\theta_1))}\\
    &\leq \frac{C}{n}\left\|M_1(\theta_1) - M_1(\theta_2) \right\|_2 + \frac{C}{n}\left\||\theta_1|^{2d_1}\Sigma_{\xi}(\theta_1) - |\theta_2|^{2d_1}\Sigma_{\xi}(\theta_2)\right\|_2
\end{align*}
The last inequality is because $\lambda_1(M_1(\theta_1) + |\theta_1|^{2d_1}\Sigma_{\xi}(\theta_1))\geq \lambda_1(M_1(\theta_1)) = |\theta|^{2d_1}\lambda_1(\Sigma_{\chi}(\theta)) \asymp Cn$. By Assumption \ref{assump:pervasive_short_mem} and \ref{assump:row_boundness_G}, we have $\|g_1(\theta)\|_2^2=\sum_{i=1}^n|g_{i1}(\theta)|^2\le\sum_{i=1}^n \|G_i(\theta)\|_2^2 \le Cn$, and $\|g_1(\theta_1)-g_1(\theta_2)\|_2
\le
\|G(\theta_1)-G(\theta_2)\|_2
\le
\sqrt{n}L_G|\theta_1 - \theta_2|$. Consequently, 
\begin{align*}
    \left\|M_1(\theta_1) - M_1(\theta_2)\right\|_2 &\leq \left\|g_1(\theta_1)g^*_1(\theta_1) - g_1(\theta_2)g^*_1(\theta_2)\right\|_2\\
    &+ \sum_{j=2}^q\left||\theta_1|^{2(d_1 - d_j)}\left\|g_j(\theta_1)g^*_j(\theta_1)\right\|_2 - |\theta_2|^{2(d_1 - d_j)}\left\|g_j(\theta_2)g^*_j(\theta_2)\right\|_2\right|\\
    &\leq C_1n|\theta_1 - \theta_2| + C_2\sum_{j=2}^q|\theta_1|^{2(d_1 - d_j)}\left\|g_j(\theta_1)g^*_j(\theta_1) - g_j(\theta_2)g^*_j(\theta_2)\right\|_2\\
    &+C_3 \sum_{j=2}^q \left||\theta_1|^{2(d_1 - d_j)} - |\theta_2|^{2(d_1 - d_j)}\right| \left\| g_j(\theta_2)g^*_j(\theta_2) \right\|_2\\
    &\leq C_4n|\theta_1 - \theta_2|  + C_3 n \sum_{j=2}^q\left||\theta_1|^{2(d_1 - d_j)} - |\theta_2|^{2(d_1 - d_j)}\right| \\
    &\leq C_4n|\theta_1 - \theta_2|  + C_3 n \sum_{j=2}^q\left||\theta_1| - |\theta_2|\right|^{2(d_1 - d_j)}\\
    &\leq C_4n|\theta_1 - \theta_2|  + C_3 n \sum_{j=2}^q\left|\theta_1 - \theta_2\right|^{2(d_1 - d_j)}\\
    &\leq \begin{cases}
        Cn|\theta_1 - \theta_2|^{2(d_1 - d_2)}, & \text{ if }q \geq 2\\
        Cn|\theta_1 - \theta_2|,& \text{ if }q = 1
    \end{cases}
\end{align*}

The last two inequalities are because of the subadditivity of $t^{\alpha}, \alpha \in (0,1)$, and the triangle inequality. 


Meanwhile, following the same logic, we have
$\left\||\theta_1|^{2d_1}\Sigma_{\xi}(\theta_1) - |\theta_2|^{2d_1}\Sigma_{\xi}(\theta_2)\right\|_2 \leq C|\theta_1 - \theta_2|^{2d_1}$.
Summing up, for any $\theta_1,\theta_2\neq 0$,
\begin{align}\label{eq:P_holder_global}
\left\|\mathbf{p}^*_1(\theta_1)\mathbf{p}_1(\theta_1) - \mathbf{p}^*_1(\theta_2)\mathbf{p}_1(\theta_2)\right\|_2
\leq C|\theta_1 - \theta_2|^{\rho_1^{(q)}}, \rho_1^{(q)} = \alpha_{12}, d_{q+1}= 0. 
\end{align}
For the case where $\theta_1\theta_2 = 0$, since $U_1(\theta)$ admits a continuous extension at $\theta = 0$, the same analysis applies, and we obtain
\begin{align}\label{eq:P_holder_at_0}
\left\|\mathbf{p}^*_1(\theta)\mathbf{p}_1(\theta) - \mathbf{p}^*_1(0)\mathbf{p}_1(0)\right\|_2
\leq C|\theta|^{\rho_1^{(q)}}, \rho_1^{(q)} = \alpha_{12}, d_{q+1}= 0.
\end{align}Therefore, $\mathbf{p}_1^*(\theta)\mathbf{p}_1(\theta)$ is H\"older continuous on $\Pi$ with H\"older exponent $\rho_1^{(q)} = \alpha_{12}$, $d_{q+1}= 0$.

Moreover, denote $P_1(\theta) = \mathbf{p}_1^*(\theta)\mathbf{p}_1(\theta)$, since for any $\theta \neq 0$, $P_1(\theta) = P_1(U_1(\theta))$, then by Lemma \ref{lemma:first_second_derivative_eigenvec_projection}, since $\lambda_1(U_1(\theta)) - \lambda_2(U_1(\theta)) \geq 1 - C_{\lambda}|\theta|^{2(d_1 - d_2)} \geq c$, the first derivative $\left\|P^{(1)}_1(\theta)\right\|_2 \leq C\left\|U^{(1)}_1(\theta)\right\|_2$.

Furthermore, \[
U_1(\theta) = \frac{A_1(\theta)}{\lambda_1(A_1(\theta))},\quad  A_1(\theta) = |\theta|^{2d_1}\Sigma(\theta) = M_1(\theta) + |\theta|^{2d_1}\Sigma_{\xi}(\theta),
\] 
\[
\left\|A^{(1)}_1(\theta)\right\|_2 \leq  \left\|M^{(1)}_1(\theta)\right\|_2 + C|\theta|^{2d_1 - 1}\Lambda_{\xi}^{+}
\]

Notice that $M_1(\theta)$ is a power law combination, taking derivatives once on \eqref{eq:M_1_decomp}, by Assumption \ref{assump:pervasive_short_mem}, $\left\|g_j^{(1)}\right\|_2 \leq C\sqrt{n}$, then $\left\|M^{(1)}_1(\theta)\right\|_2 \leq Cn|\theta|^{\min_{j = 1,\dots, q}2(d_1 - d_j) - 1}$, $\left\|A^{(1)}_1(\theta)\right\|_2 \leq C_A n |\theta|^{2(d_1 - d_2) - 1}$. $\lambda_1(A_1(\theta)) \geq \lambda_1(M_1(\theta)) \asymp C_1n$, by Lemma \ref{lemma:bound_for_first_second_order_U},  $\left\|P^{(1)}_1(\theta)\right\|_2 \leq C\left\|U^{(1)}_1(\theta)\right\|_2 \leq \frac{2}{n}\left\|A^{(1)}_1(\theta)\right\|_2 \leq C|\theta|^{2(d_1 - d_2) - 1}$.

Taking derivatives twice on \eqref{eq:M_1_decomp}, combining with $\left\|g_j^{(2)}\right\|_2 \leq C_2\sqrt{n}$, yields $\left\|M^{(2)}_1(\theta)\right\|_2 \leq Cn|\theta|^{2(d_1 - d_2) - 2}$. 
Consequently, $\left\|A_1^{(2)}(\theta)\right\|_2 \leq Cn|\theta|^{2(d_1 - d_2) - 2}$.
According to Lemma \ref{lemma:first_second_derivative_eigenvec_projection} and \ref{lemma:bound_for_first_second_order_U}, with $\lambda_2(A_1(\theta)) \leq C_2n|\theta|^{2(d_1 - d_2)}$, 
\begin{align*}
    \left\|P^{(2)}_1(\theta)\right\|_2 \leq C\left\|U^{(2)}_1(\theta)\right\|_2 \leq \frac{\left\|A_1^{(1)}(\theta)\right\|_2^2}{n^2} + \frac{\left\|A_1^{(1)}(\theta)\right\|_2 + \left\|A_1^{(2)}(\theta)\right\|_2}{n} \leq C|\theta|^{2(d_1 - d_2) - 2}.
\end{align*}

\paragraph{Second, we prove that $\mathbf{p}_k^*(\theta)\mathbf{p}_k(\theta)$s are H\"older continuous, for all $k$.} 
In this part, we aim to prove two statements. Denote $h_{jk}^{(0)}(\theta) \coloneqq \left(\prod_{i=1}^{k}\Pi_{i}(\theta)\right)g_j(\theta)$, $h_{jk}^{(r)}(\theta)$ is the $r_{th}$ derivative of $h_{jk}^{(0)}(\theta)$.  $H_{jk}^{(0)}(\theta) \coloneqq h_{jk}^{(0)}(\theta)\left(h_{jk}^{(0)}(\theta)\right)^*$. Consequently, $h_{jk}^{(0)}(\theta) = \Pi_{k}(\theta)h_{j,k-1}^{(0)}(\theta) = \left(I - P_{k}(\theta)\right)h_{j,k-1}^{(0)}(\theta)$. Under Assumption \ref{assump:gap_dominance}, for any $m  = 1, \dots, q$,
\begin{enumerate}
    \item \label{stmt:Pk_limit} \textbf{Statement 1(m):} $\mathbf{p}_m^*(0)\mathbf{p}_m(0) = P_m(0)$ exists, and 
    \begin{align}\label{eq:rank_1_P_m_0}
    P_m(0) = \frac{h_{m, m-1}^{(0)}(0)\left(h_{m, m-1}^{(0)}(0)\right)^*}{\left\|h_{m, m-1}^{(0)}(0)\right\|^2_2}.
    \end{align} then 
    \begin{align}\label{eq:h_m_m_0}
        h_{m,m}^{(0)}(0) = \left(I - P_m(0)\right)h_{m,m-1}^{(0)}(0)= \left(I - \frac{h_{m,m-1}^{(0)}(0)\left(h_{m,m-1}^{(0)}(0)\right)^*}{\left\|h_{m,m-1}^{(0)}(0)\right\|_2^2}\right)h_{m,m-1}^{(0)}(0) = 0.
    \end{align}
    \item \label{stmt:Pk_holder} \textbf{Statement 2(m):} \(P_m(\theta)\) is \(\rho^{(q)}_m\)-H\"older continuous for all \(m\leq q\), where
\[
    \rho^{(q)}_m
    \coloneqq
    \min \left\{
    \rho^{(q)}_{m-1},
    \alpha_{m,m+1},
    \min_{l\leq m-1}\left(2\rho^{(q)}_l-\alpha_{l,m}\right)
    \right\},
\]
with
\[
    \alpha_{j_1j_2}
    \coloneqq
    2|d_{j_1}-d_{j_2}|>0,
    \qquad
    d_{q+1} \coloneqq 0,
    \qquad
    \rho^{(j)}_0 \coloneqq 1, 
    \qquad
    \alpha_{0,j} \coloneqq 1,
    \qquad 
    j = 1,\dots,q.
\]
Furthermore, 
\begin{align}\label{eq:first_second_derivative_P_bnd}
    \left\|P^{(1)}_m(\theta)\right\|_2 \leq C|\theta|^{\rho^{(q)}_m - 1}, \quad \left\|P^{(2)}_m(\theta)\right\|_2 \leq C|\theta|^{\rho^{(q)}_m - 2}
\end{align}
\end{enumerate}
We have proved that these two statements hold for $m = 1$. 
By induction, we aim to prove the claim below,\\
\textbf{Claim:} For any $k\leq q$, suppose Statement 1(i) and 2(i) hold for all $i \leq k-1$, 
then Statement 1(k) and 2(k) hold.

To prove such a claim, we aim to solve another claim below first:\\
\textbf{Claim 1:} For any $k\leq q$, suppose Statement 1(i) and 2(i) hold for all $i \leq k-1$, then Statement 1(k) hold.

Consider 
\begin{align*}
    R_{k}(\theta) = \Pi_{k-1}(\theta)R_{k-1}(\theta)\Pi_{k-1}(\theta), k \geq 2,
\end{align*}
$R_1(\theta) \coloneqq \Sigma(\theta)$, 
then $R_{k}(\theta) = \sum_{j=k}^n \lambda_j(\theta)\mathbf{p}_j^*(\theta)\mathbf{p}_j(\theta)$. Hence, $\lambda_{k}(\theta), \dots, \lambda_n(\theta)$ are the nonzero eigenvalues of $R_{k}(\theta),$ and $\mathbf{p}_j(\theta), j = k, \dots, n$ are the row eigenvectors of $R_{k}(\theta)$ associated with nonzero eigenvalues. Further, let \begin{align*}
    U_{k}(\theta) = \frac{R_{k}(\theta)}{\lambda_{k}(\theta)}, \quad M_{k}(\theta) = |\theta|^{2d_{k}}\Sigma_{\chi}(\theta),\quad A_k(\theta) = |\theta|^{2d_{k}}R_k(\theta)
\end{align*} then since $|\theta|^{2d_k}\lambda_k(\theta) = |\theta|^{2d_k}\lambda_1(R_k(\theta)) = \lambda_1(|\theta|^{2d_k}R_k(\theta))$, \[
U_k(\theta) =  \frac{|\theta|^{2d_k}R_{k}(\theta)}{|\theta|^{2d_k}\lambda_{k}(\theta)} = \frac{|\theta|^{2d_k}R_{k}(\theta)}{\lambda_1(|\theta|^{2d_k}R_k(\theta))}
\]

From the definition of $R_k$, we have that for $\theta \neq 0$,
\begin{align}\label{eq:Rk_decomp}
    A_k(\theta) &= |\theta|^{2d_k}R_{k}(\theta) \\\nonumber
    &= \left(\prod_{i=1}^{k-1}\Pi_i(\theta)\right) {M_k(\theta)} \left(\prod_{i=1}^{k-1}\Pi_i(\theta)\right)+ \left(\prod_{i=1}^{k-1}\Pi_i(\theta)\right) |\theta|^{2d_k}  \Sigma_{\xi}(\theta) \left(\prod_{i=1}^{k-1}\Pi_i(\theta)\right) \\\nonumber
    \prod_{i=1}^{k-1}\Pi_i(\theta) &\coloneqq \Pi_{k-1}(\theta)\cdots \Pi_{1}(\theta).\nonumber
\end{align}

The second term 
\begin{align}\label{eq:Pi_xi_Pi_decomp}
    \left\|\left(\prod_{i=1}^{k-1}\Pi_i(\theta)\right) |\theta|^{2d_k} \Sigma_{\xi}(\theta) \left(\prod_{i=1}^{k-1}\Pi_i(\theta)\right)\right\|_2 \leq \Lambda_{\xi}^{+} {|\theta|^{2d_k}}\left\|\prod_{i=1}^{k-1}\Pi_i(\theta)\right\|_2^2 \leq \Lambda_{\xi}^{+} {|\theta|^{2d_k}} \rightarrow 0, 
\end{align} as $\theta \rightarrow 0$.

To investigate 
the first term in \eqref{eq:Rk_decomp}, 
\begin{align}\label{eq:Pi_Mk_Pi_decomp}
    \left(\prod_{i=1}^{k-1}\Pi_i(\theta)\right) M_k(\theta) \left(\prod_{i=1}^{k-1}\Pi_i(\theta)\right) &= \sum_{j=1}^{k-1}|\theta|^{-\alpha_{jk}} h_{j,k-1}^{(0)}(\theta)\left(h_{j,k-1}^{(0)}(\theta)\right)^*\\\nonumber
    &+ h_{k,k-1}^{(0)}(\theta)\left(h_{k,k-1}^{(0)}(\theta)\right)^*\\\nonumber
    &+ \sum_{j=k+1}^{q}|\theta|^{\alpha_{jk}} h_{j,k-1}^{(0)}(\theta)\left(h_{j,k-1}^{(0)}(\theta)\right)^*\nonumber
\end{align}
The third term $\rightarrow 0$, as $\theta \rightarrow 0$, since $\left\|h_{j,k-1}^{(0)}(\theta)\right\|_2 \leq \prod_{i=1}^{k-1}\left\|\Pi_i(\theta)\right\|_2\left\|g_j(\theta)\right\|_2\leq C\sqrt{n}$,
and that $\Pi_i(0)$ exists for $i \leq k-1$ as $P_i(0)$ exists for $i \leq k-1$.

Under Statement 2(k-1), Lemma \ref{lemma:holder_successive_h} holds for $m = k-1$. By Lemma \ref{lemma:holder_successive_h}, the second term 
\begin{align*}
    \lim_{\theta \rightarrow 0} h_{k,k-1}^{(0)}(\theta)\left(h_{k,k-1}^{(0)}(\theta)\right)^* = h_{k,k-1}^{(0)}(0)\left(h_{k,k-1}^{(0)}(0)\right)^*
\end{align*}

For the first term,  notice that for $j \leq k-1$,
\begin{align*}
    h_{j,j}^{(0)}(0) = \lim_{\theta \rightarrow 0}h_{j,j}^{(0)}(\theta) = \lim_{\theta \rightarrow 0}\Pi_j(\theta)h_{j,j-1}^{(0)}(\theta) = \lim_{\theta \rightarrow 0}\left(I - P_j(\theta)\right)h_{j,j-1}^{(0)}(\theta).
\end{align*}
For $j \leq k-1$, $P_j(0)$ and $h_{j,j-1}^{(0)}(0) = \left(\prod_{i=1}^{j-1}\Pi_{i}(0)\right)g_j(0)$ exist. Thus, by Statement 1(j), $h_{j,j}^{(0)}(0) = 0$.
Under Statement 2(j), Lemma \ref{lemma:holder_successive_h} holds for $m = j$, then 
\begin{align}\label{eq:h_j_j_0}
    \left\|h_{j,j}^{(0)}(\theta)\right\|_2 = \left\|h_{j,j}^{(0)}(\theta) - h_{j,j}^{(0)}(0)\right\|_2 \leq C\sqrt{n}|\theta|^{\rho^{(q)}_j},
\end{align} for $j \leq k-1$,
\begin{align}\label{eq:rate_h_j_k-1}
    \left\|h_{j,k-1}^{(0)}(\theta)\right\|_2 = \left\|\prod_{i = j+1}^{k-1}\Pi_i(\theta) h_{j,j}^{(0)}(\theta)\right\|_2 \leq \left\|h_{j,j}^{(0)}(\theta)\right\|_2 \leq C\sqrt{n}|\theta|^{\rho^{(q)}_j} 
\end{align}
The first term in \eqref{eq:Pi_Mk_Pi_decomp} $\leq \sum_{j=1}^{k-1} |\theta|^{-\alpha_{jk}}\left\|h_{j,k-1}^{(0)}(\theta)\right\|_2^2 \leq Cn\sum_{j=1}^{k-1} |\theta|^{2\rho^{(q)}_{j}-\alpha_{jk}}$. Under Assumption \ref{assump:gap_dominance}, $2\rho^{(q)}_{j}-\alpha_{jk} > 0$, the first term vanishes as $\theta \rightarrow 0$. Plugging back into \eqref{eq:Rk_decomp} yields \[
\lim_{\theta \rightarrow 0} A_k(\theta) = h_{k,k-1}^{(0)}(0)\left(h_{k,k-1}^{(0)}(0)\right)^*,
\] and
\[
\lambda_1(A_k(\theta)) = \left\|h_{k,k-1}^{(0)}(0)\right\|_2^2.
\]

Then define
\begin{align*}
    U_k(0) \coloneqq \lim_{\theta\rightarrow 0}U_k(\theta) = \lim_{\theta\rightarrow 0} \frac{A_k(\theta)}{\lambda_1(A_k(\theta))} = \frac{h_{k,k-1}^{(0)}(0)\left(h_{k,k-1}^{(0)}(0)\right)^*}{\left\|h_{k,k-1}^{(0)}(0)\right\|^2_2}.
\end{align*}
Since $U_k(0)$ is a rank-1 matrix, and $\lambda_1(U_k(0)) = 1$,
\begin{align*}
     \mathbf{p}_{k}^*(0)\mathbf{p}_{k}(0) = P_k(0) = \frac{U_k(0)}{\lambda_1(U_k(0))} = \frac{h_{k,k-1}^{(0)}(0)\left(h_{k,k-1}^{(0)}(0)\right)^*}{\left\|h_{k,k-1}^{(0)}(0)\right\|^2_2}
\end{align*}
Statement 1(k) holds.

Second, we aim to prove Statement 2(k) holds. 
By induction, we aim to prove the claim below:\\
\textbf{Claim 2:} Suppose Statement \ref{stmt:Pk_limit}(i) and \ref{stmt:Pk_holder}(i) hold for $i = 1, \dots, k-1$. Furthermore, Statement \ref{stmt:Pk_limit}(k) holds as well, then Statement \ref{stmt:Pk_holder}(k) holds.



For any $\theta_1, \theta_2 \in \Pi$,
\begin{align*}
    &\left\|P_{k}(\theta_1) - P_{k}(\theta_2) \right\|_2=\left\|\mathbf{p}^*_{k}(\theta_1)\mathbf{p}_{k}(\theta_1) - \mathbf{p}^*_{k}(\theta_2)\mathbf{p}_{k}(\theta_2)\right\|_2 \\
    &=\left\|\mathbf{p}^*_1(U_{k}(\theta_1))\mathbf{p}_1(U_{k}(\theta_1)) - \mathbf{p}^*_1(U_{k}(\theta_2))\mathbf{p}_1(U_{k}(\theta_2))\right\|_2\\
    &\leq \frac{\left\|U_k(\theta_1) - U_k(\theta_2)\right\|_2}{\lambda_1(U_k(\theta_1)) - \lambda_2(U_k(\theta_1))}= \frac{\left\|U_k(\theta_1) - U_k(\theta_2)\right\|_2}{\lambda_1(R_k(\theta_1))/\lambda_k(\theta_1) - \lambda_2(R_k(\theta_1))/\lambda_k(\theta_1)}\\
    &= \frac{\left\|U_k(\theta_1) - U_k(\theta_2)\right\|_2}{\lambda_k(\theta_1)/\lambda_k(\theta_1) - \lambda_{k+1}(\theta_1)/\lambda_k(\theta_1)}\leq C\left\|U_k(\theta_1) - U_k(\theta_2)\right\|_2
\end{align*}

If $\theta_1, \theta_2 \neq 0$, since $|\theta|^{2d_k}\lambda_k(\theta) = |\theta|^{2d_k}\lambda_1(R_k(\theta)) = \lambda_1(|\theta|^{2d_k}R_k(\theta))$,
\begin{align*}
    \left\|U_k(\theta_1) - U_k(\theta_2)\right\|_2 &= \left\|\frac{R_k(\theta_1)}{\lambda_k(\theta_1)} - \frac{R_k(\theta_2)}{\lambda_k(\theta_2)}\right\|_2\\
    &= \left\|\frac{|\theta_1|^{2d_k}R_k(\theta_1)}{\lambda_1(|\theta_1|^{2d_k}R_k(\theta_1))} - \frac{|\theta_2|^{2d_k}R_k(\theta_2)}{\lambda_1(|\theta_2|^{2d_k}R_k(\theta_2))}\right\|_2\\
    &\leq \frac{2}{\lambda_1(|\theta_1|^{2d_k}R_k(\theta_1))}\left\||\theta_1|^{2d_k}R_k(\theta_1) - |\theta_2|^{2d_k}R_k(\theta_2)\right\|_2\\
    &\leq \frac{C}{n}\left\||\theta_1|^{2d_k}R_k(\theta_1) - |\theta_2|^{2d_k}R_k(\theta_2)\right\|_2 \leq \frac{C}{n}\left\|A_k(\theta_1) - A_k(\theta_2)\right\|_2
\end{align*}
where the last two inequalities are from \eqref{eq:S_divide_lambda_S}, and by Lemma \ref{lemma:div_component_eval_specific_form}, $\lambda_1(|\theta_1|^{2d_k}R_k(\theta_1)) = |\theta_1|^{2d_k}\lambda_1(R_k(\theta_1)) = |\theta_1|^{2d_k}\lambda_k(\Sigma(\theta_1)) \asymp Cn|\theta_1|^{-2d_k}|\theta_1|^{2d_k} \asymp Cn$.

Notice that 
\begin{align}\label{eq:theta_2dk_Rk}
    A_k(\theta) &= \left(\prod_{i=1}^{k-1}\Pi_i(\theta)\right) M_k(\theta) \left(\prod_{i=1}^{k-1}\Pi_i(\theta)\right) + \left(\prod_{i=1}^{k-1}\Pi_i(\theta)\right) |\theta|^{2d_k} \Sigma_{\xi}(\theta) \left(\prod_{i=1}^{k-1}\Pi_i(\theta)\right) 
\end{align}

Since each term in $\left(\prod_{i=1}^{k-1}\Pi_i(\theta)\right) |\theta|^{2d_k} \Sigma_{\xi}(\theta) \left(\prod_{i=1}^{k-1}\Pi_i(\theta)\right)$ is bounded, by the product of bounded H\"older functions is H\"older with the minimum exponent, $\left(\prod_{i=1}^{k-1}\Pi_i(\theta)\right) |\theta|^{2d_k} \Sigma_{\xi}(\theta) \left(\prod_{i=1}^{k-1}\Pi_i(\theta)\right)$ is H\"older continous with exponent $\min\{2d_k, \rho^{(q)}_{k-1}\}$.


\begin{align}\label{eq:decomp_Pi_Mk_Pi}
    &\left\|\left(\prod_{i=1}^{k-1}\Pi_i(\theta_1)\right) M_k(\theta_1) \left(\prod_{i=1}^{k-1}\Pi_i(\theta_1)\right) - \left(\prod_{i=1}^{k-1}\Pi_i(\theta_2)\right) M_k(\theta_2) \left(\prod_{i=1}^{k-1}\Pi_i(\theta_2)\right)\right\|_2\\\nonumber
    &\leq \sum_{j = 1}^{k-1}\left\| |\theta_1|^{-\alpha_{jk}} H_{j,k-1}^{(0)}(\theta_1) - |\theta_2|^{-\alpha_{jk}} H_{j,k-1}^{(0)}(\theta_2) \right\|_2\\\nonumber
    &+ \left\| H_{k,k-1}^{(0)}(\theta_1) - H_{k,k-1}^{(0)}(\theta_2)  \right\|_2\\\nonumber
    &+ \sum_{j=k+1}^q \left\| |\theta_1|^{\alpha_{jk}} H_{j,k-1}^{(0)}(\theta_1) - |\theta_2|^{\alpha_{jk}} H_{j,k-1}^{(0)}(\theta_2) \right\|_2\nonumber
\end{align}

By Lemma \ref{lemma:holder_successive_h}, $H_{j,k-1}^{(0)}(\theta)$ is $\rho^{(q)}_{k-1}$-H\"older continuous, 
for all $j = 1, \dots, q$. Therefore, the second term becomes
\begin{align*}
    \left\| H_{k,k-1}^{(0)}(\theta_1) - H_{k,k-1}^{(0)}(\theta_2)  \right\|_2 \leq Cn|\theta_1 - \theta_2|^{\rho^{(q)}_{k-1}}.
\end{align*}
Meanwhile, the third term, $j \geq k+1$,
\begin{align*}
    &\left\||\theta_1|^{\alpha_{jk}}H_{j,k-1}^{(0)}(\theta_1) - |\theta_2|^{\alpha_{jk}}H_{j,k-1}^{(0)}(\theta_2)\right\|_2 \\
    &\leq |\theta_1|^{\alpha_{jk}}
\left\|H_{j,k-1}^{(0)}(\theta_1)-H_{j,k-1}^{(0)}(\theta_2)\right\|_2
+
\left||\theta_1|^{\alpha_{jk}}-|\theta_2|^{\alpha_{jk}}\right|
\left\|H_{j,k-1}^{(0)}(\theta_2)\right\|_2\\
&\leq Cn|\theta_1 - \theta_2|^{\rho^{(q)}_{k-1}} + Cn|\theta_1 - \theta_2|^{\alpha_{jk}} \leq Cn|\theta_1 - \theta_2|^{\min\{\rho^{(q)}_{k-1},\alpha_{jk}\}}
\end{align*} where the second last inequality is by the subadditivity of the concave function $t^{\alpha_{jk}}$, given that $\alpha_{jk} \in (0,1)$. Then, the third term can be controlled by 
\begin{align}\label{eq:third_term_rate_j_geq_k+1}
    &\left\||\theta_1|^{\alpha_{jk}}H_{j,k-1}^{(0)}(\theta_1) - |\theta_2|^{\alpha_{jk}}H_{j,k-1}^{(0)}(\theta_2)\right\|_2\\\nonumber
    &\leq Cn|\theta_1 - \theta_2|^{\min\{\rho^{(q)}_{k-1}, \min_{j\geq k+1}\alpha_{jk}\}} = Cn|\theta_1 - \theta_2|^{\min\{\rho^{(q)}_{k-1},\alpha_{k,k+1}\}}\nonumber
\end{align}
For the first term in \eqref{eq:decomp_Pi_Mk_Pi}, for each $j \leq k-1$, $|\theta_1| \leq |\theta_2|$,
\begin{align*}
    &\left\||\theta_1|^{-\alpha_{jk}} H_{j,k-1}^{(0)}(\theta_1) - |\theta_2|^{-\alpha_{jk}} H_{j,k-1}^{(0)}(\theta_2)\right\|_2 \\
    &\leq 
|\theta_2|^{-\alpha_{jk}}
\left\|H_{j,k-1}^{(0)}(\theta_1)-H_{j,k-1}^{(0)}(\theta_2)\right\|_2
+
\left||\theta_1|^{-\alpha_{jk}}-|\theta_2|^{-\alpha_{jk}}\right|
\left\|H_{j,k-1}^{(0)}(\theta_1)\right\|_2
\end{align*}
By \eqref{eq:h_j_m_upperbnd} in Lemma \ref{lemma:upperbnd_h_j_m}, when $j \leq k-1$,
\begin{align*}
    &|\theta_2|^{-\alpha_{jk}}
\left\|H_{j,k-1}^{(0)}(\theta_1)-H_{j,k-1}^{(0)}(\theta_2)\right\|_2\\
&\leq |\theta_2|^{-\alpha_{jk}}\left(\left\|h_{j,k-1}^{(0)}(\theta_1)\right\|_2 + \left\|h_{j,k-1}^{(0)}(\theta_2)\right\|_2\right)\left\|h_{j,k-1}^{(0)}(\theta_1) - h_{j,k-1}^{(0)}(\theta_2)\right\|_2\\
&\leq Cn|\theta_2|^{- \alpha_{jk}}\left(|\theta_1|^{\rho^{(q)}_j} + |\theta_2|^{\rho^{(q)}_j}\right)\left[
|\theta_1-\theta_2|^{\rho^{(q)}_j}
+
\left(|\theta_1|^{\rho^{(q)}_j}+|\theta_2|^{\rho^{(q)}_j}\right)
|\theta_1-\theta_2|^{\rho^{(q)}_{k-1}}
\right]\\
&\leq Cn|\theta_2|^{\rho^{(q)}_j - \alpha_{jk}}|\theta_1-\theta_2|^{\rho^{(q)}_j} + Cn|\theta_2|^{2\rho^{(q)}_j - \alpha_{jk}}|\theta_1-\theta_2|^{\rho^{(q)}_{k-1}}
\end{align*}
where the last inequality is from $|\theta_1| \leq |\theta_2|$. By Assumption \ref{assump:gap_dominance}, $Cn|\theta_2|^{2\rho^{(q)}_j - \alpha_{jk}}|\theta_1-\theta_2|^{\rho^{(q)}_{k-1}} \leq Cn|\theta_1-\theta_2|^{\rho^{(q)}_{k-1}}$. The first term $\leq Cn|\theta_1 - \theta_2|^{\rho^{(q)}_j}$ if $\rho^{(q)}_j \geq \alpha_{jk}$. If $\rho^{(q)}_j < \alpha_{jk}$, using $|\theta_2|\geq |\theta_1 - \theta_2|/2$, $|\theta_2|^{\rho^{(q)}_j - \alpha_{jk}} \leq C|\theta_1 - \theta_2|^{\rho^{(q)}_j - \alpha_{jk}}$. Then 
\begin{align*}
    |\theta_2|^{-\alpha_{jk}}
\left\|H_{j,k-1}^{(0)}(\theta_1)-H_{j,k-1}^{(0)}(\theta_2)\right\|_2 \leq Cn|\theta_1 - \theta_2|^{\min\{\rho^{(q)}_{k-1}, 2\rho^{(q)}_j - \alpha_{jk}\}}.
\end{align*}

Meanwhile, for $t \in (0,1), \alpha \in (0,1)$, by subadditivity of concave function $t^{\alpha}$, we have $1 = (t+1-t)^{\alpha} \leq t^{\alpha} + (1-t)^{\alpha}$, $|1-t^{\alpha}| = (1-t^{\alpha})\leq (1-t)^{\alpha} = |1-t|^{\alpha}$,
\begin{align*}
    &\left||\theta_1|^{-\alpha_{jk}}-|\theta_2|^{-\alpha_{jk}}\right|
\left\|H_{j,k-1}^{(0)}(\theta_1)\right\|_2 \leq Cn\left||\theta_1|^{-\alpha_{jk}}-|\theta_2|^{-\alpha_{jk}}\right| |\theta_1|^{2\rho^{(q)}_j}\\
&= Cn\left|1 - \left|\frac{\theta_1}{\theta_2}\right|^{\alpha_{jk}}\right||\theta_1|^{2\rho^{(q)}_j - \alpha_{jk}}\\
&\leq Cn\left|1 - \left|\frac{\theta_1}{\theta_2}\right|\right|^{\alpha_{jk}}|\theta_1|^{2\rho^{(q)}_j - \alpha_{jk}}\\
&\leq Cn \left|1 - \left|\frac{\theta_1}{\theta_2}\right|\right|^{\min\{\alpha_{jk}, 2\rho^{(q)}_j - \alpha_{jk}\}}|\theta_1|^{\min\{\alpha_{jk}, 2\rho^{(q)}_j - \alpha_{jk}\}}\\
&\leq Cn \left|\frac{\theta_1}{\theta_2}\right|^{\min\{\alpha_{jk}, 2\rho^{(q)}_j - \alpha_{jk}\}}\left|1 - \left|\frac{\theta_2}{\theta_1}\right|\right|^{\min\{\alpha_{jk}, 2\rho^{(q)}_j - \alpha_{jk}\}}|\theta_1|^{\min\{\alpha_{jk}, 2\rho^{(q)}_j - \alpha_{jk}\}}\\
&\leq Cn\left|1 - \left|\frac{\theta_2}{\theta_1}\right|\right|^{\min\{\alpha_{jk}, 2\rho^{(q)}_j - \alpha_{jk}\}}|\theta_1|^{\min\{\alpha_{jk}, 2\rho^{(q)}_j - \alpha_{jk}\}}\\
&\leq Cn\left||\theta_1| - |\theta_2|\right|^{\min\{\alpha_{jk}, 2\rho^{(q)}_j - \alpha_{jk}\}} \leq Cn\left|\theta_1 - \theta_2\right|^{\min\{\alpha_{jk}, 2\rho^{(q)}_j - \alpha_{jk}\}}.
\end{align*}Therefore, for $|\theta_1| \leq |\theta_2|$, we have $\left\||\theta_1|^{-\alpha_{jk}}H_{j,k-1}^{(0)}(\theta_1) - |\theta_2|^{-\alpha_{jk}}H_{j,k-1}^{(0)}(\theta_2)\right\|_2 \leq Cn|\theta_1 - \theta_2|^{\min\{\alpha_{jk}, \rho^{(q)}_{k-1}, 2\rho^{(q)}_j - \alpha_{jk}\}}$.

For $|\theta_1| > |\theta_2|$, decomposing in another direction yields
\begin{align*}
    &\left\||\theta_1|^{-\alpha_{jk}} H_{j,k-1}^{(0)}(\theta_1) - |\theta_2|^{-\alpha_{jk}} H_{j,k-1}^{(0)}(\theta_2)\right\|_2 \\
    &\leq 
|\theta_1|^{-\alpha_{jk}}
\left\|H_{j,k-1}^{(0)}(\theta_1)-H_{j,k-1}^{(0)}(\theta_2)\right\|_2
+
\left||\theta_1|^{-\alpha_{jk}}-|\theta_2|^{-\alpha_{jk}}\right|
\left\|H_{j,k-1}^{(0)}(\theta_2)\right\|_2\\
&\leq Cn|\theta_1 - \theta_2|^{\min\{\alpha_{jk}, \rho^{(q)}_{k-1}, 2\rho^{(q)}_j - \alpha_{jk}\}} 
\end{align*} where the last inequality is a change of variables from the previous analysis.

Overall, the first term in \eqref{eq:decomp_Pi_Mk_Pi} can be controlled by
\begin{align*}
    \sum_{j = 1}^{k-1}\left\| |\theta_1|^{-\alpha_{jk}} H_{j,k-1}^{(0)}(\theta_1) - |\theta_2|^{-\alpha_{jk}} H_{j,k-1}^{(0)}(\theta_2) \right\|_2&\leq Cn \sum_{j=1}^{k-1}|\theta_1 - \theta_2|^{\min\{\alpha_{jk}, \rho^{(q)}_{k-1}, 2\rho^{(q)}_j - \alpha_{jk}\}}\\
    &\leq Cn|\theta_1 - \theta_2|^{\min\{ \rho^{(q)}_{k-1}, \alpha_{k-1,k}, \min_{j \leq k-1} \{2\rho^{(q)}_j - \alpha_{jk}\}\}}
\end{align*}
Therefore, combining with \eqref{eq:third_term_rate_j_geq_k+1}, if $k \leq q-1$,
\begin{align*}
    &\left\|A_k(\theta_1) - A_k(\theta_2)\right\|_2 \leq  Cn|\theta_1 - \theta_2|^{\rho_k^{(q)}}, \\
    &\rho_k^{(q)} = \min\{ \rho_{k-1}^{(q)}, \alpha_{k-1,k}, \alpha_{k,k+1}, \min_{j \leq k-1} \{2\rho_j^{(q)} - \alpha_{jk}\}\}
\end{align*}
If $k = q,$ $\rho_q^{(q)} = \min\{ \rho_{q-1}^{(q)}, \alpha_{q-1,q}, 2d_q, \min_{j \leq q-1} \{2\rho_j^{(q)} - \alpha_{jq}\}\}$. Finally, $\left\|P_{k}(\theta_1) - P_{k}(\theta_2) \right\|_2 \leq C|\theta_1 - \theta_2|^{\rho_k^{(q)}}$. 

If one of $\theta_i = 0$, 
\begin{align*}
    \left\|U_k(\theta) - U_k(0)\right\|_2 &= \left\|\frac{R_k(\theta)}{\lambda_k(\theta)} - \frac{h_{k,k-1}^{(0)}(0)\left(h_{k,k-1}^{(0)}(0)\right)^*}{\left\|h_{k,k-1}^{(0)}(0)\right\|^2_2} \right\|_2\\
    &=\left\|\frac{A_k(\theta)}{\lambda_1(A_k(\theta))} - \frac{h_{k,k-1}^{(0)}(0)\left(h_{k,k-1}^{(0)}(0)\right)^*}{\lambda_1\left(h_{k,k-1}^{(0)}(0)\left(h_{k,k-1}^{(0)}(0)\right)^*\right)} \right\|_2\\
    &\leq \frac{2}{\lambda_1(A_k(\theta))}\left\|A_k(\theta) - h_{k,k-1}^{(0)}(0)\left(h_{k,k-1}^{(0)}(0)\right)^*\right\|_2\\
    &\leq \frac{C}{n}\left\|A_k(\theta) - h_{k,k-1}^{(0)}(0)\left(h_{k,k-1}^{(0)}(0)\right)^*\right\|_2
\end{align*} where the last two inequalities are from \eqref{eq:S_divide_lambda_S}, and by Lemma \ref{lemma:div_component_eval_specific_form}, $\lambda_1(|\theta|^{2d_k}R_k(\theta)) = |\theta|^{2d_k}\lambda_1(R_k(\theta)) = |\theta|^{2d_k}\lambda_k(\Sigma(\theta)) \asymp Cn|\theta|^{-2d_k}|\theta|^{2d_k} \asymp Cn$.

From \eqref{eq:Rk_decomp}, \eqref{eq:Pi_xi_Pi_decomp} and \eqref{eq:Pi_Mk_Pi_decomp} we obtain
\begin{align*}
    \left\|A_k(\theta) - h_{k,k-1}^{(0)}(0)\left(h_{k,k-1}^{(0)}(0)\right)^*\right\|_2 &\leq \Lambda_{\xi}^{+}|\theta|^{2d_k} + \sum_{j=1}^{k-1}|\theta|^{-\alpha_{jk}} \left\| h_{j,k-1}^{(0)}(\theta)\right\|_2^2\\\nonumber
    &+ \left\|h_{k,k-1}^{(0)}(\theta)\left(h_{k,k-1}^{(0)}(\theta)\right)^* - h_{k,k-1}^{(0)}(0)\left(h_{k,k-1}^{(0)}(0)\right)^*\right\|_2\\\nonumber
    &+ \sum_{j=k+1}^{q}|\theta|^{\alpha_{jk}} \left\| h_{j,k-1}^{(0)}(\theta)\right\|_2^2\nonumber
\end{align*}
By Lemma \ref{lemma:upperbnd_h_j_m}, 
\[
\sum_{j=1}^{k-1}|\theta|^{-\alpha_{jk}} \left\| h_{j,k-1}^{(0)}(\theta)\right\|_2^2 \leq Cn \sum_{j=1}^{k-1} |\theta|^{2\rho^{(q)}_j - \alpha_{jk}} \leq Cn |\theta|^{\min_{j\leq k-1} 2\rho^{(q)}_j - \alpha_{jk}}.
\]
By Lemma \ref{lemma:holder_successive_h}, we have 
\[
\left\|h_{k,k-1}^{(0)}(\theta)\left(h_{k,k-1}^{(0)}(\theta)\right)^* - h_{k,k-1}^{(0)}(0)\left(h_{k,k-1}^{(0)}(0)\right)^*\right\|_2 \leq Cn|\theta|^{\rho_{k-1}^{(q)}},
\] and $\sum_{j=k+1}^{q}|\theta|^{\alpha_{jk}} \left\| h_{j,k-1}^{(0)}(\theta)\right\|_2^2$ is at least $\alpha_{k,k+1}$-H\"older continuous.

To see the first and second derivatives of $P_k(\theta)$, we first investigate the derivatives of $h_{j,k-1}^{(r)}(\theta), r = 1, 2$. By definition, for all $ j \leq q$,
\begin{align*}
    h_{j,k-1}^{(1)}(\theta)
=
\sum_{\ell=1}^{k-1}
\left(\prod_{i=\ell+1}^{k-1}\Pi_i(\theta)\right)
\Pi_\ell^{(1)}(\theta)
\left(\prod_{i=1}^{\ell-1}\Pi_i(\theta)\right)
g_j(\theta)
+
\left(\prod_{i=1}^{k-1}\Pi_i(\theta)\right)
g_j^{(1)}(\theta),
\end{align*}
then from Statement \ref{stmt:Pk_holder}($\ell$), $\left\|\Pi_\ell^{(1)}(\theta)\right\|_2 = \left\|P_\ell^{(1)}(\theta)\right\|_2 \leq |\theta|^{\rho_{\ell}^{(q)}- 1}$,
\begin{align}\label{eq:h_j_k_1_first_order}
    \left\|h_{j,k-1}^{(1)}(\theta)\right\|_2 \leq C\sqrt{n}\sum_{\ell=1}^{k-1}|\theta|^{\rho_\ell^{(q)}- 1} \leq C\sqrt{n}|\theta|^{\rho_{k-1}^{(q)}- 1}, j \leq q
\end{align}
Particularly, if $j \leq k-1$, by relation $h_{j,k-1}^{(0)}(\theta) = \prod_{i = j+1}^{k-1}\Pi_i(\theta) h_{j,j}^{(0)}(\theta)$, 
\begin{align}\label{eq:h_j_k_1_first_order_j_leq_k_1_decomp}
    h_{j,k-1}^{(1)}(\theta)
=
\sum_{\ell=j+1}^{k-1}
\left(
\prod_{i=\ell+1}^{k-1}\Pi_i(\theta)
\right)
\Pi_\ell^{(1)}(\theta)
\left(
\prod_{i=j+1}^{\ell-1}\Pi_i(\theta)
\right)
h_{j,j}^{(0)}(\theta)
+
\left(
\prod_{i=j+1}^{k-1}\Pi_i(\theta)
\right)
h_{j,j}^{(1)}(\theta).
\end{align} then combining with \eqref{eq:h_j_j_0}, 
\begin{align}\label{eq:h_j_k_1_first_order_j_leq_k_1}
    \left\|h_{j,k-1}^{(1)}(\theta)\right\|_2 \leq C\sqrt{n}\sum_{\ell=j+1}^{k-1}|\theta|^{\rho_\ell^{(q)}- 1}|\theta|^{\rho_j^{(q)}} \leq C\sqrt{n}|\theta|^{\rho_{k-1}^{(q)}- 1}|\theta|^{\rho_j^{(q)}} \leq C\sqrt{n}|\theta|^{\rho_{j}^{(q)}- 1}, j \leq k-1.
\end{align}

Similarly, 
\[
\begin{aligned}
h_{j,k-1}^{(2)}(\theta)
=&
\sum_{\ell=1}^{k-1}
\left(\prod_{i=\ell+1}^{k-1}\Pi_i(\theta)\right)
\Pi_\ell^{(2)}(\theta)
\left(\prod_{i=1}^{\ell-1}\Pi_i(\theta)\right)
g_j(\theta) \\
&+
2\sum_{\ell=1}^{k-1}
\left(\prod_{i=\ell+1}^{k-1}\Pi_i(\theta)\right)
\Pi_\ell^{(1)}(\theta)
\left(\prod_{i=1}^{\ell-1}\Pi_i(\theta)\right)
g_j^{(1)}(\theta) \\
&+
\left(\prod_{i=1}^{k-1}\Pi_i(\theta)\right)
g_j^{(2)}(\theta) \\
&+
2\sum_{1\leq \ell_1<\ell_2\leq k-1}
\left(\prod_{i=\ell_2+1}^{k-1}\Pi_i(\theta)\right)
\Pi_{\ell_1}^{(1)}(\theta)
\left(\prod_{i=\ell_1+1}^{\ell_2-1}\Pi_i(\theta)\right)
\Pi_{\ell_2}^{(1)}(\theta)
\left(\prod_{i=1}^{\ell_1-1}\Pi_i(\theta)\right)
g_j(\theta).
\end{aligned}
\]
then from Statement \ref{stmt:Pk_holder}($\ell$), $\left\|\Pi_\ell^{(2)}(\theta)\right\|_2 = \left\|P_\ell^{(2)}(\theta)\right\|_2 \leq |\theta|^{\rho_\ell^{(q)}- 2}, \ell= 1,\dots, k-1$, \begin{align}
    \left\|h_{j,k-1}^{(2)}(\theta)\right\|_2 \leq C\sqrt{n}\sum_{\ell=1}^{k-1}|\theta|^{\rho_\ell^{(q)}- 2} \leq C\sqrt{n}|\theta|^{\rho_{k-1}^{(q)}- 2}, j \leq q
\end{align}
Particularly, differentiating on \eqref{eq:h_j_k_1_first_order_j_leq_k_1_decomp} again, for $j \leq k-1$,
\begin{align*}
h_{j,k-1}^{(2)}(\theta)
&=
2\sum_{j+1\le r<\ell\le k-1}
\left(
\prod_{i=\ell+1}^{k-1}\Pi_i(\theta)
\right)
\Pi_r^{(1)}(\theta)
\left(
\prod_{i=r+1}^{\ell-1}\Pi_i(\theta)
\right)
\Pi_\ell^{(1)}(\theta)
\left(
\prod_{i=j+1}^{r-1}\Pi_i(\theta)
\right)
h_{j,j}^{(0)}(\theta)
\\
&\quad+
\sum_{\ell=j+1}^{k-1}
\left(
\prod_{i=\ell+1}^{k-1}\Pi_i(\theta)
\right)
\Pi_\ell^{(2)}(\theta)
\left(
\prod_{i=j+1}^{\ell-1}\Pi_i(\theta)
\right)
h_{j,j}^{(0)}(\theta)
\\
&\quad+
2\sum_{\ell=j+1}^{k-1}
\left(
\prod_{i=\ell+1}^{k-1}\Pi_i(\theta)
\right)
\Pi_\ell^{(1)}(\theta)
\left(
\prod_{i=j+1}^{\ell-1}\Pi_i(\theta)
\right)
h_{j,j}^{(1)}(\theta)
\\
&\quad+
\left(
\prod_{i=j+1}^{k-1}\Pi_i(\theta)
\right)
h_{j,j}^{(2)}(\theta).
\end{align*}
Applying $\left\|h_{j,j}^{(2)}(\theta)\right\|_2 \leq C\sqrt{n}|\theta|^{\rho_{j}^{(q)}- 2}$, $\left\|h_{j,j}^{(1)}(\theta)\right\|_2 \leq C\sqrt{n}|\theta|^{\rho_{j}^{(q)}- 1}$, $\left\|h_{j,j}^{(0)}(\theta)\right\|_2 \leq C\sqrt{n}|\theta|^{\rho_{j}^{(q)}}$ yields, for $j \leq k-1$,
\begin{align*}
    \frac{1}{\sqrt{n}}\left\|h_{j,k-1}^{(2)}(\theta)\right\|_2 &\leq \sum_{r,\ell} |\theta|^{\rho^{(q)}_j + \rho^{(q)}_r + \rho^{(q)}_{\ell} - 2} + \sum_{\ell} |\theta|^{\rho^{(q)}_{\ell} - 2} |\theta|^{\rho^{(q)}_{j}} + \sum_{\ell, j} |\theta|^{\rho^{(q)}_{\ell} - 1} |\theta|^{\rho^{(q)}_{j}-1} + |\theta|^{\rho_{j}^{(q)} - 2}\\
    &\leq |\theta|^{\rho_{j}^{(q)} - 2}
\end{align*}

Combining all the bounds of $h_{j,m}^{(r)}(\theta)$, $r = 0,1,2$, we have
\begin{align}\label{eq:h_j_m_all_derivative_both_way}
    \left\|h_{j,m}^{(r)}(\theta)\right\|_2 \leq C \sqrt{n}|\theta|^{\max\{\rho_{j}^{(q)}, \rho_{m}^{(q)}\} - r}
\end{align}

Now we take a look at the derivatives of $P_k(\theta)$. Since for any $\theta \neq 0$, $P_k(\theta) = P_1(R_k(\theta)) = P_1(U_k(\theta))$, then by Lemma \ref{lemma:first_second_derivative_eigenvec_projection}, since $\lambda_1(U_k(\theta)) - \lambda_2(U_k(\theta)) \geq 1 - C_{\lambda}|\theta|^{2(d_k - d_{k+1})} \geq c$, the first derivative $\left\|P^{(1)}_k(\theta)\right\|_2 \leq C\left\|U^{(1)}_k(\theta)\right\|_2$. 

By  
\[
U_k(\theta) = \frac{A_k(\theta)}{\lambda_1(A_k(\theta))},
\] 
Lemma \ref{lemma:bound_for_first_second_order_U} implies 
\[
\left\|U^{(1)}_k(\theta)\right\|_2 \leq \frac{2}{n}\left\|A^{(1)}_k(\theta)\right\|_2
\]and that the derivative of $A^{(1)}_k(\theta)$ is the sum of derivative of \eqref{eq:Pi_Mk_Pi_decomp} and \eqref{eq:Pi_xi_Pi_decomp}. 
\begin{align}\label{eq:Pi_Mk_Pi_decomp_first_derivative}
    &\frac{\partial}{\partial \theta}\left(\prod_{i=1}^{k-1}\Pi_i(\theta)\right) M_k(\theta) \left(\prod_{i=1}^{k-1}\Pi_i(\theta)\right) \\\nonumber
    &= \sum_{j=1}^{k-1}-\alpha_{jk}|\theta|^{-\alpha_{jk}-1} h_{j,k-1}^{(0)}(\theta)\left(h_{j,k-1}^{(0)}(\theta)\right)^* + 2\sum_{j=1}^{k-1} |\theta|^{-\alpha_{jk}}\operatorname{Re}\left(h_{j,k-1}^{(1)}(\theta)\left(h_{j,k-1}^{(0)}(\theta)\right)^*\right)\\\nonumber
    &+ 2\operatorname{Re}\left(h_{k,k-1}^{(1)}(\theta)\left(h_{k,k-1}^{(0)}(\theta)\right)^*\right)\\\nonumber
    &+ \sum_{j=k+1}^{q}\alpha_{jk}|\theta|^{\alpha_{jk}-1} h_{j,k-1}^{(0)}(\theta)\left(h_{j,k-1}^{(0)}(\theta)\right)^* +2\sum_{j=k+1}^{q} |\theta|^{\alpha_{jk}}\operatorname{Re}\left(h_{j,k-1}^{(1)}(\theta)\left(h_{j,k-1}^{(0)}(\theta)\right)^*\right)\nonumber
\end{align}

By \eqref{eq:h_j_m_all_derivative_both_way} and \eqref{eq:defn_rho_q_m}, the operator norm of \eqref{eq:Pi_Mk_Pi_decomp_first_derivative} $\leq Cn|\theta|^{\rho_k^{(q)} -1}$. Similarly, the operator norm of the first derivative of \eqref{eq:Pi_xi_Pi_decomp} can be controlled by $Cn|\theta|^{\rho_k^{(q)} -1}$ as well.

The analysis for the second derivative follows the same logic. From Lemma \ref{lemma:bound_for_first_second_order_U}, 
\begin{align*}
&\left\|U_k^{(2)}(\theta)\right\|_2\leq
C\left(
\frac{\left\|A_k^{(2)}(\theta)\right\|_2}{n}
+
\frac{\left\|A_k^{(1)}(\theta)\right\|_2^2}{n^2}
\right)
\end{align*} then by Lemma \ref{lemma:first_second_derivative_eigenvec_projection}, \begin{align*}
    &\left\|P^{(2)}_k(\theta)\right\|_2 \leq C\left\|U^{(2)}_k(\theta)\right\|_2 +C\left\|U^{(1)}_k(\theta)\right\|_2^2 \leq C\left(
\frac{\left\|A_k^{(2)}(\theta)\right\|_2}{n}
+
\frac{\left\|A_k^{(1)}(\theta)\right\|_2^2}{n^2}
\right)
\end{align*}
\[
\begin{aligned}
&\frac{\partial^2}{\partial\theta^2}
\left[
\left(\prod_{i=1}^{k-1}\Pi_i(\theta)\right)
M_k(\theta)
\left(\prod_{i=1}^{k-1}\Pi_i(\theta)\right)
\right]  \\
&=
\sum_{j=1}^{k-1}
\Big[
\alpha_{jk}(\alpha_{jk}+1)|\theta|^{-\alpha_{jk}-2}
h_{j,k-1}^{(0)}(h_{j,k-1}^{(0)})^*
-4\alpha_{jk}|\theta|^{-\alpha_{jk}-1}
\operatorname{Re}\{h_{j,k-1}^{(1)}(h_{j,k-1}^{(0)})^*\}  \\
&\hspace{4.5cm}
+2|\theta|^{-\alpha_{jk}}
\operatorname{Re}\{h_{j,k-1}^{(2)}(h_{j,k-1}^{(0)})^*\}
+2|\theta|^{-\alpha_{jk}}
h_{j,k-1}^{(1)}(h_{j,k-1}^{(1)})^*
\Big] \\
&\quad
+2\operatorname{Re}\{h_{k,k-1}^{(2)}(h_{k,k-1}^{(0)})^*\}
+2h_{k,k-1}^{(1)}(h_{k,k-1}^{(1)})^* \\
&\quad
+\sum_{j=k+1}^{q}
\Big[
\alpha_{jk}(\alpha_{jk}-1)|\theta|^{\alpha_{jk}-2}
h_{j,k-1}^{(0)}(h_{j,k-1}^{(0)})^*
+4\alpha_{jk}|\theta|^{\alpha_{jk}-1}
\operatorname{Re}\{h_{j,k-1}^{(1)}(h_{j,k-1}^{(0)})^*\} \\
&\hspace{4.5cm}
+2|\theta|^{\alpha_{jk}}
\operatorname{Re}\{h_{j,k-1}^{(2)}(h_{j,k-1}^{(0)})^*\}
+2|\theta|^{\alpha_{jk}}
h_{j,k-1}^{(1)}(h_{j,k-1}^{(1)})^*
\Big],
\end{aligned}
\]
Thus, $\left\|A_k^{(2)}(\theta)\right\|_2 \leq Cn|\theta|^{\rho_k^{(q)} - 2}$.

Statement \ref{stmt:Pk_holder}(k) holds.

With everything, since $\rho_q^{(q)} \leq \cdots \leq \rho_1^{(q)}$, we now obtain $\underline{\mathbf{K}}_{n i}(\theta)$ is H\"older continuous with exponent $\rho_q^{(q)}$ where \begin{align*}
    \rho_k^{(q)}= \min \{\rho_{k-1}^{(q)}, \alpha_{k,k+1}, \min_{l\leq k-1}\{2\rho_l^{(q)} - \alpha_{l,k}\} \}, \rho^{(q)}_0 \coloneqq 1; \alpha_{0,1} \coloneqq 1, d_{q+1} \coloneqq 0.
\end{align*}
since $\rho_{k-1}^{(q)} = \min\{\rho_{k-2}^{(q)}, \alpha_{k-2,k-1}, \alpha_{k-1,k}, \min_{j \leq k-2} \{2\rho_j^{(q)} - \alpha_{j,k-1}\}\} \leq \alpha_{k-1,k}$. Furthermore, $\left\|\underline{\mathbf{K}}_{n i}^{(r)}(\theta)\right\|_2 \leq C|\theta|^{\rho^{(q)}_q - r},\quad r = 1,2.$




\end{proof}

\begin{lemma}\label{lemma:d_differ_K_piecewise_cont_fourier_coef_decay_rate} Under Assumption \ref{assump:model}, \ref{assump:semiparametric_long_memory}, \ref{assump:pervasive_short_mem}, \ref{assump: spec_den_specific_form}(a), \ref{assump:bdd_error_eval}, and \ref{assump:gap_dominance}
    \begin{align*}
        \left\|\mathbf{K}_{n i, h}\right\| = \mathcal{O}\left(|h|^{-1-\rho_q^{(q)}}\right),
    \end{align*} where $\rho_q^{(q)}$ is defined in \eqref{eq:defn_rho_q_m}.
\end{lemma}
\begin{proof} Denote $\underline{\mathbf{J}}_{n i}(\theta) = \underline{\mathbf{K}}_{n i}(\theta) - \underline{\mathbf{K}}_{n i}(0)$, $\underline{\mathbf{J}}^{(r)}_{n i}(\theta) = \underline{\mathbf{K}}^{(r)}_{n i}(\theta)$, $r = 1,2$. 
Integration by parts yields
\begin{align*}
    \mathbf{K}_{n i, h} &= \int_{\Pi} \underline{\mathbf{K}}_{n i}(\theta) e^{-\iota h \theta} d\theta = \int_{\Pi} \underline{\mathbf{J}}_{n i}(\theta) e^{-\iota h \theta} d\theta\\
    &= -\frac{\cos h\pi}{\iota h} \left(\underline{\mathbf{J}}_{n i}(\pi) - \underline{\mathbf{J}}_{n i}(-\pi)\right) + \frac{1}{\iota h}\int_{\Pi}e^{-\iota h \theta}  \underline{\mathbf{J}}^{(1)}_{n i}(\theta) d\theta\\
    &=\frac{1}{\iota h}\int_{\Pi}e^{-\iota h \theta}  \underline{\mathbf{J}}^{(1)}_{n i}(\theta) d\theta
\end{align*} since $\underline{\mathbf{J}}_{n i}(\theta)$ and $\underline{\mathbf{K}}_{n i}(\theta)$ are $2\pi-$periodic. Then by Lemma \ref{lemma:K_holder_factor_wise},
\begin{align*}
    \left\|\int_{|\theta| \leq |h|^{-1}}e^{-\iota h \theta}  \underline{\mathbf{J}}^{(1)}_{n i}(\theta) d\theta\right\|_2 &\leq C \int_{|\theta| \leq |h|^{-1}} \left\|\underline{\mathbf{J}}^{(1)}_{n i}(\theta)\right\|_2 d\theta = C \int_{|\theta| \leq |h|^{-1}} \left\|\underline{\mathbf{K}}^{(1)}_{n i}(\theta)\right\|_2 d\theta \\
    &\leq C \int_{|\theta| \leq |h|^{-1}} |\theta|^{\rho_q^{(q)} - 1} d\theta \leq C|h|^{-\rho_q^{(q)}}.
\end{align*}

Integration by parts again
\begin{align*}
    &\int_{|\theta| > |h|^{-1}}e^{-\iota h \theta}  \underline{\mathbf{J}}^{(1)}_{n i}(\theta) d\theta \\
    &= -\frac{1}{\iota h}\left[\underline{\mathbf{J}}^{(1)}_{n i}(\theta) e^{-\iota h \theta} \bigg|_{|h|^{-1}}^{\pi} + \underline{\mathbf{J}}^{(1)}_{n i}(\theta) e^{-\iota h \theta} \bigg|_{-\pi}^{-|h|^{-1}} - \int_{|\theta| > |h|^{-1}} e^{-\iota h \theta}\underline{\mathbf{J}}^{(1)}_{n i}(\theta) d\theta \right].
\end{align*} Thus, again by Lemma \ref{lemma:K_holder_factor_wise}
\begin{align*}
    \left\|\int_{|\theta| > |h|^{-1}}e^{-\iota h \theta}  \underline{\mathbf{J}}^{(1)}_{n i}(\theta) d\theta\right\|_2 &\leq |h|^{-1}\left\|\underline{\mathbf{J}}^{(1)}_{n i}(\pm \pi)\right\|_2 + |h|^{-1}\left\|\underline{\mathbf{J}}^{(1)}_{n i}(\pm |h|^{-1})\right\|_2\\
    &+ |h|^{-1}\int_{|\theta| > |h|^{-1}} \left\|\underline{\mathbf{J}}^{(2)}_{n i}(\theta)\right\|_2 d\theta\\
    &\leq |h|^{-1} + |h|^{-1}|h|^{1-\rho^{(q)}_{q}} + |h|^{-1}|h|^{1-\rho^{(q)}_{q}} \leq |h|^{-\rho^{(q)}_{q}}
\end{align*}
\end{proof}

       \subsubsection{Bias term}

        \begin{lemma}\label{lemma:bias_conv} For the process where its spectral density satisfies Assumption \ref{assump: spec_den_specific_form}, the bias term induced from the periodogram smoothing estimator vanishes as $T$ increases. 
            \begin{align}\label{eq:bias_rate}
        &\int_{\Pi}|\theta|^{2\tilde{d}} \left\|\mathbb{E}\left(\widehat{\Sigma}_{n}(\theta)\right) - \Sigma(\theta)\right\|_{op}  d\theta = n\delta_{bias}\\
        &\delta_{bias} = \mathcal{O}\left( B_T^{1-2\Delta} + \frac{1}{T B_T} +  \log T (B_T T)^{2\tilde{d}} T^{2\Delta - 1}\right)
        \end{align} where $\Delta  = \max_j d_j - \min_j d_j$. 
        If $B_T = T^{-b}$, $b \in (0,1)$, $\delta_{bias} = \mathcal{O}\left(B_T^{1-2\Delta} + \frac{1}{B_T T}\right)$.
        \end{lemma}

        \begin{proof}
        \begin{align*}
            |\theta|^{2\tilde{d}}\left\|\mathbb{E}\left(\widehat{\Sigma}_{n}(\theta)\right) - \Sigma(\theta)\right\|_{op}   &= |\theta|^{2\tilde{d}}\left\|\frac{2\pi}{B_T T}\sum_{s=-T_0, s\neq 0}^{T_0} W\left(\frac{\theta - \lambda_s}{B_T}\right)  \mathbb{E}\left(\mathbf{I}_{XX}(\lambda_s)\right) - \Sigma(\theta)\right\|_{op}  \\
            &\leq \frac{2\pi}{B_T T}|\theta|^{2\tilde{d}}\sum_{s=-T_0, s\neq 0}^{T_0} W\left(\frac{\theta - \lambda_s}{B_T}\right) \left\| \mathbb{E}\left(\mathbf{I}_{XX}(\lambda_s)\right) - \Sigma(\lambda_s)\right\|_{op}  \\
            &+ |\theta|^{2\tilde{d}} \left\|\frac{2\pi}{B_T T}\sum_{s=-T_0, s\neq 0}^{T_0}W\left(\frac{\theta - \lambda_s}{B_T}\right)  \left(\Sigma(\lambda_s) -  \Sigma(\theta)\right) \right\|_{op} \\
            &+ |\theta|^{2\tilde{d}}\left|  \frac{2\pi}{B_T T}\sum_{s=-T_0, s\neq 0}^{T_0}W\left(\frac{\theta - \lambda_s}{B_T}\right) - 1  \right| \left\|\Sigma(\theta)\right\|_{op}\\
            &\coloneqq A_1(\theta) + A_2(\theta) + A_3(\theta)
        \end{align*}

        By Lemma \ref{lm:cum_property_long_mem} (iii), we have 
        \begin{align*}
            \left\| \mathbb{E}\left(\mathbf{I}_{XX}(\lambda_s)\right) - \Sigma(\lambda_s)\right\|_{op} \leq \frac{n}{2\pi T} |\lambda_s|^{-2d-1} \log T
        \end{align*}

        Thus, applying Lemma \ref{lm:basic_summation_integral_w_theta_2_tilde_d} we have 
        \begin{align}\label{eq:A_1_rate}
            \int_{\Pi}A_1(\theta) d\theta &\leq \frac{Cn \log T}{B_T T^2} \int_{\Pi} \sum_{s=-T_0, s\neq 0}^{T_0} W\left(\frac{\theta - \lambda_s}{B_T}\right) |\lambda_s|^{-2d-1} |\theta|^{2\tilde{d}} d\theta\\
            &\leq Cn \log T (B_T T)^{2\tilde{d}} T^{2\Delta - 1}, \Delta = d - \tilde{d}
        \end{align}

        On the other hand, by Lemma \ref{lm:conv_rate_before_approx_identity}, we have 
        \begin{align*}
            A_2(\theta) &= \left\|\frac{2\pi|\theta|^{2\tilde{d}}}{B_T T}\sum_{s=-T_0, s\neq 0}^{T_0}W\left(\frac{\theta - \lambda_s}{B_T}\right)  \left(\Sigma(\lambda_s) -  \Sigma(\theta)\right) \right\|\\
                &\lesssim \int_{\Pi}  \frac{|\theta|^{2\tilde{d}}}{B_T}W\left(\frac{\theta - \lambda}{B_T}\right)\left\| \Sigma(\lambda) -  \Sigma(\theta)  \right\| d\lambda\\
                &+ \frac{n|\theta|^{-2\Delta}}{TB_T} + \frac{n|\theta|^{2\tilde{d}}}{B_T T^{1-2d}}\mathbb{1}(|\theta| \leq (2\rho + 4\pi) B_T) + \frac{n|\theta|^{-2\Delta - 1}}{T}\mathbb{1}(|\theta| > (2\rho + 4\pi) B_T)
            \end{align*}
            Take a further look inside the first term, for $\lambda, \theta \in \Pi$, by Lemma \ref{lm:increment_bound_for_Sigma},
            \begin{align*}
                    \left\| \Sigma(\lambda) -  \Sigma(\theta)  \right\| &\leq Cn\left(|\theta|^{-2d} + |\lambda|^{-2d}\right)\left|\lambda - \theta \right| + Cn\left| |\lambda|^{-2d} - |\theta|^{-2d} \right|\\
                    &+ Cn\left(|\theta|^{-2d} + |\lambda|^{-2d}\right)\mathbb{1}\{\theta \lambda < 0\}
                \end{align*}

            Thus,
            \begin{align*}
                \left\|A_2(\theta)\right\|_{L^1} &\leq C\int_{\Pi}\int_{\Pi} \frac{|\theta|^{2\tilde{d}}}{B_T}W\left(\frac{\theta - \lambda}{B_T}\right) \left\| \Sigma(\lambda) -  \Sigma(\theta)  \right\| d\lambda d\theta + \frac{Cn}{T B_T}\int_{\Pi}\left| \theta  \right|^{-2\Delta}d\theta \\
                &+\frac{C_2n}{T^{1-2d} B_T} \int_{|\theta|\leq (2\rho + 4\pi) B_T}|\theta|^{2\tilde{d}} d\theta + \frac{C_3n }{T}\int_{|\theta|>(2\rho + 4\pi)B_T}|\theta|^{-2\Delta-1} d\theta\\
                &\lesssim n\int_{\Pi}\int_{\Pi} \frac{|\theta|^{2\tilde{d}}}{B_T}W\left(\frac{\theta - \lambda}{B_T}\right)\left| |\lambda|^{-2d} - |\theta|^{-2d} \right| d\lambda d\theta\\
                &+  n\int_{\Pi}\int_{\Pi}\frac{|\theta|^{2\tilde{d}}}{B_T}W\left(\frac{\theta - \lambda}{B_T}\right)|\theta|^{-2d} \left| \lambda - \theta  \right|d\lambda d\theta\\
                &+n\int_{\Pi}\int_{\Pi}\frac{|\theta|^{2\tilde{d}}}{B_T}W\left(\frac{\theta - \lambda}{B_T}\right)|\lambda|^{-2d} \left| \lambda - \theta  \right|d\lambda d\theta\\
                &+ n\int_{\Pi}\int_{\Pi}\frac{|\theta|^{2\tilde{d}}}{B_T}W\left(\frac{\theta - \lambda}{B_T}\right)\left(|\theta|^{-2d} + |\lambda|^{-2d}\right)\mathbb{1}\{\theta \lambda < 0\} d\lambda d\theta +\frac{n}{T B_T} + \frac{n B_T^{2\tilde{d}}}{T^{1-2d}} + \frac{n}{TB_T^{\Delta}}
            \end{align*}

        By Lemma \ref{lm:calculation_W_theta_lambda_minus_theta}, the second term is of order $\mathcal{O}(B_T)$, and similarly, the third term is $\mathcal{O}(B_T)$ as well.
            
         To deal with the first term, notice that 
            \begin{align}\label{eq:continuity_translation_L1}
                &\int_{\Pi} \int_{\Pi} \frac{|\theta|^{2\tilde{d}}}{B_T}W\left(\frac{\theta - \lambda}{B_T}\right)  \left| |\lambda|^{-2d} - |\theta|^{-2d} \right| d\lambda d \theta\\\nonumber
                &\leq \int_{\Pi} \int_{\Pi} |\theta|^{2\tilde{d}} W(\nu)\left||\theta|^{-2d} -  |\theta - B_T \nu|^{-2d}\right| d\theta d\nu\\\nonumber
                &\leq \int_{\Pi}   W(\nu) \int_{\Pi} |\theta|^{2\tilde{d}} \left||\theta|^{-2d} -  |\theta - B_T \nu|^{-2d}\right| d\theta d\nu
            \end{align}
            First of all, consider a small $0< l < 1/2$,
            \begin{align*}
                \int_{0}^{2l} |\theta|^{2\tilde{d}} \left||\theta|^{-2d} -  |\theta - l|^{-2d}\right|d\theta &\leq \int_{0}^{2l} |\theta|^{-2\Delta}  d\theta + \int_{0}^{2l} |\theta|^{2\tilde{d}} |\theta - l|^{-2d}  d\theta\\
                &\leq Cl^{1-2\Delta} + l^{2\tilde{d}}\int_{0}^{2l} |\theta - l|^{-2d}  d\theta \leq Cl^{1-2\Delta}
            \end{align*}
            Also by the mean value theorem, for $\theta\in [2l,1]$, there exists a $\xi\geq \min(|\theta|, |\theta - l|)\geq {\theta}/{2}$, that
            \begin{align*}
                \int_{2l}^{1} |\theta|^{2\tilde{d}} \left||\theta|^{-2d} -  |\theta - l|^{-2d}\right|d\theta &\leq \int_{2l}^{1} 2d \xi^{-2d-1}l |\theta|^{2\tilde{d}} d\theta\\
                &\leq d2^{2d+2} l \int_{2l}^1 \theta^{-2\Delta-1} d\theta = 2l^{1-2\Delta} - 2^{2d+1}l \leq Cl^{1-2\Delta}
            \end{align*}
            The integral on $[1,\pi]$ can be simply bound by $\mathcal{O}(l)$ by the continuity on $[1,\pi]$. Likewise for $\theta\in (-\pi, 0)$. That is to say, $\int_{\Pi} |\theta|^{2\tilde{d}} \left||\theta|^{-2d} -  |\theta - l|^{-2d}\right|d\theta = \mathcal{O}(l^{1-2\Delta})$. Then 
            \begin{align*}
                \int_{\Pi}  W(\nu) \int_{\Pi} |\theta|^{2\tilde{d}} \left||\theta|^{-2d} -  |\theta - B_T \nu|^{-2d}\right| d\theta d\nu \leq CB_T^{1-2\Delta}\int_{\Pi} \nu^{1-2\Delta} d\nu = \mathcal{O}\left(B_T^{1-2\Delta}\right)
            \end{align*}

            For the fourth term $n\int_{\Pi}\int_{\Pi}\frac{|\theta|^{2\tilde{d}}}{B_T}W\left(\frac{\theta - \lambda}{B_T}\right)\left(|\theta|^{-2d} + |\lambda|^{-2d}\right)\mathbb{1}\{\theta \lambda < 0\} d\lambda d\theta$, again let \(\nu=(\theta-\lambda)/B_T\), \(\lambda=\theta-B_T\nu\). Then
\begin{align*}
&\int_{\Pi}\int_{\Pi}\frac{|\theta|^{2\tilde d}}{B_T}
W\!\left(\frac{\theta-\lambda}{B_T}\right)
\bigl(|\theta|^{-2d}+|\lambda|^{-2d}\bigr)\mathbb{1}\{\theta\lambda<0\}\,d\lambda\,d\theta\\
&\leq \int_{\mathbb R}|W(\nu)|
\int_{\Pi} |\theta|^{2\tilde d}
\bigl(|\theta|^{-2d}+|\theta-B_T\nu|^{-2d}\bigr)
\mathbb{1}\{\theta(\theta-B_T\nu)<0\}\,d\theta\,d\nu.
\end{align*}
$\mathbb{1}\{\theta(\theta-B_T\nu)<0\} \leq \mathbb{1}\{ B_T|\nu| > |\theta|  \}$, thus, again, same as the analysis for the first term, with $l = B_T|\nu|$, 
\[
\int_{\Pi} |\theta|^{2\tilde d}
\bigl(|\theta|^{-2d}+|\theta-B_T\nu|^{-2d}\bigr)
\mathbf 1\{\theta(\theta-B_T\nu)<0\}\,d\theta
\leq Cl^{1-2\Delta} = CB_T^{1-2\Delta}|\nu|^{1-2\Delta}.
\]
Therefore,
\[
\int_{\Pi}\int_{\Pi}\frac{|\theta|^{2\tilde d}}{B_T}
W\!\left(\frac{\theta-\lambda}{B_T}\right)
\bigl(|\theta|^{-2d}+|\lambda|^{-2d}\bigr)\mathbf 1\{\theta\lambda<0\}\,d\lambda\,d\theta
\leq
CB_T^{1-2\Delta}\int_{\mathbb R}|W(\nu)|\,|\nu|^{1-2\Delta}\,d\nu
\leq CB_T^{1-2\Delta}.
\]

            Summing up, we have 
            \begin{align}\label{eq:A_2_rate}
                \left\|A_2(\theta)\right\|_{L^1}\lesssim nB_T^{1-2\Delta} + \frac{n}{T B_T} + \frac{n B_T^{2\tilde{d}}}{T^{1-2d}} + nB_T + \frac{n}{TB_T^\Delta}  = \mathcal{O}\left(nB_T^{1-2\Delta} + \frac{n}{T B_T}\right)
            \end{align}

            In addition, combining with Lemma \ref{lm:rate_sum_to_int_W_only} and \ref{lm:rate_truncation_by_window},
            \begin{align}\label{eq:A_3_rate}
                \left\|A_3(\theta)\right\|_{L^1} &= \int_{\Pi}\left|  \frac{2\pi}{B_T T}\sum_{s=-T_0, s\neq 0}^{T_0}W\left(\frac{\theta - \lambda_s}{B_T}\right) - 1  \right| |\theta|^{2\tilde{d}}\left\|\Sigma(\theta)\right\|d\theta\\\nonumber
                &\leq \int_{\Pi}\left|  \frac{2\pi}{B_T T}\sum_{s=-T_0, s\neq 0}^{T_0}W\left(\frac{\theta - \lambda_s}{B_T}\right) -  \frac{1}{B_T }\int_{\Pi_1}W\left(\frac{\theta - \lambda}{B_T}\right) d\lambda  \right| \left\|\Sigma(\theta)\right\||\theta|^{2\tilde{d}}d\theta \\\nonumber
                &+ \int_{\Pi} \left|  \frac{1}{B_T}\int_{\Pi_1}W\left(\frac{\theta - \lambda}{B_T}\right) d\lambda - \frac{1}{B_T }\int_{\Pi}W\left(\frac{\theta - \lambda}{B_T}\right) d\lambda  \right|\left\|\Sigma(\theta)\right\||\theta|^{2\tilde{d}}d\theta \\\nonumber
                &+ \int_{\Pi} \left|  \frac{1}{B_T}\int_{\Pi}W\left(\frac{\theta - \lambda}{B_T}\right) d\lambda - 1  \right|\left\|\Sigma(\theta)\right\||\theta|^{2\tilde{d}}d\theta \\\nonumber
                &\leq \frac{n}{TB_T} + n\int_{\Pi}\frac{1}{B_T}\int_{-\frac{2\pi}{T}}^{\frac{2\pi}{T}} d\lambda |\theta|^{-2\Delta}d\theta + nB_T \\\nonumber
                &\leq \frac{n}{TB_T} + nB_T 
            \end{align}

            Putting (\ref{eq:A_1_rate}), (\ref{eq:A_2_rate}), (\ref{eq:A_3_rate}) together, we have the bias term
            \begin{align*}
                \int_{\Pi}|\theta|^{2\tilde{d}} \left\|\mathbb{E}\left(\widehat{\Sigma}_{n}(\theta)\right) - \Sigma(\theta)\right\|_{op}  d\theta = \mathcal{O}\left(nB_T^{1-2\Delta} + \frac{n}{T B_T} + n \log T (B_T T)^{2\tilde{d}} T^{2\Delta - 1}\right).
            \end{align*}
            Moreover, if $B_T = T^{-b}$, $b \in (0,1)$, the bias term converges to $0$ as $T\rightarrow \infty$. The dominant term in this case would be $\mathcal{O}(nB_T^{1-2\Delta} + \frac{n}{T B_T})$.

        \end{proof}

        \subsubsection{Variance Term}
        \begin{lemma} For the process where its spectral density satisfies Assumption \ref{assump: spec_den_specific_form}, the variance term induced from the periodogram smoothing estimator vanishes as T increases. \label{lm:var_conv}
            \begin{align}\label{eq:var_rate}
        var\left(\hat{\sigma}_{ij}(\theta)\right) &=\mathcal{O}\left(  \frac{|\theta|^{-4d}\log^2 T}{TB_T}\mathbb{1}(|\theta|>2\rho B_T) + \left(\frac{\log^2 T}{TB_T^{1+4d}} + \frac{\log^2 T}{T^{2-4d}B_T^{2}}\right)\mathbb{1}(|\theta|\leq 2\rho B_T) \right)
        \end{align}
        where $d$ is defined in Assumption \ref{assump: spec_den_specific_form}.
        \end{lemma}

        \begin{proof}
            Notice that
            \begin{align*}
                \mathbb{E}\left( \left| \hat{\sigma}_{ij}(\theta) \right|^2  \right) &= \frac{4\pi^2}{B_T^2T^2} \mathbb{E} \left( \sum_{s=-T_0, s\neq 0}^{T_0} W\left( \frac{\theta - \lambda_s}{B_T} \right) I_{ij}(\lambda_s)\overline{\sum_{t=1}^{T-1} W\left( \frac{\theta - \lambda_t}{B_T} \right)I_{ij}(\lambda_t)} \right)\\
                &=\frac{4\pi^2}{B_T^2T^2} \sum_{s=-T_0, s\neq 0}^{T_0}\sum_{t=1}^{T-1} W\left( \frac{\theta - \lambda_s}{B_T} \right) W\left( \frac{\theta - \lambda_t}{B_T} \right)\mathbb{E}\left(I_{ij}(\lambda_s) \overline{I_{ij}(\lambda_t)} \right)\\
                &= \frac{4\pi^2}{B_T^2T^2} \sum_{s=-T_0, s\neq 0}^{T_0}\sum_{t=1}^{T-1} W\left( \frac{\theta - \lambda_s}{B_T} \right) W\left( \frac{\theta - \lambda_t}{B_T} \right) cov\left(I_{ij}(\lambda_s), {I_{ij}(\lambda_t)} \right)\\
                &+\frac{4\pi^2}{B_T^2T^2} \sum_{s=-T_0, s\neq 0}^{T_0}\sum_{t=1}^{T-1} W\left( \frac{\theta - \lambda_s}{B_T} \right) W\left( \frac{\theta - \lambda_t}{B_T} \right) \mathbb{E}\left(I_{ij}(\lambda_s) \right)\mathbb{E}\left(\overline{I_{ij}(\lambda_t)} \right)\\
                &= \frac{4\pi^2}{B_T^2T^2} \sum_{s=-T_0, s\neq 0}^{T_0}\sum_{t=1}^{T-1} W\left( \frac{\theta - \lambda_s}{B_T} \right) W\left( \frac{\theta - \lambda_t}{B_T} \right) cov\left(I_{ij}(\lambda_s), {I_{ij}(\lambda_t)} \right) + \left|\mathbb{E}\left( \hat{\sigma}_{ij}(\theta)   \right)\right|^2
            \end{align*}
            Then 
            \begin{align*}
                var\left(\hat{\sigma}_{ij}(\theta)\right) &= \mathbb{E}\left( \left| \hat{\sigma}_{ij}(\theta) \right|^2  \right) - \left|\mathbb{E}\left( \hat{\sigma}_{ij}(\theta)   \right)\right|^2\\
                &= \frac{4\pi^2}{B_T^2T^2} \sum_{s=-T_0, s\neq 0}^{T_0}\sum_{t=-T_0, t\neq 0}^{T_0} W\left( \frac{\theta - \lambda_s}{B_T} \right) W\left( \frac{\theta - \lambda_t}{B_T} \right) cov\left(I_{ij}(\lambda_s), {I_{ij}(\lambda_t)} \right)
            \end{align*}
            and that
            \begin{align*}
                cov\left(I_{ij}(\lambda_s), {I_{ij}(\lambda_t)} \right)  = \frac{1}{4\pi^2 T^2} &[ cum\left( d_i(\lambda_s), d_j(-\lambda_s),d_i(-\lambda_t), d_j(\lambda_t)   \right)\\
                &+ cum\left( d_i(\lambda_s), d_i(-\lambda_t) \right)cum\left( d_j(-\lambda_s), d_j(\lambda_t) \right)\\
                &+ cum\left( d_i(\lambda_s), d_j(\lambda_t) \right)cum\left( d_j(-\lambda_s), d_i(-\lambda_t) \right)]
            \end{align*}
            Thus, we could denote 
            \begin{align*}
                \left| var\left(\hat{\sigma}_{ij}(\theta)\right)   \right| \leq P_1 + P_2 + P_3
            \end{align*}
            where denote $N_0 = \{k: k\neq 0, k = -T_0, \dots, T_0\}$,
            \begin{align*}
                P_1 &= \frac{1}{B_T^2T^4}\sum_{s \in N_0}\sum_{t\in N_0}W\left( \frac{\theta - \lambda_s}{B_T} \right) W\left( \frac{\theta - \lambda_t}{B_T} \right)\bigg| \text{Cum}\left( d_i(\lambda_s), d_j(-\lambda_s),d_i(-\lambda_t), d_j(\lambda_t)   \right)  \bigg|\\
                P_2 &= \frac{1}{B_T^2T^4}\sum_{s \in N_0}\sum_{t\in N_0}W\left( \frac{\theta - \lambda_s}{B_T} \right) W\left( \frac{\theta - \lambda_t}{B_T} \right)\bigg| \text{Cum}\left( d_i(\lambda_s), d_i(-\lambda_t) \right)\text{Cum}\left( d_j(-\lambda_s), d_j(\lambda_t) \right)  \bigg|\\
                P_3 &= \frac{1}{B_T^2T^4}\sum_{s \in N_0}\sum_{t\in N_0}W\left( \frac{\theta - \lambda_s}{B_T} \right) W\left( \frac{\theta - \lambda_t}{B_T} \right)\bigg| \text{Cum}\left( d_i(\lambda_s), d_j(\lambda_t) \right)\text{Cum}\left( d_j(-\lambda_s), d_i(-\lambda_t) \right)  \bigg|
            \end{align*}

        From Lemma \ref{lm:cum_property_long_mem} (ii), for $\left|\lambda\right| \leq \left|u\right|$, 
        \[
        \bigg| \text{Cum}\left( d_{q_1}(\lambda), d_{q_2}(-\lambda),d_{q_1}(-u), d_{q_2}(u)   \right)  \bigg| \leq C \left\{\left|u\right|^{-1} + T\log^3 T \right\} \left|\lambda\right|^{-2d} \left|u\right|^{-2d} .
        \]
        Then apply to $P_1$, we have, 
        \begin{align*}
            P_1 &= \frac{1}{B_T^2T^4}\sum_{s\in N_0}\sum_{|t|\geq |s|}W\left( \frac{\theta - \lambda_s}{B_T} \right) W\left( \frac{\theta - \lambda_t}{B_T} \right)\left| \text{Cum}\left( d_i(\lambda_s), d_j(-\lambda_s),d_i(-\lambda_t), d_j(\lambda_t)   \right)  \right|\\
            &+ \frac{1}{B_T^2T^4}\sum_{s\in N_0}\sum_{|t| < |s|}W\left( \frac{\theta - \lambda_s}{B_T} \right) W\left( \frac{\theta - \lambda_t}{B_T} \right)\left| \text{Cum}\left(d_j(\lambda_t), d_i(-\lambda_t) , d_j(-\lambda_s), d_i(\lambda_s) \right)  \right|\\
            &\leq  \frac{1}{B_T^2T^4}\sum_{s\in N_0}\sum_{|s|\leq |t|} W\left( \frac{\theta - \lambda_s}{B_T} \right) W\left( \frac{\theta - \lambda_t}{B_T} \right) C\left\{\left|\lambda_t\right|^{-1} + T\log^3 T \right\} \left|\lambda_s\right|^{-2d} \left|\lambda_t\right|^{-2d} \\
            &+ \frac{1}{B_T^2T^4}\sum_{s\in N_0}\sum_{|t| < |s|} W\left( \frac{\theta - \lambda_s}{B_T} \right) W\left( \frac{\theta - \lambda_t}{B_T} \right) C\left\{\left|\lambda_s\right|^{-1} + T\log^3 T \right\} \left|\lambda_t\right|^{-2d} \left|\lambda_s\right|^{-2d}\\
            &\lesssim \frac{1}{B_T^2T^4}\sum_{s\in N_0}\sum_{t\in N_0} W\left( \frac{\theta - \lambda_s}{B_T} \right) W\left( \frac{\theta - \lambda_t}{B_T} \right) \left\{\left|\lambda_t\right|^{-1}+ \left|\lambda_s\right|^{-1} + T\log^3 T \right\} \left|\lambda_s\right|^{-2d} \left|\lambda_t\right|^{-2d}\\
            &\lesssim \frac{1}{B_T^2T^4}\sum_{s\in N_0} W\left( \frac{\theta - \lambda_s}{B_T} \right)\left|\lambda_s\right|^{-2d} \sum_{t\in N_0} W\left( \frac{\theta - \lambda_t}{B_T} \right)  \left|\lambda_t\right|^{-2d-1}\\
            &+ \frac{1}{B_T^2T^4}\sum_{s\in N_0} W\left( \frac{\theta - \lambda_s}{B_T} \right)\left|\lambda_s\right|^{-2d-1} \sum_{t\in N_0} W\left( \frac{\theta - \lambda_t}{B_T} \right)  \left|\lambda_t\right|^{-2d}\\
            &+ \frac{T\log^3 T}{B_T^2T^4}\sum_{s\in N_0} W\left( \frac{\theta - \lambda_s}{B_T} \right)\left|\lambda_s\right|^{-2d} \sum_{t\in N_0} W\left( \frac{\theta - \lambda_t}{B_T} \right)  \left|\lambda_t\right|^{-2d}\\
            &\lesssim 2\left(\frac{1}{B_T T^2}\sum_{s\in N_0} W\left( \frac{\theta - \lambda_s}{B_T} \right)\left|\lambda_s\right|^{-2d}\right) \left(\frac{1}{B_T T^2} \sum_{t\in N_0} W\left( \frac{\theta - \lambda_t}{B_T} \right)  \left|\lambda_t\right|^{-2d-1}\right)\\
            &+ T\log^3 T \left(\frac{1}{B_T T^2}\sum_{s\in N_0} W\left( \frac{\theta - \lambda_s}{B_T} \right)\left|\lambda_s\right|^{-2d}\right)^2 
        \end{align*}

        From Lemma \ref{lm:basic_summation_integral_summing_from_M_B_T_theta_pointwise}, we have
        \begin{align*}
            \frac{1}{B_T T^2}\sum_{s\in N_0} W\left( \frac{\theta - \lambda_s}{B_T} \right)\left|\lambda_s\right|^{-2d} &\leq \frac{C}{B_T T^2} \sum_{s \in M_s(\theta, B_T)}|\lambda_s|^{-2d}\\
            &\leq \frac{C|\theta|^{-2d}}{T}\mathbb{1}(|\theta|>2\rho B_T) + \frac{C}{TB_T^{2d}}\mathbb{1}(|\theta|\leq 2\rho B_T)
        \end{align*}
        and 
        \begin{align*}
            \frac{1}{B_T T^2} \sum_{s\in N_0} W\left( \frac{\theta - \lambda_s}{B_T} \right)  \left|\lambda_s\right|^{-2d-1} &\leq \frac{C}{B_T T^2} \sum_{s \in M_s(\theta, B_T)}|\lambda_s|^{-2d-1}\\
            &\leq \frac{C|\theta|^{-2d-1}}{T}\mathbb{1}(|\theta|>2\rho B_T) + \frac{C}{B_T T^{1-2d}}\mathbb{1}(|\theta|\leq 2\rho B_T)
        \end{align*}
        where $M_s(\theta, B_T)$ is defined as in Lemma \ref{lm:basic_summation_integral_summing_from_M_B_T_theta_pointwise}.

        Thus,
        \begin{align*}
            &P_1 \leq \left(\frac{C|\theta|^{-4d-1}}{T^2} + \frac{C|\theta|^{-4d}\log^3 T}{T}\right) \mathbb{1}\left( |\theta| > 2\rho B_T  \right) + \left(\frac{C}{B_T^{1+2d}T^{2-2d}} + \frac{C\log^3 T}{B_T^{4d}T}\right)\mathbb{1}(|\theta|<2\rho B_T)
        \end{align*}

        For $P_2$, we have
        \begin{align*}
            P_2 &= \frac{1}{B_T^2T^4}\sum_{s\in N_0}\sum_{t\in N_0}W\left( \frac{\theta - \lambda_s}{B_T} \right) W\left( \frac{\theta - \lambda_t}{B_T} \right)\left| \text{Cum}\left( d_i(\lambda_s), d_i(-\lambda_t) \right)\text{Cum}\left( d_j(-\lambda_s), d_j(\lambda_t) \right)  \right|\\
            &= P_{21} + P_{22} + P_{23}
        \end{align*}
            where
            \begin{align*}
                P_{21} &= \frac{1}{B_T^2T^4}\sum_{l\in N_0} W^2\left( \frac{\theta - \lambda_l}{B_T} \right)\left| \text{Cum}\left( d_i(\lambda_l), d_i(-\lambda_l) \right)  \right| \left| \text{Cum}\left( d_j(-\lambda_l), d_j(\lambda_l) \right)  \right|\\
                P_{22} &= \frac{1}{B_T^2T^4}\sum_{l\in N_0} W\left( \frac{\theta - \lambda_l}{B_T} \right) W\left( \frac{\theta + \lambda_l}{B_T} \right)\left| \text{Cum}\left( d_i(\lambda_l), d_i(\lambda_l) \right)  \right| \left| \text{Cum}\left( d_j(\lambda_l), d_j(\lambda_l) \right)  \right|\\
                P_{23} &= \frac{1}{B_T^2T^4} \sum_{s\in N_0}\sum_{|s|<|t|} W\left( \frac{\theta - \lambda_s}{B_T} \right) W\left( \frac{\theta - \lambda_t}{B_T} \right) \left|  \text{Cum}\left( d_i(\lambda_s), d_i(-\lambda_t) \right) \right| \left| \text{Cum}\left( d_j(-\lambda_s), d_j(\lambda_t) \right) \right|\\
                P_{24} &= \frac{1}{B_T^2T^4} \sum_{s\in N_0}\sum_{|t|<|s|} W\left( \frac{\theta - \lambda_s}{B_T} \right) W\left( \frac{\theta - \lambda_t}{B_T} \right) \left|  \text{Cum}\left( d_i(-\lambda_t), d_i(\lambda_s) \right) \right| \left| \text{Cum}\left( d_j(\lambda_t), d_j(-\lambda_s) \right) \right|\\
                &= \frac{1}{B_T^2T^4} \sum_{t\in N_0}\sum_{|s|<|t|}W\left( \frac{\theta - \lambda_s}{B_T} \right) W\left( \frac{\theta - \lambda_t}{B_T} \right)\left| \text{Cum}\left( d_j(\lambda_s), d_j(-\lambda_t) \right) \right| \left|  \text{Cum}\left( d_i(-\lambda_s), d_i(\lambda_t) \right) \right| 
            \end{align*}

            First of all, by Lemma \ref{lm:basic_summation_integral_summing_from_M_B_T_theta_pointwise} we have
            \begin{align*}
                &\sum_{s\in N_0} W\left( \frac{\theta - \lambda_s}{B_T} \right) \left|\lambda_s\right|^{-4d} \leq C\sum_{s\in M_s(\theta, B_T)} \left|\lambda_s\right|^{-4d}\\
                &\lesssim(B_T T)|\theta|^{-4d}\mathbb{1}(|\theta|>2\rho B_T) + \left(TB_T^{1-4d} + T^{4d}\right)\mathbb{1}(|\theta|\leq 2\rho B_T)
            \end{align*}

            by Lemma \ref{lm:cum_property_long_mem} (i), and that $W(\cdot) \leq C, L_1(0) \leq CT$, we yield
            \begin{align*}
                &P_{21} \lesssim \frac{1}{B_T^2T^4} \sum_{l\in N_0} W\left( \frac{\theta - \lambda_l}{B_T} \right) \left[ \left|\lambda_l\right|^{-2d}\left( \left|\lambda_l\right|^{-1} + L_1(0)\right) \right]^2\\
                &\lesssim \frac{1}{B_T^2T^4} \sum_{l\in N_0} W\left( \frac{\theta - \lambda_l}{B_T} \right) \left|\lambda_l\right|^{-4d} T^2\\
                &\leq \frac{|\theta|^{-4d}}{B_T T}\mathbb{1}\left(|\theta|>2\rho B_T\right) + \left(\frac{1}{TB_T^{1+4d}} + \frac{1}{T^{2-4d}B_T^{2}}\right)\mathbb{1}(|\theta|\leq 2\rho B_T)
            \end{align*}

            Secondly, $L_1(x)\leq C\log T L_0(x)\leq CT\log T$,
            \begin{align*}
                &P_{22} \lesssim \frac{1}{B_T^2T^4} \sum_{l\in N_0} W\left( \frac{\theta - \lambda_l}{B_T} \right) \left[ \left|\lambda_l\right|^{-2d}\left( \left|\lambda_l\right|^{-1} + L_1(2\lambda_l)\right) \right]^2\\
                &\lesssim  \frac{1}{B_T^2T^4} \sum_{l\in N_0} W\left( \frac{\theta - \lambda_l}{B_T} \right) \left|\lambda_l\right|^{-4d} T^2 \log^2T\\
                &\lesssim \frac{|\theta|^{-4d}\log^2T}{B_T T}\mathbb{1}\left(|\theta|>2\rho B_T\right) + \left(\frac{\log^2T}{TB_T^{1+4d}} + \frac{\log^2T}{T^{2-4d}B_T^{2}}\right)\mathbb{1}(|\theta|\leq 2\rho B_T)
            \end{align*}
            
            Thirdly, by Lemma \ref{lm:basic_summation_integral_summing_from_M_B_T_theta_pointwise} we have
            \begin{align*}
                &\sum_{s\in N_0} W\left( \frac{\theta - \lambda_s}{B_T} \right)\left| \lambda_s \right|^{-2} \leq CB_T T |\theta|^{-2}\mathbb{1}(|\theta|>2\rho B_T) + CT^2 \mathbb{1}(|\theta|\leq 2\rho B_T) \\
                &\sum_{s\in N_0} W\left( \frac{\theta - \lambda_s}{B_T} \right)\left| \lambda_s \right|^{-1} \leq CB_T T |\theta|^{-1}\mathbb{1}(|\theta|>2\rho B_T) + CT \log(B_T T) \mathbb{1}(|\theta|\leq 2\rho B_T)
            \end{align*}
            Then since $L_1(\lambda)\leq C \log T L_0(\lambda),\left|\lambda_s - \lambda_t\right| \geq |\lambda_1|, s\neq t$, we have
            \begin{align*}
                P_{23} &\leq \frac{1}{B_T^2T^4} \sum_{s\in N_0}\sum_{|s|<|t|} W\left( \frac{\theta - \lambda_s}{B_T} \right) W\left( \frac{\theta - \lambda_t}{B_T} \right) \left[ \left| \lambda_s \right|^{-2d} \left( \left| \lambda_t \right|^{-1} + \frac{\log T}{\left|\lambda_s - \lambda_t\right|}  \right) \right]^2\\
                &\leq \frac{1}{B_T^2T^4} \sum_{s\in N_0}\sum_{|s|<|t|} W\left( \frac{\theta - \lambda_s}{B_T} \right) W\left( \frac{\theta - \lambda_t}{B_T} \right) \left| \lambda_s \right|^{-4d}\left| \lambda_t \right|^{-2}\\
                &+ \frac{2\log T}{B_T^2T^4} \sum_{s\in N_0}\sum_{|s|<|t|}W\left( \frac{\theta - \lambda_s}{B_T} \right) W\left( \frac{\theta - \lambda_t}{B_T} \right) \left| \lambda_s \right|^{-4d} \frac{\left| \lambda_t \right|^{-1}}{\left|\lambda_s - \lambda_t\right|}\\
                &+ \frac{\log^2 T}{B_T^2T^4} \sum_{s\in N_0}\sum_{|s|<|t|}W\left( \frac{\theta - \lambda_s}{B_T} \right) W\left( \frac{\theta - \lambda_t}{B_T} \right) \left| \lambda_s \right|^{-4d} \frac{1}{\left|\lambda_s - \lambda_t\right|^2}\\
                &\leq \frac{1}{B_T^2T^4} \sum_{s\in N_0} \sum_{t\in N_0} W\left( \frac{\theta - \lambda_s}{B_T} \right) W\left( \frac{\theta - \lambda_t}{B_T} \right) \left| \lambda_s \right|^{-4d}\left| \lambda_t \right|^{-2} \\
                &+ \frac{CT\log T}{B_T^2T^4} \sum_{s\in N_0} \sum_{t\in N_0}W\left( \frac{\theta - \lambda_s}{B_T} \right) W\left( \frac{\theta - \lambda_t}{B_T} \right) \left| \lambda_s \right|^{-4d}\left| \lambda_t \right|^{-1}\\
                &+ \frac{\log^2 T}{B_T^2T^4} \sum_{s\in N_0} \sum_{t-s=-T_0-s, t-s \neq 0, -2s}^{T_0-s}W\left( \frac{\theta - \lambda_s}{B_T} \right) W\left( \frac{\theta - \lambda_t}{B_T} \right) \left| \lambda_s \right|^{-4d} |\lambda_{t-s}|^{-2}\\
                &\lesssim \left( \frac{1}{B_T T^2} \sum_{s\in N_0} W\left( \frac{\theta - \lambda_s}{B_T} \right)\left| \lambda_s \right|^{-4d} \right) \left( \frac{1}{B_T T^2} \sum_{t\in N_0} W\left( \frac{\theta - \lambda_t}{B_T} \right)\left| \lambda_t \right|^{-2}  \right)\\
                &+ T\log T \left( \frac{1}{B_T T^2} \sum_{s\in N_0} W\left( \frac{\theta - \lambda_s}{B_T} \right)\left| \lambda_s \right|^{-4d} \right) \left( \frac{1}{B_T T^2} \sum_{t\in N_0} W\left( \frac{\theta - \lambda_t}{B_T} \right)\left| \lambda_t \right|^{-1}  \right)\\
                &+ \frac{\log^2 T}{B_T T^2} \left( \frac{1}{B_T T^2} \sum_{s\in N_0} W\left( \frac{\theta - \lambda_s}{B_T} \right)\left| \lambda_s \right|^{-4d} \right) \sum_{\Delta_{t,s}=-T_0-s,\Delta_{t,s} \neq 0, -2s }^{T_0-s} \frac{T^2}{\Delta_{t,s}^2}\\
                &\lesssim \frac{|\theta|^{-4d-2}}{T^2}\mathbb{1}(|\theta|>2\rho B_T) + \left(\frac{1}{TB_T^{1+4d}} + \frac{1}{T^{2-4d}B_T^{2}}\right)\mathbb{1}(|\theta|\leq 2\rho B_T)\\
                &+\frac{|\theta|^{-4d-1}\log T}{T}\mathbb{1}(|\theta|>2\rho B_T) + \left(\frac{1}{TB_T^{1+4d}} + \frac{1}{T^{2-4d}B_T^{2}}\right)\mathbb{1}(|\theta|\leq 2\rho B_T) (\log TB_T \log T)\\
                &+ \frac{|\theta|^{-4d}\log^2 T}{TB_T}\mathbb{1}(|\theta|>2\rho B_T) + \left(\frac{1}{TB_T^{1+4d}} + \frac{1}{T^{2-4d}B_T^{2}}\right)\mathbb{1}(|\theta|\leq 2\rho B_T) (\log^2 T)\\
                &\lesssim \frac{|\theta|^{-4d}\log^2 T}{TB_T}\mathbb{1}(|\theta|>2\rho B_T) + \left(\frac{1}{TB_T^{1+4d}} + \frac{1}{T^{2-4d}B_T^{2}}\right)\mathbb{1}(|\theta|\leq 2\rho B_T) (\log^2 T)
            \end{align*}

            Likewise, $P_{24}$ has the same upperbound as $P_{23}$. Putting $P_{21}, P_{22}, P_{23}, P_{24}$ together we have
            \begin{align*}
                P_2 \lesssim   \frac{|\theta|^{-4d}\log^2 T}{TB_T}\mathbb{1}(|\theta|>2\rho B_T) + \left(\frac{\log^2 T}{TB_T^{1+4d}} + \frac{\log^2 T}{T^{2-4d}B_T^{2}}\right)\mathbb{1}(|\theta|\leq 2\rho B_T)  
            \end{align*}

            For $P_3$, by Lemma \ref{lm:cum_property_long_mem} (i)
            \begin{align*}
                P_3 &= \frac{1}{B_T^2T^4}\sum_{s\in N_0}\sum_{t\in N_0}W\left( \frac{\theta - \lambda_s}{B_T} \right) W\left( \frac{\theta - \lambda_t}{B_T} \right)\left| \text{Cum}\left( d_i(\lambda_s), d_j(\lambda_t) \right)\text{Cum}\left( d_j(-\lambda_s), d_i(-\lambda_t) \right)  \right|\\
                &\leq \frac{1}{B_T^2T^4} \sum_{|s|= |t|}W\left( \frac{\theta - \lambda_s}{B_T} \right) W\left( \frac{\theta - \lambda_t}{B_T} \right)\left[|\lambda_s|^{-2d} \left( |\lambda_t|^{-1} + L_1(|\lambda_s + \lambda_t|) \right)\right]^2 \text{ (Lemma \ref{lm:cum_property_long_mem} (i)) }\\
                &+ \frac{1}{B_T^2T^4} \sum_{|s| < |t|}W\left( \frac{\theta - \lambda_s}{B_T} \right) W\left( \frac{\theta - \lambda_t}{B_T} \right)\left[|\lambda_s|^{-2d} \left( |\lambda_t|^{-1} + L_1(|\lambda_s + \lambda_t|) \right)\right]^2 \\
                &+ \frac{1}{B_T^2T^4} \sum_{|s| > |t|}W\left( \frac{\theta - \lambda_s}{B_T} \right) W\left( \frac{\theta - \lambda_t}{B_T} \right)\left[|\lambda_t|^{-2d} \left( |\lambda_s|^{-1} + L_1(|\lambda_s + \lambda_t|) \right)\right]^2\\
                &\leq \frac{|\theta|^{-4d}\log^2T}{B_T T}\mathbb{1}\left(|\theta|>2\rho B_T\right) + \left(\frac{\log^2T}{TB_T^{1+4d}} + \frac{\log^2T}{T^{2-4d}B_T^{2}}\right)\mathbb{1}(|\theta|\leq 2\rho B_T) \text{ (similar to $P_{21}, P_{22})$}\\
                &+ \frac{2}{B_T^2T^4} \sum_{|s|< |t|}W\left( \frac{\theta - \lambda_s}{B_T} \right) W\left( \frac{\theta - \lambda_t}{B_T} \right)\left[|\lambda_s|^{-2d} \left( |\lambda_t|^{-1} + L_1(|\lambda_s + \lambda_t|) \right)\right]^2 \text{(by symmetry)}\\
                &\lesssim \frac{|\theta|^{-4d}\log^2T}{B_T T}\mathbb{1}\left(|\theta|>2\rho B_T\right) + \left(\frac{\log^2T}{TB_T^{1+4d}} + \frac{\log^2T}{T^{2-4d}B_T^{2}}\right)\mathbb{1}(|\theta|\leq 2\rho B_T)\\
                &+ \frac{1}{B_T^2T^4} \sum_{|s|< |t|}W\left( \frac{\theta - \lambda_s}{B_T} \right) W\left( \frac{\theta - \lambda_t}{B_T} \right)|\lambda_s|^{-4d} |\lambda_t|^{-2}\\
                &+ \frac{2}{B_T^2T^4} \sum_{|s| <|t|}W\left( \frac{\theta - \lambda_s}{B_T} \right) W\left( \frac{\theta - \lambda_t}{B_T} \right)|\lambda_s|^{-4d} |\lambda_t|^{-1} \frac{\log T}{|\lambda_s + \lambda_t|}\\
                &+ \frac{1}{B_T^2T^4} \sum_{|s|< |t|}W\left( \frac{\theta - \lambda_s}{B_T} \right) W\left( \frac{\theta - \lambda_t}{B_T} \right)|\lambda_s|^{-4d} \frac{\log^2 T}{|\lambda_s + \lambda_t|^2}\\
                &\lesssim \frac{1 + \log T + \log^2 T}{B_T^2T^4} \left( \sum_{s=1}^{T-1} W\left( \frac{\theta - \lambda_s}{B_T} \right)|\lambda_s|^{-4d}  \right) \left( \sum_{t=1}^{T-1} W\left( \frac{\theta - \lambda_t}{B_T} \right)|\lambda_t|^{-2} \right)\\
                &\lesssim \frac{\log^2 T}{T^2} |\theta|^{-4d-2} \mathbb{1}(|\theta|>2\rho B_T) + \left(\frac{1}{TB_T^{1+4d}} + \frac{1}{T^{2-4d}B_T^{2}}\right)\mathbb{1}(|\theta|\leq 2\rho B_T)\log^2 T
            \end{align*}

            Therefore, 
            \begin{align*}
                var\left(\hat{\sigma}_{ij}(\theta)\right)   &\leq P_1 + P_2 + P_3 \\
                &\lesssim  \frac{|\theta|^{-4d}\log^2 T}{TB_T}\mathbb{1}(|\theta|>2\rho B_T) + \left(\frac{\log^2 T}{TB_T^{1+4d}} + \frac{\log^2 T}{T^{2-4d}B_T^{2}}\right)\mathbb{1}(|\theta|\leq 2\rho B_T) 
            \end{align*}
                
            
        \end{proof}

\section{Lemmas Used in Appendix \ref{sec:proof_lemma_for_main_lemma_1}}\label{sec:proof_lemma_for_lemma_2}
\subsection{Lemmas for Eigenvalues and Eigenvectors}
\begin{lemma}\label{lemma:first_second_derivative_eigenvec_projection}
Let $B(\theta) \in \mathbb{C}^{n\times n}$ be a twice continuously differentiable Hermitian matrix on $\Pi$. Suppose that its first eigengap is bounded below: there exists a constant $c>0$ such that
\[
\lambda_1(B(\theta))-\lambda_2(B(\theta))\geq c f_B(n)
\]
for all $\theta\in \Pi$. Let
\[
P_B(\theta) \coloneqq \mathbf{p}_1^*(B(\theta))\mathbf{p}_1(B(\theta)),
\]
where $\mathbf{p}_1(B(\theta))$ is the row eigenvector associated with the largest eigenvalue $\lambda_1(B(\theta))$. Then
\begin{align}\label{eq:first_derivative_eigenvec_characterization}
    \|P^{(1)}_B(\theta)\|_2
\le C\frac{\|B^{(1)}(\theta)\|_2}{f(n)}, \quad \lambda^{(1)}_1(B(\theta)) = \mathbf{p}_1(B(\theta)) B^{(1)}(\theta) \mathbf{p}_1^*(B(\theta))
\end{align}
and
\begin{align}\label{eq:second_derivative_eigenvec_characterization}
&\|P^{(2)}_B(\theta)\|_2
\le C\left(
\frac{\|B^{(2)}(\theta)\|_2}{f_B(n)}
+
\frac{\|B^{(1)}(\theta)\|_2^2}{f_B(n)^2}
\right), \\\nonumber
&\left|\lambda^{(2)}_1(B(\theta))\right| \leq C\left(
\|B^{(2)}(\theta)\|_2
+
\frac{\|B^{(1)}(\theta)\|_2^2}{f_B(n)}
\right)\nonumber
\end{align}
where $\cdot^{(r)}$ means the $r-$derivatives, $r = 1,2$, $C$ is the notation for constant independent of $n, T$.
\end{lemma}

\begin{proof}

By Theorem 3 in \cite{First-Order_Perturbation_Theory}, denote the reduced resolvent matrix of $B(\theta)$ w.r.t $\lambda_1(\theta)$ to be $S(\theta)$, then in the Hermitian case, $\left\|S(\theta)\right\|_2 \leq (\lambda_1(B(\theta))-\lambda_2(B(\theta)))^{-1} \leq Cf^{-1}(n)$,
\begin{align*}
    &\|P^{(1)}_B(\theta)\|_2 \leq {\|B^{(1)}(\theta)\|_2}{\left\|S(\theta)\right\|_2}
\leq \frac{\|B^{(1)}(\theta)\|_2}{\lambda_1(B(\theta))-\lambda_2(B(\theta))} \leq C\frac{\|B^{(1)}(\theta)\|_2}{f_B(n)},\\
&\|\mathbf{p}_1^{(1)}(B(\theta))\|_2 \leq {\|B^{(1)}(\theta)\|_2}{\left\|S(\theta)\right\|_2} \leq C\frac{\|B^{(1)}(\theta)\|_2}{f_B(n)}.
\end{align*}

The first derivative of the eigenvalue is given by Theorem 1 in \cite{First-Order_Perturbation_Theory}. It is also known as the Hellmann–Feynman theorem.

For the second derivatives, by (2.14) in Chapter II, Section 2.1 \cite{kato}, 
\begin{align*}
    \|P^{(2)}_B(\theta)\|_2 \leq \left\|B^{(2)}(\theta)\right\|_2\left\|S(\theta)\right\|_2 + \left\|B^{(1)}(\theta)\right\|_2^2\left\|S(\theta)\right\|_2^2 \leq C\left(
\|B^{(2)}(\theta)\|_2
+
\|B^{(1)}(\theta)\|_2^2
\right)
\end{align*}
Notice that although \cite{kato} states the theorem when $\theta = 0$, it could be trivially extended to any $\theta$.

Taking the derivative of the first order of the eigenvalue \eqref{eq:first_derivative_eigenvec_characterization} yields
\begin{align*}
    \left|\lambda^{(2)}_1(B(\theta))\right| \leq 2 \left\|B^{(1)}(\theta)\right\|_2\|\mathbf{p}_1^{(1)}(B(\theta))\|_2 + \left\|B^{(2)}(\theta)\right\|_2 \leq C\left(
\|B^{(2)}(\theta)\|_2
+
\frac{\|B^{(1)}(\theta)\|_2^2}{f_B(n)}
\right).
\end{align*}
\end{proof}

\begin{lemma}\label{lemma:bound_for_first_second_order_U}
    For Hermitian matrix functions $U(\theta), A(\theta) \in \mathbb{C}^{n\times n}$ satisfying
    \[
    U(\theta) = \frac{A(\theta)}{\lambda_1(A(\theta))}
    \] the first eigengap of $A(\theta)$ is bounded below: there exists a constant $c>0$ such that
\[
\lambda_1(A(\theta))-\lambda_2(A(\theta))\geq c f_A(n)
\]
for all $\theta\in \Pi$.
    then
    \begin{align}
        &\left\|U^{(1)}(\theta)\right\|_2 \leq \frac{\left\|A^{(1)}(\theta)\right\|_2}{\lambda_1(A(\theta))}\left[1 + \frac{\left\|A(\theta)\right\|_2}{\lambda_1(A(\theta))}\right] = \frac{2\left\|A^{(1)}(\theta)\right\|_2}{\lambda_1(A(\theta))} \\
        &\left\|U^{(2)}(\theta)\right\|_2\leq
C\left\{
\frac{\left\|A^{(2)}(\theta)\right\|_2}{\lambda_1(A(\theta))}
+
\frac{\left\|A^{(1)}(\theta)\right\|_2^2}{\lambda_1(A(\theta))^2}
+
\frac{\left\|A^{(1)}(\theta)\right\|_2^2}{\lambda_1(A(\theta))f_A(n)}
\right\} 
    \end{align}
\end{lemma}

\begin{proof}
    \begin{align*}
&\left\|U^{(2)}(\theta)\right\|_2\\
&\leq
\frac{\left\|A^{(2)}(\theta)\right\|_2}{\lambda_1(A(\theta))}
+
\frac{2\left\|A^{(1)}(\theta)\right\|_2\left|\lambda_1^{(1)}(A(\theta))\right|}{\lambda_1(A(\theta))^2}
+
\frac{\left\|A(\theta)\right\|_2\left|\lambda_1^{(2)}(A(\theta))\right|}{\lambda_1(A(\theta))^2}
+
\frac{2\left\|A(\theta)\right\|_2\left|\lambda_1^{(1)}(A(\theta))\right|^2}{\lambda_1(A(\theta))^3} 
\end{align*}
Then by \eqref{eq:second_derivative_eigenvec_characterization} we obtain the bound.
\end{proof}

\subsection{Successively Projected Vectors $h_j^{(k)}(\theta)$}

In the following two lemmas, Assumption \ref{assump:model}, \ref{assump:semiparametric_long_memory}, \ref{assump:pervasive_short_mem}, \ref{assump: spec_den_specific_form} case(a) and \ref{assump:bdd_error_eval} hold, let $G(\theta) \coloneqq (g_1(\theta), \dots, g_q(\theta))$, $P_j(\theta) \coloneqq \mathbf{p}_j^*(\theta)\mathbf{p}_j(\theta)$, $\Pi_{i}(\theta) \coloneqq I -  P_i(\theta)$, $h_j^{(m)}(\theta) \coloneqq \left(\prod_{i=1}^{m}\Pi_{i}(\theta)\right)g_j(\theta)$, $h_j^{(0)}(\theta) \coloneqq g_j(\theta)$, $H_{j}^{(m)}(\theta) \coloneqq h_j^{(m)}(\theta)\left(h_j^{(m)}(\theta)\right)^*$. 

\begin{lemma}\label{lemma:holder_successive_h}
    Suppose $P_m$ is $\rho_m$-H\"older continuous function, 
    then $h_{j}^{(m)}(\theta)$ and $H_{j}^{(m)}(\theta)$ are $\rho_m$-H\"older continuous.
\end{lemma}

\begin{proof}
    Here we use induction to prove that $h_j^{(m)}(\theta)$ is H\"older continuous. 

For $m = 1$, given any $\theta_1, \theta_2$, \begin{align}\label{eq:holder_rate_h_1_1}
    &\left\|h_j^{(1)}(\theta_1) - h_j^{(1)}(\theta_2)\right\|_2 = \left\| (I - P_1(\theta_1))g_j(\theta_1) - (I - P_1(\theta_2))g_j(\theta_2)  \right\|\\\nonumber
    &\leq \left\|g_j(\theta_1) - g_j(\theta_2)\right\|_2 + \left\|P_1(\theta_1) - P_1(\theta_2)\right\|_2\left\| g_j(\theta_1) \right\|_2 + \left\| g_j(\theta_1) - g_j(\theta_2) \right\|_2 \left\| P_1(\theta_2) \right\|_2\\\nonumber
    &\leq C_1\sqrt{n}|\theta_1 - \theta_2| + C_2\sqrt{n}|\theta_1 - \theta_2|^{\rho_1} \leq C\sqrt{n}|\theta_1 - \theta_2|^{\rho_1}.\nonumber
\end{align}


Thus, it only remains to prove that if $h_{j}^{(m-1)}(\theta)$ is $\beta_{m-1}$-H\"older continuous, $h_{j}^{(m)}(\theta)$ is $\beta_{m}$-H\"older continuous.

Since $h_j^{(m)}(\theta) = \Pi_{m}(\theta)h_j^{(m-1)}(\theta) = \left(I - P_{m}(\theta)\right)h_j^{(m-1)}(\theta)$, we have
\begin{align}\label{eq:holder_rate_h_j_k}
    &\left\|h_j^{(m)}(\theta_1) - h_j^{(m)}(\theta_2)\right\|_2 \\\nonumber
    &\leq \left\|h_j^{(m-1)}(\theta_1) - h_j^{(m-1)}(\theta_2)\right\|_2 + \left\|P_{m}(\theta_1) - P_{m}(\theta_2)\right\|_2\left\| h_j^{(m-1)}(\theta_1) \right\|_2 \\\nonumber
    &+ \left\| h_j^{(m-1)}(\theta_1) - h_j^{(m-1)}(\theta_2) \right\|_2 \left\| P_{m}(\theta_2) \right\|_2\\\nonumber
    &\leq C\left\|g_j(\theta_1)\right\|_2|\theta_1 - \theta_2|^{\beta_m} \leq C\sqrt{n}|\theta_1 - \theta_2|^{\beta_m}
\end{align}
$h_j^{(m)}(\theta)$ is $\beta_{m}$-H\"older continuous because $\beta_m \leq \beta_{m-1}$.

Since $\left\| h_j^{(m)}(\theta) \right\|_2 \leq \left\|g_j(\theta)\right\|_2 \leq C\sqrt{n}$, $\left\|H_j^{(m)}(\theta_1) - H_j^{(m)}(\theta_2)\right\|_2 \leq Cn|\theta_1 - \theta_2|^{\beta_m}$.


\end{proof}

\begin{lemma} \label{lemma:upperbnd_h_j_m}
Suppose $P_l$ is $\rho_l$-H\"older continuous function, and Statement \ref{stmt:Pk_limit}(l) holds for all $l = 1, \dots, m$. For $j \leq m$, 
    \begin{align}\label{eq:h_j_m_upperbnd}
        \left\|h_j^{(m)}(\theta_1)-h_j^{(m)}(\theta_2)\right\|_2
\leq
C\sqrt{n}\left[
|\theta_1-\theta_2|^{\rho_j}
+
\left(|\theta_1|^{\rho_j}+|\theta_2|^{\rho_j}\right)
|\theta_1-\theta_2|^{\rho_m}
\right],
    \end{align}
    and
    \begin{align}\label{eq:h_j_m_upperbnd_at_zero}
        \left\|h_j^{(m)}(\theta)\right\|_2 \leq C\sqrt{n}|\theta|^{\rho_j}
    \end{align}
\end{lemma}

\begin{proof}
    For $j \leq m$, $h_j^{(m)}(\theta) = \prod_{i = j+1}^{m}\Pi_i(\theta) h_j^{(j)}(\theta)$, then
\begin{align}\label{eq:h_j_m_decomp}
    \left\|h_j^{(m)}(\theta_1) - h_j^{(m)}(\theta_2)\right\|_2 &\leq \left\|\prod_{i = j+1}^{m}\Pi_i(\theta_1) h_j^{(j)}(\theta_1) - \prod_{i = j+1}^{m}\Pi_i(\theta_2) h_j^{(j)}(\theta_2)\right\|_2\\\nonumber
    &\leq \left\|\prod_{i = j+1}^{m}\Pi_i(\theta_1) - \prod_{i = j+1}^{m}\Pi_i(\theta_2)\right\|_2 \left\|h_j^{(j)}(\theta_1)\right\|_2 + \left\|h_j^{(j)}(\theta_1) - h_j^{(j)}(\theta_2)\right\|_2\nonumber
\end{align}

Notice that the first term, using the same logic as decomposing and triangular inequality used in Lemma \ref{lemma:holder_successive_h}, we have that the product of bounded H\"older functions is H\"older with the minimum exponent, so $\left\|\prod_{i = j+1}^{m}\Pi_i(\theta_1) - \prod_{i = j+1}^{m}\Pi_i(\theta_2)\right\|_2 \leq C|\theta_1 - \theta_2|^{\rho_m}$. Meanwhile, since Statement \ref{stmt:Pk_limit}(j) holds, \begin{align*}
    h_j^{(j)}(0) = \left(I - P_j(0)\right)h_j^{(j-1)}(0)= \left(I - \frac{h_j^{(j-1)}(0)\left(h_j^{(j-1)}(0)\right)^*}{\left\|h_j^{(j-1)}(0)\right\|_2^2}\right)h_j^{(j-1)}(0) = 0.
\end{align*}
Furthermore, since $P_j$ is $\rho_j$-H\"older continuous, by Lemma \ref{lemma:holder_successive_h}, $h_j^{(j)}(\theta)$ is $\rho_j$-H\"older continuous and
\begin{align*}
    &\left\|h_j^{(j)}(\theta)\right\|_2 = \left\|h_j^{(j)}(\theta) - h_j^{(j)}(0)\right\|_2 \leq C\sqrt{n}|\theta|^{\rho_j}, \\
    &\left\|h_j^{(m)}(\theta)\right\|_2 \leq \left\|\prod_{i = j+1}^{m}\Pi_i(\theta_1)\right\|_2 \left\|h_j^{(j)}(\theta_1)\right\|_2\leq C\sqrt{n}|\theta|^{\rho_j}.
\end{align*}
Plugging back into \eqref{eq:h_j_m_decomp}, and by symmetry yields the final results.
\end{proof}

\subsection{Successively Projected Vectors $h_{j,k}^{(r)}(\theta)$}

In the following two lemmas, Assumption \ref{assump:model}, \ref{assump:semiparametric_long_memory}, \ref{assump:pervasive_short_mem}, \ref{assump: spec_den_specific_form} case(a) and \ref{assump:bdd_error_eval} hold, let $G(\theta) \coloneqq (g_1(\theta), \dots, g_q(\theta))$, $P_j(\theta) \coloneqq \mathbf{p}_j^*(\theta)\mathbf{p}_j(\theta)$, $\Pi_{i}(\theta) \coloneqq I -  P_i(\theta)$, $h_{j,m}^{(r)}(\theta) \coloneqq \left(\prod_{i=1}^{m}\Pi_{i}(\theta)\right)g_j(\theta)$, $h_j^{(0)}(\theta) \coloneqq g_j(\theta)$, $H_{j,m}^{(r)}(\theta) \coloneqq h_{j,m}^{(r)}(\theta)\left(h_{j,m}^{(r)}(\theta)\right)^*$. 

\begin{lemma}\label{lemma:holder_successive_h}
    Suppose $P_m$ is $\rho_m$-H\"older continuous function, 
    then $h_{j,m}^{(r)}(\theta)$ and $H_{j,m}^{(r)}(\theta)$ are $\rho_m$-H\"older continuous.
\end{lemma}

\begin{proof}
    Here we use induction to prove that $h_j^{(m)}(\theta)$ is H\"older continuous. 

For $m = 1$, given any $\theta_1, \theta_2$, \begin{align}\label{eq:holder_rate_h_1_1}
    &\left\|h_{j,1}^{(0)}(\theta_1) - h_{j,1}^{(0)}(\theta_2)\right\|_2 = \left\| (I - P_1(\theta_1))g_j(\theta_1) - (I - P_1(\theta_2))g_j(\theta_2)  \right\|\\\nonumber
    &\leq \left\|g_j(\theta_1) - g_j(\theta_2)\right\|_2 + \left\|P_1(\theta_1) - P_1(\theta_2)\right\|_2\left\| g_j(\theta_1) \right\|_2 + \left\| g_j(\theta_1) - g_j(\theta_2) \right\|_2 \left\| P_1(\theta_2) \right\|_2\\\nonumber
    &\leq C_1\sqrt{n}|\theta_1 - \theta_2| + C_2\sqrt{n}|\theta_1 - \theta_2|^{\rho_1} \leq C\sqrt{n}|\theta_1 - \theta_2|^{\rho_1}.\nonumber
\end{align}


Thus, it only remains to prove that if $h_{j,m-1}^{(0)}(\theta)$ is $\rho_{m-1}$-H\"older continuous, $h_{j}^{(m)}(\theta)$ is $\rho_{m}$-H\"older continuous.

Since $h_{j,m}^{(0)}(\theta) = \Pi_{m}(\theta)h_{j,m-1}^{(0)}(\theta) = \left(I - P_{m}(\theta)\right)h_{j,m-1}^{(0)}(\theta)$, we have
\begin{align}\label{eq:holder_rate_h_j_k}
    &\left\|h_{j,m}^{(0)}(\theta_1) - h_{j,m}^{(0)}(\theta_2)\right\|_2 \\\nonumber
    &\leq \left\|h_{j,m-1}^{(0)}(\theta_1) - h_{j,m-1}^{(0)}(\theta_2)\right\|_2 + \left\|P_{m}(\theta_1) - P_{m}(\theta_2)\right\|_2\left\| h_{j,m-1}^{(0)}(\theta_1) \right\|_2 \\\nonumber
    &+ \left\| h_{j,m-1}^{(0)}(\theta_1) - h_{j,m-1}^{(0)}(\theta_2) \right\|_2 \left\| P_{m}(\theta_2) \right\|_2\\\nonumber
    &\leq C\left\|g_j(\theta_1)\right\|_2|\theta_1 - \theta_2|^{\rho_m} \leq C\sqrt{n}|\theta_1 - \theta_2|^{\rho_m}
\end{align}
$h_{j,m}^{(0)}(\theta)$ is $\rho_{m}$-H\"older continuous because $\rho_m \leq \rho_{m-1}$.

Since $\left\| h_{j,m}^{(0)}(\theta) \right\|_2 \leq \left\|g_j(\theta)\right\|_2 \leq C\sqrt{n}$, $\left\|H_{j,m}^{(0)}(\theta_1) - H_{j,m}^{(0)}(\theta_2)\right\|_2 \leq Cn|\theta_1 - \theta_2|^{\rho_m}$.


\end{proof}

\begin{lemma} \label{lemma:upperbnd_h_j_m}
Suppose $P_l$ is $\rho_l$-H\"older continuous function, and Statement \ref{stmt:Pk_limit}(l) holds for all $l = 1, \dots, m$. For $j \leq m$, 
    \begin{align}\label{eq:h_j_m_upperbnd}
        \left\|h_{j,m}^{(0)}(\theta_1)-h_{j,m}^{(0)}(\theta_2)\right\|_2
\leq
C\sqrt{n}\left[
|\theta_1-\theta_2|^{\rho_j}
+
\left(|\theta_1|^{\rho_j}+|\theta_2|^{\rho_j}\right)
|\theta_1-\theta_2|^{\rho_m}
\right],
    \end{align}
    and
    \begin{align}\label{eq:h_j_m_upperbnd_at_zero}
        \left\|h_{j,m}^{(0)}(\theta)\right\|_2 \leq C\sqrt{n}|\theta|^{\rho_j}
    \end{align}
\end{lemma}

\begin{proof}
    For $j \leq m$, $h_{j,m}^{(0)}(\theta) = \prod_{i = j+1}^{m}\Pi_i(\theta) h_{j,j}^{(0)}(\theta)$, then
\begin{align}\label{eq:h_j_m_decomp}
    \left\|h_{j,m}^{(0)}(\theta_1) - h_{j,m}^{(0)}(\theta_2)\right\|_2 &\leq \left\|\prod_{i = j+1}^{m}\Pi_i(\theta_1) h_{j,j}^{(0)}(\theta_1) - \prod_{i = j+1}^{m}\Pi_i(\theta_2) h_{j,j}^{(0)}(\theta_2)\right\|_2\\\nonumber
    &\leq \left\|\prod_{i = j+1}^{m}\Pi_i(\theta_1) - \prod_{i = j+1}^{m}\Pi_i(\theta_2)\right\|_2 \left\|h_{jj}^{(0)}(\theta_1)\right\|_2 + \left\|h_{jj}^{(0)}(\theta_1) - h_{jj}^{(0)}(\theta_2)\right\|_2\nonumber
\end{align}

Notice that the first term, using the same logic as decomposing and triangular inequality used in Lemma \ref{lemma:holder_successive_h}, we have that the product of bounded H\"older functions is H\"older with the minimum exponent, so $\left\|\prod_{i = j+1}^{m}\Pi_i(\theta_1) - \prod_{i = j+1}^{m}\Pi_i(\theta_2)\right\|_2 \leq C|\theta_1 - \theta_2|^{\rho_m}$. Meanwhile, since Statement \ref{stmt:Pk_limit}(j) holds, \begin{align*}
    h_{j,j}^{(0)}(0) = \left(I - P_j(0)\right)h_{j,j-1}^{(0)}(0)= \left(I - \frac{h_{j,j-1}^{(0)}(0)\left(h_{j,j-1}^{(0)}(0)\right)^*}{\left\|h_{j,j-1}^{(0)}(0)\right\|_2^2}\right)h_{j,j-1}^{(0)}(0) = 0.
\end{align*}
Furthermore, since $P_j$ is $\rho_j$-H\"older continuous, by Lemma \ref{lemma:holder_successive_h}, $h_{j,j}^{(0)}(\theta)$ is $\rho_j$-H\"older continuous and
\begin{align*}
    &\left\|h_{j,j}^{(0)}(\theta)\right\|_2 = \left\|h_{j,j}^{(0)}(\theta) - h_{j,j}^{(0)}(0)\right\|_2 \leq C\sqrt{n}|\theta|^{\rho_j}, \\
    &\left\|h_{j,m}^{(0)}(\theta)\right\|_2 \leq \left\|\prod_{i = j+1}^{m}\Pi_i(\theta_1)\right\|_2 \left\|h_{j,j}^{(0)}(\theta_1)\right\|_2\leq C\sqrt{n}|\theta|^{\rho_j}.
\end{align*}
Plugging back into \eqref{eq:h_j_m_decomp}, and by symmetry yields the final results.
\end{proof}

\subsection{Increment bounds for $\Sigma(\theta)$}
\begin{lemma}\label{lm:increment_bound_for_Sigma}
    Under Assumption \ref{assump:model}, \ref{assump:semiparametric_long_memory}, \ref{assump:pervasive_short_mem}, \ref{assump: spec_den_specific_form}, and \ref{assump:bdd_error_eval}, 
    \begin{enumerate}
        \item If $d_j$ differs across factors, 
        \begin{align}\label{eq:increment_bound_for_Sigma_factors}
            \left\| \Sigma(u_1) - \Sigma(u_2)  \right\| &\leq n\left( C_1|u_1|^{-2d} + C_2|u_2|^{-2d}\right) \left|u_1 - u_2\right| + C_3 n \left| |u_1|^{-2d} - |u_2|^{-2d} \right|
        \end{align}
        \item If $d_j$ differs across rows, \begin{align}\label{eq:increment_bound_for_Sigma_rows}
            \left\| \Sigma(u_1) - \Sigma(u_2)  \right\| &\leq n\left( C_1|u_1|^{-2d} + C_2|u_2|^{-2d}\right) \left|u_1 - u_2\right| + C_3 n \left| |u_1|^{-2d} - |u_2|^{-2d} \right|\\\nonumber
            &+ 2n \left(|u_2|^{-2d} + |u_1|^{-2d}\right)\left|sin \frac{\pi d}{2}\right|\mathbb{1}\left\{u_1u_2<0\right\} \nonumber
        \end{align}
    \end{enumerate}
    Furthermore, if $|u_1|, |u_2| \geq |\pi/h|$, $h$ is a nonzero integer, we have 
    \begin{align}\label{eq:increment_bound_for_Sigma_away_from_zero}
        \left\| \Sigma(u_1) - \Sigma(u_2)  \right\| &\leq Cn\left|u_1 - u_2\right|\left(   |u_1|^{-2d} + |u_2|^{-2d} + |u_1|^{-2d-1} + |u_2|^{-2d-1}\right)\\\nonumber
        &+ 2n \left(|u_2|^{-2d} + |u_1|^{-2d}\right)\nonumber
    \end{align}
\end{lemma}

\begin{proof}
    \begin{enumerate}
        \item If $d_j$ differs across factors,
        let $\widetilde{D}(u) = diag\left( |u|^{ - d_l} \right), l = 1,...,q$, then
\begin{align*}
    \Sigma^{\chi}(u_2) - \Sigma^{\chi}\left(u_1  \right) = G(u_2)\widetilde{D}(u_2)G^{\ast}(u_2) - G(\theta_1)\widetilde{D}(u_1)G^{\ast}(u_1).
\end{align*}
By Assumption \ref{assump:pervasive_short_mem}, we have
\begin{align}\label{eq:d_differ_shocks_decomp}
    \left\| \Sigma^{\chi}(u_2) - \Sigma^{\chi}\left( u_1  \right) \right\|_{op} &\leq \left\| G(u_2) - G(u_1)  \right\| \left\|  \widetilde{D}^2(u_2)\right\| \left\|G^{\ast}(u_2)\right\|\\\nonumber
    &+ \left\|G(u_1)\right\| \left\|G^{\ast}(u_2)\right\| \left\|  \widetilde{D}^2(u_2) - \widetilde{D}^2(u_1)\right\|\\\nonumber
    &+ \left\| G^{\ast}(u_2) - G^{\ast}(u_1)  \right\| \left\|  \widetilde{D}^2(u_1)\right\| \left\|G(u_1)\right\|\nonumber
\end{align}
By Assumption \ref{assump:pervasive_short_mem}, the first and the third terms are controlled with $n|u_2 - u_1||u_2|^{-2d}$, $n|u_2 - u_1||u_1|^{-2d}$. As for the second term, notice that 
\begin{align*}
    \left\|  \widetilde{D}^2(u_2) - \widetilde{D}^2(u_1)\right\| = \max_{j =1,...,q} \left| |u_1|^{-2d_j} - |u_2|^{-2d_j} \right|
    \end{align*}
    Define the function $f_{a, b}(x) = |a^{-x} - b^{-x}|$, $x>0$, $a,b \in (0,1)$.
        Without loss of generality, assume $a > b \in (0,1)$, we have $a^{-x}<b^{-x}$, hence $f_{a, b}(x) = b^{-x} - a^{-x}$, and that
\[
f'_{a, b}(x)= -\ln(b)\,b^{-x} + \ln(a)\,a^{-x}.
\]
Since $0<b<a\leq1$, we have $-\ln(b)>-\ln(a)\geq 0$ and $b^{-x}>a^{-x}$, so $f'_{a, b}(x)>0$ for all $x\ge 0$.
Therefore, $f_{a, b}(x)$ is increasing on $[0,\infty)$, and for any $d_j\le d$,
\[
\bigl|a^{-2d_j}-b^{-2d_j}\bigr|
= f_{a, b}(2d_j)\le f_{a, b}(2d)
= \bigl|a^{-2d}-b^{-2d}\bigr|.
\] When $a<b$ the argument is symmetric; if $a=b$ both sides are $0$.
Thus, when $|u_1|, |u_2| \leq 1$ we have \[
\left\|  \widetilde{D}^2(u_2) - \widetilde{D}^2(u_1)\right\| \leq \left| |u_1|^{-2d} - |u_2|^{-2d} \right|.
\]
When $|u_1|\leq 1 < |u_2| \leq \pi$, $\max_{j = 1,\dots, q}\left| \left| u_1/u_2 \right|^{-2d_j} - 1\right| \leq \left| \left| u_1/u_2 \right|^{-2d} - 1\right|$ and that
\begin{align*}
&\left\|  \widetilde{D}^2(u_2) - \widetilde{D}^2(u_1)\right\| \leq \max_{j = 1,\dots, q}\left| \left| u_1/u_2 \right|^{-2d_j} - 1\right| \max_{j = 1,\dots, q} |\theta|^{-2d_j}\\
&\leq \left| \left| u_1/u_2 \right|^{-2d} - 1\right| |u_2|^{-2\tilde{d}} = \left| |u_1|^{-2d} - |u_2|^{-2d} \right| |u_2|^{2\Delta} \leq \left| |u_1|^{-2d} - |u_2|^{-2d} \right|.
\end{align*}
When $|u_2|\leq 1 < |u_1| \leq \pi$, the argument is symmetric. When $|u_1|, |u_2| \leq (1,\pi]$, $f_{a,b}(\cdot)$ is continuous. Then $\left\|  \widetilde{D}^2(u_2) - \widetilde{D}^2(u_1)\right\| \leq \max_{j = 1,\dots, q}\left| |u_1|^{-2d_j} - |u_2|^{-2d_j} \right| \leq C\left| |u_1| - |u_2| \right| \leq C \left| u_1 - u_2 \right| \leq C \left| u_1 - u_2 \right|\left(|u_1|^{-2d} \wedge |u_2|^{-2d}\right)$.
Plugging back into (\ref{eq:d_differ_shocks_decomp}), and by Assumption \ref{assump:bdd_error_eval}, $\left\| \Sigma_{\xi}(u_1) - \Sigma_{\xi}(u_2)  \right\| \leq C|u_1 - u_2|$ we have for $u_1, u_2 \in \Pi$, 
\begin{align*}
     \left\| \Sigma(u_1) - \Sigma(u_2)  \right\| \leq n\left( C_1|u_1|^{-2d} + C_2|u_2|^{-2d}\right) \left|u_1 - u_2\right| + C_3 n \left| |u_1|^{-2d} - |u_2|^{-2d} \right|
\end{align*}
\item If $d_j$ differs across rows. Using notations from Assumption \ref{assump: spec_den_specific_form}, we have 
\begin{align*}
    \Sigma_{\chi}(u_2) - \Sigma_{\chi}\left( u_1  \right) &= D_n(u_2)G(\theta_2) G^{\ast}(u_2)D_n^{\ast}(u_2) - D_n(u_1)G(u_1) G^{\ast}(u_1)D_n^{\ast}(u_1)\\
    &= n D_n(u_2) S(u_2) D_n^{\ast}(u_2) - n D_n(u_1) S(u_1) D_n^{\ast}(u_1).
\end{align*}
Then
\begin{align}\label{eq:d_differ_rows_decomp}
    \left\| \Sigma_{\chi}(u_2) - \Sigma_{\chi}\left( u_1  \right) \right\|_{op} &\leq n\left\| D_n(u_2) - D_n(u_1)  \right\| \left\|  S(u_2)\right\| \left\|D_n^{\ast}(u_2)\right\|\\\nonumber
    &+ n\left\|D_n(u_1)\right\| \left\|D_n^{\ast}(u_2)\right\| \left\|  S(u_2) - S(u_1)\right\|\\\nonumber
    &+ n\left\| D_n^{\ast}(u_2) - D_n^{\ast}(u_1)  \right\| \left\|  D_n(u_1)\right\| \left\|S(u_1)\right\|\nonumber
\end{align}
By Assumption \ref{assump:pervasive_short_mem}, 
the second term is controlled with $n|u_2 - u_1||u_1|^{- d} |u_2|^{- d} \leq n|u_2 - u_1|\left(|u_1|^{- 2d} + |u_2|^{- 2d}\right)/2$, while the first and the third term is controlled with $n \left\|D_n(u_1)\right\| \left\|  D_n(u_2) - D_n(u_1)\right\|$, $n \left\|D_n(u_2)\right\|\left\|  D_n(u_2) - D_n(u_1)\right\|$.

Notice that $1-e^{-\iota \theta} = e^{-\iota \theta/2}(e^{\iota \theta/2} - e^{-\iota \theta/2}) = 2 \sin \frac{\theta}{2} \iota e^{-\iota \theta/2}$, $|1-e^{-\iota \theta}| = \sqrt{(1-\cos \theta)^2 + \sin^2 \theta} = \sqrt{2(1-\cos \theta)} = 2 \left|\sin \frac{\theta}{2}\right|$. For $\theta \in [-\pi, \pi)$, we always have $\sin(\theta/2)$ and $\theta$ share the same sign. Then
\begin{align}\label{eq:1_exp_iota_theta_property}
    1-e^{-\iota \theta} &=  2\left|\sin \frac{\theta}{2}\right|\iota e^{-\iota \frac{\theta}{2}} sgn(\theta) = 2\left|\sin \frac{\theta}{2}\right| e^{-\iota (\frac{\theta}{2} - \frac{\pi}{2}sgn(\theta))} \\\nonumber
    &= |\theta|  \frac{\sin (\theta/2)}{\theta/2} e^{-\iota (\frac{\theta}{2} - \frac{\pi}{2}sgn(\theta))}\nonumber
\end{align}
and
\begin{align*}
    (1-e^{-\iota \theta})^{-d_j} = |\theta|^{-d_j} l_j(\theta), \quad l_j(\theta) = e^{-\frac{\pi}{2} sgn(\theta) \iota d} \left(\frac{\sin (\theta/2)}{\theta/2}\right)^{-d} e^{\iota d \theta/2}.
\end{align*}

Thus,
\begin{align*}
    &\left\|  D_n(u_2) - D_n(u_1)\right\| = \left\|diag\left(|u_2|^{-d_j}l_j(u_2) - |u_1|^{-d_j}l_j(u_1)\right)\right\|\\
    &\leq \left\| diag\left[ \left(|u_2|^{-d_j} - |u_1|^{-d_j}\right) l_j(u_2)\right]  \right\| +  \left\| diag\left[ |u_1|^{-d_j} \left(l_j(u_2) - l_j(u_1)\right)\right]  \right\|\\
    &\leq \max_{j = 1, \dots, n} \left| |u_2|^{-d_j} - |u_1|^{-d_j} \right| \left|l_j(u_2)\right| + \max_{j = 1, \dots, n} |u_1|^{-d_j} \left| l_j(u_2) - l_j(u_1)  \right|
\end{align*}

Further denote $l^{(1)}_j(\theta) = e^{-\frac{\pi}{2} sgn(\theta) \iota d_j}$, $l^{(2)}_j(\theta) = \left(\frac{\sin (\theta/2)}{\theta/2}\right)^{-d_j} e^{\iota d_j \theta/2}$, then
\begin{align*}
    \max_{j = 1,\dots,n}\left| l_j(u_1) - l_j(u_2) \right| &\leq \max_{j = 1,\dots,n}  \left| l_j^{(1)}(u_1) - l_j^{(1)}(u_2) \right|\cdot \left| l_j^{(2)}(u_1) \right|\\
    &+ \max_{j = 1,\dots,n} \left|l_j^{(1)}(u_2)\right| \cdot \left| l_j^{(2)}(u_1) - l_j^{(2)}(u_2) \right|\\
    &\leq  \max_{j = 1,\dots,n}  \left| l_j^{(1)}(u_1) - l_j^{(1)}(u_2) \right| + |u_1 - u_2|
\end{align*}
since $\left|l^{(1)}_j(\theta)\right|, \left|l^{(2)}_j(\theta)\right| \leq 1$, $l^{(2)}_j(\theta)$ is continuous over $\Pi$.

Notice that if $sgn(u_1) = sgn(u_2)$, $\left| l_j^{(1)}(u_1) - l_j^{(1)}(u_2) \right| = 0$, and if $sgn(u_1) + sgn(u_2) = 0$,
\begin{align*}
    \left| l_j^{(1)}(u_1) - l_j^{(1)}(u_2) \right| = \left| e^{-\frac{\pi}{2} sgn(\theta_1) \iota d_j} - e^{\frac{\pi}{2} sgn(\theta_1) \iota d_j} \right| = 2\left|sin \frac{\pi d_j}{2}\right|
\end{align*}
For $d_j < 0.5$, $sin \frac{\pi d_j}{2}$ increases with $d_j$, $\max_{j = 1,\dots,n} \left| l_j^{(1)}(u_1) - l_j^{(1)}(u_2) \right| = 2\left|sin \frac{\pi d}{2}\right|$. Thus,
\begin{align*}
    \left\|  D_n(u_2) - D_n(u_1)\right\| &\leq \max_{j = 1, \dots, n} \left| |u_2|^{-d_j} - |u_1|^{-d_j} \right| \\
    &+ \max_{j = 1, \dots, n} |u_1|^{-d_j} 
    \left(2\left|sin \frac{\pi d}{2}\right|\mathbb{1}\left\{u_1u_2<0\right\} + |u_1 - u_2|\right)
\end{align*}

Symmetrically, 
\begin{align*}
    \left\|  D_n(u_2) - D_n(u_1)\right\| &\leq \max_{j = 1, \dots, n} \left| |u_2|^{-d_j} - |u_1|^{-d_j} \right| \\
    &+ \max_{j = 1, \dots, n} |u_2|^{-d_j} 
    \left(2\left|sin \frac{\pi d}{2}\right|\mathbb{1}\left\{u_1u_2<0\right\} + |u_1 - u_2|\right)
\end{align*}
holds as well. Analogous to the analysis for $d_j$ differs across factors, when $|u_1|, |u_2| \leq 1$, 
\begin{align*}
&\text{first term } + \text{ third term in (\ref{eq:d_differ_rows_decomp})} \\
&\leq Cn\left(|u_1|^{-d} + |u_2|^{-d}\right)\left||u_1|^{-d} - |u_2|^{-d}\right| + n \left(|u_2|^{-2d} + |u_1|^{-2d}\right)\max_j\left| l_j(u_2) - l_j(u_1)  \right|\\
&= Cn\left||u_1|^{-2d} - |u_2|^{-2d}\right| + n \left(|u_2|^{-2d} + |u_1|^{-2d}\right) \left(2\left|sin \frac{\pi d}{2}\right|\mathbb{1}\left\{u_1u_2<0\right\} + |u_1 - u_2|\right);
\end{align*}
When $|u_1| \in (1,\pi], |u_2| \leq 1$, 
\begin{align*}
    &\text{first term } + \text{ third term in (\ref{eq:d_differ_rows_decomp})} \\&\leq Cn\left(|u_1|^{-\tilde{d}} + |u_2|^{-d}\right)\left||u_1|^{-d} - |u_2|^{-d}\right||u_1|^{\Delta}+ n \left(|u_2|^{-2d} + |u_1|^{-2d}\right)\max_{j}\left| l_j(u_2) - l_j(u_1)  \right|\\
    &\leq Cn|u_1|^{2\Delta} \left(|u_1|^{-d} + |u_2|^{-d}\right)\left||u_1|^{-d} - |u_2|^{-d}\right| + n \left(|u_2|^{-2d} + |u_1|^{-2d}\right)\max_j\left| l_j(u_2) - l_j(u_1)  \right|\\
    &\leq Cn\left||u_1|^{-2d} - |u_2|^{-2d}\right| + n \left(|u_2|^{-2d} + |u_1|^{-2d}\right)\max_j\left| l_j(u_2) - l_j(u_1)  \right|\\
    &\leq Cn\left||u_1|^{-2d} - |u_2|^{-2d}\right| + n \left(|u_2|^{-2d} + |u_1|^{-2d}\right)\left(2\left|sin \frac{\pi d}{2}\right|\mathbb{1}\left\{u_1u_2<0\right\} + |u_1 - u_2|\right)
\end{align*}
When $|u_1|, |u_2| \in (1,\pi]$, 
\begin{align*}
    &\text{first term } + \text{ third term in (\ref{eq:d_differ_rows_decomp})} \\
    &\leq Cn\left(|u_1|^{-d} + |u_2|^{-d}\right)\left||u_1| - |u_2|\right| + n \left(|u_2|^{-2d} + |u_1|^{-2d}\right)\max_j\left| l_j(u_2) - l_j(u_1)  \right|\\
    &\leq Cn\left(|u_1|^{-2d} + |u_2|^{-2d}\right)\left|u_1 - u_2\right| + n \left(|u_2|^{-2d} + |u_1|^{-2d}\right)\max_j\left| l_j(u_2) - l_j(u_1)  \right|
\end{align*}
Plugging back into (\ref{eq:d_differ_rows_decomp}), we obtain the final result.
    \end{enumerate}

Furthermore, if $|u_1|, |u_2| \geq |\pi/h|$, $h$ is a nonzero integer, by Mean Value Theorem, \[
\left||u_1|^{-2d} - |u_2|^{-2d}\right| \leq C (|u_1|^{-2d-1} \wedge |u_2|^{-2d-1})\left|u_1 - u_2\right|,
\]
and that $\left| l_j(u_2) - l_j(u_1)  \right| \leq 2$.
Together with \eqref{eq:increment_bound_for_Sigma_factors} and \eqref{eq:increment_bound_for_Sigma_rows}, this implies \eqref{eq:increment_bound_for_Sigma_away_from_zero}.
    
\end{proof}


\subsection{Summation and Integration rates}
\begin{lemma}\label{lm:basic_summation_lambda_s}
For any $a \in (0,2), a \neq 1$,
    \begin{align*}
        \frac{1}{T^2} \sum_{s=-T_0, s\neq 0}^{T_0} |\lambda_s|^{-a}  = \mathcal{O}\left( \max\left\{ T^{-1}, T^{a-2} \right\} \right)
    \end{align*}
    and that \begin{align*}
        \frac{1}{T^2} \sum_{s=-T_0, s\neq 0}^{T_0} |\lambda_s|^{-1}  = \mathcal{O}\left(\frac{\log T}{T} \right)
    \end{align*}
\end{lemma}

\begin{proof}
    When $a > 0$, 
            \begin{align*}
                \left| \sum_{s=1}^{T_0} \frac{1}{s^a}  - \int_{1}^{T_0 +1} \frac{1}{t^a} d t  \right| &=  \left| \sum_{s=1}^{T_0} \int_{s}^{s+1} \frac{1}{s^a} dt  - \sum_{s=1}^{T_0} \int_{s}^{s+1} \frac{1}{t^a} d t  \right| \\
                &\leq \sum_{s=1}^{T_0} \int_{s}^{s+1} \left|  \frac{1}{s^a} - \frac{1}{t^a}  \right| dt\\
                &\leq \sum_{s=1}^{T_0} \int_{s}^{s+1} \frac{1}{s^a} - \frac{1}{(s+1)^a} dt\\
                &= \sum_{s=1}^{T_0} \frac{1}{s^a} - \frac{1}{(s+1)^a}\\
                &= 1-\frac{1}{T_0^a} \leq 1
            \end{align*}
            Then
            \begin{align*}
            \sum_{s=1}^{T_0} \frac{1}{s^a} &\leq 1 + \int_{1}^{T_0+1} \frac{1}{t^a} d t = \frac{(T_0+1)^{1-a}}{1-a} + \frac{a}{a-1}  = \mathcal{O}\left( \max\left\{ T^{1-a}, 1 \right\} \right), \text{ if } a \neq 1, a > 0.
            \end{align*}
            and \[ \sum_{s=1}^{T_0} \frac{1}{s} \leq 1 + \log T_0 \leq C\log T \].
            Then if $a\in (0,2), a\neq 1$,
            \begin{align*}
        \frac{1}{T^2} \sum_{s=-T_0, s\neq 0}^{T_0} |\lambda_s|^{-a}  \leq 2 T^{a-2} \sum_{s=1}^{T_0} s^{-a} = \mathcal{O}\left( \max\left\{ T^{-1}, T^{a-2} \right\} \right)
    \end{align*}
    and that \begin{align*}
        \frac{1}{T^2} \sum_{s=-T_0, s\neq 0}^{T_0} |\lambda_s|^{-1}  = \mathcal{O}\left(\frac{\log T}{T^{2-a}} \right)
    \end{align*}
\end{proof}

\begin{lemma}\label{lm:basic_summation_integral_summing_from_M_B_T_theta_pointwise}
  Let 
  \begin{align*}
M_s(\theta,B_T) &\coloneqq
\Bigl\{ s\in\mathbb{Z}\setminus\{0\}: \lambda_s = \frac{2\pi s}{T} \in \left( \theta-\rho B_T, \theta+\rho B_T\right)\Bigr\},\\
M(\theta,B_T) &\coloneqq
\Bigl\{ \lambda \in \Pi: \lambda \in \left( \theta-\rho B_T, \theta+\rho B_T\right)\Bigr\}\\
\widetilde{M}_s(\theta,B_T) &\coloneqq
\Bigl\{ s\in\mathbb{Z}\setminus\{0\}: I_s \cap M(\theta, B_T) \neq \emptyset \Bigr\}
\end{align*}where $I_s = [\lambda_s, \lambda_{s+1}),$ for $s = 1,..., T_0-1$, $I_s = [\lambda_{s-1}, \lambda_s),$ for $s = -T_0 +1, \cdots, -1$. $\rho$ is a constant independent from $T$ and $n$. Usually, it is defined by Assumption \ref{assump:W_kernel}, the compact support of kernel $W(\cdot)$. Then for any $a\in(0,2)$, 
\begin{enumerate}
    \item If $|\theta|>2\rho B_T$, 
    \[
    \sum_{s\in M_s(\theta,B_T)} |\lambda_s|^{-a} \lesssim (B_T T) |\theta|^{-a}
    \]\\
    If $|\theta|>2(\rho + 2\pi) B_T$, 
    \[
    \sum_{s\in \widetilde{M}_s(\theta,B_T)} |\lambda_s|^{-a} \lesssim (B_T T) |\theta|^{-a}
    \]
    \item If $|\theta|\le 2\rho B_T$,
    \begin{align*}
        \sum_{s\in M_s(\theta,B_T)} |\lambda_s|^{-a} \lesssim \begin{cases}
T\,B_T^{\,1-a}, & 0<a<1,\\
T\log(B_TT), & a=1,\\
T^{a}, & 1<a<2,
\end{cases}
    \end{align*}\\
    If $|\theta| \leq 2(\rho + 2\pi) B_T$, 
    \begin{align*}
        \sum_{s\in \widetilde{M}_s(\theta,B_T)} |\lambda_s|^{-a} \lesssim \begin{cases}
T\,B_T^{\,1-a}, & 0<a<1,\\
T\log(B_TT), & a=1,\\
T^{a}, & 1<a<2,
\end{cases}
    \end{align*}
\end{enumerate}

\end{lemma}

\begin{proof}
    \begin{enumerate}
        \item If $|\theta|>2\rho B_T$, for $s \in M(\theta,B_T)$, we have $|\lambda_s| \geq |\theta| - |\lambda_s - \theta| \geq |\theta| - \frac{1}{2}|\theta| = \frac{1}{2}|\theta|$. Then
        \[
        \sum_{s\in M_s(\theta,B_T)} |\lambda_s|^{-a} \leq 2^a |\theta|^{-a} \left| M_s(\theta,B_T)\right| \lesssim B_T T |\theta|^{-a}.
        \]

        Meanwhile, for all $\theta \in \Pi$ and $s \in \widetilde{M}_s(\theta, B_T)$, there exists such a $\lambda^* \in I_s \cap M(\theta, B_T)$, and that $|\lambda_s - \theta| \leq |\lambda_s - \lambda^*| + |\lambda^* - \theta| \leq \frac{2\pi}{T} + \rho B_T$. Since $TB_T \geq 1$, we further write $|\lambda_s - \theta| \leq (2\pi + \rho)B_T$, for all $s \in \widetilde{M}_s(\theta, B_T)$ and $\theta \in \Pi$. Consequently, $\left| \widetilde{M}_s(\theta,B_T)\right| \leq \frac{T}{2\pi}*2(2\pi + \rho) B_T \lesssim TB_T$.

        If $|\theta| > (2\rho + 4\pi)B_T$, $|\lambda_s| \geq |\theta| - |\lambda_s - \theta| \geq |\theta| - (2\pi + \rho)B_T \geq |\theta| - \frac{1}{2}|\theta| = \frac{1}{2}|\theta|$. Then
        \[
        \sum_{s\in \widetilde{M}_s(\theta,B_T)} |\lambda_s|^{-a} \leq 2^a |\theta|^{-a} \left| \widetilde{M}_s(\theta,B_T)\right| \lesssim B_T T |\theta|^{-a}.
        \]

        \item If $|\theta|\leq 2\rho B_T$, for $s \in M_s(\theta,B_T)$, we have $|\lambda_s| \leq |\theta| + |\lambda_s - \theta| \leq 2\rho B_T + \rho B_T = 3\rho B_T$. That is, we only consider the case when $|s| \leq \frac{3\rho}{2\pi} B_T T, s\neq 0$. Therefore, following the same calculation in Lemma \ref{lm:basic_summation_lambda_s}, we have when $a \neq 1$
        \[
        \sum_{s\in M_s(\theta,B_T)} |\lambda_s|^{-a} \lesssim T^a \sum_{|s| \leq \frac{3\rho}{2\pi} B_T T, s\neq 0} |s|^{-a} \lesssim T^a \max\{(B_T T)^{1-a}, 1\} \leq \max \{TB_T^{1-a}, T^a\}
        \]
        and when $a = 1$, 
        \[
        \sum_{s\in M_s(\theta,B_T)} |\lambda_s|^{-a} \lesssim T \sum_{|s| \leq \frac{3\rho}{2\pi} B_T T, s\neq 0} |s|^{-1} \lesssim T \log B_T T .
        \]

        If $|\theta| \leq (2\rho + 4\pi)B_T$, $|\lambda_s| \leq |\theta| + |\lambda_s - \theta| \leq (2\rho + 4\pi)B_T + (\rho + 2\pi)B_T \leq 3(\rho + 2\pi)B_T$, then $|s| \leq 3\left({\rho}/{2\pi} + 1\right) B_T T, s \neq 0$. Similarly,
        \begin{align*}
            \sum_{s \in \widetilde{M}_s(\theta, B_T)}|\lambda_s|^{-a} \lesssim \max\{TB_T^{1-a}, T^a\} \text{ if } a \neq 1;\;
            \sum_{s \in \widetilde{M}_s(\theta, B_T)}|\lambda_s|^{-1} \lesssim T \log B_T T
        \end{align*}
    \end{enumerate}
\end{proof}

\begin{lemma}\label{lm:basic_summation_integral_w_theta_2_tilde_d}
For any $p > 1 + q >1$,
    \begin{align*}
        \frac{1}{B_T T^2} \int_{\Pi} |\theta|^{q} \sum_{s=-T_0, s\neq 0}^{T_0} W\left(\frac{\theta - \lambda_s}{B_T}\right) |\lambda_s|^{-p} d\theta = \mathcal{O}\left( T^{p - q - 2} + B_T^{q} T^{p-2} \right)
    \end{align*}
\end{lemma}

\begin{proof}
Notice that since $\left|\lambda_s + uB_T\right|^{q} \leq \left|\lambda_s\right|^{q} + B_T^{q} |u|^{2\tilde{d}}$,
\begin{align*}
    \int_{\Pi} W\left(\frac{\theta - \lambda_s}{B_T}\right) |\theta|^{q} d\theta &\leq B_T \int_{\Pi} W(u) \left|\lambda_s + uB_T\right|^{q} du\\
    &\leq B_T \int_{\Pi} W(u) du \left|\lambda_s\right|^{q} + B_T^{1+q} \int_{\Pi} W(u) u^q du \\
    &\leq B_T |\lambda_s|^{q} + B_T^{1+q}
\end{align*}
Thus,
    \begin{align*}
        LHS &\leq \frac{1}{B_T T^2} \sum_{s=-T_0, s\neq 0}^{T_0}  \int_{\Pi} W\left(\frac{\theta - \lambda_s}{B_T}\right) |\theta|^{q} d\theta |\lambda_s|^{-p}\\
        &\leq \frac{1}{T^2} \sum_{s=-T_0, s\neq 0}^{T_0} |\lambda_s|^{q-p} + \frac{B_T^q}{T^2} \sum_{s=-T_0, s\neq 0}^{T_0} |\lambda_s|^{-p}\\
        &\leq T^{p-q-2} + B_T^q T^{p-2}
    \end{align*}
\end{proof}

        \begin{lemma}\label{lm:conv_rate_before_approx_identity} Under Assumption \ref{assump:model}, \ref{assump: spec_den_specific_form}, \ref{assump:bdd_error_eval}, and \ref{assump:pervasive_short_mem}, for almost every $\theta \in \Pi,$
            \begin{align*}
            &\left\|\frac{2\pi|\theta|^{2\tilde{d}}}{B_T T}\sum_{s=-T_0, s\neq 0}^{T_0}W\left(\frac{\theta - \lambda_s}{B_T}\right)  \left(\Sigma(\lambda_s) -  \Sigma(\theta)\right) \right\|\\
                &\lesssim \int_{\Pi} \left\| \frac{|\theta|^{2\tilde{d}}}{B_T}W\left(\frac{\theta - \lambda}{B_T}\right)  \left(\Sigma(\lambda) -  \Sigma(\theta)\right)  \right\| d\lambda\\
                &+ \frac{n|\theta|^{-2\Delta}}{TB_T} + \frac{n|\theta|^{2\tilde{d}}}{B_T T^{1-2d}}\mathbb{1}(|\theta| \leq (2\rho + 4\pi) B_T) + \frac{n|\theta|^{-2\Delta - 1}}{T}\mathbb{1}(|\theta| > (2\rho + 4\pi) B_T)
            \end{align*}
        \end{lemma}

        \begin{proof}
            Let 
            \begin{align*}
                H_{\theta}(\lambda) = \frac{|\theta|^{2\tilde{d}}}{B_T}W\left(\frac{\theta - \lambda}{B_T}\right)  \left(\Sigma(\lambda) -  \Sigma(\theta)\right)
            \end{align*}
            and $I_s = [\lambda_s, \lambda_{s+1}),$ for $s = 1,..., T_0-1$, $I_s = [\lambda_{s-1}, \lambda_s),$ for $s = -T_0 +1, \cdots, -1$. Denote $\Pi_1 = \Pi \setminus \bigl[-\lambda_1, \lambda_1\bigr)$, we have $\Pi_1 = \cup_{s=-T_0, s\neq 0}^{T_0} I_s \cup [-\pi, -\lambda_{T_0}) \cup [\lambda_{T_0}, \pi)$.
            
            Following almost the same analysis as shown in Lemma \ref{lm:basic_summation_lambda_s}, we have
            \begin{align*}
                &\left\|\frac{2\pi}{T} \sum_{s=-T_0, s\neq 0}^{T_0} H_{\theta}( \lambda_s) - \int_{\Pi_1} H_{\theta}(\lambda) d\lambda \right\| \\
                &= \left\| \sum_{s=-T_0, s\neq 0}^{T_0}\int_{I_s} H_{\theta}( \lambda_s)d \lambda - \sum_{s=-T_0, s\neq 0}^{T_0}\int_{I_s} H_{\theta}(\lambda) d\lambda - \int_{-\lambda_{T_0}}^{-\lambda_1} H_{\theta}(\lambda) d\lambda - \int_{\lambda_1}^{\lambda_{T_0}} H_{\theta}(\lambda) d\lambda \right\|\\
                &\leq \sum_{s=-T_0, s\neq 0}^{T_0}\int_{I_s}\left\| H_{\theta}(\lambda_s) - H_{\theta}(\lambda) \right\| d\lambda + \int_{-\pi}^{-\lambda_{T_0}} \left\|H_{\theta}(\lambda)\right\| d\lambda + \int_{\lambda_{T_0}}^{\pi} \left\|H_{\theta}(\lambda)\right\| d\lambda
            \end{align*}

            For the edge term, here we will only derive a crude bound as it is not the dominant term in the remainder. Remind that $T_0 = \lfloor \frac{T-1}{2} \rfloor$, $\lambda_{T_0} = \frac{2\pi T_0}{T}$, then $\pi - \lambda_{T_0} \leq 2\pi/T$. Hence,
            \begin{align*}
                \int_{\lambda_{T_0}}^{\pi} \left\|H_{\theta}(\lambda)\right\| d\lambda &\leq \frac{|\theta|^{2\tilde{d}}}{B_T} \int_{\lambda_{T_0}}^{\pi} \left\|\Sigma(\lambda)\right\| d\lambda + \frac{|\theta|^{2\tilde{d}}}{B_T}\left\|\Sigma(\theta)\right\| (\pi - \lambda_{T_0} ) \\
                &\leq \frac{|\theta|^{2\tilde{d}}}{B_T} n \pi^{-2d} (\pi - \lambda_{T_0}) + \frac{|\theta|^{2\tilde{d}}}{B_T} n |\theta|^{-2d}(\pi - \lambda_{T_0} ) \leq \frac{n |\theta|^{-2\Delta}}{B_T T}.
            \end{align*}

            Notice that since $W(\cdot)$ has compact support on $(-\rho, \rho)$ by Assumption \ref{assump:W_kernel}, then $H_{\theta}(\lambda)$ can be written as
            \begin{align*}
                H_{\theta}(\lambda) = \frac{|\theta|^{2\tilde{d}}}{B_T}  W\left(\frac{\theta - \lambda}{B_T}\right)  \left(\Sigma(\lambda) -  \Sigma(\theta)\right) \mathbb{1}(|\theta - \lambda| \leq \rho B_T)
            \end{align*}

            Meanwhile,
            \begin{align*}
                H_{\theta}(\lambda_s) - H_{\theta}(\lambda) &= \frac{|\theta|^{2\tilde{d}}}{B_T} \left[ W\left(\frac{\theta - \lambda_s}{B_T}\right) - W\left(\frac{\theta - \lambda}{B_T}\right) \right]\left( \Sigma(\lambda_s) -  \Sigma(\theta) \right)\\
                &+  \frac{|\theta|^{2\tilde{d}}}{B_T} W\left(\frac{\theta - \lambda}{B_T}\right) \mathbb{1} \left( |\theta - \lambda| \leq \rho B_T  \right)\left( \Sigma(\lambda_s) -  \Sigma(\lambda) \right)\\
                &\coloneqq P_1 + P_2
            \end{align*}
            By the Lipschitz continuity of $W(\cdot)$ we have $\left|W\left(\frac{\theta - \lambda_s}{B_T}\right) - W\left(\frac{\theta - \lambda}{B_T}\right)\right| \leq L |\lambda_s - \lambda|/B_T$. Furthermore, if $|\theta - \lambda_s| > \rho B_T$ and $|\theta - \lambda| > \rho B_T$, $\left|W\left(\frac{\theta - \lambda_s}{B_T}\right) - W\left(\frac{\theta - \lambda}{B_T}\right)\right|  = 0$. Thus, we could write
            \begin{align}\label{eq:increment_kernel_W}
                \left|W\left(\frac{\theta - \lambda_s}{B_T}\right) - W\left(\frac{\theta - \lambda}{B_T}\right)\right| \leq L\frac{|\lambda_s - \lambda|}{B_T}\left[\mathbb{1}\left( |\theta - \lambda_s| \leq \rho B_T  \right) + \mathbb{1}\left( |\theta - \lambda| \leq \rho B_T  \right)\right]
            \end{align}

            Then
            \begin{align*}
                \sum_{s=-T_0, s\neq 0}^{T_0}\int_{I_s}\left\| P_1 \right\| d\lambda &\leq \frac{|\theta|^{2\tilde{d}}}{B_T} \sum_{s=-T_0, s\neq 0}^{T_0}\int_{I_s}\frac{|\lambda_s - \lambda|}{B_T}\mathbb{1}\left( |\theta - \lambda_s| \leq \rho B_T  \right) \left\| \Sigma(\lambda_s) -  \Sigma(\theta)  \right\| d\lambda \\
                &+ \frac{|\theta|^{2\tilde{d}}}{B_T} \sum_{s=-T_0, s\neq 0}^{T_0}\int_{I_s}\frac{|\lambda_s - \lambda|}{B_T}\mathbb{1}\left( |\theta - \lambda| \leq \rho B_T  \right) \left\| \Sigma(\lambda_s) -  \Sigma(\theta)  \right\| d\lambda\\
                &\coloneqq P_{1,1} + P_{1,2}
            \end{align*}
            Applying $\int_{I_s} |\lambda_s - \lambda| d\lambda = \mathcal{O}\left(T^{-2}\right)$, the first term becomes
            \begin{align*}
                P_{1,1} \leq \frac{|\theta|^{2\tilde{d}}}{B_T^2T^2}\sum_{s=-T_0, s\neq 0}^{T_0}\mathbb{1}\left( |\theta - \lambda_s| \leq \rho B_T  \right) \left\| \Sigma(\lambda_s) -  \Sigma(\theta)  \right\|
            \end{align*}
            Denote 
            \[
            M_s(\theta,B_T):=
\Bigl\{ s\in\mathbb{Z}\setminus\{0\}: \lambda_s = \frac{2\pi s}{T} \in \left( \theta-\rho B_T, \theta+\rho B_T\right)\Bigr\}
            \]
            $|M_s(\theta, B_T)| \leq TB_T$, then by Lemma \ref{lm:basic_summation_integral_summing_from_M_B_T_theta_pointwise}
            \begin{align*}
                &\sum_{s=-T_0, s\neq 0}^{T_0}\mathbb{1}\left( |\theta - \lambda_s| \leq \rho B_T  \right) \left\| \Sigma(\lambda_s) -  \Sigma(\theta)  \right\| = \sum_{s \in M_s(\theta, B_T)}\left\| \Sigma(\lambda_s) -  \Sigma(\theta)  \right\|\\
                &\leq \sum_{s \in M_s(\theta, B_T)} \left\| \Sigma(\lambda_s)  \right\| + \sum_{s \in M_s(\theta, B_T)} \left\| \Sigma(\theta)  \right\| \leq Cn\sum_{s \in M_s(\theta, B_T)} |\lambda_s|^{-2d} + Cn|\theta|^{-2d}|M_s(\theta, B_T)|\\
                &\leq \begin{cases}
                    2CnB_T T |\theta|^{-2d} , & |\theta| > 2\rho B_T\\
                    CnT B_T^{1-2d} + CnB_T T |\theta|^{-2d} , & |\theta| < 2\rho B_T
                \end{cases} \\
                &\leq CnB_T T |\theta|^{-2d}
            \end{align*}
            where the last inequality is because for $|\theta| < 2\rho B_T$, $CnB_T T |\theta|^{-2d}$ dominates. Thus, $P_{1,1} \leq Cn|\theta|^{-2\Delta}/(B_T T)$.

            On the other hand, let 
            \begin{align*}
            &M(\theta,B_T):=
\Bigl\{ \lambda \in \Pi: \lambda \in \left( \theta-\rho B_T, \theta+\rho B_T\right)\Bigr\}\\
            &\widetilde{M}_s(\theta,B_T):=
\Bigl\{ s\in\mathbb{Z}\setminus\{0\}: I_s \cap M(\theta, B_T) \neq \emptyset \Bigr\}
            \end{align*}
            then
            \begin{align}\label{eq:P_1_2_lemma}
                P_{1,2} &= \frac{|\theta|^{2\tilde{d}}}{B_T} \sum_{s=-T_0, s\neq 0}^{T_0}\int_{I_s}\frac{|\lambda_s - \lambda|}{B_T}\mathbb{1}\left( |\theta - \lambda| \leq \rho B_T  \right) d\lambda \left\| \Sigma(\lambda_s) -  \Sigma(\theta)  \right\| \\\nonumber
                &=\frac{|\theta|^{2\tilde{d}}}{B_T} \sum_{s=-T_0, s\neq 0}^{T_0}\int_{I_s \cap M(\theta, B_T)}\frac{|\lambda_s - \lambda|}{B_T} d\lambda \left\| \Sigma(\lambda_s) -  \Sigma(\theta)  \right\|\\\nonumber
                &= \frac{|\theta|^{2\tilde{d}}}{B_T} \sum_{s \in \widetilde{M}_s(\theta, B_T)} \int_{I_s \cap M(\theta, B_T)}\frac{|\lambda_s - \lambda|}{B_T} d\lambda \left\| \Sigma(\lambda_s) -  \Sigma(\theta)  \right\|\\\nonumber
                &\leq \frac{|\theta|^{2\tilde{d}}}{B_T} \sum_{s \in \widetilde{M}_s(\theta, B_T)} \int_{I_s}\frac{|\lambda_s - \lambda|}{B_T} d\lambda \left\| \Sigma(\lambda_s) -  \Sigma(\theta)  \right\|
                \leq \frac{|\theta|^{2\tilde{d}}}{B_T^2 T^2}\sum_{s \in \widetilde{M}_s(\theta, B_T)}  \left\| \Sigma(\lambda_s) -  \Sigma(\theta)  \right\|\nonumber
            \end{align} 
            Notice that $\widetilde{M}_s(\theta,B_T)$ is almost the same as $M_s(\theta,B_T)$, with potentially two less than $2\pi/T$ segments, $|\widetilde{M}_s(\theta,B_T)| \leq CT(4\pi/T + 2\rho B_T) \leq CTB_T$. 
            By Lemma \ref{lm:basic_summation_integral_summing_from_M_B_T_theta_pointwise}, we have 
            \[
            \sum_{s \in \widetilde{M}_s(\theta, B_T)} |\lambda_s|^{-2d} \lesssim \begin{cases}
                (B_T T)|\theta|^{-2d}, & |\theta|> (2\rho + 4\pi)B_T\\
                T B_T^{1-2d}, & |\theta| \leq  (2\rho + 4\pi)B_T
            \end{cases}
            \]
            Then same as before, we have $P_{1,2} \leq Cn|\theta|^{-2\Delta}/(B_T T)$. Altogether we have $\sum_{s=-T_0, s\neq 0}^{T_0}\int_{I_s}\left\| P_1 \right\| d\lambda \leq Cn|\theta|^{-2\Delta}/(B_T T)$.

            Meanwhile, under Assumption \ref{assump: spec_den_specific_form}, by Lemma \ref{lm:increment_bound_for_Sigma}, we have for $\lambda \in I_s$,
            \begin{align*}
                \left\| \Sigma(\lambda_s) - \Sigma(\lambda)  \right\| &\leq \left\| \Sigma_{\chi}(\lambda_s) - \Sigma_{\chi}(\lambda)  \right\| + \left\| \Sigma_{\xi}(\lambda_s) - \Sigma_{\xi}(\lambda)  \right\|\\
                &\leq Cn\left| \lambda_s - \lambda \right| \left( |\lambda_s|^{-2d} + |\lambda|^{-2d} +  |\lambda_s|^{-2d-1} + |\lambda|^{-2d-1} \right) + C\left| \lambda_s - \lambda \right|\\
                &\leq Cn\left| \lambda_s - \lambda \right| |\lambda_s|^{-2d-1}
            \end{align*}
            where the last inequality is because $|\lambda|^{-2d} < \{ |\lambda_s|^{-2d}, |\lambda|^{-2d-1} \} < |\lambda_s|^{-2d-1}$, for $\lambda \in I_s$.

            Then 
            \begin{align*}
                \sum_{s=-T_0, s\neq 0}^{T_0}\int_{I_s}\left\| P_2 \right\| d\lambda &\leq \frac{|\theta|^{2\tilde{d}}}{B_T} \sum_{s=-T_0, s\neq 0}^{T_0}\int_{I_s}\mathbb{1}\left( |\theta - \lambda| \leq \rho B_T  \right) \left\| \Sigma(\lambda_s) -  \Sigma(\lambda)  \right\| d\lambda \\
                &\leq Cn \frac{|\theta|^{2\tilde{d}}}{B_T} \sum_{s=-T_0, s\neq 0}^{T_0}\int_{I_s}\mathbb{1}\left( |\theta - \lambda| \leq \rho B_T  \right) \left| \lambda_s - \lambda \right| d\lambda |\lambda_s|^{-2d-1}\\
                &\leq Cn \frac{|\theta|^{2\tilde{d}}}{B_T}  \sum_{s\in \widetilde{M}_s(\theta, B_T)} \int_{I_s} \left| \lambda_s - \lambda \right| d\lambda |\lambda_s|^{-2d-1}\\
                &\leq Cn \frac{|\theta|^{2\tilde{d}}}{B_T T^2}  \sum_{s\in \widetilde{M}_s(\theta, B_T)} |\lambda_s|^{-2d-1}\\
                &\leq \left\{\begin{aligned}
                    &\frac{Cn|\theta|^{2\tilde{d}}}{B_T T^{1-2d}}, & |\theta| \leq (2\rho + 4\pi) B_T \\
                    &\frac{Cn|\theta|^{-2\Delta - 1}}{T}, & |\theta| > (2\rho + 4\pi) B_T
                \end{aligned}\right.
            \end{align*}

            Finally, we have
            \begin{align*}
                &\left\|\frac{2\pi}{T} \sum_{s=-T_0, s\neq 0}^{T_0} H_{\theta}( \lambda_s) - \int_{\Pi_1} H_{\theta}(\lambda) d\lambda \right\| \\
                &\leq \sum_{s=-T_0, s\neq 0}^{T_0}\int_{I_s}\left\| H_{\theta}(\lambda_s) - H_{\theta}(\lambda) \right\| d\lambda + \int_{-\pi}^{-\lambda_{T_0}} \left\|H_{\theta}(\lambda)\right\| d\lambda + \int_{\lambda_{T_0}}^{\pi} \left\|H_{\theta}(\lambda)\right\| d\lambda\\
                &\leq \sum_{s=-T_0, s\neq 0}^{T_0}\int_{I_s}\left\| P_1 \right\| d\lambda + \sum_{s=-T_0, s\neq 0}^{T_0}\int_{I_s}\left\| P_2 \right\| d\lambda + \frac{Cn|\theta|^{-2\Delta}}{B_T T}\\
                &\leq \frac{Cn|\theta|^{-2\Delta}}{B_T T} + \left\{\begin{aligned}
                    &\frac{Cn|\theta|^{2\tilde{d}}}{B_T T^{1-2d}}, & |\theta| \leq (2\rho + 4\pi) B_T \\
                    &\frac{Cn|\theta|^{-2\Delta - 1}}{T}, & |\theta| > (2\rho + 4\pi) B_T
                \end{aligned}\right.
            \end{align*}
            
        \end{proof}

        \begin{lemma}\label{lm:calculation_W_theta_lambda_minus_theta}
        For $t\in (0,0.5), \alpha \in (0,1]$,
            \begin{align*}
                \int_{\Pi}\int_{\Pi}\frac{1}{B_T}W\left(\frac{\theta - \lambda}{B_T}\right)|\theta|^{-2t} \left| \lambda - \theta  \right|^{\alpha}d\lambda d\theta = \mathcal{O}\left( B_T^{\alpha} \right)
            \end{align*}
        \end{lemma}
        \begin{proof}
            Let $\nu = (\theta - \lambda)/B_T,$ $W(\cdot)$ is bounded, then
            \begin{align*}
                LHS &= B_T^{\alpha} \int_{0}^{2\pi B_T} \int_{0}^{\frac{\theta}{B_T}} \theta^{-2t} W(\nu) \nu^{\alpha} d\nu d\theta + B_T^{\alpha} \int_{2\pi B_T}^{2\pi } \int_{0}^{2\pi} \theta^{-2t} W(\nu) \nu^{\alpha} d\nu d\theta\\
                &\lesssim B_T^{1-2t+\alpha} + B_T^{\alpha} \lesssim B_T^{\alpha}.
            \end{align*}
        \end{proof}

        \begin{lemma}\label{lm:rate_sum_to_int_W_only} Assumption \ref{assump:W_kernel} holds.
            \begin{align*}
                \left|  \frac{2\pi}{B_T T}\sum_{s=-T_0, s\neq 0}^{T_0}W\left(\frac{\theta - \lambda_s}{B_T}\right)  -  \int_{\Pi_1}\frac{1}{B_T}W\left(\frac{\theta - \lambda}{B_T}\right) d\lambda \right| = \mathcal{O}\left(\frac{1}{T B_T}\right)
            \end{align*}
        \end{lemma}
        \begin{proof}
            \begin{align*}
                &\left|  \frac{2\pi}{B_T T}\sum_{s=-T_0, s\neq 0}^{T_0}W\left(\frac{\theta - \lambda_s}{B_T}\right)  -  \int_{\Pi_1}\frac{1}{B_T}W\left(\frac{\theta - \lambda}{B_T}\right) d\lambda \right|\\
                & =\left|\sum_{s=-T_0, s\neq 0}^{T_0} \int_{I_s}\frac{1}{B_T}W\left(\frac{\theta-\lambda_s}{B_T}\right) d \lambda - \sum_{s=-T_0, s\neq 0}^{T_0} \int_{I_s} \frac{1}{B_T} W\left(\frac{\theta-\lambda}{B_T}\right) d \lambda\right| \\
                & \leq \sum_{s=-T_0, s\neq 0}^{T_0} \int_{I_s} \frac{1}{B_T}\left|W\left(\frac{\theta-\lambda_s}{B_T}\right)-W\left(\frac{\theta -\lambda}{B_T}\right)\right| d \lambda \\
                & \leq \frac{L}{B_T^2} \sum_{s=-T_0, s\neq 0}^{T_0}\int_{I_s}|\lambda_s - \lambda| d\lambda \mathbb{1}\left(|\theta - \lambda_s| \leq \rho B_T\right)\\
                &+ \frac{L}{B_T^2} \sum_{s=-T_0, s\neq 0}^{T_0}\int_{I_s}|\lambda_s - \lambda| \mathbb{1}\left(|\theta - \lambda| \leq \rho B_T\right) d\lambda \; \;\text{ (by (\ref{eq:increment_kernel_W}))}
            \end{align*}
            The first term $\leq C|M_s(\theta, B_T)| /(T^2 B_T^2)  = \frac{C}{TB_T}$.
            Analogous to the proof for $P_{1,2}$ in Lemma \ref{lm:conv_rate_before_approx_identity} (\ref{eq:P_1_2_lemma}), we have the second term $\leq C|\widetilde{M}_s(\theta, B_T)| /(T^2 B_T^2)  = \frac{C}{TB_T}$. 
        \end{proof}

        \begin{lemma}\label{lm:rate_truncation_by_window} Assumption \ref{assump:model}, \ref{assump:semiparametric_long_memory}, \ref{assump:pervasive_short_mem}, and \ref{assump:W_kernel} hold. For $T$ sufficiently large such that $B_T \leq \frac{\pi}{2\rho}$,
            \begin{align*}
                \int_{\Pi}\left|  \int_{\Pi} \frac{1}{B_T} W\left(\frac{\theta - \lambda}{B_T}\right) d\lambda - 1  \right|\left\|\Sigma(\theta)\right\||\theta|^{2\tilde{d}} d\theta = \mathcal{O}\left(nB_T\right)
            \end{align*}
        \end{lemma}
        \begin{proof}
            \begin{align*}
                &\int_{\Pi}\left|  \int_{\Pi} \frac{1}{B_T} W\left(\frac{\theta - \lambda}{B_T}\right) d\lambda - 1  \right|\left\|\Sigma(\theta)\right\||\theta|^{2\tilde{d}} d\theta \\
                & = \int_{\Pi}\left|  \int_{\Pi} \frac{1}{B_T} W\left(\frac{\theta - \lambda}{B_T}\right) d\lambda - \int_{-\rho}^{\rho} W(\nu) d\nu  \right|\left\|\Sigma(\theta)\right\||\theta|^{2\tilde{d}} d\theta\\
                &= Cn\int_{\Pi}\left|  \int_{\frac{\theta - \pi}{B_T}}^{\frac{\theta + \pi}{B_T}} W(\nu) d\nu - \int_{-\rho}^{\rho} W(\nu) d\nu  \right||\theta|^{-2\Delta} d\theta
            \end{align*}
            Take a look at the inner term, we notice that when $B_T<\pi/\rho$, for $\theta \in (-\pi + \rho B_T, \pi - \rho B_T)$, $\left|  \int_{\frac{\theta - \pi}{B_T}}^{\frac{\theta + \pi}{B_T}} W(\nu) d\nu - \int_{-\rho}^{\rho} W(\nu) d\nu  \right| = 0$. Let's denote $E_T = \Pi \ (-\pi + \rho B_T, \pi - \rho B_T) = (-\pi, -\pi + \rho B_T] \cup [\pi - \rho B_T, \pi]$. Thus, 
            \begin{align*}
                &Cn\int_{\Pi}\left|  \int_{\frac{\theta - \pi}{B_T}}^{\frac{\theta + \pi}{B_T}} W(\nu) d\nu - \int_{-\rho}^{\rho} W(\nu) d\nu  \right||\theta|^{-2\Delta} d\theta \leq Cn\int_{E_T}\left\| W(\nu) \right\|_{L_1}|\theta|^{-2\Delta} d\theta\\
                &\leq Cn\int_{E_T}|\theta|^{-2\Delta} d\theta \leq Cn\int_{-\pi}^{-\pi + \rho B_T}|\theta|^{-2\Delta} d\theta + Cn\int_{\pi - \rho B_T}^{\pi}|\theta|^{-2\Delta} d\theta
            \end{align*}
            For $B_T \leq \frac{\pi}{2\rho}$, $\int_{-\pi}^{-\pi + \rho B_T}|\theta|^{-2\Delta} d\theta$ and  $\int_{\pi - \rho B_T}^{\pi}|\theta|^{-2\Delta} d\theta \leq C\rho B_T$.
        \end{proof}

\subsection{Properties of the periodogram under long memory}        
In this section, we will extend Lemma 1 in \cite{kim_cumulant_long_memory} to the multivariate case. Before that, let's first introduce some definitions borrowed from \cite{kim_cumulant_long_memory}.
Define functions $L_0$ and $L_1$ as 
\begin{align}\label{eq:defn_L}
    L_m(\lambda) = \begin{cases}
        e^{-m}T, & \text{ if } |\lambda| < \frac{e^m}{T}\\
        |\lambda|^{-1}\{ \log (T|\lambda|)  \}^m, & \text{ if } \frac{e^m}{T} <|\lambda| \leq \pi,
    \end{cases} \text{ for } \lambda \in \Pi, m = 0,1.
\end{align}
$L_m(\lambda)$ is symmetric, and periodically extended to $\mathbb{R}$. 
Further denote $H(\lambda) = \sum_{t=1}^T e^{-\iota t \lambda}$, discrete Fourier transform $\mathbf{d}(\lambda) = \sum_{t=1}^T X_{t}exp(-\iota \lambda t)$, $d_i(\lambda)$ is the $i_{th}$ entry of $\mathbf{d}(\lambda)$, $\lambda \in \Pi$. 

Here are some properties of $L_m(\lambda)$ and $H(\lambda)$. 
\begin{lemma}\label{lemma:property_L}
    For every integer $T$, 
    \begin{enumerate}[(i)]
        \item $L_1(\lambda)\leq C \log T L_0(\lambda)$
        \item $|H(\lambda)| \leq C L_0(\lambda)$
        \item $|H(\lambda)| \leq 1/|\lambda|$
        \item $\int_0^{\pi} L_0^2(\lambda) d\lambda \leq CT$
        \item $\int_{0}^{\pi} L_0(b_1 + \lambda) L_0(b_2 - \lambda) d\lambda \leq CL_1(b_1 + b_2)$, $b_1, b_2 \in \mathbb{R}$
        \item $2\pi H(\alpha - \gamma) = \int_{\Pi} H(\alpha - \beta)H(\beta - \gamma)d\beta$
    \end{enumerate}
\end{lemma}

\begin{proof}
    (i) For $|\lambda| \leq 1/T$, $L_1(\lambda) = e^{-1} L_0(\lambda)$; $1/T \leq |\lambda| \leq e/T$, $L_1(\lambda) = e^{-1}T = e^{-1}T|\lambda| L_0(\lambda) \leq L_0(\lambda)$; $e/T \leq |\lambda| \leq \pi$, $L_1(\lambda) = |\lambda|^{-1} \log T|\lambda| \leq L_0(\lambda) \log(T\pi) \leq C \log T L_0(\lambda)$. \\
    (ii) See equation (20) in \cite{nordman_lahiri} or equation (6) in \cite{Dahlhaus}.\\
    (iii) By (ii) and the definition of $L_0(\cdot)$.\\
    (iv),(v) See Lemma 2 in \cite{Dahlhaus} or Lemma 2 in \cite{kim_cumulant_long_memory}.\\
    (vi) It is easy to check, or see Equation 4 in \cite{Dahlhaus}.
\end{proof}

        \begin{lemma}\label{lm:cum_property_long_mem}
  Let $1 \leq k_1 \leq k_2 \leq T$ , and $a_1, a_2, \dots, a_l \in \{\pm \lambda_{k_1}, \pm \lambda_{k_2}\}$, $|a_1| \leq \dots \leq |a_l|$ with $2 \leq l \leq 4$. Suppose $\{X_t\}$ is a real-valued vector stationary linear process and that Assumption \ref{assump:model}, \ref{assump:semiparametric_long_memory}, \ref{assump:pervasive_short_mem}, and \ref{assump:bdd_error_eval} hold. Then, for a generic constant $C > 0$ not depending on $n$ or $1 \leq k_1 \leq k_2 \leq T$, it holds that for $d = \max_i d_i$ 
  \begin{itemize}
    \item[(i)] $\left| \text{Cum} \left( d_{i}(a_1), d_{j}(a_2) \right) \right| \leq C |a_1|^{-2d} \left( |a_2|^{-1} + L_1(a_1 + a_2) \right).$
    
    \item[(ii)] $\left| \text{Cum} \left( d_{q_1}(a_1), \dots, d_{q_4}(a_4) \right) \right| \leq C\left\| Q^{Z}(0,0,0) \right\|_{op} \left\{ |a_4|^{d-1} |a_{3}|^{-d} + T[\log T]^{3} \right\} \prod_{j=1}^{4} |a_j|^{-d}.$ where $\{Z_t = (u_t, \epsilon_t)^{T}; t\in\mathbb{Z}\}$ is second–order white noise, with
nonsingular covariance matrix $\Gamma_{Z}$ and finite fourth–order moments. $Q^{Z}(0,0,0)$ is the fourth-order cumulant of $Z_{t}$. The definition of $\|Q\|_{op}$ follows from $\|Q\|_{2,2,2,2}$ in \cite{lim2006singularvalueseigenvaluestensors}.
    
    \item[(iii)] If, in addition, Assumption \ref{assump: spec_den_specific_form} holds,
    \begin{align}\label{eq:expectation_periodogram_spectral_density}
        \left\|  \text{Cum}(\mathbf{d}(a_1), \mathbf{d}(-a_1)) - 2\pi T\Sigma(a_1)  \right\|_{op} \leq Cn |a_1|^{-2d-1} \log T
    \end{align}
\end{itemize}

\end{lemma}

\begin{proof}
This lemma generalizes Lemma 1 of \cite{kim_cumulant_long_memory} in two respects: first, it treats multivariate processes; second, it is developed within a factor–model framework that includes idiosyncratic noise in addition to the long–memory common component. 

    (i) 
    \begin{align*}
        Cum\left( d_{i}(a_1), d_{j}(a_2)  \right) &= \sum_{t_1, t_2} e^{-\iota (a_1 t_1 + a_2t_2)}  Cum \left( X_{it_1},  X_{jt_2} \right)\\
        &= \sum_{t_1, t_2} e^{-\iota (a_1 t_1 + a_2t_2)} \int_{\Pi} exp(\iota \omega (t_1-t_2))\sigma_{ij}(\omega) d\omega
    \end{align*}
    Plugging in the definition of $H(\lambda)$, 
    \begin{align*}
        Cum\left( d_{i}(a_1), d_{j}(a_2)  \right) &= \int_{\Pi} H(a_1 - \omega) H(a_2 + \omega) \sigma_{ij}(\omega) d\omega\\
        &= \int_{|\omega| \leq |a_1|/2} H(a_1 - \omega) H(a_2 + \omega) \sigma_{ij}(\omega) d\omega \\
        &+ \int_{|\omega| > |a_1|/2} H(a_1 - \omega) H(a_2 + \omega) \sigma_{ij}(\omega) d\omega
    \end{align*}
Notice that for $|\omega| \leq |a_1|/2$, $\frac{1}{|a_1 - \omega|} \leq \frac{2}{|a_1|}$, $\frac{1}{|a_2 + \omega|} \leq \frac{2}{|a_2|}$, and $\sigma_{ii}(\omega) \leq |\omega|^{-2d} g_{ii}(\theta) + \sigma^{\xi}_{ii}(\omega) \leq C|\omega|^{-2d} $, $|\sigma_{ij}(\omega)| \leq \sqrt{\sigma_{ii}(\omega)\sigma_{jj}(\omega)} \leq C|\omega|^{-2d}$. 
    Then applying Lemma \ref{lemma:property_L} (iii) yields
    \begin{align*}
        &\left|\int_{|\omega| \leq |a_1|/2} H(a_1 - \omega) H(a_2 + \omega) \sigma_{ij}(\omega) d\omega \right| \leq C \int_{|\omega| \leq |a_1|/2}  \frac{1}{|a_1 - \omega|}  \frac{1}{|a_2 + \omega|} |\omega|^{-2d} d\omega\\
        &\leq C |a_1|^{-1} |a_2|^{-1} \int_{|\omega| \leq |a_1|/2} |\omega|^{-2d} d\omega \leq C |a_1|^{-2d} |a_2|^{-1} 
    \end{align*}

On the other hand, applying Lemma \ref{lemma:property_L} (ii) 
\begin{align*}
    &\left|\int_{|\omega| > |a_1|/2} H(a_1 - \omega) H(a_2 + \omega) \sigma_{ij}(\omega) d\omega \right|\\
    &\leq C \int_{|\omega| > |a_1|/2} |\omega|^{-2d} \left| H(a_1 - \omega) \right| \left| H(a_2 + \omega) \right| d\omega \\
    &\leq C |a_1|^{-2d} \int_{|\omega| > |a_1|/2} L_0(a_1 - \omega)  L_0(a_2 + \omega) d\omega\\
    &\leq C |a_1|^{-2d} \int_{\Pi}L_0(a_1 - \omega)  L_0(a_2 + \omega) d\omega\\
    &\leq C |a_1|^{-2d} L_1(a_1 + a_2)
\end{align*}

Thus,
\begin{align*}
    \left|  Cum\left( d_{i}(a_1), d_{j}(a_2)  \right)  \right| \leq C|a_1|^{-2d}\left(  |a_2|^{-1} + L_1(a_1 + a_2) \right)
\end{align*}

(ii) 
\begin{align}\label{eq:relation_cum_dft_X}
    Cum\left( d_{q_1}(a_1), \dots, d_{q_4}(a_4) \right) &= \sum_{t_1, t_2, t_3, t_4} e^{-\iota \sum_{l=1}^4 a_l t_l} Cum \left( X_{q_1 t_1}, X_{q_2 t_2}, X_{q_3 t_3},X_{q_4 t_4}\right)
\end{align}
Let $Q_{q_1, q_2, q_3, q_4}^X(t_2-t_1, t_3-t_1, t_4-t_1)$ 
be the joint fourth cumulant of $X_{q_1 t_1}, X_{q_2 t_2}, X_{q_3 t_3},X_{q_4 t_4}$, and let $\widetilde{Q}_{q_1, q_2, q_3, q_4}^X(\omega_1, \omega_2, \omega_3)$ be the fourth-order spectral density such that 
\begin{align}\label{eq:defn_fourth_order_cum}
    \widetilde{Q}_{q_1, q_2, q_3, q_4}^X(\omega_1, \omega_2, \omega_3) = (2\pi)^{-3} \sum_{u_1, u_2, u_3} Q_{q_1, q_2, q_3, q_4}^X(u_1, u_2, u_3)e^{ -\iota \sum_{i=1}^3 u_i \lambda_i }
\end{align}

By \eqref{eq:xi_representation} in Assumption \ref{assump:model}, we may write the model as 
\begin{align*}
    X_t = \chi_t + \xi_t = B(L)u_t + B_{\xi}(L)\epsilon_t = Y(L)Z_t = \sum_{m\in \mathbb{Z}}\mathbf{Y}_m Z_{t-m}
\end{align*}
where $Y(L) = (B(L), B_{\xi}(L)) \in \mathbb{C}^{n \times (q+n)}$, $\mathbf{Y}_m \in \mathbb{C}^{n \times (q+n)}$,   $\{Z_t = (u_t^\top, \epsilon_t^\top)^{\top}; t\in\mathbb{Z}\}$ is second–order white noise, with
nonsingular covariance matrix $\Gamma_{Z}$ and finite fourth–order moments. Then the spectral density of $X_t$ can be written as 
\begin{align*}
    \Sigma(\theta) = \frac{1}{2\pi} \mathbf{Y}(\theta) \Gamma_{Z} \mathbf{Y}^{*}(\theta)
\end{align*}
where $\mathbf{Y}(\theta) = \sum_m \mathbf{Y}_m e^{-\iota \theta m} = \{y_{ij}(\theta)\}_{i = 1,\dots, n; j = 1, \dots, n+q}$, $y_{ij}(\theta) = \sum_m \mathbf{Y}_{m,i,j} e^{-\iota \theta m}$.

Let $\mathbf{y}_i(\theta) = (y_{i1}(\theta), \dots, y_{i, n+q}(\theta)) = (b_{i1}(\theta), \dots, b_{iq}(\theta), b_{\xi,i1}(\theta) ,\dots,   b_{\xi,in}(\theta))$, $i = 1,\dots, n$. Thus, according to Assumption \ref{assump: spec_den_specific_form},
\begin{align*}
    \left\|\mathbf{y}_i(\theta)\right\|^2  = \sum_{j = 1}^q |b_{ij}(\theta)|^2 + \sum_{l = 1}^n |b_{\xi,il}(\theta)|^2 \leq |\theta|^{-2d} \sum_{j = 1}^q |g_{ij}(\theta)|^2 + \sum_{l = 1}^n |b_{\xi, il}(\theta)|^2
\end{align*}

By Assumption \ref{assump:bdd_error_eval}, we have $\left\|B_{\xi}(\theta)\right\|_{op} \leq C_{\xi}$, where $C_{\xi}$ does not depend on $n$ and $T$. Let $B_{\xi,i}(\theta)$ to be the $i_{th}$ row of $B_{\xi}(\theta)$, then $\left\| B_{\xi,i}(\theta) \right\|_2^2 = e_i^\top B_{\xi}(\theta)B_{\xi}^*(\theta)e_i \leq \lambda_{1}(B_{\xi}(\theta)B_{\xi}^*(\theta)) =\left\|B_{\xi}(\theta)\right\|_{op}^2 \leq C^2_{\xi} $. That is, for $i  = 1,\dots,n$, $\sum_{l = 1}^n |b_{\xi,il}(\theta)|^2 \leq C^2_{\xi}$.

By Assumption \ref{assump:row_boundness_G}, we have $\left\|\mathbf{y}_i(\theta)\right\|^2 \leq C_1|\theta|^{-2d} + C_{\xi}^2 \leq C|\theta|^{-2d}$, where $C$ is independent of $n, T$. Thus, by Parseval's theorem, we have $\sum_{m \in \mathbb{Z}}|Y_{m,i,j}|^2 = \frac{1}{2\pi} \int_{\Pi} |y_{ij}(\theta)|^2 d\theta \leq \int_{\Pi}|\theta|^{-2d} d\theta < \infty$.

From \cite{hosoya} Lemma A2.1, specifically the last equation of the proof, we have 
\begin{align}\label{eq:hosoya_cum_X_Z}
    Q_{q_1q_2q_3q_4}^{X}(u_1, u_2, u_3) = &\sum_{\alpha_1,\alpha_2,\alpha_3,\alpha_4=1}^{n+q} \iiint \left\{   \prod_{j=1}^3 exp(\iota u_j v_j)\right\} y_{q_1 \alpha_1}(v_1 + v_2 + v_3)  \\\nonumber
    &\times y_{q_2 \alpha_2}(-v_1)y_{q_3 \alpha_3}(-v_2)y_{q_4 \alpha_4}(-v_3) \widetilde{Q}_{\alpha_1, \alpha_2, \alpha_3, \alpha_4}^Z (v_1, v_2, v_3) dv_1 dv_2 dv_3\nonumber
\end{align}
if $\sum_{m =0}^{\infty}|Y_{m, i,j }|^2 < \infty$ and if $\sum_{j_1, j_2, j_3  = -\infty}^{\infty} |Q_{\alpha_1, \alpha_2, \alpha_3, \alpha_4}^Z (j_1, j_2, j_3)| < \infty$.

Importantly, although \cite{hosoya} formulates $Y(L)$ as an one-sided filter, the argument essentially only requires the square summability over all $\mathbb{Z}$, that is, $\sum_{m \in \mathbb{Z}}|Y_{m,i,j}|^2 < \infty$. Thus, the result can be extended to the two-sided filter trivially. Thus, (\ref{eq:hosoya_cum_X_Z}) holds in our setting regardless of whether $Y(L)$ is one-sided.

Plugging (\ref{eq:hosoya_cum_X_Z}) back into (\ref{eq:relation_cum_dft_X}), 
\begin{align*}
    &Cum\left( d_{q_1}(a_1), \dots, d_{q_4}(a_4) \right) \\
    &= \iiint \prod_{i=1}^3 H(a_i-v_i)H(\sum_{j=1}^3 v_j + a_4) \sum_{\alpha_1,\alpha_2,\alpha_3,\alpha_4=1}^{n+q} y_{q_1 \alpha_1}(v_1 + v_2 + v_3)\\
    &\times y_{q_2 \alpha_2}(-v_1)y_{q_3 \alpha_3}(-v_2)y_{q_4 \alpha_4}(-v_3) \widetilde{Q}_{\alpha_1,\alpha_2,\alpha_3,\alpha_4}^{Z}(v_1 + v_2 + v_3, v_2, v_3) dv_1 dv_2 dv_3
\end{align*}
Notice that $Z_t$ is a stationary white noise, then 
\begin{align*}
    Q^Z_{\alpha_1\alpha_2\alpha_3\alpha_4}(u_1, u_2, u_3) = Q^Z_{\alpha_1\alpha_2\alpha_3\alpha_4}(0,0,0) \mathbb{1}(u_1=u_2=u_3=0)
\end{align*}
and analogous to (\ref{eq:defn_fourth_order_cum}), 
\begin{align*}
    \widetilde{Q}^Z_{\alpha_1\alpha_2\alpha_3\alpha_4}(v_1, v_2, v_3) &=(2\pi)^{-3} \sum_{u_1, u_2, u_3}^{\infty} Q_{\alpha_1\alpha_2\alpha_3\alpha_4}^Z (u_1, u_2, u_3) e^{-\iota \sum_{j=1}^3 u_j v_j}  \\ 
    &=(2\pi)^{-3} Q^Z_{\alpha_1\alpha_2\alpha_3\alpha_4}(0,0,0)
\end{align*}
Let $\omega_i = a_i - v_i$, we have
\begin{align}
    &Cum\left( d_{q_1}(a_1), \dots, d_{q_4}(a_4) \right) \\\nonumber
    &=(2\pi)^{-3} \iiint H\left(\sum_{i=1}^4 a_i - \sum_{i=1}^3 \omega_i\right)\prod_{i=1}^3 H(w_i) M(\omega_1,\omega_2,\omega_3)d\omega_1d\omega_2d\omega_3\\\nonumber
    &M(\omega_1,\omega_2,\omega_3) =  \sum_{\alpha_1,\alpha_2,\alpha_3,\alpha_4=1}^{n+q} Q^Z_{\alpha_1\alpha_2\alpha_3\alpha_4}(0,0,0)y_{q_1 \alpha_1}\left(  \sum_{i=1}^3 (a_i-\omega_i)\right) \prod_{i=1}^3 y_{q_{i+1}\alpha_{i+1}}(\omega_i - a_i)\nonumber
\end{align}
This is the multivariate extension to \cite{yajima} Equation (6) in the proof of Lemma 3.

Notice that 
\begin{align*}
&\left|M(\omega_1,\omega_2,\omega_3)\right| \\
    &=\left|\sum_{\alpha_1,\alpha_2,\alpha_3,\alpha_4=1}^{n+q} Q^Z_{\alpha_1\alpha_2\alpha_3\alpha_4}(0,0,0)y_{q_1 \alpha_1}\left(  \sum_{i=1}^3 (a_i-\omega_i)\right) \prod_{i=1}^3 y_{q_{i+1}\alpha_{i+1}}(\omega_i - a_i)\right|\\
    &=\left|\bigg\langle Q^{Z}(0,0,0),\, y_{q_1}\left(\sum_{i=1}^3 (a_i-\omega_i)\right)\otimes y_{q_2}(\omega_1-a_1)\otimes y_{q_3}(\omega_2-a_2)\otimes y_{q_4}(\omega_3-a_3)\bigg\rangle\right|\\
    &\leq \left\| Q^{Z}(0,0,0) \right\|_{op} \left\|y_{q_1}\left(\sum_{i=1}^3 (a_i-\omega_i)\right)   \right\|_2 \left\| y_{q_2}(\omega_1-a_1) \right\|_2\left\|  y_{q_3}(\omega_2-a_2) \right\|_2  \left\|  y_{q_4}(\omega_3-a_3) \right\|_2
\end{align*}
where $y_{q_j}\left(u\right) = \left(y_{q_j,1}\left(u\right),\dots,  y_{q_j,n+q}\left(u\right)   \right)$. 


The second equality follows from the definition of the inner product for two tensors of the same dimension, for example, see page 3 in \cite{defn_inner_prod_tensor}, or Section 2 in \cite{tensor}.
The third inequality follows from \[
\begin{aligned}
\left|
\left\langle Q, u_1\otimes u_2\otimes u_3\otimes u_4 \right\rangle
\right|
&=
\prod_{j=1}^4 \|u_j\|_2
\left|
\left\langle Q, \frac{u_1}{\|u_1\|_2}\otimes \frac{u_2}{\|u_2\|_2}\otimes \frac{u_3}{\|u_3\|_2}\otimes \frac{u_4}{\|u_4\|_2} \right\rangle
\right| \\
&\leq \prod_{j=1}^4 \|u_j\|_2 \sup_{\|v_1\|_2=\cdots=\|v_4\|_2=1}
\left|
\left\langle
Q,\,
v_1\otimes v_2\otimes v_3\otimes v_4
\right\rangle
\right|. \\
&\leq 
\prod_{j=1}^4 \|u_j\|_2 \, \|Q\|_{op}.
\end{aligned}
\]
where the definition of $\|Q\|_{op}$ follows from $\|Q\|_{2,2,2,2}$ in \cite{lim2006singularvalueseigenvaluestensors}. 

By Assumption \ref{assump:model} (i), and $u_t \in \mathbb{R}^{q}$, $\left\|Q^{u}(0,0,0)\right\|_2 = \mathcal{O}(1)$. Meanwhile, for any $v_j = (a^{\top}_j, b^{\top}_j)^{\top}$, $\left\|v_j\right\|_2 = 1$,  $j = 1,2,3,4$, since 
\begin{align*}
    &cum\left(v_1^{\top}Z_{\alpha_1,t}, v_2^{\top}Z_{\alpha_2,t}, v_3^{\top}Z_{\alpha_3,t}, v_4^{\top}Z_{\alpha_4,t}\right)\\ 
    &= cum\left(a_1^{\top}u_{\alpha_1,t}, a_2^{\top}u_{\alpha_2,t}, a_3^{\top}u_{\alpha_3,t}, a_4^{\top}u_{\alpha_4,t}\right) + cum\left(b_1^{\top}{\epsilon}_{\alpha_1,t}, b_2^{\top}{\epsilon}_{\alpha_2,t}, b_3^{\top}{\epsilon}_{\alpha_3,t}, b_4^{\top}{\epsilon}_{\alpha_4,t}\right)\\
    &\leq \left\langle
Q^u(0,0,0),
a_1\otimes a_2\otimes a_3\otimes a_4
\right\rangle + \left\langle
Q^{\epsilon}(0,0,0),
b_1\otimes b_2\otimes b_3\otimes b_4
\right\rangle \\
&\leq \left\|Q^u(0,0,0)\right\|_{op}\prod_{j=1}^4 \|a_j\|_2 + \left\|Q^{\epsilon}(0,0,0)\right\|_{op}\prod_{j=1}^4 \|b_j\|_2\\
&\leq \left\|Q^u(0,0,0)\right\|_{op} + \left\|Q^{\epsilon}(0,0,0)\right\|_{op}
\end{align*} where the first inequality is by multilinearity of cumulants, and the definition of the inner product for two tensors of the same dimension. Taking the supremum over all unit vectors $v_j$ yields
\begin{align}\label{eq:upnd_Z_Q}
    \left\|Q^Z(0,0,0)\right\|_{op} \leq \left\|Q^u(0,0,0)\right\|_{op} + \left\|Q^{\epsilon}(0,0,0)\right\|_{op}
\end{align}

By \eqref{eq:eps_Q_finite}, \eqref{eq:upnd_Z_Q} implies $\left\|Q^Z(0,0,0)\right\|_{op} = \mathcal{O}(1)$. Thus,
\begin{align*}
    \left| Cum\left( d_{q_1}(a_1), \dots, d_{q_4}(a_4) \right) \right| &\leq C \iiint_{\Pi^3} \left| H\left(\sum_{i=1}^4 a_i - \sum_{i=1}^3 \omega_i\right) \right|\left\|y_{q_1}\left(\sum_{i=1}^3 (a_i-\omega_i)\right)   \right\|_2\\
    &\times \prod_{i=1}^3 \left[ \left|H(\omega_i)\right|\left\|y_{q_{i+1}}(\omega_i - a_i)\right\|_2   \right] d\omega_1d\omega_2d\omega_3
\end{align*}
Using $||y_i (\theta)||_2^2 \leq C |\theta|^{-2d}$, the rest of the proof is essentially the same as Lemma 3 in \cite{nordman_lahiri}. The only thing we found is the inequality below can be tightened in the way that for $B_j^c = \{\omega_j: |\omega_j - a_j|\leq |a_j|/8\}$, $B = \cap_{j=1}^3 B_j^c$,
\begin{align}\label{eq:to_prove_4th_order_in_B}
    &\int_{B_3^c} \left\|y_{q_1}\left(\sum_{i=1}^3 (a_i-\omega_i)\right)   \right\|_2 \left\|y_{q_{4}}(\omega_3 - a_3)\right\|_2 d \omega_3, \text{ while }\omega_2 \in B_2^c, \omega_3 \in B_3^c\\\nonumber
    &\leq \left( \int_{B_3^c}\left\|y_{q_1}\left(\sum_{i=1}^3 (a_i-\omega_i)\right)   \right\|_2^2 d\omega_3 \right)^{\frac{1}{2}} \left( \int_{B_3^c}\left\|y_{q_{4}}(\omega_3 - a_3)\right\|_2^2 d\omega_3 \right)^{\frac{1}{2}}\\\nonumber
    &\leq C\left( \int_{B_3^c}\left|\sum_{i=1}^3 (a_i-\omega_i)   \right|^{-2d} d\omega_3 \right)^{\frac{1}{2}} \left( \int_{B_3^c}\left|\omega_3 - a_3\right|^{-2d} d\omega_3 \right)^{\frac{1}{2}}\\\nonumber
\end{align}

For $\omega = (\omega_1, \omega_2, \omega_3) \in B$, let $v_i = a_i-\omega_i$, $i = 1,2,3$, $u = v_1 + v_2 + v_3$, then 
\begin{align*}
    \int_{B_3^c}\left|\sum_{i=1}^3 (a_i-\omega_i)   \right|^{-2d} d\omega_3 &= \int_{|v_3| \leq |a_3|/8} |v_1 + v_2 + v_3|^{-2d} dv_3\\
    &= \int_{v_1 +v_2 - |a_3|/8}^{v_1 +v_2 + |a_3|/8} |u|^{-2d} du\\
    &\leq C\left( (v_1 + v_2 +|a_3|/8)^{1-2d} - \max\left(v_1 + v_2 -|a_3|/8, 0\right)^{1-2d}   \right)\\
    &\leq C(|v_1| + |v_2| +|a_3|/8)^{1-2d}\\
    &\leq C(|a_1| + |a_2| +|a_3|)^{1-2d} \leq C|a_3|^{1-2d}
\end{align*}
since $|a_1|\leq |a_2|\leq |a_3|$.
Meanwhile, $\int_{B_3^c}\left|\omega_3 - a_3\right|^{-2d} d\omega_3 \leq C|a_3|^{1-2d}$. Thus, (\ref{eq:to_prove_4th_order_in_B}) $\leq C|a_3|^{1-2d}$. For small $|a_3|$ such a bound is tighter than $|a_3|^{1/2 - d}$, which is what \cite{nordman_lahiri} obtained in their Lemma 3. 
The rest of the proof remains the same, and in the end we achieve the upper bound of $\left| \text{Cum} \left( d_{q_1}(a_1), \dots, d_{q_4}(a_4) \right) \right| \leq C \left\{ |a_4|^{d-1} |a_{3}|^{-d} + T[\log T]^{3} \right\} \prod_{j=1}^{4} |a_j|^{-d}$.\\

(iii) For simplicity, write $a_1 = \lambda_j$ for some fixed $1\leq j \leq \lfloor T/2 \rfloor$. Remind that 
\begin{align*}
    &Cum\left( \mathbf{d}(a_1), \mathbf{d}(-a_1)  \right) = \int_{\Pi} H(a_1 - \omega) H(-a_1 + \omega) \Sigma(\omega) d\omega\\
    &2\pi T = 2\pi H(0) = \int_{\Pi} H(\alpha - \omega)H(\omega - \alpha)d\omega
\end{align*}
Let the remainder be $\widetilde{R}(a_1)$, then
\begin{align*}
    \widetilde{R}(a_1)&=Cum\left( \mathbf{d}(a_1), \mathbf{d}(-a_1)  \right) - 2\pi T \Sigma(a_1)\\
    &= \int_{\Pi} H(a_1 - \omega) H(-a_1 + \omega) \Sigma(\omega) d\omega - \int_{\Pi} H(a_1 - \omega)H(\omega - a_1)d\omega \Sigma(a_1)\\
    &= \int_{\Pi} H(a_1 - \omega) H(-a_1 + \omega) \left(\Sigma(\omega) - \Sigma(a_1)\right) d\omega\\
    &= \int_{\Pi} H(\omega) H(-\omega) \left(\Sigma(a_1 - \omega) - \Sigma(a_1)\right) d\omega\\
    &= \int_{\Pi} H(\omega) H(-\omega) \left(\Sigma_{\chi}(a_1 - \omega) - \Sigma_{\chi}(a_1)\right) d\omega\\
    &+ \int_{\Pi} H(\omega) H(-\omega) \left(\Sigma_{\xi}(a_1 - \omega) - \Sigma_{\xi}(a_1)\right) d\omega\\
    &\coloneqq \widetilde{R}_{\chi}(a_1) + \widetilde{R}_{\xi}(a_1)
\end{align*}
Notice that
\begin{align}\label{eq:R_xi_a1}
    \left\|\widetilde{R}_{\xi}(a_1)\right\| &\leq \int_{\Pi} L_0^2(\omega) \left\|\Sigma_{\xi}(a_1 - \omega) - \Sigma_{\xi}(a_1)\right\| d\omega\\\nonumber
    &\leq \int_{\Pi} L_0^2(\omega) |\omega| d\omega\\\nonumber
    &\leq C \int_{0}^{\frac{1}{T}} L_0^2(\omega)\omega d\omega + \int_{\frac{1}{T}}^{\pi} \omega^{-2}\omega d\omega\\\nonumber
    &\leq C\frac{1}{T} \int_{\Pi} L_0^2(\omega)d\omega + \log T \leq C \log T\nonumber
\end{align}
The last inequality is from Lemma \ref{lemma:property_L}.

Now let's deal with $\widetilde{R}_{\chi}(a_1)$. 
\begin{enumerate}
    \item When $d_j$ differs across shocks but remains the same across rows. Let $d = \max_{j = 1,\dots, q} d_j$
    
    According to Assumption \ref{assump:pervasive_short_mem}, and Lemma \ref{lm:increment_bound_for_Sigma}, we have
\begin{align*}
    &\left\| \Sigma_{\chi}(a_1 - \omega) - \Sigma_{\chi}(a_1)  \right\| \\
    &\leq n|\omega| \left( |a_1 - \omega|^{-2d} +  |a_1|^{-2d} \right) + n |a_1|^{-2d} \left| \left| 1-\frac{\omega}{a_1} \right|^{-2d} - 1  \right|
\end{align*}
The last inequality is because for a function $\phi_u(t) = \left|u^{-2t} - 1\right|, t\geq 0$, it is easy to check that for all $u\geq 0$, $\phi_u(t)$ is increasing in $t$, thus, $\max_{j=1,\dots, q} \phi_u(d_j) \leq \phi_u(\max_{j=1,\dots, q} d_j)$. Then same as in the proof of Lemma 1 (iii) in \cite{kim_cumulant_long_memory}, we have
\begin{align*}
    n\int_{\Pi} L_0^2(\omega) |a_1|^{-2d} \left| \left| 1-\frac{\omega}{a_1} \right|^{-2d} - 1  \right| d\omega \leq n |a_1|^{-2d-1} \log T.
\end{align*}
Meanwhile, similar to (\ref{eq:R_xi_a1}),
\begin{align*}
    n\int_{\Pi} L_0^2(\omega) |\omega|  d\omega |a_1|^{-2d} \leq n|a_1|^{-2d} \log T
\end{align*}
As for the rest of the term, we split it into several segments, for $|\omega| < \frac{\pi}{T}$, $|a_1 - \omega| \asymp |a_1|$, $L_0(\omega) \leq T$,
\begin{align*}
    n\int_{-\frac{\pi}{T}}^{\frac{\pi}{T}} L_0^2(\omega) |\omega| |a_1 - \omega|^{-2d}  d\omega \leq Cn T^2 |a_1|^{-2d} \int_{0}^{\frac{\pi}{T}} \omega d\omega \leq C n|a_1|^{-2d}.
\end{align*}
For $\frac{\pi}{T}<|\omega|<\frac{a_1}{2}$, $|a_1 - \omega| \asymp |a_1|$, $L_0(\omega) \leq |\omega|^{-1}$,
\begin{align*}
    n\int_{\frac{\pi}{T}}^{a_1/2} L_0^2(\omega) |\omega| |a_1 - \omega|^{-2d}  d\omega \leq n |a_1|^{-2d} \int_{\frac{\pi}{T}}^{a_1/2} \omega^{-1} d\omega \leq Cn |a_1|^{-2d} \log T.
\end{align*}
For $\frac{1}{2}a_1<|\omega|<\frac{3}{2}a_1$, $L_0(\omega) \leq |\omega|^{-1} \leq C |a_1|^{-1}$,
\begin{align*}
    n\int_{a_1/2}^{3a_1/2} L_0^2(\omega) |\omega| |a_1 - \omega|^{-2d}  d\omega &\leq n |a_1|^{-1} \int_{a_1/2}^{3a_1/2} |a_1 - \omega|^{-2d} d\omega\\ 
    &\leq Cn |a_1|^{-1} \int_{0}^{a_1/2} |u|^{-2d} du \leq Cn |a_1|^{-2d} .
\end{align*}
For $\frac{3}{2}a_1<|\omega|<\pi$, $|a_1 - \omega|^{-2d} \leq C|\omega|^{-2d}$, 
\begin{align*}
    n\int_{3a_1/2}^{\pi} L_0^2(\omega) |\omega| |a_1 - \omega|^{-2d}  d\omega \leq C n \int_{3a_1/2}^{\pi} |\omega|^{-2} |\omega| |\omega|^{-2d}  d\omega \leq C n |a_1|^{-2d}
\end{align*}
Altogether, 
\begin{align}\label{eq:eqn_L_w_a1_w_2d}
    n\int_\Pi L_0^2(\omega) |\omega| |a_1 - \omega|^{-2d}  d\omega \leq Cn|a_1|^{-2d}\log T.
\end{align}
Thus,
\begin{align*}
        \left\|\widetilde{R}_{\chi}(a_1)\right\| \leq n\int_{\Pi} L_0^2(\omega)\left\| \Sigma_{\chi}(a_1 - \omega) - \Sigma_{\chi}(a_1)  \right\| d \omega \leq Cn|a_1|^{-2d-1}\log T
    \end{align*}

\item When $d_j$ differs across rows but remains the same across shocks. Let $d = \max_{j = 1,\dots, n} d_j$.
 According to Assumption \ref{assump:pervasive_short_mem}, and Lemma \ref{lm:increment_bound_for_Sigma}, we have
 \begin{align*}
     \left\| \Sigma_{\chi}(a_1 - \omega) - \Sigma_{\chi}(a_1)  \right\| &\leq Cn|\omega|\left(|a_1 - \omega|^{-2d} + |a_1|^{-2d}\right) + Cn |a_1|^{-2d} \left| \left| 1-\frac{\omega}{a_1} \right|^{-2d} - 1  \right|\\
     &+ n\left(|a_1 - \omega|^{-2d} + |a_1|^{-2d}\right) \sin\frac{\pi d}{2} \mathbb{1}\{a_1(a_1 - \omega) < 0\}\\
     &\leq C_1n\left(|a_1 - \omega|^{-2d} + |a_1|^{-2d}\right) |\omega| \\
     &+ C_2n\left(|a_1 - \omega|^{-2d}+ |a_1|^{-2d}\right)\mathbb{1}\{a_1(a_1 - \omega) < 0\} \\
     &+ C_3n |a_1|^{-2d} \left| \left| 1-\frac{\omega}{a_1} \right|^{-2d} - 1  \right|
 \end{align*}
Again, $n\int_{\Pi} L_0^2(\omega) |a_1|^{-2d} \left| \left| 1-\frac{\omega}{a_1} \right|^{-2d} - 1  \right| d\omega \leq n |a_1|^{-2d-1} \log T$. By (\ref{eq:eqn_L_w_a1_w_2d}), the first term is again $\leq Cn|a_1|^{-2d}\log T$.

For the second term, $\mathbb{1}\{a_1(a_1 - \omega)<0\} = \mathbb{1}\{ |\omega| > |a_1|\}$, then
\begin{align*}
    &n \int_{\Pi} L_0^2(\omega) \left(|a_1 - \omega|^{-2d} + |a_1|^{-2d}\right)\mathbb{1}\{a_1(a_1 - \omega) < 0\} d\omega\\
    &\leq n \int_{|\omega| > |a_1|} L_0^2(\omega) |a_1 - \omega|^{-2d} d\omega + n |a_1|^{-2d} \int_{|\omega| > |a_1|} L_0^2(\omega) d\omega
\end{align*}
From the definition of $L_0(\cdot)$, we have $n |a_1|^{-2d} \int_{|\omega| > |a_1|} L_0^2(\omega) d\omega \leq n |a_1|^{-2d} \int_{|\omega| > |a_1|} |\omega|^{-2} d\omega \leq Cn|a_1|^{-2d-1}$.

Meanwhile, when $|\omega| \geq |a_1|$, $L_0(\omega) \leq |\omega|^{-1}$, 
\begin{align*}
    n \int_{|\omega| \geq |a_1|} L_0^2(\omega) |a_1 - \omega|^{-2d} d\omega &\leq n \int_{|\omega| \geq |a_1|} |\omega|^{-2} |a_1 - \omega|^{-2d} d\omega =:n(I_1+I_2)
\end{align*}
and 
\begin{align*}
I_1 &=\int_{-\pi}^{-|a_1|}|\omega|^{-2}|a_1-\omega|^{-2d}\,d\omega
\lesssim |a_1|^{-2d}\int_{-\pi}^{-|a_1|}|\omega|^{-2}\,d\omega
\lesssim |a_1|^{-2d-1},\\
I_2 &= \int_{|a_1|}^{\pi}\omega^{-2}|\omega-a_1|^{-2d}\,d\omega =\int_0^{\pi-|a_1|}(|a_1|+u)^{-2}u^{-2d}\,du\\
&\lesssim |a_1|^{-2}\int_0^{|a_1|}u^{-2d}\,du+\int_{|a_1|}^{\pi - |a_1|}u^{-2d-2}\,du\\
&\leq C|a_1|^{-2d-1}
\end{align*}

Thus,
\begin{align*}
        \left\|\widetilde{R}_{\chi}(a_1)\right\| &\leq n\int_{\Pi} L_0^2(\omega)\left\| \Sigma_{\chi}(a_1 - \omega) - \Sigma_{\chi}(a_1)  \right\| d \omega \\
        &\leq C_1n |a_1|^{-2d} \log T + C_2n|a_1|^{-2d-1} + C_3n|a_1|^{-2d-1}\log T\\
        &\leq Cn|a_1|^{-2d-1}\log T
    \end{align*}
\end{enumerate}
    Therefore, $\left\|\widetilde{R}(a_1)\right\| \leq Cn|a_1|^{-2d-1}\log T$.
\end{proof}


\bibliographystyle{johd}
\bibliography{bib}


\end{document}